\documentclass [11pt,oneside,a4paper,mathscr]{amsart}

\usepackage{amscd,amsmath,amssymb,euscript}
\usepackage[frame,cmtip,curve,arrow,matrix,line,graph]{xy}
\usepackage{tikz-cd}
\usepackage{adjustbox}
\usetikzlibrary{matrix}

\usepackage{hyperref}

\title{K-theoretic Hall algebras for quivers with potential}
\author{Tudor P\u adurariu}
\address{Department of Mathematics, Massachusetts Institute of Techonology, 
182 Memorial Drive, Cambridge, MA 02139}
\email{tpad@mit.edu}

\date{}

\newtheorem{thm}{Theorem}[section]
\newtheorem{cor}[thm]{Corollary}
\newtheorem{prop}[thm]{Proposition}

\theoremstyle{definition}
\newtheorem{defn}[thm]{Definition}

\newtheorem{thm*}[thm]{Theorem$^*$}

\newtheorem{conjecture}[thm]{Conjecture}

\newcommand{\comment}[1]{}

\renewcommand{\leq}{\leqslant}
\renewcommand{\geq}{\geqslant}

\newcommand{\NN}{\overline{\mathbb{N}}}
\newcommand{\A}{\mathcal A}

\newcommand{\OO}{\mathcal O}

\newcommand{\WW}{\overline{\mathbb{W}}}

\newcommand{\X}{\mathcal{X}}
\newcommand{\Y}{\mathcal{Y}}

\newcommand{\MM}{\mathbb{M}}
\newcommand{\oM}{\overline{\mathbb{M}}}

\newcommand{\C}{\mathbb{C}}

\newcommand{\ee}{\underline{e}}
\newcommand{\dd}{\underline{d}}

\begin{document}
\maketitle
\setcounter{tocdepth}{1}
\tableofcontents

\section{Introduction}

 Given a quiver with potential $(Q,W)$, Kontsevich-Soibelman constructed a Hall algebra on the cohomology of the stack of representations of $(Q,W)$. As shown by Davison-Meinhardt, this algebra comes with a filtration whose associated graded algebra is supercommutative. A special case of this construction is related to work of Nakajima, Varagnolo, Maulik-Okounkov etc. about geometric constructions of Yangians and their representations; indeed, given a quiver $Q$, there exists an associated pair $(\widetilde{Q},\widetilde{W})$ for which the CoHA is conjecturally the positive half of the Yangian $Y_{\text{MO}}(\mathfrak{g}_Q)$. 

 The goal of this article is to extend these ideas to K-theory.  More precisely, we construct a K-theoretic Hall algebra using category of singularities, define a filtration whose associated graded algebra is a deformation of a symmetric algebra, and compare the $\text{KHA}$ and the $\text{CoHA}$ using the Chern character. As before, we expect our construction for the special class of quivers $(\widetilde{Q},\widetilde{W})$ to recover the positive part of quantum affine algebra $U_q(\widehat{\mathfrak{g}_Q})$ defined by Okounkov-Smirnov, but for general pairs $(Q,W)$ we expect new phenomena.

\subsection{Quivers with potential.}
The Donaldson-Thomas invariants of a Calabi-Yau $3$-fold $Y$ are virtual counts of curves on $Y,$ and they can be defined using the geometry of 
the moduli stack of sheaves $\mathcal{M}(Y,\beta)$ of a given class $\beta\in H^{\text{even}}(Y,\mathbb{Z})$.
One can define DT invariants for other categories, such as the categories $\text{Rep}\,(Q,W)$ of representations of quivers with potential $(Q,W)$. These invariants are defined using the vanishing cycle sheaf $\varphi_{\text{Tr}\,(W)}\mathbb{Q}$ of the regular function 
$$\text{Tr}\,(W):\X(d)\to\mathbb{A}^1$$
on the stack $\X(d)$ of representation of $Q$ of a given dimension $d$.

For any Calabi-Yau $3$-fold $Y,$ the Hilbert-Chow morphism from the stack $\mathcal{M}(Y,\beta)$ to the coarse space $M(Y,\beta)$
$$\pi:\mathcal{M}(Y,\beta)\to M(Y,\beta)$$
can be locally described using the Hilbert-Chow morphism from the stack $\X(d)$ to the coarse space $X(d)$ of representations of $(Q,W)$:
$$\pi:\text{crit}\,(\text{Tr}\,(W))\subset \X(d) \to
\text{crit}\,(\text{Tr}\,(W))\subset X(d),$$ see \cite{j}, \cite{t}.
In some particular cases this description is global, such as for $Y=\C^3$ and $\beta\in\mathbb{N}$, where $Q_3$ is the quiver with one vertex and three loops $x,y,$ and $z$, and $W=xyz-xzy$. 
It is thus worthwhile to study the DT theory of these categories $\text{Rep}\,(Q,W)$, and try to generalize the constructions and results to the general case of a Calabi-Yau $3$-fold.
\\

\subsection{Cohomological Hall algebras.} 
Hall algebras originate in work of Steinitz \cite{ste}, Hall \cite{hall} who studied an algebra with basis isomorphism classes of abelian $p$-groups and multiplication encoding the possible extensions of two such groups. In general, a Hall algebra is an algebra with basis isomorphism classes of objects in a given category and with multiplication encoding extension of such objects.
Ringel \cite{rin} showed that the Hall algebra of the abelian category of representations of a quiver $Q$ associated to $\mathfrak{g}$ over a finite field recovers the positive part of the quantum groups $U_q(\mathfrak{g})$ for $\mathfrak{g}$ a simple Lie algebra. For a quiver with potential $(Q,W)$, Kontsevich-Soibelman \cite{ks} constructed a Hall algebra generated by vector spaces originating in Donaldson-Thomas theory that we now explain.
\\

For a quiver with potential $(Q,W)$, dimension vector $d\in \mathbb{N}^{I}$ and generic King stability condition $\theta\in \mathbb{Q}^{I}$, we denote by $\X^{ss}(d)\subset \X(d)$ the locus of $\theta$ semistable representations, and by $X^{ss}(d)$ the coarse space of $\X^{ss}(d)$. 
The CoHA, constructed by Kontsevich-Soibelman \cite{ks}, is an algebra with underlying $\mathbb{N}^{I}$-graded
vector space:
$$\text{CoHA}\,(Q, W):=\bigoplus_{d\in\mathbb{N}^{Q_0}}H^{\cdot}(\X^{ss}(d), \varphi_{\text{Tr}\,W}\mathbb{Q}),$$ where the multiplication $m=p_*q^*$ is defined using the stack $\X(d,e)$ parametrizing pairs of representations $0\subset A\subset B$ with $A$ of dimension $d$ and $B$ of dimension $d+e$ and the natural maps: 
$$\X^{ss}(d)\times \X^{ss}(e)\xleftarrow{q} \X^{ss}(d,e)\xrightarrow{p} \X^{ss}(d+e).$$

The critical locus of $\text{Tr}(W):\X(d)\to\mathbb{A}^1$ over zero is $\X(Q,W,d)$, the moduli of representations of dimension $d$ of the Jacobi algebra $$\text{Jac}(Q,W)=\mathbb{C}\langle Q\rangle/\left(\frac{\partial W}{\partial e}, e\in E\right),$$ so the vector space $H^{\cdot}(\X(d),\varphi_{\text{Tr}\,W}\mathbb{Q})$ is the cohomology of the (usually singular) space $\X(Q,W,d)$ with coefficients in (a shift of) a perverse sheaf.
Using framed quivers, Davison-Meinhardt \cite{dm} and Soibelman \cite{s} constructed representations of these algebras.
\\


\subsection{CoHAs and Yangians}
There are multiple constructions called CoHAs and Yangians in the literature. The Kontsevich-Soibelman CoHAs are defined in a more general framework than Yangians, but CoHAs
recover only the positive part of the Yangians. It is an interesting problem to construct full Yangians using Hall algebras. 
\\

We mention some results about geometric construction of Yangians.  
In \cite{nakd}, Nakajima showed that Kac-Moody algebras act by correspondences on the Borel-Moore homology of Nakajima quiver varieties associated to quivers with no loops. Ginzburg-Vasserot in type A and Varagnolo for 
Kac-Moody algebras \cite{vara} give a geometric construction of the Yangian. Schiffmann-Vasserot \cite{sv3} construct the corresponding Yangian for the Jordan quiver and show it acts on the Borel-Moore homology of $\text{Hilb}(\mathbb{A}^2,n)$,
and Maulik-Okounkov \cite{mo} construct an Yangian $Y_{\text{MO}}$ for a general quiver which acts on homology of Nakajima quiver varieties.
\\

We next mention some of the other constructions called CoHA in the literature related to quivers.
Schiffmann-Vasserot \cite{sv2} consider algebras also called $\text{CoHA}$ for a general quiver which are conjecturally related to $Y_{\text{MO}}$. In \cite{yz}, Yang-Zhao construct algebras corresponding to the positive part of the Schiffmann-Vasserot CoHAs, and they discuss the relation to the Yangian defined by Drinfeld when $Q$ has no loops \cite{yz3}.  
The Kontsevich-Soibelman $\text{CoHA}$ is related to these other $\text{CoHA}$s. Indeed, the positive part of the Schiffmann-Vasserot algebra can be recovered as the Kontsevich-Soibelman $\text{CoHA}$ for a certain quiver with potential $(\widetilde{Q},\widetilde{W})$ introduced by Ginzburg \cite{gi} that one constructs from $Q$, called the tripled quiver \cite{rs}, \cite{yz2}, see Subsection \ref{tripledef}. In \cite{d}, Davison conjectured that $\text{CoHA}(\widetilde{Q}, \widetilde{W})$ is the positive part of the Maulik-Okounkov Yangian $Y_{\text{MO}}$.
\\

There are also $\text{CoHA}$s beyond the quiver case, for example the Minets \cite{min} and Sala-Schiffmann \cite{ss} $\text{CoHA}$s associated to Higgs bundles on a curve $C$, and the Kapranov-Vasserot \cite{kv} $\text{CoHA}$ for sheaves on a surface $S$. One expects to construct a Kontsevich-Soibelman $\text{CoHA}$ for all Calabi-Yau $3$-folds, and that the Minets, Sala-Schiffmann and the Kaparnov-Vasserot CoHAs are particular cases of the Kontsevich-Soibelman $\text{CoHA}$ using dimensional reduction for $\text{Tot}(\omega_C)\times \mathbb{A}^1$
and $\text{Tot}(\omega_S)$, respectively.

\subsection{The Jordan quiver}
For the Jordan quiver $Q$
\begin{tikzcd}
1 \arrow[out=0,in=90,loop,swap,"f"]
\end{tikzcd}
with one vertex and one loop, the tripled quiver
\begin{tikzcd}
1 \arrow[out=0,in=90,loop,swap,"x"]
  \arrow[out=120,in=210,loop,swap,"y"]
  \arrow[out=240,in=330,loop,swap,"z"]
\end{tikzcd}
has one vertex, three loops, and potential $W=xyz-xzy.$ 
By dimensional reduction \cite[Appendix A]{d}, there exists an isomorphism
$$\text{CoHA}(\widetilde{Q},\widetilde{W})\cong
\bigoplus_{d\geq 0} H^{\cdot}_c(\mathcal{P}(d),\mathbb{Q}),$$ where $\mathcal{P}(d)$ is the stack of commuting matrices of dimension $d$, and the right hand side has an algebra structure defined by correspondences. 
The framed representations \cite{dm}, \cite{s} of the $\text{CoHA}(\widetilde{Q},\widetilde{W})$ for vector $f=1$ in this case are given by 
$$\bigoplus_{d\geq 0} H^{\cdot}(\text{Hilb}\,(\C^3,n),\varphi_{\text{Tr}\,W}\mathbb{Q}).$$ 

In this example, $\text{CoHA}(\widetilde{Q},\widetilde{W})$ is the positive part of the affine Yangian of $\mathfrak{gl}_1$ \cite{d}, \cite{s}, \cite{yz} and it acts on:
$$\bigoplus_{d\geq 0}H^{BM}_{\cdot}(\text{Hilb}\,(\C^2,n), \mathbb{Q}).$$
Schiffmann-Vasserot \cite{sv1}, Maulik-Okounkov \cite{mo} constructed an action of the full affine Yangian of $\mathfrak{gl}_1$ on the above vector space.
\\

\subsection{The PBW theorem for COHA} The $\text{CoHA}(Q,W)$ has nice properties even if $(Q,W)$ is not a tripled quiver, such as the existence of a localized coproduct \cite{d} compatible with the product, a wall-crossing theorem \cite[Theorem B]{dm}, and a PBW theorem for a generic stability condition \cite[Theorem C]{dm} that we now briefly explain. The map from the stack of representations to the coarse space $\pi: \X^{ss}(d)\to X^{ss}(d)$ induces a perverse filtration on $H^{\cdot}(\X(d)^{ss},\varphi_{\text{Tr}\,W}\mathbb{Q})$ such that its first piece is:
$$P^{\leq 1}= \bigoplus_{d\in\mathbb{N}^{Q_0}} H^{\cdot}(X^{ss}(d), \varphi_{\text{Tr}\,W}IC_d).$$  
Davison-Meinhardt \cite[Theorem C]{dm} proved that the associated graded of (a twisted version of) CoHA with respect to this perverse filtration is a supercommutative algebra on $P^{\leq 1}\otimes H^{\cdot}(B\C^*)$: $$\text{gr}^P\, \text{CoHA}\,(Q, W)=\text{Sym}\,\left(\bigoplus_{d\in\mathbb{N}^{Q_0}} H^{\cdot}(X^{ss}(d), \varphi_{\text{Tr}\,W}IC)\otimes H(B\C^*)\right).$$
As a corollary, there exists a natural Lie algebra structure on $P^{\leq 1}$ called the BPS Lie algebra. To prove the PBW theorem, Davison and Meinhardt first prove a more general version at the level of sheaves on $X^{ss}(d)$ in the potential zero case, statement which follows from an explicit description of the summands that appear in the decomposition theorem for the map $\pi:\X^{ss}(d)\to X^{ss}(d)$ \cite[Proposition 4.3 and Theorem 4.6]{mr}. They deduce the general statement by applying the vanishing cycle functor, which commutes with pushforward along the map $\pi:\X^{ss}(d)\to X^{ss}(d)$ because it can be approximated by proper maps.

\subsection{K-theoretic Hall algebras.} The purpose of the present article is to formulate and prove analogous results in K-theory. In \cite[Section 8.1]{ks}, Kontsevich-Soibelman propose the category of singularities $D_{sg}(\X(d)_0)$ as a categorification of the vanishing cohomology $H^{\cdot}(\X(d), \varphi_{\text{Tr}\,W}\mathbb{Q})$. Indeed, given a smooth variety $X$ with a regular function $f:X\to \mathbb{A}^1$ with $0$ the only critical value, the ($\mathbb{Z}/2\mathbb{Z}$ periodic) vanishing cohomology $H^{\cdot}(X,\varphi_f\mathbb{Q})$ can be recovered as the periodic cyclic homology of the category of singularities of the central fiber $D_{sg}(X_0) =\text{Coh}\,(X_0)/\,\text{Perf}\,(X_0)$ by a theorem of Efimov \cite{e}. The category $D_{sg}(X_0)$ satisfies a Thom-Sebastiani theorem and dimensional reduction theorem analogous to the ones for $H^{\cdot}(X,\varphi_f\mathbb{Q})$.
\\

For a quiver with potential $(Q,W)$, we construct two versions of $K$-theoretic Hall algebra using $K_0$ of the category of singularities $D_{sg}(\X_0)$ and of its idempotent completion.
The definition of the KHA as a $\mathbb{N}^{I}$-graded vector space is:
$$\text{KHA}(Q,W)=\bigoplus_{d\in\mathbb{N}^{I}} K_0(D_{sg}(\X(d)_0)),$$ and similarly for $\text{KHA}^{\text{id}}(Q,W)$.
For $Y=\text{crit}\,(f)$, we will also use the notation $K_{\text{crit}}(Y)$ for $K_0(D_{sg}(X_0))$. This definition depends on the function $f$, not only on the critical locus $Y$. Using this notation, the graded $d$ piece of the KHA is $K_{\text{crit}}(\X(Q,W,d))$.
The multiplication is defined using correspondences as in cohomology. 
We construct $T$-equivariant version of the $\text{KHA}$, where $T$ is a torus fixing the potential. We also define a compatible coproduct in Section \ref{bialgebra2}.

\begin{thm}\label{thm:1}
There exist natural braided bialgebra structures on $KHA$, $KHA^{id}$, and their $T$-equivariant versions.
\end{thm}

Using framed representations, we can construct representations of the $\text{KHA}s$ analogoues to the ones constructed in \cite{dm}, \cite{s} in cohomology.

\begin{thm}\label{exrep}
There is a natural action of $KHA(\widetilde{Q},\widetilde{W})$ on $$\bigoplus_{d\in\mathbb{N}^I} K_0(D_{sg}(\X^{ss}(f,d)_0)),$$ where $f$ is a framing vector $f\in\mathbb{N}^I$. 
\end{thm}

For example, for the pair $(Q_3,W)$ considered above, the above theorem says that $\text{KHA}_T(Q_3,W)$ acts on $\bigoplus_{d\geq 0} K_{\text{crit}}^T(\text{Hilb}\,(\mathbb{A}^3,d))$.
\\

\subsection{KHAs and quantum affine algebras}
For a quiver $Q$, we expect $\text{KHA}(\widetilde{Q}, \widetilde{W})$ to be related to positive parts of quantum affine algebras. We mention some previous geometric constructions of quantum affine algebras in the literature. Ginzburg-Vasserot \cite{gv} constructed geometric representations of $U_q(\widehat{\mathfrak{sl}}_n)$ using the K-theory of flag varieties. Nakajima \cite{n} gave a geometric construction of the quantum affine algebras $U_q(\widehat{\mathfrak{g}})$ for a Kac-Moody algebra $\mathfrak{g}$ using the K-theory of quiver varieties. Schiffmann-Vasserot \cite{sv1} and Feigin-Tsymbaliuk \cite{ft} discussed the case of the Jordan quiver, when the corresponding algebra $U_q(\widehat{\widehat{\mathfrak{gl}_1}})$ appears in the literature also as the Elliptic Hall algebra, the Feigin-Odeskii shuffle algebra etc. Okounkov-Smironov \cite{os} constructed a quantum group $U_q(\widehat{\mathfrak{g}}_Q)$ for a general quiver $Q$ which acts on the K-theory of Nakajima quiver varieties; here, the Lie algebra $\mathfrak{g}_Q$ was constructed by Maulik-Okounkov \cite{mo} and, for general $Q$, it is larger than the corresponding Kac-Moody Lie algebra $\mathfrak{g}_{KM}$.
\\

We mention a couple of other constructions called KHA in the literature, see Subsection \ref{further} for more examples.
Yang-Zhao \cite{yz} consider a K-theoretic Hall algebra that should be related to our construction via dimensional reduction. For the affine $\widehat{A}_n$ quiver, Negu\c{t} \cite{ne1} constructs the toroidal algebra of type $A_n$ using a KHA construction.
\\

We expect the $\text{KHA}(\widetilde{Q}, \widetilde{W})$ constructed in this article to be related to the positive part $U^{>}_q\left(\widehat{\mathfrak{g}_Q}\right)$ of the Okounkov-Smirnov quantum group:

\begin{conjecture}
There exists an isomorphism for a natural one dimensional torus: $$\text{KHA}_{\C^*}(\widetilde{Q},\widetilde{W})\cong U^{>}_q\left(\widehat{\mathfrak{g}_Q}\right).$$  
\end{conjecture}

In this direction, we have proved:

\begin{thm}
Let $Q$ be an arbitrary quiver, and consider the tripled quiver $(\widetilde{Q},\widetilde{W})$. 
Then the $KHA(\widetilde{Q},\widetilde{W})$ acts naturally on the $K_0$ of Nakajima quiver varieties 
$$\bigoplus_{d\in\mathbb{N}^I} K_0(N^{\text{ss}}(f,d)).$$ 
Similar statements also hold in the $T$-equivariant case.
\end{thm}

For example, when $Q$ is the Jordan quiver,
dimensional reduction gives a description of the underlying vector space of $\text{KHA}(\widetilde{Q},\widetilde{W})$ using the stack $\mathcal{P}(d)$ of representations of the preprojective algebra of $Q$:
$$\text{KHA}_T(\widetilde{Q},\widetilde{W})=\bigoplus_{d\geq 0} K^T_0(\text{Coh}\,\mathcal{P}(d)).$$
The above theorem says that $KHA_T(\widetilde{Q},\widetilde{W})$ acts on 
$\bigoplus_{n\geq 0} K^T_0(\text{Hilb}\,(\C^2,n))$.


\begin{thm}\label{ex}
(a) Let $A_n$ be the Dynkin quiver 
and consider its tripled quiver $(\widetilde{A_n},\widetilde{W})$, where $\widetilde{A_n}$ is the quiver
\begin{tikzcd}
1 \arrow[out=60, in=120,loop] \arrow[r, shift left=2, "x_1"] 
& 2 \arrow[out=60, in=120,loop] \arrow[r, shift left=2, "x_2"] \arrow[l, "y_1"]
& \cdots \arrow[l, "y_2"] \arrow[r, shift left=2, "x_{n}"]
& n \arrow[out=60, in=120,loop]  \arrow[l, "y_{n}"]
\end{tikzcd}
with potential $W=\sum_{i=1}^{n}\omega_i(y_ix_i-x_{i-1}y_{i-1}).$ Here the set of vertices is $I=\{1,\cdots, n\}$ and $\omega_i$ is the loop at the vertex $i$. 
Then for a certain one dimensional torus $T$ defined in Section \ref{Dynkin}, we have an isomorphism:
$$\text{KHA}_T(\widetilde{A_n},\widetilde{W})\cong U_q^{>}(L\mathfrak{sl}_{n+1}).$$ 

(b) Let $(\widetilde{Q},\widetilde{W})$ be the tripled quiver for the Jordan quiver. Then for a certain two dimensional torus $T$ defined in Section \ref{Jordan}, there exists a natural inclusion of the Feigin-Odeskii shuffle algebra:
$$\text{SH}\subset \text{KHA}_T(\widetilde{Q},\widetilde{W}).$$
\end{thm}

\subsection{The PBW theorem for KHA} We next discuss the PBW theorem. Assume that $Q$ is symmetric, $W$ is arbitrary, and that the stability condition $\theta$ is trivial, but we expect these results to hold for general $(Q,W)$ and generic $\theta$, which is the generality in which the cohomological PBW theorem holds \cite{dm}. 

In order to prove a PBW theorem for these algebras, we need to find a replacement for the perverse filtration and the decomposition theorem for the Hilbert-Chow morphism $\pi:\X(d)\to X(d)$ and for the BPS invariants 
$H^{\cdot}(X(d), \varphi_{\text{Tr}\,W} IC_d)$. In general, there is no categorification for intersection cohomology, and also no categorification for the decomposition theorem. For example, for an elliptic $K3$ surface $\pi: X\to\mathbb{P}^1$, the decomposition $R\pi_*IC_X$ is not concentrated in one perverse dimension, but the category $D^b\text{Coh}\,(X)$ does not admit semi-orthogonal decompositions because $X$ is Calabi-Yau.
Even more, there are no replacements of intersection cohomology or the decomposition theorem in K-theory.

A striking feature of symmetric quiver geometries is that these notions have analogues that we can use to state and prove a PBW theorem.
The BPS invariants are replaced with certain categories $\mathbb{M}(d)\subset D_{sg}(\X(d)_0)$ inspired by work of \v{S}penko-van den Bergh \cite{sp}; the decomposition theorem for the map $\pi:\X(d)\to X(d)$ is categorified by semi-orthogonal decompositions of the categories $D_{sg}(\X(d)_0)$ coming from geometric invariant theory, inspired by work of Halpern-Leistner \cite{hl} and \v{S}penko-van den Bergh \cite{sp}; and the perverse filtration is replaced by the filtration induced by the above semi-orthogonal decomposition.
\\

We define the categories $\mathbb{M}(d)$ in Subsection \ref{defMM}. 
They are subcategories of $\oM(d)$ with certain boundary conditions explained in Subsection \ref{defMM}. In the zero potential case, the category $\overline{\mathbb{N}}(d)\subset D^b(\X(d))$ is defined as the full subcategory generated by the sheaves $V(\chi)\otimes\OO_{\X(d)}$ for $\chi$ such that:
$$\chi+\rho\in\frac{1}{2}\bigoplus_{\beta\text{ wt of }R(d)}[0,\beta]\subset M_{\mathbb{R}},$$ where $\rho$ is half the sum of positive roots of $G(d)$ and where the right hand side is the Minkowski sum after all weights $\beta$ of the $G(d)$-representation $R(d)$. 
In the case of a general potential $W,$ consider the inclusion of the central fiber $i:\X(d)_0\to \X(d)$, and define $\oM(d)\subset D_{sg}(\X(d)_0)$ as the subcategory 
$$\oM(d)=\text{Coh}'(\X(d)_0)/\text{Perf}'(\X(d)_0),$$ where $\text{Coh}'(\X(d)_0)\subset \text{Coh}(\X(d)_0)$ is the subcategory of sheaves $F$ such that $i_*F\in\overline{\mathbb{N}}$, and $\text{Perf}'(\X(d)_0)\subset \text{Coh}'(\X(d)_0)$ is the subcategory
of perfect complexes.
\\

The categories $\oM(d)$ admit a decomposition in $\oM(d)_w$, where $w\in\mathbb{Z}$ is the weight with respect to the diagonal character. 
In Section \ref{5}, we construct a semi-orthogonal decomposition: 
 $$D_{sg}(\X(d)_0)=\langle \cdots, \oM(d)\rangle, $$ where the complement of $\oM(d)$ has a further semi-orthogonal decomposition with factors $p_{\dd *}q_{\dd}^*\left(\oM(d_1)_{w_1}\boxtimes\cdots\boxtimes \oM(d_k)_{w_k}\right)$, where $p_{\dd}: \X(\dd)\to \X(d)$. We first prove the statement in the zero potential case, and then deduce from it the general potential case.
\\

The categories $\overline{\mathbb{N}}(d)$ and $\oM(d)$ may contain sheaves supported on attracting loci, and in general we cannot remove them with a further semi-orthogonal decomposition. 
It is possible to find such a decomposition if there exists a generic $\delta\in M_{\mathbb{R}}^W$ as in \cite[Definition 3.1]{hl2}, and in this case the category $\mathbb{N}(d)\subset \overline{\mathbb{N}}(d)$ is the same as the category defined by \v{S}penko-van den Bergh in \cite{sp}. Their category is Calabi-Yau, and thus cannot be further decomposed. 
A sheaf $F$ is in $\mathbb{N}(d)\subset \overline{\mathbb{N}}(d)$ or in $\MM(d)\subset \oM(d)$ if its restriction on these attracting loci contained in $\overline{\mathbb{N}}(d)$ or in $\oM(d)$ satisfy some further conditions as described in Subsection \ref{defMM}. The K-theory of $\oM(d)$ can be then decomposed in $K_0(\MM(d))$ and summands corresponding to sheaves supported on attracting loci. The following is proved in Section \ref{5}:

\begin{thm}
Let $(Q,W)$ be a symmetric quiver with potential. There exists a filtration induced by the inclusion $\MM(d)\subset \oM(d)$ and by the above semi-orthogonal decomposition of $D_{sg}(\X(d)_0)$ such that:
$$\text{gr}^F\text{KHA}\,(Q,W)= \text{dSym}\,\left(\bigoplus_{d\in\mathbb{N}^{I}} K(\mathbb{M}(d))\right),$$
where the rights hand side is an algebra called the deformed symmetric algebra which has relations:
$$x_{d,w}x_{e,v}=\left(x_{e,v}q^{f(e,d)}\right)\left(x_{d,w}q^{-g(e,d)}\right),$$ 
for $x_{d,w}\in K_0(\MM(d)_w)$ and $x_{e,v}\in K_0(\MM(e)_v)$, and factors $q^{f(e,d)}$ and $q^{-g(e,d)}$ which depend on the $d$ and $e$ only.
The same result holds for $\text{KHA}^{id}$.
\end{thm}

Using semi-orthogonal decompositions, we also show a wall-crossing theorem that explains how $\text{KHA}(Q,W)$ decomposes into smaller algebras with fixed slope $\text{KHA}(Q,W,\mu)$. 
\\

\subsection{Comparison between KHA and CoHA} Next, we compare the $\text{KHA}$ and the $\text{CoHA}$. We assume that $W$ is homogeneous with respect to some nonnegative weights associated to the edges.
In section \ref{ch}, we construct a Chern character: 
$$\text{ch}: K_0(D_{sg}(\X(d)_0))\to H^{\cdot}(\X(d),\varphi_{\text{Tr}\,W}\mathbb{Q})^{\wedge}$$ with values in the completion of vanishing cohomology. We expect it to be related to the Chern character for the category $D_{sg}(\X(d)_0)$: 
$$\text{ch}: K_0(D_{sg}(\X(d)_0))\to HP(D_{sg}(\X(d)_0)).$$
Let $\text{gr}\,\text{KHA}$ be the associated graded with respect to the cohomological degree filtration.

\begin{thm}
The Chern character induces a bialgebra morphism: 
$$\text{ch}:\text{gr}\, \text{KHA}(Q,W)\to \text{CoHA}(Q,W)^{\wedge}.$$
\end{thm}

Next, we focus on the study of $\text{gr}\,\text{KHA}$. The first result is that it satisfies a PBW theorem:

\begin{thm}
There exists a filtration $E^{\cdot}$ on (a twisted version of) $\text{gr}_I\text{KHA}$ induced by the semi-orthogonal decomposition from Section \ref{5} such that the associated graded is a symmetric algebra: $$\text{gr}^E\,\text{gr}\,\text{KHA}\,(Q,W)= \text{Sym}\,\left(\bigoplus_{d\in\mathbb{N}^{I}} \text{gr}\,K(\mathbb{M}_0(d))[\![u]\!]\right).$$ As a corollary, there exists a Lie algebra, called the $\text{KBPS}$ Lie algebra, on the first piece of the filtration $E^{\leq 1}=\bigoplus_{d\in\mathbb{N}^{I}} \text{gr}\,K(\mathbb{M}_0(d))$.
\end{thm}

We next show that the two filtrations $E^{\cdot}$ on $\text{gr}\,\text{KHA}$ and $P^{\cdot}$ on $\text{CoHA}$ are compatible. This result is surprising because $E^{\cdot}$ is defined using semi-orthogonal decompositions, while $P^{\cdot}$ is defined using the topology of the Hilbert-Chow morphism $\pi:\X(d)\to X(d)$. We show this statement by observing that $E^{\leq 1}$ and $P^{\leq 1}$ are the subspaces of primitive elements in the two algebras. In order to show that $E^{\leq 1}$ is primitive for the coproduct defined in Section \ref{bialgebra}, we give an equivalent construction of the coproduct using semi-orthogonal decompositions in Subsection \ref{sodcop}.

\begin{thm}
The filtration $E^{\cdot}$ on $\text{gr}\, KHA$ coming from the semi-orthogonal decomposition and the perverse filtration $P^{\cdot}$ on $H^{\cdot}(\X(d),\varphi_{\text{Tr}\,W}\mathbb{Q})$ are compatible via $\text{ch}$, so we obtain a Lie algebra morphism: $$\text{ch}:\text{KBPS}\to \text{BPS}.$$
\end{thm}

When $Q$ is arbitrary and $W=0$, the Chern character $\text{ch}:\text{gr}\,\text{KHA}\to\text{CoHA}$ induces an isomorphism of bialgebras and $\text{ch}:\text{KHA}\to \text{CoHA}$ is injective, so we can recover part the intersection cohomology of the coarse spaces $X(d)$ using the category $\MM(d)$:

\begin{cor}
The Chern character map $\text{ch}: \text{gr}\,K_0(\MM(d))\to IH^{\cdot}(X(d))$ is an isomorphism. Thus:
$$\text{dim}\,K_0(\MM(d))=\text{dim}\,\text{gr}\,K_0(\MM(d))=\text{dim}\,IH^{\text{even}}(X(d)).$$
\end{cor}

\subsection{Further directions}\label{further}
We expect to define a $\text{KHA}$ for any Calabi-Yau $3$-fold $X$, but the definition seems out of reach at the moment because there is no categorical/ K-theoretic replacement of the Joyce sheaf in general, even if there are replacements in the local case. However, using dimensional reduction, we expect to define the $\text{KHA}$ for $\text{Tot}\,(\omega_S)$ using the moduli stacks of sheaves on the surface $S$. We should thus obtain a $\text{KHA}$ for every surface $S$, including $\text{Tot}(\omega_C)$ for $C$ a curve. In future work, we hope to prove a PBW theorem and define a $\text{KBPS}$ Lie algebra for these geometries, which might require replacing $K_0$ with coarser invariants, such as $K_0^{\text{num}}$. 

It is interesting to see whether, for a curve $C$, the KBPS Lie algebras for the local surfaces $\text{Tot}\,(\omega_C)$ satisfy any of the properties in \cite[Section 8]{sch}. It is also interesting to see whether there any similarities between the KBPS of a surface and the Lie algebras defined by Looijenga-Lunts \cite{ll} and Verbitsky \cite{ver}.
\\

There are other Hall algebras constructions in the literature using K-theory, for example those of Kapranov-Vasseot \cite{kv}, Negu\c t \cite{ne2}, Yang-Zhao \cite{yz}, Zhao \cite{z}. We expect these constructions to be related to the $\text{KHA}$ above. 
\\

As we mentioned in the introduction, we expect the $\text{KHA}$ for a tripled quiver $(\widetilde{Q},\widetilde{W})$ to be related to the positive parts of the quantum groups constructed by Okounkov and Smirnov \cite{os}. In particular, we expect $\text{KHA}(\widetilde{Q},\widetilde{W})$ to be
the positive part of a quantum group, and it is an interesting problem to construct the full quantum group using Hall algebras. 
Further, it is natural to guess that $\text{KHA}(Q,W)$ is the positive part of a quantum group for more general pairs $(Q,W)$. 
\\

In future work, we plan to compute explicitly some the categories $\mathbb{M}(d)$ for the tripled quiver $Q_3$ with potential $W=xyz-xzy$, which give the categorical BPS invariants for $\C^3$. There are other CY$3$ geometries which can be described globally using quiver with potentials, such as the orbifold $\C^3/(\mathbb{Z}/n\mathbb{Z})$, $\text{Tot}_{\mathbb{P}^1}(\OO(-1)\oplus \OO(-1))$, and $\text{Tot}\,(\omega_{\mathbb{P}^2})$, and it will be interesting to compute the categories $\mathbb{M}(d)$ in these cases as well and see how they compare to be the motivic or cohomological DT invariants.
\\

We expect the KBPS and BPS Lie algebras for a tripled quiver $(\widetilde{Q}, \widetilde{W})$ to be isomorphic. For this, we need to show that the Chern character map $\text{ch}: G_0(\mathcal{P}(d))\to H^{BM}_{\cdot}(\mathcal{P}(d))$ is an isomorphism, where $\mathcal{P}(d)$ is the stack of representations of dimension $d$ of the preprojective algebra of $Q$. This expectation is motivated by a theorem of Davison \cite{d2} which says that the cohomology with compact supports of $\mathcal{P}(d)$ is pure, of Tate type.
\\

A striking feature of the symmetric quiver geometries is the existence of categorifications of the intersection cohomology of $X(d)$ and of the decomposition theorem for the map $\pi:\X(d)\to X(d)$. It is interesting to see other examples where these constructions can be categorified.

\subsection{Outline of the paper} In Section \ref{2} we review the notions needed about quivers with potentials, derived categories, and geometric invariant theory. We discuss some features of critical K-theory, for example the Thom-Sebastiani theorem, dimensional reduction, and localization.

In Section \ref{3}, we define the multiplications on $\text{KHA}$, $\text{KHA}^{\text{id}}$, and on their $T$-equivariant versions, and show they are associative. We also show that the product can be defined using $T(d)$-equivariant K-theory.

In section \ref{bialgebra2}, we define a coassociative coproduct on $\text{KHA}$ and a braiding $R$. We then show that the product, coproduct, and braiding define a braided bialgebra structure on $\text{KHA}$. We also show that the analogous result holds for $\text{CoHA}$.

In section \ref{ch}, we define a Chern character $\text{ch}:K_0(D_{sg}(\X(d)_0))\to H^{\cdot}(\X(d),\varphi_{\text{Tr}\,W}\mathbb{Q})^{\vee}$ and show that $\text{ch}$ defines a bialgebra morphism $\text{ch}:\text{gr}\,\text{KHA}\to\text{CoHA}.$

In section \ref{4}, we prove a wall-crossing formula for $\text{KHA}$ and show that the analogue of a theorem of \cite{r} holds in K-theory.

From section \ref{5} on, we assume that $Q$ is symmetric and that the stability condition $\theta$ is zero. In section \ref{5}, we prove the PBW theorem for $\text{KHA}$ and $\text{KHA}^{\text{id}}$. For more details about the proof, see the first page of that section.

In section \ref{8}, we discuss properties of  $\text{gr}\,\text{KHA}$. We first provide a different construction of the coproduct of $\text{KHA}$ using semi-orthogonal decompositions. 
We prove a PBW theorem for $\text{gr}\,\text{KHA}$, and deduce as a corollary the existence of a $\text{KBPS}$ Lie algebra. We show that $\text{BPS}$ and $\text{KBPS}$ are the space of primitives of the bialgebras $\text{CoHA}$ and $\text{KHA}$, respectively.
This implies that the filtrations on $\text{gr}\,\text{KHA}$ and $\text{CoHA}$ are compatible, and thus that the Chern character induces a Lie algebra morphism $\text{ch}: \text{KBPS}\to \text{BPS}$. 

In Section \ref{6}, we compute examples of $\text{KHA}$. We discuss the zero potential case, when $\text{KHA}$ has a shuffle product description, and the tripled Dynkin type A and Jordan cases, when the $\text{KHA}$s are affine quantum groups. 

In Section \ref{7}, we construct representations of $\text{KHA}$ using framed quivers, and explain how starting from a quiver $Q$ we can construct a pair $(\widetilde{Q},\widetilde{W})$ such that $\text{KHA}(\widetilde{Q},\widetilde{W})$ acts on the K-theory of the Nakajima quiver varieties of $Q$. 
\\

\subsection{Acknowledgements} I would like to thank my PhD advisor Davesh Maulik for suggesting the problems discussed in the present paper and for his constant help and encouragement throughout the project. Many thanks to Ben Davison and Daniel Halpern-Leistner for answering my questions related to their work. I would like to thank Pavel Etingof, Andrei Okounkov, and Andrei Negu\c{t} for useful conversations about the project.
\\

\subsection{Notations} All the schemes and stacks considered are over $\mathbb{C}$. We will denote by $(Q,W)$ a quiver with potential, by $d\in\mathbb{N}^I$ a dimension vector, and by $\theta\in\mathbb{Q}^I$ a King stability condition.

We denote by $\mathcal{X}(d)=R(d)/G(d)$ the moduli of representations of $Q$ of dimension $d$. We also consider the related stacks $\Y(d)=R(d)/T(d)$ and $\mathcal{Z}(d)=R(d)/B(d)$, where $T(d)\subset B(d)\subset G(d)$ is a compatible choice under multiplication of maximal torus and Borel of $G(d)$. We denote by $X(d)$ the coarse space of $\X(d)$. We use the notation $\mathcal{K}$ for a Koszul stack and $\mathcal{P}$ for the stack of representations of a preprojective algebra. We will use Roman letters $X$ etc. for schemes , and calligraphic letters $\X$ etc. for stacks. 

The maps related to attracting loci from geometric invariant theory are denoted by $p$ for the proper map and $q$ for the affine bundle map, see Subsection \ref{window}. 

We denote by $D^b(-)$ the derived category of coherent sheaves on a scheme or stack, by $\text{Perf}(-)\subset D^b(-)$ its subcategory of perfect complexes, and by $D_{sg}(-)$ the category of singularities. The categories $\overline{\mathbb{N}}(d)$, $\mathbb{N}(d)$, $\oM(d)$, and $\MM(d)$ are defined in Subsections \ref{defnmagic} and \ref{defMM}. All the categories considered are triangulated except for Proposition \ref{sod}, when we need a dg enhancement for $D^b(\X(d))$ and for $\overline{\mathbb{N}}(d)$. 

We will use weights only for the groups $T(d)$ and $G(d)$ as above. We will use the notation $\chi$ for a weight of $T(d)$, $\rho$ for half the sum of positive roots of a given reductive group, $\beta$ for a simple root of $T(d)$, and $W_d$ for the Weyl group of $G(d)$. We use $w\chi$ for the standard Weyl group action and $w*\chi=w(\chi+\rho)-\rho$ for the shifted Weyl group action.
For a weight $\chi$ of $G(d)$, let $V(\chi)$ be the representation with highest weight $\chi^+$ in the shifted Weyl orbit of $\chi$ if there is such a weight, and zero otherwise. 
For more notations and conventions about reductive groups, see Subsection \ref{notations}. The definition of the polytope $\overline{\mathbb{W}}$ is in Subsection \ref{polytope}.

We use the notation $\dd$ for a partition of $d$. We use superscripts $d=(d^i)$ to denote the dimensions at a vertex $i\in I$, and subscripts $\dd=(d_i)$ to denote dimension vectors $d_i\in\mathbb{N}^I$ that appear in a partition of $d$. We will use two ordering in the set of partitions: the lexicographic order $\succ$, which is total, and a partial one defined in Subsection \ref{compa}. 

We will use three filtrations on the K-theoretic algebras: the semi-orthogonal filtration $F^{\cdot}$ on $\text{KHA}$, see Definitions \ref{filtration} and \ref{secondfiltrations}; the cohomological filtration $I^{\cdot}$ on $\text{KHA}$, and we drop $I$ from the notation in the associated graded in this case $\text{gr}\,\text{KHA}$; and the filtration $E^{\cdot}$ on $\text{gr}\,\text{KHA}$ induced by $F^{\cdot}$, see Definition \ref{E}. 

We denote by $K_{\cdot}(-)$ the algebraic K-theory of the category $\text{Perf}(-)$ of a space and by $G_{\cdot}(-)$ the algebraic K-theory of the category $D^b(-)$ of a space. We denote by $T$ a torus preserving the potential, see Subsection \ref{multi}, by $\text{ch}$ the Chern character, and by $\varphi$ the vanishing cycle functor.

\section{Background material}\label{2}

\subsection{Quivers with potential}\label{quivers}
Let $Q=(I, E, s, t)$ be a quiver with vertex set $I$, edge set $E$, and source and target maps $s, t: E\to I$. Let $d=(d^i)_{i\in I}\in \mathbb{N}^{I}$ be a dimension vector of $Q$, and $f\in\mathbb{N}^{I}$ a framing vector. The stack of representations of dimension $d$ of the quiver $Q$ is the quotient stack $\X(d)=R(d)/ G(d)$, where $$R(d)=\prod_{e\in E} \text{Hom}\,(V^{s(e)}, V^{t(e)})$$ for vector spaces $V^i$ of dimension $d^i$, and $G(d)$ is the group $$G(d)=\prod_{I\in I} GL(V^i)$$ that naturally acts on $R(d)$ by changing the basis of the vector spaces $V^i$. We denote by $X(d)=R(d)//G(d)$ the coarse space of $\X(d)$.

A potential $W$ is a linear combination of cycles in $Q$. Such a potential determines a regular function: 
$$\text{Tr}\,(W):\X(d)\to\mathbb{A}^1.$$ We will assume throughout the paper that $0$ is the only critical value. The critical locus of this function is the moduli of representations of the Jacobi algebra $\text{Jac}\,(Q,W)=\mathbb{C}Q/(\frac{\partial{W}}{\partial{e}})$, where $\mathbb{C}Q$ is the path algebra of $Q$, and $(\frac{\partial{W}}{\partial{e}})$ is the two-sided ideal in $\mathbb{C}Q$ generated by the derivatives $\frac{\partial W}{\partial e}$ of $W$ along all edges $e\in E$.
\\

\subsection{King stability conditions and moduli of representations of a quiver}
Given a tuple of integers $\theta=(\theta^i)\in\mathbb{Q}^{I}$, we define the slope of a dimension vector $d\in\mathbb{N}^{I}$ by: $$\mu(d):=\frac{\sum_{i\in I}\theta^id^i}{\sum_{i\in I}d^i}\in\mathbb{Q}\cup\{\infty\}.$$ 
For a slope $\mu\in \mathbb{Q}\cup \{\infty\}$, define $\Lambda_{\mu}\subset \mathbb{N}^{I}$ the monoid of dimension vectors $d$ with slope $\mu$ together with $d=0$. We define a pairing on the dimension vectors $\mathbb{N}^{I}$ by the formula: $$(d, e)=\sum_{i\in I} d^ie^i-\sum_{a\in E} d^{s(a)}e^{t(a)}.$$ Observe that $\text{dim}\,\X(d)=-(d,d)$.
We call the tuple of integers $\theta$ $\mu-$generic if $(d,e)=(e,d)$ for every $d,e\in \Lambda_{\mu}$, and generic if it holds for all slopes $\mu$. 

Given $\theta$, we call a representation $V$ \text{it}{(semi)stable} if for every proper subrepresentation $W\subset V$, $$\mu(W)<(\leq) \mu(V).$$ The locus of stable $R^{s}(d)$ and semistable representations $R^{ss}(d)$ inside $R(d)$ are open. We consider the moduli stack $\X^{ss}(d)=R^{ss}(d)/G(d)$ of semistable representations of dimension $d$ with corresponding coarse space $X^{ss}(d)=R^{ss}(d)//G(d)$; it has an open substack $\X^{s}=R^{s}(d)/G(d)$ with corresponding coarse space $X^{s}(d)=R^{s}(d)//G(d)\subset X^{ss}(d).$ 
We will use the notation $IC_d=IC_{X^{ss}(d)}$ when $X^{s}(d)$ is non-empty, and $0$ otherwise.
\\

\subsection{Moduli of framed representations}\label{framed}
Fix a framing vector $f\in \mathbb{N}^{I}$. We define a new quiver $Q^f$ by adding a new vertex $v$ to the vertices $I$ of the initial quiver, and by adding $f^i$ edges from the vertex $v$ to the vertex $i\in I$.
The dimension vector $d\in\mathbb{N}^{I}$ can be extended to a dimension vector for the new quiver $(1,d)\in\mathbb{N}\times\mathbb{N}^{I}$. 
The stability condition $\theta$ is extended to a stability condition for the new quiver $\theta^f=(\theta', \theta)\in \mathbb{Q}\times\mathbb{Q}^{I}$ as follows: let $\mu\in\mathbb{Q}$, and define $\theta'=\mu+\varepsilon$ for a small positive rational $\varepsilon>0$.
Assume that $d\in\mathbb{N}^I$ has slope $\mu$. Then a $(1,d)$-representation $V^f$ of $Q^f$ is $\theta^f$-semistable if and only if it is $\theta^f$-stable, and this holds if and only if the underlying $Q$-representation $V$ is semistable, and for all proper $Q^f$ subrepresentations $U^f\subset V^f$, the underlying $Q$ subrepresentation $U\subset V$ has slope $\mu(U)<\mu=\mu(V)$.

We denote by $\X^{ss}(f,d)= R^{ss}(f,d)/G(d)$ the moduli space of $\theta^f$ stable representations of $Q^f$ of dimension $(1,d)$. It comes with a proper map:
$$\pi_{f,d}: \X^{ss}(f,d)\to X^{ss}(d).$$
For any dimension vector $d$, we consider $\eta\in H^2(\X^{ss}(f,d))$ be the Chern class of the tautological line bundle. It is also the image of $l$ under the natural map: 
$$\mathbb{C}[l]=H^{\cdot}(\C^*)\to H^{\cdot}(R(d)\times\C^{fd}/G(d)\times\C^*)\to H^{\cdot}(\X^{ss}(f,d)).$$
Meinhardt-Reineke identified the explicit summands appearing in the decomposition theorem for the map $\pi_{f,d}$ \cite[Proposition 4.3 and Theorem 4.6]{mr}, \cite[Theorem 4.10]{dm}, see Subsection \ref{succ} for the definition of a partition of $d\in\mathbb{N}^I$.

\begin{thm}\label{mere}
The decomposition theorem for the proper map $\pi_{f,d}:\X^{ss}(f,d)\to X^{ss}(d)$ has the form:
$$R\pi_{f,d*}\mathbb{Q}[-(d,d)]=\bigoplus S^{m_1}IC_{d_1}\boxtimes\cdots\boxtimes S^{m_k}IC_{d_k}[-\deg{\tau_1}-\cdots-\deg{\tau_k}-\sigma(\dd)],$$ where the sum is after all the partitions $\dd: m_1d_1+\cdots+m_kd_k=d$ of $d$ with $d_1>\cdots>d_k$, all monomials $\tau_i\in \mathbb{Q}[\eta_1,\cdots,\eta_{m_i}]^{\mathfrak{S}_{m_i}}/(\eta^{fd_i})$, and $\sigma$ is defined by $\sigma(\dd)=m_1+\cdots+m_k$.
\end{thm}

\subsection{PBW theorem for CoHA}\label{CoHA}
The cohomological Hall algebra for $(Q,W)$ and King stability condition $\theta$ is defined as: $$\text{CoHA}(Q,W)=\bigoplus_{d\in\mathbb{N}^I} H^{\cdot}(\X^{ss}(d), \varphi_{\text{Tr}\,W}IC_{\X^{ss}(d)}),$$ where the multiplication $m=p_{d,e*}q_{d,e}^*$ is defined via the stack of extensions of dimensions $d$ by dimension $e$: $$\X^{ss}(d)\times \X^{ss}(e)\xleftarrow{q_{d,e}} \X^{ss}(d,e)\xrightarrow{p_{d,e}} \X^{ss}(d+e).$$
Here $\X(d,e)=R(d,e)/G(d,e)$ the stack of pairs of representations $0\subset A\subset B$, where $A$ is a dimension $d$ representation, and $B/A$ is a dimension $e$ representation. 
For fixed dimension vector $d\in\mathbb{N}^I$ and cohomological degree $i$, we have that: $$H^i(\X^{ss}(d),\varphi_{\text{Tr}\,W} IC_{\X^{ss}(d)})=H^i(\X^{ss}(f,d),\varphi_{\text{Tr}\,W} IC_{\X^{ss}(d)}),$$ for $f$ large enough. 
The map $\pi_{f,d}:\X^{ss}(f,d)\to X^{ss}(d)$ induces a perverse filtration: 
\begin{multline*}
    P^{\leq i}=H^{\cdot}(X^{ss}(d),\varphi_{\text{Tr}\,W} {}^pR^{\leq i}\pi_{f,d*}\mathbb{Q}[-(d,d)])\to H^{\cdot}(X^{ss}(d),\varphi_{\text{Tr}\,W} R\pi_{f,d*}\mathbb{Q}[-(d,d)])\cong\\
    H^{\cdot}(\X^{ss}(d), \varphi_{\text{Tr}\,W}IC_{d})
\end{multline*}
for large enough $f$, and thus induces a filtration $P^{\leq i}$ of $\text{CoHA}$.
Theorem \ref{mere} implies that:
$$P^{\leq 1}=\bigoplus_{d\in\mathbb{N}^I} H^{\cdot}(X(d),\varphi_{\text{Tr}\,W} IC_{\X(d)}).$$
We next explain a modification of the multiplication \cite[Section 2.6]{ks}, \cite[Section 1.6]{dm} which is used in the formulation of the PBW theorem for CoHAs.
\begin{defn}
\label{supercom} 
Let $V=(\mathbb{Z}/2\mathbb{Z})^I$, and define the bilinear map $\tau:V\otimes V\to \mathbb{Z}/2\mathbb{Z}$ by the formula:
$$\tau(d,e)=\chi(d,e)+\chi(d,d)\chi(e,e).$$
We have that $\tau(d,d)=0$, so there exists a bilinear form $\psi:V\otimes V\to\mathbb{Z}/2\mathbb{Z}$ such that: $$\psi(d,e)+\psi(e,d)=\tau(d,e).$$ Define the $\psi$-twisted multiplication by the formula:
$$a*b:=(-1)^{\psi(d,e)}ab.$$
One checks that the $\psi$-twisted multiplication is associative from the associativity of the standard multiplication and the bilinearity of $\psi$. Let the degree of $a$ in $H^{\cdot}(\X(d), \varphi_{\text{Tr}\,W}IC_{\X(d)})$ be $\text{deg}\,(a)=\psi(d,d)$. 
We have that: $$a*b=(-1)^{\psi(d,d)+\psi(e,e)}b*a,$$ so the algebra is supercommutative for the above grading. Use the notation $\text{CoHA}^{\psi}$ for $\text{CoHA}$ with the twisted multiplication $\psi$.
\end{defn}

The following PBW theorem was proved in \cite[Theorem C]{dm}:

\begin{thm}
The inclusion $P^{\leq 1}\subset \text{CoHA}^{\psi}$ induces an isomorphism of supercommutative algebras:
$$\text{Sym}\,(P^{\leq 1}[u])\to \text{gr}^P\,\text{CoHA}^{\psi},$$
where $u$ has perverse degree $2$. In particular, $P^{\leq 1}\subset \text{CoHA}^{\psi}$ is closed under the Lie bracket $[x,y]=xy-yx$, and thus it is a Lie algebra under $[\,,]$ called the BPS Lie algebra.
\end{thm}

\subsection{The tripled quiver}\label{tripledef}
The following construction was introduced by Ginzburg \cite{gi} and it is used in conjunction with dimensional reduction to obtain representations of a preprojective CoHA or KHA or the cohomology of K-theory of Nakajima quiver varieties.

Let $Q=(I,E)$ be a quiver.
The double quiver $\overline{Q}$ has vertex set $I$ and edge set $E\cup \overline{E}$, where for every edge $e\in E$ we add an edge of opposite orientation $\overline{e}$. The tripled quiver $\widetilde{Q}$ has vertex set $I$, and to the edges of $\overline{Q}$ we add a loop $\omega_i$ at every vertex $i\in I$. The potential $\widetilde{W}$ is defined by: $$\widetilde{W}=\sum_{e\in E} \omega_{s(e)}\bar{e}e-\omega_{t(e)}e\bar{e}.$$

\subsection{Nakajima quiver varieties}\label{nakquivvar}

Let $Q$ be a quiver, $d\in\mathbb{N}^I$ a dimension vector, $\theta$ a stability condition, and $f\in\mathbb{N}^I$ a framing vector. Extend $\theta$ to the stability condition $\theta'$ for $Q^f$ as in Subsection \ref{framed}. Associated to $\theta$ we have a character $\chi_{\theta}=\prod_{i\in I} \det(g^i)^{m\theta^i}: G(d)\to\mathbb{C}^*$ for $m$ a positive integer such that $m\theta_i$ are all integers.
The action of $G(d)=G(1,d)/\mathbb{C}^*$ on $R(1,d)$ induces a moment map:
$$\mu: T^*R(1,d)\to \mathfrak{g}(d)^{\vee}\cong \mathfrak{g}(d).$$
Define the Nakajima quiver variety $N^{\text{ss}}(f,d)$ by the GIT quotient:
$$N^{\text{ss}}(f,d):=R(1,d)^{\theta^f-\text{ss}}//_{\chi_{\theta}}G(d).$$

\subsection{Semi-orthogonal decompositions.}\label{sodadjoint}
Let $\A$ be a triangulated category, and let $\A_i\subset \A$, for $1\leq i\leq n$, be full triangulated subcategories. We say that $\A$ has a semi-orthogonal decomposition: $$\A=\langle \A_m,\cdots,\A_1\rangle$$ if
for every objects $A_i\in\mathcal{A}_i$ and $A_j\in\mathcal{A}_j$ and $i<j$ we have
$\text{Hom}\,(A_i,A_j)=0$, 
and the smallest full triangulated subcategory of $\mathcal{A}$ containing $\mathcal{A}_i$ for $1\leq i\leq m$ is $\mathcal{A}$.

Equivalently, the subcategories $\A_i$ for $1\leq i\leq n$ form a semi-orthogonal decomposition of $\A$ if the inclusion $\A_i\subset \langle \A_i,\cdots,\A_1\rangle$ admits a left adjoint for all $i$, or, alternatively, the inclusion $\A_i\subset \langle \A_n,\cdots,\A_i\rangle$ admits a right adjoint for all $i$, see \cite{b}, \cite[Section 2]{ku} for more details.

\subsection{Window categories.}\label{window}
Let $\X=V/G$ be a quotient stack where $G$ be a reductive group and $V$ is a $G$ representation. 
Given a linearization $\mathcal{L}$, the semistable stack is $\X^{\mathcal{L}-ss}=V^{\mathcal{L}-ss}/G$, where $V^{\mathcal{L}-ss}\subset V$ is the open locus of semistable points. We will drop the linearization $\mathcal{L}$ from notation. Halpern-Leistner \cite{hl} constructed categories $\mathbb{G}_w\subset D^b(\X)$ which are equivalent to $D^b(\X^{ss})$ under the restriction map to the open locus $\X^{ss}\subset \X$.
\\

The semistable locus is obtained by removing unstable points. We next recall the criterion for determining the unstable locus. 
For any pair $(\lambda, Z)$, where $\lambda$ is a character of $G$ and $Z$ is a component of the fixed locus of $\lambda$, define: 
$$\text{inv}\,(\lambda, Z):=-\frac{\langle \lambda,\mathcal{L}|_Z\rangle}{|\lambda|},$$
where $\langle \lambda, \mathcal{L}|_Z\rangle=\text{weight of }\lambda\text{ on }\mathcal{L}|_Z$ and $|\lambda|$ is a Weyl invariant quadratic form on the group of characters of $G$. 
Further, for any pair $(\lambda, Z)$ as above, consider the diagram:

\begin{tikzcd}
\mathcal{S}=S/P\arrow{r}{p} \arrow{d}{q}& \mathcal{X}=V/G\\
\mathcal{Z}=Z/L,
\end{tikzcd}
\\
where $S\subset V$ is the subset of points $x$ such that $\lim_{z\to 0} \lambda(z)x\in Z$, and $L$ and $P$ are the Levi and parabolic groups corresponding to $\lambda$. The map $q$ is an affine bundle map and the map $p$ is proper. If $G$ is abelian, then the map $p$ is a closed immersion.
\\

If there are no pairs with $\text{inv}<0$, then there are no unstable points. If there are pairs with $\text{inv}<0$, choose the pair $(\lambda, Z)$ that maximises $\text{inv}(\lambda, Z)$ and consider the diagram defined above:
$$\mathcal{Z}=Z/L\xleftarrow{q}\mathcal{S}=S/P
\xrightarrow{p} \X=V/G.$$
For this choice of $(\lambda, Z)$, the map $p$ is a closed immersion.
\\

The image of $p$ in $S\times_P G/G\subset V/G$ is unstable and closed. We repeat the above process for its complement and continue until there are no pairs $(\lambda, Z)$ such that $\text{inv}(\lambda, Z)<0$. These unstable loci are called Kempf-Ness loci, and they stratify $\X^{\text{uns}}=\X-\X^{\text{ss}}$.
For more details on the Kempf-Ness stratification, see \cite[Section 2.1]{hl}.
\\
The category $D^b(\mathcal{Z})$ has a semi-orthogonal decomposition in subcategories $$D^b(\mathcal{Z})_w=\{E\in D^b(\mathcal{Z})\text{ such that the }\lambda-\text{weights of }E\text{ are }w\}$$ for $w\in \mathbb{Z}$. The category $D^b(\mathcal{S})$ has a similar semi-orthogonal decomposition:
$$D^b(\mathcal{S})=\langle D^b(\mathcal{S})_w, w\in\mathbb{Z}\rangle \text{ and }q^*:D^b(\mathcal{Z})_w\cong D^b(\mathcal{S})_w.$$
Assume that the Kempf-Ness loci are given by $(\lambda_i, Z_i)$, and that the attracting loci are $S_i$. Let $j_i:Z_i\to V$ be the inclusion.
Choose an integer $w_i$ for any index $i$ of a stratum, and let $n_i=\langle \lambda_i, \det N_{\mathcal{S}/\X}\rangle$. Consider the full subcategory category $\mathbb{G}_w$ of $D^b(\X)$ defined by:
$$\mathbb{G}_w=\{F\in D^b(\X)\text{ such that } w_i\leq \langle \lambda_i, j_i^*F\rangle\leq w_i+n_i-1\}.$$

The main theorem in \cite{hl} is:
\begin{thm}\label{hl}
There exists a semi-orthogonal decomposition: $$D^b(\X)=\langle U_{<w},\mathbb{G}_w, U_{\geq w}\rangle,$$ where the categories $U_{<w}$ and $U_{\geq w}$ contain sheaves supported on the unstable locus.
The restriction functor 
$\text{res}:D^b(\X)\to D^b(\X^{ss})$ induces an equivalence of categories: $$\text{res}:\mathbb{G}_w\to D^b(\X^{ss}).$$
\end{thm}

We state the precise statement in the one-stratum case. Consider the categories:
$$D^b(\mathcal{S})_{<w}=\langle q^*D^b(\mathcal{Z})_i\text{ for }i<w\rangle\text{ and }D^b(\mathcal{S})_{\geq w}=\langle q^*D^b(\mathcal{Z})_i\text{ for }i\geq w\rangle.$$ We denote by $D^b_{\mathcal{S}}(\X)$ the category of complexes on $\X$ supported on $\mathcal{S}$, and let $D^b_{\mathcal{S}}(\X)_{<w}$ be the subcategory of $D^b(\X)$ generated by $p_*\left(D^b(\mathcal{S})_{<w}\right)$ and $D^b_{\mathcal{S}}(\X)_{\geq w}$ the subcategory of $D^b(\X)$ generated by 
$p_*\left(D^b(\mathcal{S})_{\geq w}\right).$ There is a semi-orthogonal decomposition:
$$D^b_{\mathcal{S}}(\X)=\langle D^b_{\mathcal{S}}(\X)_{<w},\,D^b_{\mathcal{S}}(\X)_{\geq w}\rangle.$$
The semi-orthogonal decomposition from Theorem \ref{hl} is:
$$D^b(\X)=\langle D^b_{\mathcal{S}}(\X)_{<w},\mathbb{G}_w, D^b_{\mathcal{S}}(\X)_{\geq w}\rangle.$$
The adjoint to the inclusion $D^b_{\mathcal{S}}(\X)_{\geq w}\subset D^b(\X)$ is given by \cite[Lemma 3.37]{hl}:
$$\Phi(F):=\beta_{\geq w}R\Gamma_{\mathcal{S}}(F):=\beta_{\geq w}\left(R\mathcal{H}om(\OO_\X/I_{\mathcal{S}}^N,F)\right)\in D^b_{\mathcal{S}}(\X)_{\geq w}$$ for large enough $N$. The adjoint to the inclusion $D^b_{\mathcal{S}}(\X)_{< w}\subset D^b(\X)$ is given by $\Phi(F):=\mathbb{D}\beta_{\geq n+1-w}R\Gamma_{\mathcal{S}}\mathbb{D}(F)\in D^b_{\mathcal{S}}(\X)_{< w}$, where $\mathbb{D}$ is the Serre duality functor.

\subsection{The \v{S}penko-van den Bergh noncommutative resolution.}\label{sp}
Let $\X=V/G$ be a quotient stack, where $G$ is a reductive group and $V$ is a symmetric linear $G-$representation. 
Fix a maximal torus and a Borel subgroup $T\subset B\subset G$. 
Let $M$ be lattice of $G$ weights, $M_{\mathbb{R}}=M\otimes_{\mathbb{Z}}\mathbb{R}$, and let $W$ be the Weyl group of $G$. Assume that the weights of $V$ span $M_{\mathbb{R}}$.
We denote by: $$\overline{\mathbb{W}}:=\bigoplus_{\beta\text{ wt of }V} [0,\beta]\subset M_{\mathbb{R}},$$
where the sum is taken after all weights $\beta$ of the $G$-representation $V$.
Let $\delta\in M_{\mathbb{R}}^W$. Define $\mathbb{N}\subset D^b(\X)$ as the full subcategory generated by sheaves $V(\chi)\otimes\OO_{\X}$, where $\chi$ is a dominant $G$-weight such that:
$$\chi+\rho+\delta\in\frac{1}{2}\overline{\mathbb{W}}.$$
Here $\rho$ is half the sum of positive roots of $G$.

The category can be also characterized as follows.
For any character $\lambda:\C^*\to G$, define $$n_{\lambda}=\langle \lambda, V^{\lambda>0}\rangle-\langle \lambda, 2\rho^{\lambda>0}\rangle,$$ where $2\rho^{\lambda>0}$ is the sum of the weights which pair positively with $\lambda$ in the adjoint representation of $G$. 
Then $\mathbb{N}\subset D^b(\X)$ is the subcategory of sheaves $F$ such that:
$$-\frac{1}{2}n_{\lambda}+\langle \lambda, \delta\rangle \leq \langle \lambda, i_{\lambda}^*F\rangle \leq \frac{1}{2}n_{\lambda}+\langle \lambda, \delta\rangle$$ for any character $\lambda:\C^*\to G$, where $i_{\lambda}:V^{\lambda}\to V$ is the inclusion of the $\lambda$-fixed locus.  
\\

For a character $\lambda: \C^*\to G$, consider the fixed locus $Z$, the attracting locus $S$, the Levi group $L$, and the parabolic group $P$:
$$Z/L \xleftarrow{q} S/P \xrightarrow{p} V/G.$$
\v{S}penko-van den Bergh \cite{sp}, \cite{sp2} proved the following:

\begin{thm}
For any $\delta\in M^W_{\mathbb{R}}$, there exists a semi-orthogonal decomposition 
$$D^b(\X)=\langle \cdots, \mathbb{N} \rangle,$$
where the complement of $\mathbb{N}$ is generated by complexes supported on attracting loci. 

Further, if there exists $\delta\in M^W_{\mathbb{R}}$ such that $\langle \lambda, \delta\rangle+\frac{1}{2}n_{\lambda}$ is not an integer for all non-trivial characters $\lambda$, then $\mathbb{N}$ is a Calabi-Yau category and a crepant non-commutative resolution of the singularity $V//G$.
\end{thm}

We want to apply the above theorem for the stacks $\X(d)$ defined in Subsection \ref{quivers}, but in general there are no generic $\delta\in M^W_{\mathbb{R}}$ as in the second paragraph of the above theorem; for example, there are no such $\delta$ for any quiver with one vertex. 
\\

We will make explicit the complement of the above semi-orthogonal decomposition for $\X(d)$ when $\delta=0$ in Section \ref{5}. 
For $\chi\in M_{\mathbb{R}}$, consider the shifted Weyl action $w*\chi:=w(\chi+\rho)-\rho$, and denote by $\chi^+$ the dominant weight in the orbit of $\chi$ for shifted Weyl action if there exists such a weight, and zero otherwise.
An important ingredient in the above semi-orthogonal decomposition is the following:

\begin{prop}
Let $\chi$ be a $P$-weight. The complex $p_{\lambda *}(V(\chi)\otimes\OO_{\mathcal{S}})$ has a resolution by the vector bundles $V((\chi-\sigma_I)^+)\otimes\OO_\X$ where $\sigma_I$ is a partial sum of weights for $I\subset\{\beta\text{ weight of }N^{\lambda<0}\}$.
\end{prop}

Indeed, for the closed immersion $i:S/P\to V/P$, the complex $i_*(V(\chi)\otimes\OO_{\mathcal{S}})$ has a resolution by the bundles $V(\chi-\sigma_I)\otimes\OO_\X$, where $\sigma_I$ is a partial sum of weights for $I\subset\{\beta\text{ weight of }N^{\lambda<0}\}$, and these bundles appear in cohomological position $|I|$. 
By the Borel-Bott-Weyl Theorem for the map $V/P\to V/G$, we obtain that the complex $p_*(V(\chi)\otimes\OO_{\mathcal{S}})$
has a resolution by the vector bundles $V((\chi-\sigma_I)^+)\otimes\OO_\X$ for all sets $I$ as above.
See \cite[Section 3.2]{hls} for more details.


\subsection{Categories of singularities.}\label{singul}
Let $\mathcal{Y}=Y/G$ be a quotient stack, where $Y$ is an affine scheme with an action of a reductive group $G$. The category of singularities of $\mathcal{Y}$ is a triangulated category defined as he quotient of triangulated categories:
$$D_{sg}(\mathcal{Y}):=D^b\,(\mathcal{Y})/\text{Perf}(\mathcal{Y}),$$
where $\text{Perf}(\mathcal{Y})\subset D^b\,(\mathcal{Y})$ is the full subcategory of perfect complexes. 
If $\mathcal{Y}$ is smooth, the category of singularities is trivial. 
In our situation, $\mathcal{Y}$ will be the zero fiber of a regular function on a smooth quotient stack:
$$f:\X\to\mathbb{A}^1.$$
For a category $\mathcal{A}$, denote by $\mathcal{A}^{hf}\subset \mathcal{A}$ the full subcategory of homologically finite objects, that is, objects $A$ such that for every object $B$ of $\mathcal{A}$, the morphism space $\text{Hom}\,(A,B[i])$ is trivial except for finitely many shifts $i\in\mathbb{Z}$. If $\mathcal{A}\subset D^b(\X_0)$ is a category which admits an adjoint $\Phi:D^b(\X_0)\to\mathcal{A}$, then $\mathcal{A}^{hf}\subset \mathcal{A}$ is the full subcategory of perfect complexes \cite[Proposition 1.11]{o1}.
For a subcategory $\mathcal{A}\subset D^b(\X)$, define: 
$$\mathcal{A}^0=\{F\in D^b(\X_0)\text{ such that }i_*F\in \mathcal{A}\},$$
where $i:\X_0\to\X$ is the closed immersion. We use the notation $\mathcal{A}^{hf}$ for $(\mathcal{A}^0)^{hf}$.
For a subcategory $\mathcal{A}\subset D^b(\X_0)$, definea triangulated category: $$D_{sg}(\mathcal{A})=\mathcal{A}^0/\mathcal{A}^{hf}.$$
The categories $D^b(\X)$ for $\X$ a smooth quotient stack and $\mathbb{N}\subset D^b(\X)$ defined in Subsection \ref{sp} admit dg enhancements, and they are $D^b(\mathbb{A}^1)$-linear with respect to a potential $f:\X\to\mathbb{A}^1$. 
The next proposition is due to Halpern-Leistner-Pomerleano \cite[Lemmas 1.17 and 1.18 ]{hlp} and Orlov \cite[Proposition 1.10]{o1}:

\begin{prop}\label{sod}
Let $\X$ be a smooth quotient stack with a regular function $f:\X\to\mathbb{A}^1$. Consider $D^b(\mathbb{A}^1)$-linear dg categories $\mathcal{A}_i\subset D^b(\X)$ such that: $$D^b(\X)=\langle \mathcal{A}_i\rangle.$$
(a) We have semi-orthogonal decomposition:
$$D^b(\X_0)=\langle \mathcal{A}^0_i\rangle\text{ and }\text{Perf}\,(\X_0)=\langle \mathcal{A}_i^{hf}\rangle.$$
(b) We have a 
semi-orthogonal decomposition:
$$D_{sg}(\X_0)=\langle D_{sg}(\mathcal{A}_i)\rangle.$$ Further, $\mathcal{A}_i^{hf}\subset \mathcal{A}^0_i$ is the subcategory of perfect complexes in $\mathcal{A}^0_i$.
\end{prop}

\begin{proof}
(a) The first semi-orthogonal decomposition follows from \cite[Lemma $1.17$]{hlp} and the interpretation of the categories $D^b(\X_0)$ and $\mathcal{A}^0_i$ as: $$D^b(\X_0)=\text{Fun}_{D^b(\mathbb{A}^1)}(D^b(\text{pt}), D^b(\X))\text{ and }
\mathcal{A}^0_i=\text{Fun}_{D^b(\mathbb{A}^1)}(D^b(\text{pt}), \mathcal{A}_i)$$ via the Fourier-Mukai transformation \cite{bzfn}. The second semi-orthogonal decomposition follows from the first one and \cite[Proposition 1.10]{o1}.

(b) Both statements follow from \cite[Proposition 1.10]{o1}.
\end{proof}

\textbf{Remark.} One can define $D_{sg}^{id}(\X_0)$ directly from $D^b(\X_0)$ as in \cite[Section 1]{hlp}; the analogoue of part $(b)$ is Proposition $1.11$
in loc. cit.
\\

An application of this proposition is the following. Assume that we are given a stability condition $\theta$ for the action of $G$ on a smooth variety $V$, and let $\X=V/G$.
Then the semi-orthogonal decomposition $$D^b(\X)=\langle \cdots, \mathbb{G}_w, \cdots \rangle,$$ where $\mathbb{G}_w\cong D^b(\X^{ss})$ implies that there exists a semi-orthogonal decomposition 
$$D_{sg}(\X_0)=\langle \cdots, \mathbb{D}_w, \cdots\rangle,$$ where $\mathbb{D}_w\cong D_{sg}(\X^{ss}_0)$.
The following result also follows directly from Proposition \ref{sod}:

\begin{cor}\label{uns}
Let $\mathcal{X}$ be a quotient stack with Kempf-Ness stratum $\mathcal{Z}\xleftarrow{q}\mathcal{S}\xrightarrow{p}\X$. By Theorem \ref{hl}, the functor:
$$p_*q^*: D^b(\mathcal{Z})_w\to D^b(\mathcal{X})$$ is fully faithful. 
Theorem \ref{hl} gives a semi-orthogonal decomposition:
$$D^b(\mathcal{Y})=\langle p_*q^*D^b(\mathcal{Z})_{< w}, \mathbb{G}_w, p_*q^*D^b(\mathcal{Z})_{\geq w}\rangle.$$ 
Then Proposition \ref{sod} says that we have a semi-orthogonal decomposition: 
$$D_{sg}(\mathcal{Y}_0)=\langle p_*q^*D_{sg}(\mathcal{Z}_0)_{< w}, D_{sg}(\mathbb{G}_w), p_*q^*D_{sg}(\mathcal{Z}_0)_{\geq w}\rangle.$$
\end{cor}

\subsection{Dimensional reduction}

In this section, we prove the following:

\begin{thm}\label{dimred0}
Let $V\times\mathbb{A}^n$ be a representation of a reductive group $G$, and let $\X=X\times\mathbb{A}^n/G$. Consider a regular function:
$$g=t_1f_1+\cdots+t_nf_n:\X\to\mathbb{A}^1,$$
where $t_1,\cdots, t_n$ are the coordinates on 
$\mathbb{A}^n$ and 
$f_1,\cdots, f_n: X\to\mathbb{A}^1$ are 
$G$-equivariant regular functions. Let $\X_0\subset\X$ be the zero locus of $g$ and let $i:\mathcal{Z}\hookrightarrow \mathcal{X}_0$ be the zero locus of $f_1,\cdots, f_n$. The pushforward map induces an isomorphism:
$$i_*: G_0(\mathcal{Z})\cong K_0(D_{sg}(\X_0)).$$
\end{thm}

\textbf{Remark.} The isomorphism $G_0(\mathcal{Z}')\cong K_0(D_{sg}(\X_0))$, where $\mathcal{Z}'$ is the Koszul stack associated to the regular functions $f_1,\cdots, f_n$ follows from a derived equivalence $$D^b(\mathcal{Z}')\cong D_{sg}^{gr}(\X_0)$$ due to Isik \cite{i}. Here $D_{sg}^{gr}$ is the graded category of singularities, defined as the quotient of triangulated categories $$D_{sg}^{gr}(\X_0)=D^b\text{Coh}^{gr}(\X_0)/\text{Perf}^{gr}(\X_0)$$ where $\text{Coh}^{gr}(\X_0)$ denotes the category of graded coherent with respect to the grading induced by the weight $1$ action of $\mathbb{G}_m$ on $\mathbb{A}^n$. 
We obtain $D_{sg}(\X_0)$ by collapsing the $\mathbb{Z}$-grading to a $\mathbb{Z}/2\mathbb{Z}$ and we thus have an isomorphism $K_0(D_{sg}^{gr}(\X_0))=K_0(D_{sg}(\X_0))$ \cite[Proposition 1.22]{hlp}. The derived equivalence constructed by Isik uses Koszul duality, and it is not clear whether it is compatible with the dimensional reduction in cohomology \cite[Appendix A]{d}:
$$i_*: H_{\cdot}^{BM}(\mathcal{Z})\to H_{\cdot}^{BM}(\X_0,\varphi_g\mathbb{Q}).$$

We prefer to use the above statement because the dimensional reduction theorems in K-theory and cohomology are compatible. The only place where we actually need to use Theorem \ref{dimred0} and not Isik's theorem is Proposition \ref{dimred}, but in that case a straightforward argument shows that $G_0(\mathcal{Z})\twoheadrightarrow K_0(D_{sg}(\X_0))$ and thus $\text{gr}\,G_0(\mathcal{Z})\cong \text{gr}\,K_0(D_{sg}(\X_0))$.

\begin{prop}
Let $\mathcal{U}_0=\X_0-\mathcal{Z}$. 
Then:

(a) $G_1(\X_0)\twoheadrightarrow G_1(\mathcal{U}_0)$ and

(b) $G_0(\mathcal{Z})\hookrightarrow G_0(\mathcal{X}_0)$.
\end{prop}

\begin{proof}
The two statements are equivalent from the excision long exact sequence in equivariant K-theory:
$$\cdots\to G_1(\X_0)\to G_1(\mathcal{U}_0)\to G_0(\mathcal{Z})\to G_0(\mathcal{X}_0)\to\cdots$$
We assume first that $n=1$. Then $\X_0=(Z\times\mathbb{A}^1)/G\cup (X\times 0)/G$. The two statements are clear if $f=0$, so we assume otherwise. Consider the composition:
$$G_0(Z/G)\xrightarrow{\pi^*} G_0(Z\times\mathbb{A}^1/G)\to G_0(\X_0)\twoheadrightarrow G_0(Z\times\mathbb{G}_m/G).$$ The second map is surjective because $Z\times\mathbb{G}_m\subset\X_0$ is open. The composition is an isomorphism, and $\pi^*$ is also an isomorphism, so the map (b) is indeed injective.

Assume next that $n>1$. Let $\mathcal{V}_0\subset \mathcal{X}_0$ be the locus where $f_n\neq 0$ and let $\mathcal{W}_0=(X\times\mathbb{A}^{n-1})_0/G$ be its complement. The regular function considered here is: 
$$t_1f_1+\cdots+t_{n-1}f_{n-1}: X\times\mathbb{A}^{n-1}/G\to\mathbb{A}^1.$$
Then $\mathcal{V}_0\subset \mathcal{U}_0$.  Its complement is $\mathcal{W}_0\cap\mathcal{U}_0\subset \mathcal{W}_0$, which is the locus of points with $t_1f_1+\cdots+t_{n-1}f_{n-1}=0$ but such that not all the functions $f_i$ are zero. 
By the induction hypothesis, we have a surjection $G_1(\mathcal{W}_0)\twoheadrightarrow G_1(\mathcal{W}_0\cap \mathcal{U}_0)$.
The diagram:

\begin{tikzcd}
G_1(\mathcal{W}_0)\arrow[r, twoheadrightarrow]\arrow{d}& G_1(\mathcal{W}_0\cap \mathcal{U})\arrow{d}\\
G_1(\X_0)\arrow{d}\arrow{r}& G_1(\mathcal{U})\arrow{d}\\
G_1(\mathcal{V}_0)\arrow{r}{\text{id}}& G_1(\mathcal{V}_0)
\end{tikzcd}
\\
implies that indeed $G_1(\X_0)\twoheadrightarrow G_1(\mathcal{U})$.

\end{proof}

\begin{prop}\label{pushzero}
The pushforward map $i_*:K_0(\X_0)\to K_0(\X)$ is zero.
\end{prop}

\begin{proof}
We have that $K_0(\X)=K_0^{\text{num}}(\X_0)$, and thus the map factors through $i_*:K_0^{\text{num}}(\X_0)\to K_0^{\text{num}}(\X)$. 
The space $K_0^{\text{num}}(\X_0)$ is generated by vector bundles $\mathcal{O}\otimes \Gamma$, where $\Gamma$ is a representation of $G(d)$. The pushforward $i_*(\OO\otimes \Gamma)$ has a resolution $$0\to \OO_{\X(d)}\otimes \Gamma\to \OO_{\X(d)}\otimes \Gamma\to i_*(\OO\otimes \Gamma)\to 0,$$ so its image is indeed zero in $K_0(\X(d))$. 
\end{proof}

\begin{prop}\label{factorsthru}
The map $K_0(\X_0)\to G_0(\X_0)$ factors through $K_0(\X_0)\to K_0(\mathcal{U}_0)$.
\end{prop}

\begin{proof}
Let $\mathcal{U}\subset\X$ be the locus where not all $f_i$ are zero. We have that $\X-\X_0=\mathcal{U}-\mathcal{U}_0$.
We argue that $G_0(\mathcal{U})\to G_0(\mathcal{U}-\mathcal{U}_0)$ is an isomorphism. When $n=1$ this is immediate. 
Let $\mathcal{V}\subset \mathcal{U}$ be the locus where $f_n\neq 0$, and let $\mathcal{W}\subset \mathcal{U}$ be its complement. Consider the diagram:

\begin{tikzcd}
G_0(\mathcal{W})\arrow{r} \arrow{d}{\text{iso}}& G_0(\mathcal{U})\arrow[r, twoheadrightarrow] \arrow{d} & G_0(\mathcal{U}-\mathcal{W})\arrow{d}{\text{iso}}\\
G_0(\mathcal{W}-\mathcal{W}_0)\arrow{r}& G_0(\mathcal{U}-\mathcal{U}_0)\arrow[r, twoheadrightarrow] & G_0(\mathcal{U}-\mathcal{W}).
\end{tikzcd}
\\
By induction, the left and right vertical maps are isomorphisms, so the middle horizontal map is an isomorphism as well.
\\

Next, 
consider the diagram:

\begin{tikzcd}
 & K_0(\X_0)\arrow{dl} \arrow{d} \arrow{dr}{\text{zero}} & & \\
 G_1(\X-\X_0)\arrow{r}\arrow{d}{\text{id}}& G_0(\X_0) \arrow{r} \arrow{d}& G_0(\X)\arrow{d} \arrow[r, twoheadrightarrow]& G_0(\X-\X_0)\arrow{d}{\text{id}}\\
 G_1(\mathcal{U}-\mathcal{U}_0)\arrow[r, twoheadrightarrow] & G_0(\mathcal{U}_0)\arrow{r}{\text{zero}}& G_0(\mathcal{U})\arrow{r}{\text{iso}} & G_0(\mathcal{U}-\mathcal{U}_0).
\end{tikzcd}
\\
and assume that the map $G_0(\mathcal{U})\to G_0(\mathcal{U}-\mathcal{U}_0)$ is an isomorphism.
The horizontal sequences are excision long exact sequences.
 The map $K_0(\X_0)\to G_0(\X_0)$ factors through: $$K_0(\X_0)\to G_1(\X-\X_0)=G_0(\mathcal{U}_0)\oplus G_1(\mathcal{U}_0),$$ and thus through $K_0(\X_0)\to G_0(\mathcal{U}_0)$.
\end{proof}

\begin{proof}[Proof of Theorem \ref{dimred0}]
Consider the diagram:

\begin{tikzcd}
 & K_0(\X_0)\arrow{d}\arrow[r, twoheadrightarrow]& K_0(\mathcal{U})\arrow{d}{\text{id}}\\
G_0(\mathcal{Z})\arrow[r, hookrightarrow] \arrow{dr}& G_0(\X_0)\arrow[r, twoheadrightarrow] \arrow[d, twoheadrightarrow]& G_0(\mathcal{U})\\
 & K_0(D_{sg}(\X_0))&
\end{tikzcd}
\\
The middle horizontal line is an excision long exact sequence, and the middle vertical line follows from the definition $D_{sg}(\X_0)=D^b(\X_0)/\text{Perf}(\X_0)$.
Proposition \ref{factorsthru} and diagram chasing implies that the images of $G_0(\mathcal{Z})\to G_0(\X_0)$ and $K_0(\X_0)\to G_0(\X_0)$ intersect in zero only, so the induced map $G_0(\mathcal{Z})\to K_0(D_{sg}(\X_0))$ is indeed an isomorphism.

\end{proof}

\subsection{Thom-Sebastiani theorem.}\label{TSiso} Consider quotient stacks $\X$ and $\mathcal{Y}$ with regular functions $f:\X\to\mathbb{A}^1$ and $g:\mathcal{Y}\to\mathbb{A}^1$. Consider the regular function: 
$$f+g:\X\times\mathcal{Y}\to \mathbb{A}^1.$$

\begin{thm}
The inclusion $i:\X_0\times \mathcal{Y}_0\to (\X\times\mathcal{Y})_0$ induces an equivalence 
$$i_*: D_{sg}(\X_0)\times D_{sg}(\mathcal{Y}_0)\cong D_{sg}((\X\times\mathcal{Y})_0).$$
\end{thm}

We have the following corollary for summands of geometric categories:
\begin{cor}\label{corts}
Let $f:\X\to\mathbb{A}^1$ and $g:\Y\to\mathbb{A}^1$ be two regular functions on quotient stacks, and consider the regular function $f+g:\X\times\Y\to\mathbb{A}^1$. Consider subcategories $A\subset D^b(\X)$ and $B\subset D^b(\Y)$ whose inclusions admit (left or right) adjoints; we thus have an inclusion $A\boxtimes B\subset D^b(\X\times\Y)$ which admits an adjoint. Pushforward along the closed immersion $i:\X_0\times\Y_0\to (\X\times\Y)_0$ induces an equivalence of
categories of singularities:  $$i_*: D_{sg}(A)\boxtimes D_{sg}(B)\cong D_{sg}(A\boxtimes B).$$
\end{cor}

\begin{proof}
Consider the commutative diagram:

\begin{tikzcd}
D_{sg}\left(A\boxtimes B\right) \arrow{r}{\text{ff}} & D_{sg}\left(\X\boxtimes\Y\right) \\ D_{sg}\left(A\right)\boxtimes D_{sg}\left(B\right) \arrow{u}{i_*} \arrow{r}{\text{ff}} & 
D_{sg}\left(\X\right)\boxtimes D_{sg}\left(\Y\right), \arrow{u}{\text{eq}}
\end{tikzcd}
\\
where $\text{ff}$ denotes a fully faithful functor and $\text{eq}$ denotes an equivalence of categories. 
This implies that the left vertical map is also fully faithful. Let $\Phi$ be the adjoint to the inclusion of the horizontal maps in the above diagram. The following diagram commutes:

\begin{tikzcd}
D_{sg}\left(A\boxtimes B\right)  & D_{sg}\left(\X\boxtimes\Y\right) \arrow{l}{\Phi}\\ D_{sg}\left(A\right)\boxtimes D_{sg}\left(B\right) \arrow{u}{i_*}  & 
D_{sg}\left(\X\right)\boxtimes D_{sg}\left(\Y\right). \arrow{l}{\Phi} \arrow{u}{\text{eq}}
\end{tikzcd}
\\
The maps $\Phi$ are essentially surjective, so the left vertical map is essentially surjective as well, and thus we obtain the desired conclusion.
\end{proof}

\subsection{Functoriality of categories of singularities.}\label{functo}
A reference for this subsection is \cite{pv}. Let $\mathcal{X}$ be a quotient stack, and let $f:\mathcal{X}\to\mathbb{A}^1$ be a regular function with $0$ the only critical value. A morphism $q:\mathcal{Y}\to \X$ induces a functor: $$q^*: D_{sg}(\X_0)\to D_{sg}(\mathcal{Y}_0),$$ where $\mathcal{Y}_0$ is the zero fiber of the regular function $qf:\mathcal{Y}\to\mathbb{A}^1$. 
If the morphism $q$ is proper, it also induces a functor:
$$q_*:D_{sg}(\mathcal{Y}_0)\to D_{sg}(\X_0).$$
Let $T$ be a torus and assume that $\mathcal{X}$ and $\mathcal{Y}$ have a $T$ action and that $q$ is $T$ equivariant. Then the above functors admit $T$ versions.

\subsection{Localization theorems for categories of singularities.}\label{local}
Let $\X=V/G$ be a smooth quotient stack with $V$ a representation of $G$ and let $f:\X\to\mathbb{A}^1$ a regular function with $0$ a critical value. Assume that we are also given a linearization $\chi:G\to\C^*$, and let $\X^{\text{ss}}\subset\X$ be the open locus of semistable points. Then $K_0(BG)$ acts on $K_0(D_{sg}(\X^{\text{ss}}_0))$ via the tensor product.

Next, assume that $E\to V$ is a $G$-equivariant bundle, and let $\mathcal{E}=E/G$. Define the Euler class: $$\text{eu}(\mathcal{E}):=i^*i_*(1)\in K_0(BG).$$

\begin{prop}\label{zerodiv}
Let $E\to V$ and $\mathcal{E}$ be as above.
Assume that the exists $\C^*\subset G$ such that $E^{\C^*}=V$. 
Then the class $\text{eu}\,(\mathcal{E})$ is not a zero divisor in $K_0(D_{sg}(\X^{\text{ss}}_0)).$
\end{prop}

\begin{proof}
It suffices to prove the statement for $G=T$ a torus; the statement follows then from Proposition \ref{abel2}.
The action of $K_0(BG)$ on $K_0(D_{sg}(\X^{\text{ss}}_0))$ respects the $\mathbb{C}^*$ weights, so we get maps:
$$K_0(BG)_v\otimes K_0(D_{sg}(\X^{\text{ss}}_0))_w\to K_0(D_{sg}(\X^{\text{ss}}_0))_{v+w}.$$
Consider the filtrations $I^j\subset K_0(BG)$ and $K_0(D_{sg}(\X^{\text{ss}}_0))$ induced by the powers of the kernel of the maps $K_0(BG)\to K_0(B\C^*)$ and $K_0(D_{sg}(\X^{\text{ss}}_0))\to K_0(D_{sg}(V^{\text{ss}}_0/\C^*))$. These filtrations respect the $\C^*$-weights and they are compatible under the above action. We thus get maps:
$$\text{gr}_0K_0(BG)_v\otimes \text{gr}_{\cdot}K_0(D_{sg}(\X^{\text{ss}}_0))\to \text{gr}_{\cdot}K_0(D_{sg}(\X^{\text{ss}}_0)).$$ 
Let $I$ be the set of $\C^*$ weights on the normal bundle $N_{V/E}$; the hypothesis says that $0$ is not in $I$. We have that:
$$\text{eu}(\mathcal{E})=\prod_{i\in I}(1-q^i)\in K_0(B\C^*).$$ Choose $v\in\mathbb{Z}$ the smallest exponent with nonzero coefficient in the above product. Then: $$\text{gr}_0\,\text{eu}(\mathcal{E})=\pm 1\in \text{gr}_0K_0(B\C^*)_v,$$ and this implies that $\text{eu}(\mathcal{E})$ is not a zero divisor for $K_0(D_{sg}(\X^{\text{ss}}_0))$.

\end{proof}

\begin{thm}\label{loc}
Let $G$ be a reductive acting on a smooth variety $V$, and let $\X=V/G$. Consider a regular function $f:\X\to\mathbb{A}^1$.
Let $\C^*\hookrightarrow G$ be a subgroup, and define the ideal: 
$$I:=\text{ker}\,(K_0(BG)\to K_0(B\C^*)).$$ Let $\mathcal{Z}=V^{\C^*}/G$.
The restriction map to the fixed locus:
$$K_0(D_{sg}(\X_0))_I\to K(D_{sg}(\mathcal{Z}_0))_I$$
is an isomorphism.
\end{thm}

\begin{proof}
It is enough to show the statement for the case of $G=T$ a torus, the general result follows from Proposition \ref{abel2}. 

Consider the diagram $\mathcal{Z}\leftarrow \mathcal{S} \rightarrow \X$ for the attracting locus of the above character $\lambda:\C^*\to T.$ Let $\mathcal{U}\subset \X$ be the complement of $\mathcal{S}$.
By Proposition \ref{uns} and Theorem \ref{hl}, we have semi-orthogonal decompositions:
$$D_{sg}(\X_0)=\langle D_{sg}(\mathcal{Z}_0)_{<w}, \mathbb{D}_w, D_{sg}(\mathcal{Z}_0)_{\geq w}\rangle\text{ and }D_{sg}(\mathcal{Z}_0)=\langle D_{sg}(\mathcal{Z}_0)_{<w}, D_{sg}(\mathcal{Z}_0)_{\geq w}\rangle,$$
where $\mathbb{D}_w\cong D_{sg}(\mathcal{U}_0)$ by Theorem \ref{hl}.
By passing to $K_0$, we have the decomposition:
$$K_0(D_{sg}(\X_0))=K_0(D_{sg}(\mathcal{Z}_0))\oplus K_0(D_{sg}(\mathcal{U}_0)).$$
It is thus enough to show that $K_0(D_{sg}(\mathcal{U}_0))_I=0$. By the definition of the category of singularities, the map:
$$G_0(\mathcal{U}_0)\to K_0(D_{sg}(\mathcal{U}_0))$$ is surjective, so it suffices to show that $G_0(\mathcal{U}_0)_I=0$. Using the excision long exact sequence:
$$\cdots\to G_0(\mathcal{Z}_0)\to G_0(\X_0)\to G_0(\mathcal{U}_0)\to 0,$$ it is enough to show that: $$i_*:G_0(\mathcal{Z}_0)_I\to G_0(\X_0)_I$$ is an isomorphism, which follows from the localization theorem of Takeda \cite{tak}.

\end{proof}

\subsection{Chern character for singular varieties}\label{singularGRR}
Consider a singular quotient stack $\X_0$ with an embedding in a smooth quotient stack $\X$.
We construct a Chern character map 
$$\text{ch}': G_0(\X_0)\to H^{BM}_{\cdot}(\X_0)$$ such that its associated graded with respect to the cohomological filtration:
$$\text{ch}': \text{gr}\,G_0(\X_0)\to H^{BM}_{\cdot}(\X_0)$$

(a) commutes with proper pushforwards to schemes $\mathcal{Y}_0$ which fit in a diagram:

\begin{tikzcd}
\X_0 \arrow{d}{p_0} \arrow{r}{i} & \X \arrow{d}{p}\\
\mathcal{Y}_0 \arrow{r}{j} & \mathcal{Y},
\end{tikzcd}
\\
where $i$ and $j$ are closed immersions and $p$ is proper, and

(b) commutes with smooth pullbacks and pullbacks to a closed substack $\mathcal{Y}_0$ which fit in a diagram:

\begin{tikzcd}
\Y_0 \arrow{d}{q_0} \arrow{r}{i} & \Y \arrow{d}{q}\\
\mathcal{X}_0 \arrow{r}{j} & \mathcal{X},
\end{tikzcd}
\\
where $i$ and $j$ are closed immersions and $q$ is smooth or a closed immersion.
\\

First, there are Chern characters $\text{ch}: K_0(\X)\to H^{BM}_{\cdot}(\X)$ that commute with proper pushforward up to a Todd class, and with smooth pullbacks; this follows from the analogous statements for smooth varieties via the Edidin-Graham \cite{eg}, Totaro \cite{tot} construction. The result about restriction to closed smooth subvarieties follows, for example, from the fact that the cycle map for a smooth variety $A_{\cdot}(X)\to H^{BM}_{\cdot}(X)$ for a smooth variety commutes with restrictions to smooth subvarieties $Y\subset X$ \cite[Section 19]{f}; the result extends to a closed immersion $\Y\to\X$ where $\X$ and $\Y$ are smooth stacks by the Edidin-Graham, Totaro construction.
Consider the diagram:

\begin{tikzcd}
K_1(\X-\X_0)\arrow{r} \arrow{d}{\text{ch}}& K_0(\X_0) \arrow{r}  & K_0(\X)  \arrow{d}{\text{ch}} \\
H^{BM}_{\cdot+1}(\X-\X_0) \arrow{r} & H^{BM}_{\cdot}(\X_0) \arrow{r} & H^{BM}_{\cdot}(\X).
\end{tikzcd}
\\
The stacks $\X$ and $\X-\X_0$ are smooth, so we have Chern characters $\text{ch}$ for $\X$ and $\X-\X_0$ with the above properties.
We also obtain a Chern character $\text{ch}': G_0(\X_0)\to H^{BM}_{\cdot}(\X_0)$ which makes the above diagram commutative. Because the Chern characters for $\X$ and $\X-\X_0$ have the above properties, a standard diagram chasing shows that $\text{ch}'$ will also have the above properties.

\subsection{Abelian quotient stacks.}\label{abeliansubsection}
Consider the quotient stacks $\mathcal{Y}(d)=R(d)/T(d)$ and $\mathcal{Z}(d)=R(d)/B(d)$. There are natural morphisms:
$$\mathcal{Y}(d) \xleftarrow{\Pi} \mathcal{Z}(d) \xrightarrow{\pi} \X(d).$$
The map $\Pi^*: K_0(D_{sg}(\X_0/T))\to K_0(D_{sg}(X_0/B))$ is an isomorphism. 
We can describe the underlying vector space for the KHA in terms of the stacks $\mathcal{Y}(d)$ using the following:

\begin{prop}\label{abel2}
(a) Let $G$ be a reductive group with Weyl group $W$, $X$ a variety with a $G$ action, $f:X\to \C$ a $G$ invariant regular function. 
Then the image of $$\pi_*\Pi^*:K_0(D_{sg}(X_0/T))\to K_0(D_{sg}(X_0/G))$$ factors through the symmetrization map 
$S:K_0(D_{sg}(X_0/T))\to K_0(D_{sg}(X_0/T))^{W}$ which sends $x$ to $\sum_{w\in W} wx$ and induces an isomorphism:
$$\pi_*\Pi^*: K_0(D_{sg}(X_0/T))^{W}\cong K_0(D_{sg}(X_0/G)).$$
(b) The map:
$$(\Pi^*)^{-1}\pi^*: K_0(D_{sg}(\X_0/G))\to K_0(D_{sg}(X_0/T))$$ is an injection and has image $K_0(D_{sg}(X_0/T))^{W}$.
\end{prop}

\begin{proof}

First, observe that:
$$\Pi^*: K_0(D_{sg}(X_0/T))\cong K_0(D_{sg}(X_0/B)) \cong K_0(D_{sg}(X_0\times_B G/G)).$$ 
The map $\pi$ is an extension of $\mathbb{P}(E)$ bundles over $X$ for $E$ vector bundles over $X$, so the derived category $D^b(X\times_B G)$ admits a semi-ortohogonal decomposition where all the summands are equivalent to $D^b(X)$, and the summands are in bijection to the number of cells $K_0(G/B)$. Further, $D_{sg}(X_0\times_B G/G)$ admits a similar semi-orthogonal decomposition with summands $D_{sg}(X_0/T)$.
This implies that 
$$K_0(D_{sg}(X_0\times_B G/G))=K_0(D_{sg}(X_0/G))\otimes K_0(G/B).$$
The action of the Weyl group $W$ on the right hand side is trivial on $K_0(D_{sg}(X_0/G))$, and is the natural one on $K_0(G/B)$.
The projection map $\pi_*:K_0(G/B)\to K_0(\text{pt.})$ 
has the form: $$\pi_*(y)=\sum_{w\in W} wy,$$ from which we get that indeed the map $\pi_*: K_0(D_{sg}(X_0\times_B G/G))\to K_0(D_{sg}(X_0/G))$ is $\pi_*(x)=\sum_{w\in W} wx.$
This implies the desired conclusion.

(b) We have that $\pi_*\pi^*=\OO_{X_0/G}$ and the conclusion follow from $(a)$.
\end{proof}

\section{Definition of the KHA}\label{3}
\subsection{Definition of the multiplication.}\label{multi}
Let $(Q,W)$ be a quiver with potential. 
For $d\in\mathbb{N}^I$ a dimension vector, we consider the stack of representations $\X(d)=R(d)/G(d)$, see Subsection \ref{quivers} for definitions. 
The potential $W$ determines a regular function 
$$\text{Tr}(W_d):\X(d)\to\mathbb{A}^1.$$ For $d,e\in\mathbb{N}^I$ two dimension vectors, 
consider the stack $\X(d,e)=R(d,e)/G(d,e)$ of pairs of representations $0\subset A\subset B$, where $A$ has dimension $d$ and $B/A$ has dimension $e$. Let $\theta\in\mathbb{Q}^I$ be a King stability condition, and define the slope: $$\mu(d)=\frac{\sum_{i\in I}\theta^i d^i}{\sum_{i\in I} d^i}.$$
For a fixed slope $\mu$, let $\Lambda_{\mu}\subset \mathbb{N}^I$ be the monoid of dimension vectors with slope $\mu$. We denote by $\X^{ss}(d)\subset \X(d)$ the locus of $\theta$-semistable representations.
Consider the diagram of attracting loci, see Subsection \ref{window}, where $q$ is an affine bundle and $p$ is a proper map: 

\begin{tikzcd}
  \X^{ss}(d,e) \arrow[r, "p_{d,e}"] \arrow[d, "q_{d,e}"]
    & \X^{ss}(d+e) \\
  \X^{ss}(d)\times\X^{ss}(e).  
  \end{tikzcd}
  \\
The potentials are compatible with respect to these maps: 
$$p_{d,e}^*\text{Tr}(W_{d+e})=q_{d,e}^*(\text{Tr}(W_d)+\text{Tr}(W_e)).$$ 
We use $p$ and $q$ instead of $p_{d,e}$ and $q_{d,e}$ when there is no danger of confusion.
For every edge $e\in I$, let $\mathbb{G}_m$ act on $\text{Hom}\,(\C^{s(e)},\C^{t(e)})$ by scalar multiplication. We denote by $T\subset \prod_{i\in I}\mathbb{G}_m$ a torus which leaves $W$ invariant.
Consider the maps: 
$$\mu: BT\times BT\to BT \text{ and } \delta: BT\to BT\times BT$$ induced by the multiplication map and the diagonal map, respectively:
$$\mu: T\times T\to T \text{ and }\delta: T\to T\times T.$$ 

\begin{defn}
Consider the functors:
   
\begin{multline}
D_{sg}(\X^{ss}(d)_0)\boxtimes D_{sg}(\X^{ss}(e)_0)\xrightarrow{TS}  D_{sg}\left((\X^{ss}(d)\times\X^{ss}(e))_0\right)\xrightarrow{q^*}\\
D_{sg}(\X^{ss}(d,e)_0)\xrightarrow{p_*} D_{sg}(\X^{ss}(d+e)_0).
\end{multline}
The functor $\text{TS}$ is the Thom-Sebastiani equivalence, see Subsection \ref{TSiso}; the functors $q^*$ and $p_*$ are induced by the results in Subsection \ref{functo} from the functors $$D^b(\X^{ss}(d)\times\X^{ss}(e))\xrightarrow{q^*} D^b(\X^{ss}(d,e))\xrightarrow{p} D^b(\X^{ss}(d+e)).$$ 
Taking idempotent completions, we obtain functors:
\begin{multline}
D^{id}_{sg}(\X^{ss}(d)_0)\boxtimes D^{id}_{sg}(\X^{ss}(e)_0)\xrightarrow{TS}  D^{id}_{sg}\left((\X^{ss}(d)\times\X^{ss}(e))_0\right)\xrightarrow{q^*}\\
D^{id}_{sg}(\X^{ss}(d,e)_0)\xrightarrow{p_*} D^{id}_{sg}(\X^{ss}(d+e)_0).
\end{multline}

\end{defn}

Consider the $\Lambda_{\mu}$-graded vector spaces: 
\begin{itemize}
\item $KHA(Q,W)=\bigoplus_{d\in\Lambda_{\mu}} K_0\left(D_{sg}(\X^{ss}(d)_0)\right),$
\item $KHA^{id}(Q,W)=\bigoplus_{d\in\Lambda_{\mu}} K_0\left(D^{id}_{sg}(\X^{ss}(d)_0)\right).$

\end{itemize}

The multiplication maps on $KHA$ and $KHA^{id}$ are induced by these functors after passing to $K_0$.
\\

We next define the $T$-equivariant multiplication:

\begin{defn}\label{multiplication}
Consider the maps:

\begin{multline*}
    \X^{ss}(d)/T\times\X^{ss}(e)/T \xleftarrow{\delta} (\X^{ss}(d)\times\X^{ss}(e))/T \xleftarrow{q_{d,e}} \X^{ss}(d,e)/T \xrightarrow{p_{d,e}}\\ \X^{ss}(d+e)/T. 
\end{multline*}
They induce functors:
\begin{multline}
D_{sg}\left(\X^{ss}(d)_0/T\right)\boxtimes D_{sg}\left(\X^{ss}(e)_0/T\right)\xrightarrow{TS}  D_{sg}\left((\X^{ss}(d)/T\times\X^{ss}(e)/T)_0\right)\xrightarrow{\delta^*}\\
D_{sg}\left(\left(\X^{ss}(d)\times\X^{ss}(e)\right)_0/T\right)\xrightarrow{q^*}
D_{sg}\left(\X^{ss}(d,e)_0/T\right)\xrightarrow{p_*} D_{sg}\left(\X^{ss}(d+e)_0/T\right).
\end{multline}
We obtain similar functors after passing to $D_{sg}^{id}$.
\end{defn}
The functors above induce a multiplication of the $\Lambda_{\mu}$-graded vector space:
$$KHA_T(Q,W)=\bigoplus_{d\in\Lambda_{\mu}} K^T_0\left(D_{sg}\left(\X^{ss}(d)_0\right)\right).$$
There are similar definitions for the multiplication functors for $D^{id}_{sg}$ and they induce algebra structures on $\text{KHA}_T^{id}$. 
\\

In this section, we prove the following:

\begin{thm}\label{thm1}
The multiplications defined above for $KHA$, $KHA^{id}$, and their $T$-equivariant versions are associative. 
\end{thm}

Before we start proving the theorem, we collect some technical results.

\begin{prop}\label{com1}
Consider the cartesian diagram of smooth quotient stacks:

\begin{tikzcd}
  \mathcal{W} \arrow[r, "a"] \arrow[d, "d"]
    & \mathcal{Z} \arrow[d, "b"] \\
  \mathcal{X} \arrow[r, "c"]
&\mathcal{Y}
\end{tikzcd}
\\
with $a$ and $c$ proper. Consider a regular function $f:\Y\to\C$, and let $\mathcal{Y}_0$ etc. denote the zero fibers of $f$ etc. inside the quotient stacks. The following diagram commutes:

\begin{tikzcd}
  D_{sg}(\mathcal{X}_0) \arrow[r, "c_*"] \arrow[d, "d^*"]
    & D_{sg}(\mathcal{Y}_0) \arrow[d, "b^*"] \\
  D_{sg}(\mathcal{W}_0)\arrow[r, "a_*"]
&D_{sg}(\mathcal{Z}_0).
\end{tikzcd}
\\
 The similar statement holds for $D^{id}_{sg}$.
\end{prop}

\begin{proof}
The diagram for the zero fibers commutes:

\begin{tikzcd}
  D^b(\mathcal{X}_0) \arrow[r, "c_*"] \arrow[d, "d^*"]
    & D^b(\mathcal{Y}_0) \arrow[d, "b^*"] \\
  D^b(\mathcal{W}_0) \arrow[r, "a_*"]
& D^b(\mathcal{Z}_0).
\end{tikzcd}
\\
The functors $a_*,c_*, b^*, d^*$ extend to the categories $D_{sg}$, so the diagram will commute in those cases as well. After taking idempotent completions of the categories involved, the diagram commutes.
\end{proof}

\begin{prop}\label{com2}
Consider the commutative diagram of smooth quotient stacks: 

\begin{tikzcd}
  \mathcal{W} \arrow[r, "a"] \arrow[d, "d"]
    & \mathcal{Z} \arrow[d, "b"] \\
  \mathcal{X} \arrow[r, "c"]
&\mathcal{Y}. \end{tikzcd}
\\
Consider a regular function $f:\Y\to\C$, and let $\mathcal{Y}_0$ etc. denote the zero fibers of $f$ etc. inside the quotient stacks. The following diagram commutes:

\begin{tikzcd}
  D_{sg}(\mathcal{Y}_0) \arrow[r, "c^*"] \arrow[d, "b^*"]
    & D_{sg}(\mathcal{X}_0) \arrow[d, "b^*"] \\
  D_{sg}(\mathcal{Z}_0) \arrow[r, "a^*"]
&D_{sg}(\mathcal{W}_0).
\end{tikzcd}
\\
The similar diagram commutes for $D^{id}_{sg}$. Also, if all the morphisms are proper, the diagram with only proper pushforwards commutes for $D_{sg}$ and $D_{sg}^{id}$.
\end{prop}

\begin{proof}
The diagram for the zero fibers commutes:

\begin{tikzcd}
  D^b(\mathcal{Y}_0) \arrow[r, "c^*"] \arrow[d, "b^*"]
    & D^b(\mathcal{X}_0) \arrow[d, "b^*"] \\
  D^b(\mathcal{Z}_0) \arrow[r, "a^*"]
&D^b(\mathcal{W}_0)
\end{tikzcd}
\\
Pullback functors extend to $D_{sg}$, so the diagram commutes after passing to these categories. Further, if the maps are proper, the diagram

\begin{tikzcd}
  D^b(\mathcal{W}_0) \arrow[r, "a_*"] \arrow[d, "d_*"]
    & D^b(\mathcal{Z}_0) \arrow[d, "b_*"] \\
  D^b(\mathcal{X}_0) \arrow[r, "c_*"]
&D^b(\mathcal{Y}_0) \end{tikzcd}
\\
commutes. Proper pushforward functors extend to $D_{sg}$, so the diagram commutes after passing to these categories. By passing to idempotent completions, the diagrams also commute.

\end{proof}

We now begin the proof of Theorem \ref{thm1}:

\begin{proof}
We will show that multiplication is associative at the categorical level. We assume that $\theta=0$, and the general case follows from this one by restricting to the appropriate open sets. First, we treat the zero potential case. We need to show that the following diagram is commutative:

\begin{tikzcd}
  D^b(\X(d))\boxtimes D^b(\X(e))\boxtimes D^b(\X(f)) \arrow[r, "q_{e,f}^*"] \arrow[d, "q_{d,e}^*"]
    & D^b(\X(d))\boxtimes D^b(\X(e,f)) \arrow[d, "q_1^*"] \arrow[r, "p_{e,f*}"] & D^b(\X(d))\boxtimes D^b(\X(e+f)) \arrow[d, "q_{d,e+f}^*"] \\
  D^b(\X(d,e))\boxtimes D^b(\X(f)) \arrow[d, "p_{d,e*}"] \arrow[r, "q_2^*"] 
& D^b(\X(d,e,f)) \arrow[d, "p_*"] \arrow[r, "p_{e,f*}"] & D^b(\X(d,e+f)) \arrow[d, "p_{d,e+f*}"] \\
D^b(\X(d+e))\boxtimes D^b(\X(f)) \arrow[r, "q_{d+e,f*}"] & D^b(\X(d+e,f)) \arrow[r, "p_{d+e,f*}"] & D^b(\X(d+e+f)).
\end{tikzcd}
\\
Here $\X(d,e,f)$ is the stacks of triples of representations $0\subset A\subset B\subset C$ such that $A$ is of dimension $d$, $B/A$ is of dimension $e$, and $C/B$ is of dimension $f$. The upper right and the lower left corners come from cartesian diagrams, for example for the upper right corner:

\begin{tikzcd}
  \X(d,e,f) \arrow[r] \arrow[d]
    & \X(d,e+f) \arrow[d] \\
  \X(d)\times\X(e,f) \arrow[r]
&\X(d)\times\X(e+f).
\end{tikzcd}
\\
The maps from and to $\X(d,e,f)$ in the above diagram are defined via the cartesian diagrams for the upper right and lower left corners. In particular, these two corners commute from proper base change.
For the upper left corner, the following diagram commutes:

\begin{tikzcd}
  \X(d,e,f) \arrow[r] \arrow[d]
    & \X(d,e)\times \X(f) \arrow[d] \\
  \X(d+e,f) \arrow[r]
&\X(d+e)\times\X(f) 
\end{tikzcd}
\\
because both compositions send 
$$(V_0\subset V_1\subset V)\in \X(d,e,f)\text{ to }(V_1, V/V_1)\in \X(d+e)\times \X(f).$$ A similar proof works for the lower right corner.
\\

We next explain why the $T$ equivariant multiplication is associative. We need to show that the following diagram commutes:

\begin{tikzcd}
  D^b(\X(d)/T\times \X(e)/T\times \X(f)/T) \arrow[r, "q_{e,f}^*"] \arrow[d, "q_{d,e}^*"]
    & D^b(\X(d)/T\times \X(e,f)/T) \arrow[d, "q_1^*"] \arrow[r, "p_{e,f*}"] & D^b(\X(d)/T\times \X(e+f)/T) \arrow[d, "q_{d,e+f}^*"] \\
  D^b(\X(d,e)/T\times \X(f)/T) \arrow[d, "p_{d,e*}"] \arrow[r, "q_2^*"] 
& D^b(\X(d,e,f)/T) \arrow[d, "p_*"] \arrow[r, "p_{e,f*}"] & D^b(\X(d,e+f)/T) \arrow[d, "p_{d,e+f*}"] \\
D^b(\X(d+e)/T\times\X(f)/T) \arrow[r, "q_{d+e,f*}"] & D^b(\X(d+e,f)/T) \arrow[r, "p_{d+e,f*}"] & D^b(\X(d+e+f)/T).
\end{tikzcd}
\\
The lower left and upper right corners come from cartesian diagrams; the maps from and to $\X(d,e,f)/T$ are defined using these cartesian diagrams. Further, they commute by proper base change. The lower right corner commutes as before. For the upper left corner, consider the diagram:

\begin{tikzcd}
\X(d)/T\times\X(e)/T\times\X(f)/T &
(\X(d)\times\X(e))/T\times\X(f)/T\arrow{l}&
\X(d)/T\times\X(e,f)/T \arrow{l}\\
(\X(d)\times\X(e))/T\times \X(f)/T \arrow{u} &
(\X(d)\times\X(e)\times \X(f))/T \arrow{u}\arrow{l}
&(\X(d)\times\X(e,f))/T \arrow{l} \arrow{u}\\
\X(d,e)/T\times\X(f)/T\arrow{u}&
(\X(d,e)\times\X(f))/T\arrow{u}\arrow{l}& \X(d,e,f)/T.\arrow{u}\arrow{l}
\end{tikzcd}
\\
The lower right corner is commutative, and upper right and lower left corners are both cartesian. The upper left corner is also cartesian because the maps are induced from $$(\X(d)/T)\times(\X(e)/T)\times (\X(f)/T)\to BT\times BT\times BT$$ via the maps in the cartesian diagram
\begin{tikzcd}
  BT \arrow[r, "\delta"] \arrow[d, "\delta"]
    & BT\times BT \arrow[d, "\delta\times 1"] \\
  BT\times BT \arrow[r, "1\times\delta"]
& BT \times BT\times BT, 
\end{tikzcd}
which are induced by the maps in the commutative diagram
\begin{tikzcd}
  T \arrow[r, "\Delta"] \arrow[d, "\Delta"]
    & T\times T \arrow[d, "\Delta\times 1"] \\
  T\times T \arrow[r, "1\times \Delta"]
& T \times T\times T. 
\end{tikzcd}
\\
This proves the statement in the zero potential case. 
\\

Assume that $W$ is an arbitrary potential. Use Propositions \ref{com1} and \ref{com2} to see that the following diagram commutes:

\begin{tikzcd}
  D_{sg}\left((\X(d)\times\X(e)\times\X(f))_0\right) \arrow[r, "q_{e,f}^*"] \arrow[d, "q_{d,e}^*"]
    & D_{sg}\left((\X(d)\times \X(e,f))_0\right) \arrow[d, "q_1^*"] \arrow[r, "p_{e,f*}"] & D_{sg}\left((\X(d)\times \X(e+f))_0\right) \arrow[d, "q_{d,e+f}^*"] \\
  D_{sg}\left((\X(d,e)\times \X(f))_0\right) \arrow[d, "p_{d,e*}"] \arrow[r, "q_2^*"] 
& D_{sg}\left(\X(d,e,f)_0\right) \arrow[d, "p_*"] \arrow[r, "p_{e,f*}"] & D_{sg}\left(\X(d,e+f)_0\right) \arrow[d, "p_{d,e+f*}"] \\
D_{sg}\left((\X(d+e)\times \X(f))_0\right) \arrow[r, "q_{d+e,f*}"] & D_{sg}\left(\X(d+e,f)_0\right) \arrow[r, "p_{d+e,f*}"] & D_{sg}\left(\X(d+e+f)_0\right).
\end{tikzcd}
\\
Further, the diagrams for $D^{id}_{sg}$ and for the $T$-equivariant version commute as well. 
The diagrams we want to show commute are:

\begin{tikzcd}
D_{sg}(\X(d)_0)\boxtimes D_{sg}(\X(e)_0)\boxtimes D_{sg}(\X(f)_0) \arrow{d}{m_{d,e}\times\text{id}} \arrow{r}{\text{id}\times m_{e,f}} & D_{sg}(\X(d)_0)\boxtimes D_{sg}(\X(e+f)_0) \arrow{d}{m_{d,e+f}}\\
D_{sg}(\X(d+e)_0)\boxtimes D_{sg}(\X(f)_0) \arrow{r}{m_{d+e,f}} & D_{sg}(\X(d+e+f)_0),
\end{tikzcd}
\\
and its analogues. They differ by the above diagram by some Thom-Sebastiani equivalences, which we now take into account. First, the following diagram commutes:

\begin{tikzcd}
  D_{sg}(\X(d)_0)\boxtimes D_{sg}(\X(e)_0)\boxtimes D_{sg}(\X(f)_0) \arrow[r, "\text{id}\times i_*"] \arrow[d, "i_*\times \text{id}"]
    & D_{sg}(\X(d)_0)\boxtimes D_{sg}\left((\X(e)\times \X(f))_0\right) \arrow[d, "i_*"]  \\
      D_{sg}\left((\X(d)\times \X(e))_0\right)\boxtimes D_{sg}(\X(f)_0) \arrow[r, "i_*"] \arrow[d, "q_{d,e}^*"]
    & D_{sg}\left((\X(d)\times\X(e)\times \X(f))_0\right) \arrow[d, "q_1^*"] \\
  D_{sg}(\X(d,e)_0)\boxtimes D_{sg}(\X(f)_0) \arrow[d, "p_{d,e*}"] \arrow[r, "i_*"] 
& D_{sg}\left((\X(d,e)\times \X(f))_0\right) \arrow[d, "p_*"]\\
D_{sg}(\X(d+e)_0)\boxtimes D_{sg}(\X(f)_0) \arrow[r, "i_*"] & D_{sg}\left((\X(d+e)\times \X(f))_0\right).
\end{tikzcd}
\\
The top and lower squares are immediate. The middle square follows from proper base change applied to the diagram:

\begin{tikzcd}
\X(d,e)_0\times\X(f)_0 \arrow{r}{i} \arrow{d}{q_{d,e}} & (\X(d,e)\times\X(f))_0 \arrow{d}{q_{d,e}} \\
(\X(d)\times\X(e))_0\times \X(f)_0 \arrow{r}{i} & 
(\X(d)\times \X(e)\times \X(f))_0.
\end{tikzcd}
\\
There are similar commutative diagrams for $D_{sg}^{id}$ and their $T$-equivariant versions.
We can similarly show that the following diagram commutes:

\begin{tikzcd}
D_{sg}(\X(d)_0)\boxtimes D_{sg}\left((\X(e)\times \X(f))_0\right) \arrow{r}{p_{e,f*}q_{e,f}^*} \arrow{d}{i_*} & 
D_{sg}(\X(d)_0)\boxtimes D_{sg}(\X(e+f)_0) \arrow{d}{i_*} \\
D_{sg}\left((\X(d)\times\X(e)\times\X(f))_0\right) \arrow{r}{p_{e,f*}q_{e,f}^*} &
D_{sg}\left((\X(d)\times\X(e+f))_0\right).
\end{tikzcd}
\\
Once again, the analogous diagrams for $D_{sg}^{id}$ and their $T$-equivariant versions commute. Putting together these diagrams for the Thom-Sebastiani equivalences, we obtain that the multiplication is associative.

\end{proof}

\subsection{KHA using abelian stacks.} In \cite[Section 4.2]{d}, Davison explains how to define the multiplication on $\text{CoHA}$ using $T$-equivariant cohomology. In this subsection, we discuss the analogue for $\text{KHA}$. Consider the stacks $\mathcal{Y}(d)=R(d)/T(d)$ and $\mathcal{Z}(d)=R(d)/B(d)$. They come with attracting loci maps $p$ and $q$ as the stacks $\X(d)$. There are affine bundle maps $\Pi$ and proper maps $\pi$ as follows:
$$\Y(d)\xleftarrow{\Pi} \mathcal{Z}(d)\xrightarrow{\pi}\X(d).$$
We have a commutative diagram:

\begin{tikzcd}
\mathcal{Y}(d)\times\mathcal{Y}(e) & \mathcal{Z}(d)\times\mathcal{Z}(e) \arrow{l}{\Pi} \arrow{r}{\pi} & \X(d)\times\X(e) \\
\mathcal{Y}(d,e) \arrow{u}{q} \arrow{d}{p} & \mathcal{Z}(d,e) \arrow{l}{\Pi} \arrow{r}{\pi} \arrow{u}{q} \arrow{d}{p} & \X(d,e) \arrow{u}{q} \arrow{d}{p} \\
\mathcal{Y}(d+e) & \mathcal{Z}(d+e) \arrow{l}{\Pi} \arrow{r}{\pi} & \X(d+e).
\end{tikzcd}
\\
Consider the closed immersion $i:R(d,e)\to R(d+e),$ and define
$$\text{eu}_1(d,e):=i^*i_*(1)\in K_0(BT(d)).$$ Let $\text{eu}_2(d,e)$ be the Euler class of the 
proper map $G(d+e)/G(d,e)\to \text{pt}$. Denote by $K_0(D_{sg}(\mathcal{Y}(d+e)_0))_{\text{loc}}$ the localization of $K_0(D_{sg}(\mathcal{Y}(d+e)_0))$ at the ideal of $K_0(BG(d+e))$ generated by $\text{eu}_2(d,e)$. 

Define the abelian multiplication for the stacks $\mathcal{Y}(d)$ by:
$$m^{ab}:=p_*q^*:K_0(D_{sg}(\mathcal{Y}(d)_0))\boxtimes K_0(D_{sg}(\mathcal{Y}(e)_0))\to K_0(D_{sg}(\mathcal{Y}(d+e)_0)).$$
Consider the symmetrization map $$S_d: K_0(D_{sg}(\mathcal{Y}(d)_0))\to K_0(D_{sg}(\mathcal{Y}(d)_0))^{W_d}$$ which sends $x\to \sum_{\sigma\in W_{d}}\sigma\,x.$
Define similarly
$$S_{d,e}: K_0(D_{sg}(\mathcal{Y}(d+e)_0))^{W_d\times W_e}\to K_0(D_{sg}(\mathcal{Y}(d+e)_0))^{W_{d+e}}$$ by sending $x\to \sum_{\sigma\in W_{d+e}/W_d\times W_e} \sigma x$.
Define the multiplication $\widetilde{m}$ to be compared with the the multiplication $m$ from Definition \ref{multiplication} as follows:
\begin{multline*}
    \widetilde{m}:= S_{d,e}\,\text{eu}_2(d,e)^{-1}p_*q^*:K_0(D_{sg}(\mathcal{Y}(d)_0))^{W_d}\boxtimes K_0(D_{sg}(\mathcal{Y}(e)_0))^{W_e}\to\\ K_0(D_{sg}(\mathcal{Y}(d+e)_0))^{W_{d+e}}_{\text{loc}}.\end{multline*}

\begin{prop}\label{abel}
The multiplication $\widetilde{m}$ defined above coincides with the multiplication $m$ from Definition \ref{multiplication} via the isomorphism:
\begin{multline*}
\Theta_d: K_0(D_{sg}(\X(d)_0))\xrightarrow{\pi^*} 
K_0(D_{sg}(\mathcal{Z}(d)_0))\xrightarrow{(\Pi^*)^{-1}}K_0(D_{sg}(\mathcal{Y}(d)_0))\xrightarrow{S_d} \\
K_0(D_{sg}(\mathcal{Y}(d)_0))^{W_d}.
\end{multline*}
In other words, we have a commutative diagram:

\begin{tikzcd}
K_0(D_{sg}(\mathcal{X}(d)_0))\boxtimes K_0(D_{sg}(\mathcal{X}(e)_0))\arrow{r}{m_{d,e}} \arrow{d}{\Theta_d\boxtimes\Theta_e}&  K_0(D_{sg}(\X(d+e)_0))_{\text{loc}}\arrow{d}{\Theta_{d+e}} \\
K_0(D_{sg}(\mathcal{Y}(d)_0))^{W_d}\boxtimes K_0(D_{sg}(\mathcal{Y}(e)_0))^{W_e}\arrow{r}{\widetilde{m}_{d,e}} &  K_0(D_{sg}(\mathcal{Y}(d+e)_0))^{W_{d+e}}_{\text{loc}}.
\end{tikzcd}
\end{prop}

\begin{proof}
We suppress $D_{sg}$ from the notations to make the diagram easier to read. We first show that the following diagram commutes:

\begin{tikzcd}
K_0(\mathcal{Y}(d)\times\mathcal{Y}(e)) \arrow{d}{q^*} \arrow{r}{\Pi^*} & K_0(\mathcal{Z}(d)\times\mathcal{Z}(e)) \arrow{d}{q^*} & K_0(\X(d)\times\X(e)) \arrow{l}{\pi^*} \arrow{d}{q^*}\\
K_0(\mathcal{Y}(d,e))  \arrow{d}{p_*} \arrow{r}{\Pi^*} & K_0(\mathcal{Z}(d,e))   \arrow{d}{p_*} & K_0(\X(d,e))  \arrow{d}{p_*\,\text{eu}_2} \arrow{l}{\pi^*}\\
K_0(\mathcal{Y}(d+e)) \arrow{r}{\Pi^*} & K_0(\mathcal{Z}(d+e))  & K_0(\X(d+e)). \arrow{l}{\pi^*} \\
\end{tikzcd}
\\
The top squares clearly commute. The maps $\Pi^*$ are isomorphisms, and the maps $\pi^*$ are isomorphisms onto the $W$-fixed points. The lower left corner commutes from proper base change. To show that the lower right corner commutes, let 
$$\mathcal{W}=\left(R(d,e)\times_{P(d,e)} G(d+e)\right)/B(d+e),$$
which fits in a cartesian diagram:

\begin{tikzcd}
\mathcal{W} \arrow{r}{\widetilde{\pi}} \arrow{d}{\widetilde{p}} & \X(d,e) \arrow{d}{p}\\
\mathcal{Z}(d+e) \arrow{r}{\pi} & \X(d+e).
\end{tikzcd}
\\
Consider the inclusion map $i: \mathcal{Z}(d,e) \to \mathcal{W}$. Then $i\widetilde{\pi}=\pi$, $\widetilde{p}i=p$, and $i_*i^*(1)=\text{eu}_2$, so the following diagram commutes:

\begin{tikzcd}
 K_0(\mathcal{Z}(d,e))   \arrow{d}{p_*} & K_0(\X(d,e))  \arrow{d}{p_*\,\text{eu}_2} \arrow{l}{\Pi^*}\\
K_0(\mathcal{Z}(d+e))   & K_0(\X(d+e)). \arrow{l}{\Pi^*} \\
\end{tikzcd}
\\
Finally, we need to compare the abelian K-theory with its Weyl invariant part. This follows from the following commutative diagram:

\begin{tikzcd}
K_0(D_{sg}(\Y(d)_0))\boxtimes K_0(D_{sg}(\Y(e)_0))\arrow{r}\arrow{d}{S_d\boxtimes S_e}& K_0(D_{sg}(\Y(d+e)_0))\arrow{d}{S_{d+e}}\\
K_0(D_{sg}(\Y(d)_0))^{W_d}\boxtimes K_0(D_{sg}(\Y(e)_0))^{W_e}\arrow{r}{S_{d,e}}& K_0(D_{sg}(\Y(d+e)_0))^{W_{d+e}}.
\end{tikzcd}
\end{proof}

An immediate corollary is:
\begin{cor}\label{image}
 The image of $\widetilde{m}_{d,e}$ lies in $$K_0\left(D_{sg}(\mathcal{Y}(d+e)_0)\right)^{W_{d+e}}\subset K_0\left(D_{sg}(\mathcal{Y}(d+e)_0)\right)_{\text{loc}}^{W_{d+e}}.$$
\end{cor}

\textbf{Remark.} One can define $m^{ab}$ and $\widetilde{m}$ for $D_{sg}^{id}$ and in the $T$-equivariant case. They can be compared to the multiplication from Definition \ref{multiplication} as above.

\section{KHA as a bialgebra}\label{bialgebra2}
\subsection{Definition of the coproduct}
We will continue using the notations from the previous section, so $(Q,W)$ is a quiver with potential and $\theta\in\mathbb{Q}^I$ is a King stability condition.
Let $d,e\in\mathbb{N}^I$ be dimension vectors. Recall the maps:
$$\X^{ss}(e)\times\X^{ss}(e)\xleftarrow{q_{d,e}}\X^{ss}(d,e)\xrightarrow{p_{d,e}}\X^{ss}(d+e)$$
and we denote by the same letters the maps on the representations, on the corresponding groups, and on the stacks $\Y(d)=R(d)/T(d)$, and on their semistable loci, for example:
$$R(d)\times R(e)\xleftarrow{q_{d,e}}R(d,e)\xrightarrow{p_{d,e}} R(d+e).$$
Consider the $G(d,e)\times\C^*$ representation $R(d,e)\oplus\C$, 
where $\C^*$ acts on $R(d,e)$ by the character
$$\lambda_{d,e}=(z\text{Id}_d,\text{Id}_e):\C^*\to G(d,e),$$
so in particular with weights $0$ on $R(d)\times R(e)$ and with weights $1$ on $R(d,e)^{\lambda_{d,e}>0}$, $\C^*$ acts with weights $1$ on $\C$, and $G(d,e)$ acts trivially on $\C$. Consider the quotient stacks:
$$\widetilde{\X}^{ss}(d,e)=R^{ss}(d,e)\oplus\C/G(d,e)\times\C^*\subset\widetilde{\X}(d,e)=R(d,e)\oplus\C/G(d,e)\times\C^*.$$
The stack $\widetilde{\X}(d,e)$ comes with natural maps:

\begin{tikzcd}
\widetilde{\mathcal{X}}(d,e) \arrow{d}{\widetilde{q}_{d,e}} \arrow{r}{\tau} & \mathcal{X}(d,e) \arrow{r}{p_{d,e}} & \mathcal{X}(d+e)\\
\mathcal{X}(d)\times\mathcal{X}(e)
\end{tikzcd}
\\
where 
$\tau:\widetilde{X}(d,e)\to\X(d,e)=R(d,e)/G(d,e)$ is induced by the projections onto the first factors $\text{pr}_1: R(d,e)\oplus\C\to R(d,e)$ and $\text{pr}_1: G(d,e)\times\C^*\to G(d,e)$. The morphism: 
$$\widetilde{q}_{d,e}:\widetilde{X}(d,e)\to \X(d)\times\X(e)$$ is induced by
the projections $R(d,e)\oplus\C\xrightarrow{\text{pr}_1} R(d,e)\xrightarrow{q_{d,e}} R(d)\times R(e)$ and $G(d,e)\times\C^*\xrightarrow{\text{pr}_1} G(d,e)\xrightarrow{q_{d,e}} G(d)\times G(e)$.
These maps restricts to similar maps for $\widetilde{\X}^{ss}$.
Define similarly the stacks:
$$\widetilde{\mathcal{Y}}^{ss}(d,e):=R^{ss}(d,e)\oplus\C/T(d,e)\times\C^*\subset\widetilde{\mathcal{Y}}(d,e):=R(d,e)\oplus\C/T(d,e)\times\C^*.$$
The stack $\widetilde{\Y}(d,e)$ comes with analogous maps:

\begin{tikzcd}
\widetilde{\mathcal{Y}}(d,e) \arrow{d}{\widetilde{q}_{d,e}} \arrow{r}{\tau} & \mathcal{Y}(d,e) \arrow{r}{p_{d,e}} & \mathcal{Y}(d+e)\\
\mathcal{Y}(d)\times\mathcal{Y}(e).
\end{tikzcd}
\\
These maps restricts to similar maps for $\widetilde{\Y}^{ss}$.
Both maps $\widetilde{q}$ are cohomologically proper, and we can define pushforward along them as follows.

\begin{defn}\label{proper}
Suppose we have a map of quotient stacks $$\pi: W/P\times\C^*\to V/L,$$ where $P$ and $L$ are affine algebraic groups with $L\subset P$, $W$ is a $P$ representation, $V$ is an $L$ representation such that $W^L=V$, 
and $\C^*$ fixes $V$ and acts with weights $1$ on $N_{V/W}$. We denote by 
$$\pi_*:D^b\,(W/P\times\C^*)\to D^b\,(V/L)$$
the composition of the maps $\pi_*:=p_*j^*i_*$, where 
\begin{itemize}
    \item $i:W/P\times\C^*\hookrightarrow W/L\times\C^*$ is a closed immersion,
    \item $j:W-V/L\times\C^*\hookrightarrow W/L$ is an open immersion,
    \item $p:W-V/L\times\C^*=\mathbb{P}_{V/L}(N_{V/W})\to V/L$ is a proper map.
\end{itemize}
We will use the same notation for the map induced on categories of singularities and on their $K_0$. The analogous functors exist for vanishing cohomology and we will use the same notations in that context.
\end{defn}

\begin{prop}
Consider a cartesian diagram of quotient stacks of the form $V/G$, where $G$ is an affine algebraic group and 
$V$ is a $G$ representation:

\begin{tikzcd}
\X^1\arrow{d}{b} \arrow{r}{\pi^1}& \Y^1 \arrow{d}{a}\\
\X^2 \arrow{r}{\pi^2}& \Y^2,
\end{tikzcd}
\\ 
such that the maps $\pi^1$ and $\pi^2$ are proper as in Definition \ref{proper}. Then proper base change holds: $a^*\pi^1_*=\pi^2_*b^*$.
\end{prop}

\begin{proof}
Consider the diagram:

\begin{tikzcd}
\mathcal{W}^1\arrow{d}{d} \arrow{r}{j^1}& \mathcal{Z}^1\arrow{d}{c}& \X^1\arrow{l}{i^1}\arrow{r}{\pi^1}\arrow{d}{b}& \Y^1\arrow{d}{a}\\
\mathcal{W}^2\arrow{r}{j^2}& \mathcal{Z}^2& \X^2\arrow{l}{i^2}\arrow{r}{\pi^2}& \Y^2,
\end{tikzcd}
\\
where the stacks are in Definition \ref{proper}, and all squares are cartesian. In particular, we have that $d^*j^{1*}=j^{2*}c^*$ and $c^*i^1_*=i_*^2b^*$. We also have a cartesian square:

\begin{tikzcd}
\mathcal{W}^1\arrow{d}{b} \arrow{r}{p^1}& \Y^1 \arrow{d}{a}\\
\mathcal{W}^2 \arrow{r}{p^2}& \Y^2.
\end{tikzcd}
\\
Using proper base change for this square, we get that $a^*p^1_{*}=p^2_*d^*$.
We need to check that: $$a^*\pi_{*}^1=a^*p^1_{*}j^{1*}i^1_*=p_*^2j^{2*}i^2_*a^*=\pi^2_*b^*,$$
which follows by combining the above three equalities.
\end{proof}

\begin{defn}\label{coproduct}
Define a coproduct map using the stacks $\X$ by:
\begin{multline*}
    \Delta_{d,e}:=TS^{-1}\text{sw}_{e,d}\,\widetilde{q}_{e,d*}\tau^*p_{e,d}^*: K_0\left(D_{sg}(\X^{ss}(d+e)_0)\right)\to\\ K_0\left(D_{sg}(\X^{ss}(d)_0)\right)\boxtimes K_0\left(D_{sg}(\X^{ss}(e)_0)\right).
\end{multline*}
We use the notation $TS$ for the Thom-Sebastiani isomorphism and 
$$ K_0\left(D_{sg}(\X^{ss}(e)_0)\right)\boxtimes K_0\left(D_{sg}(\X^{ss}(d)_0)\right)\xrightarrow{\text{sw}_{e,d}} K_0\left(D_{sg}(\X^{ss}(d)_0)\right)\boxtimes K_0\left(D_{sg}(\X^{ss}(e)_0)\right)$$ is the functor that sends $A\boxtimes B$ to $B\boxtimes A$.
The $T$-equivariant version is defined by:
\begin{multline*}
D_{sg}(\X^{ss}(d+e)_0/T)\xrightarrow{\text{sw}_{e,d}\,\widetilde{q}_{e,d*}\tau^*p_{e,d}^*} D_{sg}\left((\X^{ss}(d)\times\X^{ss}(e))_0/T\right) \xrightarrow{\mu^*}\\
D_{sg}\left((\X^{ss}(d)\times\X^{ss}(e))_0/T\times T\right)\xrightarrow{TS^{-1}} D_{sg}(\X^{ss}(d)_0/T)\boxtimes D_{sg}(\X^{ss}(e)_0/T).
\end{multline*}
We obtain functors for $D_{sg}^{id}$ after passing to the idempotent completion.
We define the abelian coproduct by:
\begin{multline*}
    \Delta^{ab}_{d,e}:=TS^{-1}\,\widetilde{q}_{e,d*}\tau^*p_{e,d}^*: K_0\left(D_{sg}(\mathcal{Y}^{ss}(d+e)_0)\right)\to\\ K_0\left(D_{sg}(\mathcal{Y}^{ss}(e)_0)\right)\boxtimes K_0\left(D_{sg}(\mathcal{Y}^{ss}(d)_0)\right).
\end{multline*}
We define the coproduct to be compared with $\Delta$ by: 
\begin{multline*}
    \widetilde{\Delta}_{d,e}:=TS^{-1}\text{sw}_{e,d}\,\widetilde{q}_{e,d*}\tau^*p_{e,d}^*(\text{eu}_2(d,e)-): K_0\left(D_{sg}(\mathcal{Y}^{ss}(d+e)_0)\right)\to\\ K_0\left(D_{sg}(\mathcal{Y}^{ss}(d)_0)\right)\boxtimes K_0\left(D_{sg}(\mathcal{Y}^{ss}(e)_0)\right).
\end{multline*}
We will use the same notation for the above restriction to the Weyl-fixed subspace:
$$\widetilde{\Delta}_{d,e}: K_0\left(D_{sg}(\mathcal{Y}^{ss}(d+e)_0)\right)^{W_{d+e}}\to\\ K_0\left(D_{sg}(\mathcal{Y}^{ss}(d)_0)\right)^{W_d}\boxtimes K_0\left(D_{sg}(\mathcal{Y}^{ss}(e)_0)\right)^{W_e}.$$
There are similar definitions for $D_{sg}^{id}$ and in the $T$-equivariant versions.
\\

All the functors admit version in cohomology because vanishing cycles commute with proper pushforward, we have natural maps $\tau^*p_{d,e}^*\varphi\to \varphi \tau^*p_{d,e}^*$, and the Thom-Sebastiani theorem holds.
\end{defn}

\begin{prop}\label{abelianyetagain}
Recall the injections from Subsection \ref{abeliansubsection}: $$(\Pi^*)^{-1}\pi^*:K_0(D_{sg}(\X^{ss}(d)_0))\to K_0(D_{sg}(\Y^{ss}(d)_0)).$$ The following diagram commutes:

\begin{tikzcd}
K_0(D_{sg}(\X^{ss}(d+e)_0)) \arrow{r}{\Delta_{d,e}} \arrow{d}{(\Pi^*)^{-1}\pi^*}& K_0(D_{sg}(\X^{ss}(d)_0))\boxtimes\,K_0(D_{sg}(\X^{ss}(e)_0))\arrow{d}{(\Pi^*)^{-1}\pi^*}\\
K_0(D_{sg}(\Y^{ss}(d+e)_0))\arrow{r}{\widetilde{\Delta}_{d,e}}& K_0(D_{sg}(\Y^{ss}(d)_0))\boxtimes\,K_0(D_{sg}(\Y^{ss}(e)_0)).
\end{tikzcd}
\\
Further, consider the isomorphism from Subsection \ref{abeliansubsection}:
$$\Theta_d: K_0(D_{sg}(\X^{ss}(d)_0))\xrightarrow{(\Pi^*)^{-1}\pi^*} 
K_0(D_{sg}(\mathcal{Y}^{ss}(d)_0))\xrightarrow{S_d} K_0(D_{sg}(\mathcal{Y}^{ss}(d)_0))^{W_d}.$$
Let $S_{d,e}$ be the map:
$$S_{d,e}: K_0(D_{sg}(\mathcal{Y}^{ss}(d+e)_0))^{W_d\times W_e}\to K_0(D_{sg}(\mathcal{Y}^{ss}(d+e)_0))^{W_{d+e}}$$ that sends $x\to \sum_{\sigma\in W_{d+e}/W_d\times W_e} \sigma x$.
The following diagram commutes:

\begin{tikzcd}
K_0(D_{sg}(\X^{ss}(d+e)_0))\arrow{r}{\Delta_{d,e}}\arrow{d}{\Theta_{d+e}} &
K_0(D_{sg}(\mathcal{X}^{ss}(d)_0))\boxtimes K_0(D_{sg}(\mathcal{X}^{ss}(e)_0)) \arrow{d}{S_{d,e}(\Theta_d\boxtimes\Theta_e)} \\
K_0(D_{sg}(\mathcal{Y}^{ss}(d+e)_0))^{W_{d+e}}\arrow{r}{\widetilde{\Delta}_{d,e}}&
K_0(D_{sg}(\mathcal{Y}^{ss}(d)_0))^{W_d}\boxtimes K_0(D_{sg}(\mathcal{Y}^{ss}(e)_0))^{W_e}.
\end{tikzcd}
\end{prop}

\begin{proof}
The proof is similar to the one for Proposition \ref{abel}; we will use the same notations for the morphisms and we will drop $D_{sg}$ and the semistable notations. Consider the diagram:

\begin{tikzcd}
K_0(\mathcal{Y}(e)\times\mathcal{Y}(d))  \arrow{r}{\Pi^*} & K_0(\mathcal{Z}(e)\times\mathcal{Z}(d))  & K_0(\X(e)\times\X(d)) \arrow{l}{\pi^*} \\
K_0(\widetilde{\mathcal{Y}}(e,d))  \arrow{r}{\Pi^*} \arrow{u}{\widetilde{q}_*(\text{eu}_2)}& K_0(\widetilde{\mathcal{Z}}(e,d))   \arrow{u}{\widetilde{q}_*} & K_0(\widetilde{\mathcal{X}}(e,d))  \arrow{u}{\widetilde{q}_*} \arrow{l}{\pi^*}\\
K_0(\mathcal{Y}(d+e)) \arrow{u}{\tau^* p^*}\arrow{r}{\Pi^*} & K_0(\mathcal{Z}(d+e)) \arrow{u}{\tau^* p^*}  & K_0(\X(d+e)). \arrow{l}{\pi^*} \arrow{u}{\tau^* p^*}\\
\end{tikzcd}
\\
It is immediate to see that the bottom two squares commute. The top right square commutes from proper base change. For the top left corner, let 
$$\mathcal{Z}'=R(e,d)\oplus\C/B(e)\times B(d)\times\C^*.$$ The stack $\mathcal{Z}'$ fits in a cartesian diagram:

\begin{tikzcd}
\widetilde{\Y}(e,d)\arrow{d}{\widetilde{q}} & \mathcal{Z}' \arrow{d}{\widetilde{q}} \arrow{l}{\Pi}\\
\Y(e)\times\Y(d) & \mathcal{Z}(e)\times\mathcal{Z}(d). \arrow{l}{\Pi\times\Pi} 
\end{tikzcd}
\\
After using proper base change for this diagram, it is enough to show that $i^*i_*(1)=\text{eu}_2(d,e)$ for the closed immersion 
$$i:\widetilde{\mathcal{Z}}(e,d)=R(e,d)\oplus\C/B(d+e)\times\C^*\to \mathcal{Z}'=R(e,d)\oplus\C/B(e)\times B(d)\times\C^*,$$ which follows from the fact that $\text{eu}_2=s^*s_*(1)$ is the Euler class of the immersion $$s:B(e)\times B(d)\hookrightarrow B(d+e).$$
This ends the proof that the first diagram commutes.
Finally, we compare the abelian K-theory and its Weyl invariant part by the diagram:

\begin{tikzcd}
K_0(\Y(d+e)) \arrow{d} \arrow{r}{S_{d+e}}& K_0(\Y(d+e))^{W_{d+e}}\arrow{d}\\
K_0(\Y(d))\boxtimes K_0(\Y(e)) \arrow[r, "S_{d,e}(S_d\boxtimes S_e)"]& K_0(\Y(d))^{W_d}\boxtimes K_0(\Y(e))^{W_e}.
\end{tikzcd}

\end{proof}

\textbf{Remark.} The analogues statements hold for $D_{sg}^{id}$, in the $T$-equivariant cases, and in cohomology.
\\


\begin{prop}\label{coassoc}
The coproduct $\Delta$ on $\text{KHA}_T$ is coassociative.
\end{prop}

\begin{proof}
Let $d,e,f\in\mathbb{N}^I$ be dimension vectors. We need to show that the following diagram commutes:

\begin{tikzcd}
K^T(d+e+f) \arrow{r}{\Delta_{d,e+f}} \arrow{d}{\Delta_{d+e,f}}& K^T(d)\boxtimes K^T(e+f) \arrow{d}{1\boxtimes \Delta_{e,f}}\\
K^T(d+e)\boxtimes K^T(f) \arrow{r}{\Delta_{d,e}\boxtimes 1} & K^T(d)\boxtimes K^T(e) \boxtimes K^T(f),
\end{tikzcd}
\\
where we use the notation $K^T(d)=K^T_0(D_{sg}(\X^{ss}(d)_0))$. We will prove the corresponding categorical statement. We assume that the stability condition is trivial, and the conclusion for arbitrary $\theta$ follows by restricting to the appropriate open sets. We first discuss the non-equivariant statement.
We will use the definition of comultiplication $\widetilde{\Delta}$ via the stacks $\mathcal{Y}$; the conclusion will follow from Proposition \ref{abelianyetagain}.  
The $\text{eu}_2$ factors appearing in the definition of $\widetilde{\Delta}$ are equal because: 
$$\text{eu}_2(d,e+f)\,\text{eu}_2(e,f)=\text{eu}_2(d+e,f)\,\text{eu}_2(d,e).$$
We need to check that the coproduct $\Delta^{ab}$ is coassociative for categories:

\begin{tikzcd}
D_{sg}(\Y(d+e+f)_0)\arrow{r}{\Delta^{ab}_{d,e+f}}\arrow{d}{\Delta^{ab}_{d+e,f}}& D_{sg}(\Y(e+f)_0)\boxtimes D_{sg}(\Y(d)_0)\arrow{d}{\text{Id}\boxtimes \Delta^{ab}_{e,f}}\\
D_{sg}(\Y(f)_0)\boxtimes D_{sg}(\Y(d+e)_0)\arrow{r}{\Delta^{ab}_{d,e}\boxtimes\text{Id}}& D_{sg}(\Y(f)_0)\boxtimes D_{sg}(\Y(e)_0)\boxtimes D_{sg}(\Y(d)_0).
\end{tikzcd}
\\
We first check the statement in the potential zero case, so we need to show that the following diagram commutes:

\begin{tikzcd}
D^b(\mathcal{Y}(d+e+f)) \arrow{d}{\tau^*p^*} \arrow{r}{\tau^*p^*} & D^b(\widetilde{\mathcal{Y}}(e+f,d)) \arrow{d}{\tau^*p^*} \arrow{r}{\widetilde{q}_*} & D^b(\mathcal{Y}(e+f)\times \mathcal{Y}(d)) \arrow{d}{\tau^*p^*}\\
D^b(\widetilde{\mathcal{Y}}(f,d+e)) \arrow{d}{\widetilde{q}_*} \arrow{r}{\tau^*p^*} & D^b(R(f,e,d)\oplus\C^2/T'\times(\C^*)^2) \arrow{d}{\widetilde{q}_*} \arrow{r}{\widetilde{q}_*} & D^b(\widetilde{\mathcal{Y}}(f,e)\times \mathcal{Y}(d)) \arrow{d}{\widetilde{q}_*}\\
D^b(\mathcal{Y}(f)\times \mathcal{Y}(d+e)) \arrow{r}{\tau^*p^*} & D^b(\mathcal{Y}(f) \times \widetilde{\mathcal{Y}}(e,d)) \arrow{r}{\widetilde{q}_*} & D^b(\mathcal{Y}(f)\times \mathcal{Y}(e)\times \mathcal{Y}(d)). 
\end{tikzcd}
\\
Here, $(\C^*)^2\subset T'=T(d+e+f)$ is the torus generated by the characters $(\text{Id}_d,z\text{Id}_e,\text{Id}_f)$ and $(\text{Id}_d,\text{Id}_e,w\text{Id}_f)$, and $\C^2$ is its adjoint representations. The group $(\C^*)^2$ acts on $R(f,e,d)$ via $T'$. 
The definition is such that the lower left corner commutes by proper base change; the top right corner also commutes by proper base change.
It is clear that the other two squares commute. 
\\

In the $T$-equivariant case, we need to show that the diagram:

\begin{tikzcd}
D^b(\Y(d+e+f)/T)\arrow{r}{\Delta^{ab}_{d,e+f}}\arrow{d}{\Delta^{ab}_{d+e,f}}& D^b(\Y(e+f)/T\times \Y(d)/T)\arrow{d}{\Delta^{ab}_{e,f}\boxtimes\text{Id}}\\
D^b(\Y(f)/T\times \Y(d+e)/T)\arrow{r}{\text{Id}\boxtimes\Delta^{ab}_{d,e}}& D^b(\Y(f)/T\times \Y(e)/T\times \Y(e)/T)
\end{tikzcd}
\\
commutes. We write the diagram as follows:

\begin{tikzcd}
D^b_T(\Y(d+e+f)) \arrow{r}{\Delta^{ab}_{d,e+f}} \arrow{d}{\Delta^{ab}_{d+e,f}}&
D^b_T\left(\Y(e+f)\times \Y(d)\right) \arrow{r}{\mu^*} \arrow{d}{\Delta^{ab}_{e,f}\boxtimes\text{Id}}&
D^b_{T^2}(\Y(e+f)\times \Y(d)) \arrow{d}{\Delta^{ab}_{e,f}\boxtimes\text{Id}}\\
D^b_T(\Y(f)\times\Y(d+e)) \arrow{d}{\mu^*} \arrow{r}{\text{Id}\boxtimes\Delta^{ab}_{d,e}}&
D^b_T\left(\Y(f)\times \Y(e)\times \Y(d)\right) \arrow{r}{\mu^*}\arrow{d}{\mu^*}&
D^b_{T^2}\left(\left(\Y(f)\times \Y(e)\right)\times\Y(d)\right)  \arrow{d}{\mu^*\boxtimes\text{Id}}\\
D^b_{T^2}(\Y(f)\times\Y(d+e)) \arrow{r}{\text{Id}\boxtimes\Delta^{ab}_{d,e}}& D^b_{T^2}\left(\Y(f)\times (\Y(d)\times\Y(e))\right) \arrow{r}{\text{Id}\boxtimes\mu^*}& D^b_{T^3}\left(\Y(f)\times\Y(e)\times\Y(d)\right).
\end{tikzcd}
\\
One proceeds as in the proof of Theorem \ref{thm1} to argue that the four squares commute. 
The top left corner commutes from the proof in the non-equivariant case.
The top right and bottom left corners can be split each in two squares, one of which is immediate to see to commute, and the other commutes from proper base change. The bottom right corner commutes as in Theorem \ref{thm1}, it amounts to proving that the following diagram commutes:

\begin{tikzcd}
D^b(BT)\arrow{r}{\mu^*}\arrow{d}{\mu^*}& D^b(BT\times BT)\arrow{d}{\text{Id}\times \mu^*}\\
D^b(BT\times BT)\arrow{r}{\mu^*\times \text{Id}}& D^b(BT\times BT\times BT),
\end{tikzcd}
\\
which follows from the associativity of multiplication on $T$.
\\

We next assume that $W$ is arbitrary. 
For simplicity of notation, we write the proof in the non-equivariant case, but the proof in the general case is the same. From the zero potential case, we deduce that the following diagram commutes:

\begin{tikzcd}
D_{sg}(\Y(d+e+f)_0)\arrow{r}{\Delta^{ab}_{d,e+f}}\arrow{d}{\Delta^{ab}_{d+e,f}}& D_{sg}((\Y(e+f)\times\Y(d))_0)\arrow{d}{\Delta^{ab}_{e,f}\times\text{Id}}\\
D_{sg}((\Y(f)\times \Y(d+e))_0)\arrow{r}{\text{Id}\times\Delta^{ab}_{d,e}}& D_{sg}((\Y(f)\times\Y(e)\times\Y(d))_0).
\end{tikzcd}
\\
This diagram differs from the coassociativity diagram by some Thom-Sebastiani isomorphisms, and we deduce the statement as in the proof of Theorem \ref{thm1}, for example, the following diagram commutes, where we use the top square by proper base change:

\begin{tikzcd}
D_{sg}((\Y(e+f)\times\Y(d))_0)\arrow{d}{\tau^*p^*}& D_{sg}(\Y(e+f)_0)\boxtimes D_{sg}(\Y(d)_0)\arrow{d}{\tau^*p^*}\arrow{l}{i_*}\\
D_{sg}((\widetilde{\Y}(f,e)\times \Y(d))_0)\arrow{d}{\widetilde{q}_*}& D_{sg}(\widetilde{\Y}(f,e)_0)\boxtimes D_{sg}(\Y(d)_0)\arrow{d}{\widetilde{q}_*}\arrow{l}{i_*}\\
D_{sg}((\Y(f)\times\Y(e)\times\Y(d))_0)& D_{sg}((\Y(f)\times\Y(e)_0)\boxtimes D_{sg}(\Y(d)_0).\arrow{l}{i_*}
\end{tikzcd}
\\
The proofs for $D_{sg}^{id}$ and in the $T$-equivariant cases are similar.

\end{proof}

\subsection{Braiding} 
We next define a braiding $R$ on the $\text{KHA}$. Let $d, e\in\mathbb{N}^I$ be two dimension vectors. We will use the notation $$K(d)=K_0(D_{sg}(\X^{ss}(d)_0))\cong K_0(D_{sg}(\Y^{ss}(d)_0))^{W_d}$$ in this subsection.
Define $$R_{d,e}^{ab}=\text{sw}_{d,e}\Delta^{ab}_{e,d}m_{d,e}^{ab}:
D_{sg}\left(\mathcal{Y}^{ss}(d)_0\times\mathcal{Y}^{ss}(e)_0\right)\to D_{sg}\left(\mathcal{Y}^{ss}(e)_0\times\mathcal{Y}^{ss}(d)_0\right).$$
Passing to idempotent completions we obtain a definition of the braiding for $D_{sg}^{id}$.
\\

Let $m_{d,e}'=m^{ab}_{d,e}\,\text{eu}_2^{-1}(d,e)$ and 
$\Delta'_{e,d}=\Delta^{ab}_{e,d}\,\text{eu}_2(e,d)$.
Define $R_{d,e}=\Delta'_{e,d}m_{d,e}'$:
$$
K_0(D_{sg}\left(\mathcal{Y}^{ss}(d)_0))\boxtimes K_0(D_{sg}(\mathcal{Y}^{ss}(e)_0\right))\xrightarrow{R_{d,e}} 
K_0(D_{sg}\left(\mathcal{Y}^{ss}(e)_0))\boxtimes K_0(D_{sg}(\mathcal{Y}^{ss}(d)_0\right)).$$ We do not need to localize in the above definition by Corollary \ref{image}. After taking $W_d\times W_e$ invariants, we obtain maps:
$$R_{d,e}, R_{d,e}^{ab}:
K(d)\boxtimes \, K(e) \to
K(e)\boxtimes K(d).$$
The definitions for $D_{sg}^{id}$ and in the $T$-equivariant cases are similar.

\begin{prop}\label{braiding}
The matrices $R_{d,e}^{ab}$ and $R_{d,e}$ define braidings, that is, 
$$R^{ab}_{d+e,f}=\left(R^{ab}_{d,f}\boxtimes 1\right)\left(1\boxtimes R^{ab}_{e,f}\right)\text{ and }R_{d+e,f}=\left(R_{d,f}\boxtimes 1\right)\left(1\boxtimes R_{e,f}\right).$$
\end{prop}

\begin{proof}
We first discuss the proof for $D_{sg}$, the arguments for $D_{sg}^{id}$ is similar. 
It suffices to prove the first statement, because the $\text{eu}_2$ factors are equal:
$$\frac{\text{eu}_2(f,d+e)}{\text{eu}_2(d+e,f)}=\frac{\text{eu}_2(f,d)}{\text{eu}_2(d,f)}\,\frac{\text{eu}_2(f,e)}{\text{eu}_2(e,f)}.$$
Further, it suffices to prove the braiding property before taking Weyl invariants. We need to show that the following diagram commutes:

\begin{tikzcd}
K(d)\boxtimes K(e)\boxtimes K(f) \arrow{r}{m^{ab}_{d,e}\boxtimes 1}\arrow{d}{1\boxtimes m^{ab}_{e,f}}& K(d+e)\boxtimes K(f) \arrow{r}{\text{sw}\,\Delta^{ab}_{e,d}\boxtimes 1} \arrow{d}{m^{ab}_{d+e,f}} & K(e)\boxtimes K(d)\boxtimes K(f)\arrow{d}{1\boxtimes m^{ab}_{d,f}}\\
K(d)\boxtimes K(e+f)\arrow{r}{m^{ab}_{d+e,f}}& 
K(d+e+f)\arrow{d}{\text{sw}\,\Delta^{ab}_{e+f,d}}\arrow{r}{\text{sw}\,\Delta^{ab}_{e,d+f}}& K(e)\boxtimes K(d+f)\arrow{d}{1\boxtimes \text{sw}\,\Delta^{ab}_{f,d}}\\
 & K(e+f)\boxtimes K(d)\arrow{r}{\text{sw}\,\Delta^{ab}_{e,f}\boxtimes 1}& K(e)\boxtimes K(f)\boxtimes K(d).
\end{tikzcd}
\\
The top left corner and the lower right corner commute from Theorem \ref{thm1} and Proposition \ref{coassoc}. 
To show that the top right corner commutes, we prove the corresponding categorical statement. We can thus assume that the stability condition is trivial. We first prove the zero potential case, and then deduce the general potential case as in the proofs of Theorem \ref{thm1} and Propositions \ref{coassoc}. Consider the $T\times \C^*$ representation: 
$$R'(d,e,f)=R(d)\oplus R(e)\oplus R(f)\oplus R(d,e)^{>0}\oplus R(d,f)^{>0},$$ where the $\C^*$ action is induced by $\lambda_{d,e}:\C^*\to T(d+e)\hookrightarrow T(d+e+f)$.
Define the stacks $\Y(d,e,f)$ and $\overline{\Y}(d,e,f)$ by: 
$$\widetilde{\Y}(d,e,f)=R'(d,e,f)/T\text{ and } \overline{\Y}(d,e,f)=R'(d,e,f)\oplus\C/T\times\C^*.$$
The following diagram commutes:

\begin{tikzcd}
\Y(d+e)\times\Y(f) & \Y(d,e)\times\Y(f) \arrow{l}{p_{d,e}} & 
\widetilde{\Y}(d,e)\times\Y(f) \arrow{l}{\tau} \arrow{r}{\widetilde{q}_{d,e}} & \Y(d)\times\Y(e)\times\Y(f)\\
\Y(d+e,f) \arrow{d}{p_{d+e,f}} \arrow{u}{q_{d+e,f}} & \widetilde{\Y}(d,e,f) \arrow{d}{p} \arrow{l}{p} \arrow{u}{q} & \overline{\Y}(d,e,f) \arrow{u}{q} \arrow{l}{\tau} \arrow{r}{\widetilde{q}} \arrow{d}{p} & \Y(e)\times \Y(d,f)\arrow{u}{q_{d,f}} \arrow{d}{p_{d,f}}\\
\Y(d+e+f) & \Y(d+f,e) \arrow{l}{p_{e,d+f}} & \widetilde{\Y}(d+f,e) \arrow{l}{\tau} \arrow{r}{\widetilde{q}} & \Y(e)\times\Y(d+f).
\end{tikzcd}
\\
The maps denoted by $q$ above are the natural projections. 
The lower left, lower center, and top right squares are all cartesian, and we use them to define all the other maps from or to $\widetilde{\Y}(d,e,f)$ and $\overline{\Y}(d,e,f)$.
Using proper base change along these squares and the commutativity of the other squares, we obtain that the following diagram commutes:

\begin{tikzcd}
D^b(\Y(d+e))\boxtimes D^b(\Y(f)) \arrow{d}{m^{ab}_{d+e,f}} \arrow{r}{\text{sw}\,\Delta^{ab}_{e,d}\boxtimes\text{Id}}&
D^b(\Y(e))\boxtimes D^b(\Y(d))\boxtimes D^b(\Y(f)) \arrow{d}{\text{Id}\boxtimes m^{ab}_{d,f}}\\
D^b(\Y(d+e+f))\arrow{r}{\text{sw}\,\Delta^{ab}_{e,d+f}}& D^b(\Y(e))\boxtimes D^b(\Y(d+f)).
\end{tikzcd}
\\
For an arbitrary potential $W$, the diagram above implies that the following diagram commutes:

\begin{tikzcd}
D_{sg}(\Y(d+e)_0)\boxtimes D_{sg}(\Y(f)_0) \arrow{d}{m^{ab}_{d+e,f}} \arrow{r}{\text{sw}\,\Delta^{ab}_{e,d}\boxtimes\text{Id}}&
D_{sg}(\Y(e)_0)\boxtimes D_{sg}(\Y(d)_0)\boxtimes D_{sg}(\Y(f)_0) \arrow{d}{\text{Id}\boxtimes m^{ab}_{d,f}}\\
D_{sg}(\Y(d+e+f)_0)\arrow{r}{\text{sw}\,\Delta^{ab}_{e,d+f}}& D_{sg}(\Y(e)_0)\boxtimes D_{sg}(\Y(d+f)_0).
\end{tikzcd}
\\
One can check the compatibility with the Thom-Sebastiani isomorphism as in the proofs of Theorem \ref{thm1} and Proposition \ref{coassoc}, and this ends the proof.

\end{proof}

\subsection{Compatibility between the product and the coproduct} 
In this subsection, we use the notations $K(d)=K_0(D_{sg}(\X^{ss}(d)_0))$ and $K(d,e)=K_0(D_{sg}(\X^{ss}(d,e)_0))$. 
The theorem that we prove in this section is:

\begin{thm}\label{bialgebra}
The product, coproduct, and the $R$ matrix $(m, \Delta, R)$ define a braided bialgebra structure on $\text{gr}\,KHA$, that is, for dimension vectors $d,d',e,e'\in\mathbb{N}^I$ such that $d+e=d'+e'$, the following diagram commutes:

\begin{tikzcd}
  K(d)\boxtimes K(e) \arrow[r, "m"] \arrow[d, "\Delta\boxtimes \Delta"]
    & K(d+e) \arrow[d, "\Delta"] \\
  \bigoplus_P K(d_1)\boxtimes K(d_2)\boxtimes K(e_1)\boxtimes K(e_2) \arrow[r, "\widetilde{m}^2"]
&K(d')\boxtimes K(e').
\end{tikzcd}
\\
Here $\widetilde{m}^2:=(m\boxtimes m)(1\boxtimes R\boxtimes 1)$ and
$P$ be the set of quadruples of dimension vectors $(d_1,d_2,e_1,e_2)$ such that: 
\[d=d_1+d_2,\, e=e_1+e_2,\,d'=d_1+e_1,\text{ and } e'=d_2+e_2.\]

\end{thm}

Before we start the proof of Theorem \ref{bialgebra}, we collect some preliminary results. We will be using localization with respect to a $\C^*\subset G(d+e)$, and these computations are done in the localization of $\text{KHA}$ with respect to
the ideal $I_{d,e}$ defined in Theorem \ref{loc}. 
\begin{prop}\label{resp}
The abelian product and coproduct $m^{ab}_{d,e}$ and $\Delta^{ab}_{d,e}$ respect the $\bigoplus_{d\in\mathbb{N}^I}K_0(BT(d))$-module structure, see Subsection \ref{local}.  
\end{prop}

\begin{proof}
Let $x\in K(d), y\in K(e), l_d\in K_0(BT(d)), l_e\in K_0(BT(e))$. Denote by $l_{d+e}:=l_{d}\boxtimes l_e$ the product in $ K_0(BT(d))\boxtimes K_0(BT(e))\cong K_0(BT(d+e))$. 
Use the projection formula to write:
$$p_{d,e*}\left(q_{d,e}^*(xl_d\boxtimes yl_e)\right)=p_{d,e*}\left( q^*_{d,e}(x\boxtimes y)\otimes p_{d,e}^*(l_{d+e})\right)=p_{d,e*}q_{d,e}^*(x\boxtimes y)\,l_{d+e}.$$ The proof for comultiplication is similar.
\end{proof}

Note that for the natural section $s:R(d)\times R(e)\to R(d,e)$ we have:
$$\text{eu}(s: R(d)\times R(e)\to R(d,e))=\text{eu}_1(e,d).$$

\begin{prop}\label{compo}
(a) Let $i:\Y^{ss}(d)\times\Y^{ss}(e)\to \Y^{ss}(d+e)$ be the natural closed immersion. The composition: 
\begin{multline*}
    K_0(D_{sg}(\Y^{ss}(d)_0))\boxtimes K_0(D_{sg}(\Y^{ss}(e)_0)) \xrightarrow{m^{ab}_{d,e}} K_0(D_{sg}(\Y^{ss}(d+e)_0))\xrightarrow{TS^{-1}i^*}\\ K_0(D_{sg}(\Y^{ss}(d)_0))\boxtimes K_0(D_{sg}(\Y^{ss}(e)_0))
\end{multline*} 
is multiplication by $\text{eu}_1\,(d,e)$.

(b) After localization with respect to the ideal $I_{d,e}$, the $R_{d,e}^{ab}$ and $R_{d,e}$ matrices
are equal to:
$$R_{d,e}^{ab}=\frac{\text{eu}_1(d,e)}{\text{eu}_1(e,d)}\,\text{sw}\text{ and }R_{d,e}=\frac{\text{eu}(d,e)}{\text{eu}(e,d)}\,\text{sw}.$$
\end{prop}

\begin{proof}
(a) Let $s:\Y(d)\times\Y(e)\to\Y(d,e)$ be the natural section of the bundle $q_{d,e}$. Then $i=p_{d,e}\,s.$ 
Consider the maps:
\[\mathcal{Y}(d)\times\Y(e)\xleftarrow{q}\Y(d,e)\xrightarrow{p}\Y(d+e).\]
For $x\in K_0(D_{sg}(\Y^{ss}(d)_0))\boxtimes K_0(D_{sg}(\Y^{ss}(e)_0))$ we have that:
$$p^*p_*(x)=x\,\text{eu}_1(d,e).$$
We now compute:
$$i^*m^{ab}_{d,e}(x)=i^*p_*q^*(x)=s^*q^*x\,\text{eu}_1(d,e)=x\,\text{eu}_1(d,e).$$
(b) In the localization with respect to $I_{d,e}$, we have that:
$$R_{d,e}=R_{d,e}^{ab}\frac{\text{eu}_2(e,d)}{\text{eu}_2(d,e)},$$
so it suffices to show the result for $R_{d,e}^{ab}$. We will use Proposition \ref{resp} to commute Euler factors past $m^{ab}$ and $\Delta^{ab}$.
The definition of the $R_{d,e}^{ab}$ matrix is, after composing with the Thom-Sebastiani isomorphisms:

\begin{tikzcd}
K_0\left(D_{sg}\left((\mathcal{Y}(d)\times\mathcal{Y}(e))_0\right)\right)\arrow{r}{p_{d,e*}q_{d,e}^*} & K_0(D_{sg}(\mathcal{Y}(d+e)_0)) \arrow{r}{\widetilde{q}_{d,e*}\tau^*p^*_{d,e}}&  K_0\left(D_{sg}\left((\mathcal{Y}(e)\times\mathcal{Y}(d))_0\right)\right)\\
 & K_0\left(D_{sg}\left((\mathcal{Y}(d)\times\mathcal{Y}(e))_0\right)\right). \arrow{u}{i_*}& 
\end{tikzcd}
\\
For $x\in K_0\left(D_{sg}\left((\mathcal{Y}(d)\times\mathcal{Y}(e))_0\right)\right)$, we have that: 
$$\left(\widetilde{q}_{d,e*}\tau^*p_{d,e}^*i_*\right)\left(i^*p_{d,e*}q_{d,e}^*(x)\right)=\text{sw}\,R^{ab}_{d,e}\left(i_*i^*(x)\right)=\text{eu}_1(d,e)\,\text{eu}_1(e,d)\,\text{sw}\,R^{ab}_{d,e}(x).$$
By part (a), we have that $i^*p_{d,e*}q_{d,e}^*(x)=
\text{eu}_1(d,e)\,(x)$.
Next, for $s$ the natural section $s: R(d)\times R(e)\to R(d,e)$, we have that:
\begin{multline*}
    \left(\widetilde{q}_{d,e*}\tau^*p_{d,e}^*i_*(x)\right)=
\left(\widetilde{q}_{d,e*}\tau^*p_{d,e}^*p_{d,e*}s_*(x)\right)=
\left(\widetilde{q}_{d,e*}\tau^*\left(\text{eu}_1(d,e)s_*(x)\right)\right)=\\
\text{eu}_1(d,e)\left(\widetilde{q}_{d,e*}\tau^*s_*(x)\right).
\end{multline*}
Consider the cartesian diagram:

\begin{tikzcd}
\widetilde{\Y}(d,e) \arrow{d}{\tau} & R(d)\times R(e)\oplus\C/T(d+e)\times\C^*\arrow{l}{\widetilde{s}} \arrow{d}{\widetilde{\tau}}\\
\Y(d,e) & \Y(d)\times \Y(e).\arrow{l}{s}
\end{tikzcd}
\\
Consider the section of $\widetilde{\tau}$: $$t:\Y(d)\times \Y(e)\to R(d)\times R(e)\oplus\C/T(d+e)\times\C^*.$$ We have that $t_*(x)=\widetilde{\tau}^*(x).$ Continuing the computation from above, we have that:
$$\widetilde{q}_{d,e*}\tau^*s_*(x)=\widetilde{q}_{d,e*}\widetilde{s}_*\widetilde{\tau}^*(x)=
\widetilde{q}_{d,e*}\widetilde{s}_*t_*(x)=x.$$
Putting together these computations, we get that:
$$R_{d,e}^{ab}=\frac{\text{eu}_1(d,e)}{\text{eu}_1(e,d)}\,\text{sw}.$$
\end{proof}

\begin{proof}[Proof of Theorem \ref{bialgebra}]
Unraveling the definitions of $m$ and $\Delta$, the diagram we need to prove commutes is:

\begin{tikzcd}
  K(d)\boxtimes K(e) \arrow[r] \arrow[d, ]
    & K(d,e)  \arrow[r, ""] & K(d+e) \arrow[d] \\
  \bigoplus_P K(d_1,d_2)\boxtimes K(e_1,e_2) \arrow[d, ""]  
&  & K(d',e') \arrow[d, ""] \\
\bigoplus_P K(d_1)\boxtimes K(d_2)\boxtimes K(e_1)\boxtimes K(e_2) \arrow[r, ""] & \bigoplus_P K(d_1,e_1)\boxtimes K(d_2,e_2) \arrow[r, ""] & K(d')\boxtimes K(e')
\end{tikzcd}
\\
We use the localization theorem for the character 
$$\lambda: \C^*\xrightarrow{(z\text{Id}_d, \text{Id}_e)} G(d)\times G(e)\hookrightarrow G(d+e).$$ The fixed loci of $\C^*$ on $\X(d,e)$ and $\X(d+e)$ are both $\X(d)\times\X(e)$.
The fixed loci of $\C^*$ on $\X(d',e')$ are $\bigoplus_P \X(d_1,d_2)\times\X(e_1,e_2)$. Indeed, consider the diagram:

\begin{tikzcd}
  R(d',e')\times_{G(d',e')} G(d+e) \arrow[r, "i_{d',e'}"] 
    & R(d+e)\times_{G(d',e')} G(d+e) \arrow[d, "\pi_{d',e'}"] \\
&R(d+e),
\end{tikzcd}
\\
where $i_{d',e'}$ is a closed immersion and $\pi_{d',e'}$ is a projection.
The $\C^*$ fixed locus on $R(d+e)$ is $R(d)\times R(e)$. The $\C^*$ fixed locus on the fiber $$G(d+e)/G(d',e')=\prod_{v\in I} \text{Gr}(d'^{v}, d^v+e^v)$$ has components indexed by the set $P$; 
indeed, if $V$ is a vector space of dimension $d^v+e^v$ and we have a splitting $V=D\oplus E$, where $D$ had dimension $d^v$ and is the locus fixed by $\C^*$ and $E$ has dimension $e^v$ and is the locus of weight $1$, a subspace $W\subset V$ of dimension $d'^v$ can have a subspace of dimension $0\leq d^v_1\leq d'^v$ fixed by $\C^*$, and this choices determine the fixed components of $\C^*$ on $\text{Gr}(d'^v,d^v+e^v)$. Choose such a subspace $D'\subset V$ of dimension $d'^v$; it determines a splitting $D'=D_1\oplus E_1$ where $\C^*$ acts with weight $0$ on $D_1$ of dimension $d^v_1$ and with weight $1$ on $E_1$ of dimension $e^v_1$; we thus obtain two dimension vectors $d_1$ and $e_1\in \mathbb{N}^I$.
Let $d_2=d-d_1$ and $e_2=e-e_1$. The dimension vectors $(d_1,d_2,e_1,e_2)$ give an element of $P$. Further, each such component is isomorphic to: 
$$\prod_{v\in I}\left(\text{Gr}(d^v_1,d^v)\times \text{Gr}(d^v_2,e^v)\right)=G(d)/G(d_1,d_2)\times G(e)/G(e_1,e_2).$$ The $\C^*$ fixed component of $R(d+e)\times_{G(d',e')} G(d+e)$
corresponding to the partition $(d_1,d_2,e_1,e_2)$ is thus isomorphic to:
 $$\left(R(d)\times_{G(d_1,d_2)} G(d)\right)\oplus \left(R(e)\times_{G(e_1,e_2)} G(e)\right).$$
By intersecting these fixed loci with $R(d',e')\times_{G(d',e')}G(d+e)$, we see that the $\C^*$ fixed loci on $R(d',e')\times_{G(d',e')} G(d+e)$ are: $$\left(R(d_1, d_2)\times_{G(d_1,d_2)} G(d)\right)\oplus \left(R(e_1, e_2)\times_{G(e_1,e_2)} G(e)\right).$$ After taking the quotient by $G(d+e)$, we get that the $\C^*$ fixed loci on $\X(d',e')$ are indeed $\X(d_1,d_2)\times\X(e_1,e_2)$. 
The diagram of fixed loci from above becomes:

\begin{tikzcd}
K(d)\boxtimes K(e) \arrow[r] \arrow[d, ]
    &  K(d)\boxtimes K(e) \arrow[d] \\
  \bigoplus_P K(d_1,d_2)\boxtimes K(e_1,e_2) \arrow[d, ""]  
 & \bigoplus_P K(d_1,d_2)\boxtimes K(e_1,e_2) \arrow[d, ""] \\
\bigoplus_P K(d_1)\boxtimes K(d_2)\boxtimes K(e_1)\boxtimes K(e_2) \arrow[r, ""] &  \bigoplus_P K(d_1)\boxtimes K(d_2)\boxtimes K(e_1)\boxtimes K(e_2).
\end{tikzcd}
\\
From now on, we fix an element $(d_1,d_2,e_1,e_2)\in P$. It is enough to show that the diagram where we replace the sum by the term corresponding to $(d_1,d_2,e_1,e_2)$ commutes. 
Using Proposition \ref{zerodiv}, it is enough to show it commutes after localization at the ideal $I_{d,e}$, see Proposition \ref{zerodiv} for its definition. 
We will use the abelian formulas for multiplication and comultiplication. Recall from Proposition \ref{change} that both these operations respect the $K_0(BG)$-module structure, so they commute with arbitrary Euler classes. The Euler classes that appear in the definitions for multiplication, comultiplication, and the R-matrix are the same for both compositions:
\[\frac{\text{eu}_2(d_1,d_2)\text{eu}_2(e_1,e_2)}{\text{eu}_2(d_1,e_1)\text{eu}_2(d_2,e_2)}=\frac{\text{eu}_2(d',e')}{\text{eu}_2(d,e)}\,\frac{\text{eu}_2(d_2,e_1)}{\text{eu}_2(e_1,d_2)}.\]
Using Proposition \ref{compo} for the $R^{ab}$ matrix and for $m^{ab}_{d_1,e_1}$ and $m^{ab}_{d_2,e_2}$, we get that the second line
is multiplication by:  
$$1\boxtimes\,\text{sw}\boxtimes\,1\,\left(\text{eu}_1(d_1,e_1)\boxtimes \text{eu}_1(d_2,e_2)\frac{\text{eu}_1(d_2,e_1)}{\text{eu}_1(e_1,d_2)}\right).$$
This means that the composition of the left and down arrows is: 
$$1\boxtimes\,\text{sw}\boxtimes\,1\left(
\text{eu}_1(d_1,e_1)\boxtimes \text{eu}_1(d_2,e_2)\,\frac{\text{eu}_1(d_2,e_1)}{\text{eu}_1(e_1,d_2)}\Delta^{ab}_{d_1,d_2}(x)\boxtimes \Delta^{ab}_{e_1,e_2}(y)\right),$$
which can also be written as:
\begin{equation}\label{result}
1\boxtimes\,\text{sw}\boxtimes\,1\left(
\frac{\text{eu}_1(d,e)}{\text{eu}_1\,(d_1,e_2)\text{eu}_1\,(e_1,d_2)}\,
\Delta_{d_1,d_2}(x)\boxtimes \Delta_{e_1,e_2}(y)\right).
\end{equation}
To compute the composition of the up and right arrow, consider the diagram:

\begin{tikzcd}
K(d)\boxtimes K(e) \arrow{r}{p_*q^*}& K(d+e) \arrow{r}{i_{d,e}^*} \arrow{d}{p^*} & K(d)\boxtimes K(e) \arrow{d}{p^*\times p^*}\\
 & K(e',d') \arrow{d}{\tau^*} \arrow{r} &
K_0(d_2,d_1)\boxtimes K(e_2,e_1) \arrow{d}{\tau^*\times\tau^*}\\
 & K_0(D_{sg}(\widetilde{\Y}(e',d')_0)) \arrow{d}{\widetilde{q}_*} \arrow{r} & 
K_0(D_{sg}(\widetilde{\Y}(d_2,d_1)_0)\boxtimes D_{sg}(\widetilde{\Y}(e_2,e_1)_0)) \arrow{d}{e^{-1}\widetilde{q}_*\times\widetilde{q}_*}\\
 & K(e')\boxtimes\,K(d') \arrow{r}{i^*} &
 K(d_2)\,\boxtimes\,K(d_1)\,\boxtimes\,K(e_2)\,\boxtimes\,K(e_1),
\end{tikzcd}
\\
where $e=\text{eu}_1\,(d_1,e_2)\,\text{eu}_2\,(e_1,d_2).$ The top two squares commute.
To explain why the lower square commutes, let 
 $$\mathcal{Z}:=R(d_2,d_1)\oplus R(e_2,e_1)\oplus R(e_2, d_1)^{>0}\oplus R(d_2, e_1)^{>0}\oplus\C/T\times \C^*,$$ where the $\C^*$ action is induced from $\widetilde{\Y}(d',e')$. We denote by $\mathcal{W}$ the product without the $\C$ summand and the $\C^*$ quotient. By proper base change and the Thom-Sebastiani isomorphisms, the following diagram commutes:

\begin{tikzcd}
K_0(D_{sg}(\widetilde{\Y}(e', d')_0)) \arrow{d}{\widetilde{q}_*} \arrow{r} & 
K_0(D_{sg}(\mathcal{Z})_0) \arrow{d}{\widetilde{q}_*\times\widetilde{q}_*}\\
K_0(e')\boxtimes\,K_0(d') \arrow{r}{i^*} &
 K_0(d_2)\,\boxtimes\,K_0(d_1)\,\boxtimes\,K_0(e_2)\,\boxtimes\,K_0(e_1).
\end{tikzcd}
\\
Let $i: \widetilde{\Y}(d_2,d_1)\times\widetilde{\Y}(e_2,e_1)\to \mathcal{Z}$.
It is enough to check that the following diagram commutes:

\begin{tikzcd}
K_0(D_{sg}(\mathcal{Z}_0)) \arrow{r}{i^*} \arrow{d}{\text{e}\tilde{q}_*} & K_0(D_{sg}(\widetilde{\Y}(d_2,d_1)_0)\boxtimes D_{sg}(\widetilde{\Y}(e_2,e_1)_0)) \arrow{dl}{\widetilde{q}_*\times\widetilde{q}_*}\\
K(d_2)\boxtimes K(d_1)\boxtimes K(e_2)\boxtimes K(e_1).
\end{tikzcd}
\\
We have that $\widetilde{q}\boxtimes\widetilde{q}=\widetilde{q}i$, so it is enough to show that $i_*(1)=e$. 
For this, let $\Y'=R(d_2,d_1)\oplus R(e_2,e_1)\oplus\C/T\times\C^*,$ where the $\C^*$ action is as in $\widetilde{\Y}(e',d')$. The map $i$ factors as follows:

\begin{tikzcd}
\widetilde{\Y}(d_2,d_1)\times\widetilde{\Y}(e_2,e_1)\arrow{d}{i} \arrow{r}{r} & \Y'\arrow{dl}{i'}\\
\mathcal{Z}.
\end{tikzcd}
\\
The map $i'$ is a closed immersion. We have that $r_*(1)=1$, so it suffices to show that $i'_*(1)=e$.
Consider the following cartesian diagram:

\begin{tikzcd}
\Y' \arrow{r}{i'} & \mathcal{Z} \\
\Y(d^2)\times\Y(d^1)\times\Y(e^2)\times\Y(e^1) \arrow{r}{s} \arrow{u}{\pi} & \Y(d^2,d^1)\times\Y(e^2,e^1). \arrow{u}{\pi}
\end{tikzcd}
\\
Use proper base change along the above square and $s_*(1)=e$ to obtain that $i'_*(1)=e$.
Now, we compute the composition along the up and right arrow to be:
$$1\boxtimes\,\text{sw}\boxtimes\,1\left(\frac{\text{eu}_1(d,e)}{\text{eu}_1\,(d^1,e^2)\text{eu}_1\,(e^1,d^2)}\,
\Delta^{ab}_{d^1,d^2}\boxtimes\Delta^{ab}_{e^1,e^2}\right).$$
Comparing with Equations \ref{result}, we see that the product and the coproduct are indeed compatible.

\end{proof}

\begin{thm}
The conclusion of Theorem \ref{bialgebra} holds for $\text{CoHA}$ as well. 
\end{thm}

\begin{proof}
The functors used to define the product, coproduct, and the $R$-matrix in $K$-theory are also defined for critical cohomology, and we can use them to define these operations for $\text{CoHA}$. The proof of Theorem \ref{bialgebra} uses Propositions \ref{change} and \ref{compo}; all the functors and theorems used in their proofs exist and hold in cohomology, for example for $i:Y\to X$ a $T$-equivariant closed immersion and a regular function $f:X/T\to\mathbb{A}^1$, we have that $$i^*i_*(x)=\text{eu}_1(N_{Y/X})\,x\in H^{\cdot}(Y,\varphi_f\mathbb{Q}),$$ see \cite[Proposition 2.16]{d}.  Further, the proof of Theorem \ref{bialgebra} uses the Localization Theorem \ref{loc} and Proposition \ref{zerodiv}, both of which hold in critical cohomology and were proved in \cite{d}, see Proposition $4.1$ and the proof of Theorem $5.13$ in loc. cit.
\end{proof}

\section{The Chern character}\label{ch}
In this section, we assume that the potential $W$ is
homogeneous of positive degree with respect to some non-negative weights attached to the edges. We will construct a Chern character in the following situation. 
Let $\X=V/G$ be a quotient stack, where $G$ is a reductive group and $V$ is an affine variety with a $G$-action. Consider a regular map $f:\X\to\mathbb{A}^1$ homogeneous of positive degree with respect to an action of $\mathbb{G}_m$ with non-negative weights on $\X$. Let $n+1=\text{dim}\,\X$. The fibers of $f$ over non-zero points are all isomorphic. 

We construct a Chern character map: 
$$\text{ch}:K_0(D_{sg}(\X_0))\to H^{\cdot}(\X_0,\varphi_f \mathbb{Q})$$ such that its associated graded with respect to the cohomological filtration:
$$\text{ch}: \text{gr}\,K_0(D_{sg}(\X_0))\to H^{\cdot}(\X_0,\varphi_f \mathbb{Q})$$
commutes with pullbacks $f:\Y\to\X$, where $f$ is smooth or a closed immersion, and
commutes with proper pushforward $p:\X\to\Y$. 
\\

If $f$ is zero on $\X$, then $$\X_0=\X\times_{\mathbb{A}^1}0=\text{Spec}\,\left(\OO_{\X}\xrightarrow{0}\OO_{\X}\right).$$
Consider the morphism $i:\X_0\to\X$ induced by 
\begin{tikzcd}
\OO_{\X}\arrow{r}{0}& \OO_{\X}\\
0\arrow{u} \arrow{r}& \OO_{\X}. \arrow{u}{\text{id}}
\end{tikzcd}
\\
The pushforward $i_*$ induces an isomorphism:
$$i_*: K_0(D_{sg}(\X_0))\cong K_0(\X).$$
Consider the Chern character induced by the Chern character for $\X$:
$$K_0(D_{sg}(\X_0))\xrightarrow{i_*} K_0(\X)\xrightarrow{\text{ch}} H^{\cdot}(\X).$$
Next, assume that $f$ is nonzero.
Consider the diagram:

\begin{tikzcd}
K_0(\X_0)\arrow{d}{\text{ch}} \arrow{r}{i_0^*}& G_0(\X_0) \arrow{d}{\text{ch}} \arrow[r, twoheadrightarrow] & K_0(D_{sg}(\X_0)) \\
H^{\cdot}(\X_0) \arrow{r}{\cap \,o} & H^{\cdot+2}(\X,\X-\X_0)=H^{BM}_{2n-\cdot}(\X_0), & 
\end{tikzcd}
\\
where $o\in H^2(\X,\X-\X_0)=H^{BM}_{2n}(\X_0)$ is the fundamental class of $\X_0$ and the left lower map is $$H^{\cdot}(\X_0)\cong H^{\cdot}(\X)\xrightarrow{\cap\, o} H^{BM}_{2n-\cdot}(\X_0).$$

We recall the definition of vanishing and nearby cycles \cite[Constructible sheaves]{kash}.
Consider the diagram:

\begin{tikzcd}
\widetilde{X} \arrow{r}{\widetilde{f}}\arrow{d}{p} & \C \arrow{d}{\text{exp}}\\
X \arrow{r}{f}& \C.
\end{tikzcd}
\\
Let $i_0:\X_0\to \X$ be the inclusion of the zero fiber. Consider the complex $K=[p_!\mathbb{Q}_{\C}\to \mathbb{Q}_{\C}]$, and define the nearby and vanishing cycle sheaves as: 
$$\psi_f\mathbb{Q}=i_0^*R\mathcal{H}om(f^*p_!\mathbb{Q}_{\mathbb{C}},\mathbb{Q})\text{ and }\varphi_f\mathbb{Q}=i_0^*R\mathcal{H}om(f^*K,\mathbb{Q}),$$
where we have used $\mathcal{H}om$ for the $\text{Hom}$ sheaf in the category of constructible sheaves on $\X$. Both $\varphi_f$ and $\psi_f$ come with an action $T$ of the monodromy around $0$. 
They fit is a distinguished triangle:
$$\psi_f\mathbb{Q}[-1]\xrightarrow{\text{can}} \varphi_f\mathbb{Q}\to i^*\mathbb{Q}\to.$$
Using the results in \cite[Section 7]{dp} for homogeneous regular functions, we can identify the long exact sequence from the above distinguished triangle with a long exact sequence in relative cohomology:

\begin{tikzcd}
H^{\cdot}(\X_0,\varphi_f\mathbb{Q}) \arrow{r}& H^{\cdot}(\X_0) \arrow{r}& H^{\cdot}(\X_0,\psi_f\mathbb{Q})\\
H^{\cdot}(\X,\X_1)\arrow{r}\arrow{u}{\text{iso.}} &H^{\cdot}(\X)\arrow{r} \arrow{u}{\text{id.}}& H^{\cdot}(\X_1).\arrow{u}{\text{iso.}}
\end{tikzcd}
\\
Using the Wang long exact sequence, the monodromy $T$ fits in a long exact sequence:
$$\cdots\to H^{\cdot-1}(\X_1)\xrightarrow{T-\text{id}} H^{\cdot-1}(\X_1)\cong H^{\cdot}(\X-\X_0,\X_1)\to H^{\cdot}(\X-\X_0)\to H^{\cdot}(\X_1)\to\cdots.$$
We also need to recall the definition of the variation triangle. For this, 
consider the distinguished triangle:

\begin{tikzcd}
p_!\mathbb{Q}_{\C}\arrow{d} \arrow{r}{1-T} & p_!\mathbb{Q}_{\C} \arrow{d}{\text{tr}} \arrow{r}{\text{id}}& \mathbb{Q}_{\mathbb{C}^*}\arrow{d}\\
0\arrow{r}& \mathbb{Q}_{\C}\arrow{r}{\text{id}}& \mathbb{Q}_{\C}.
\end{tikzcd}
\\
It induces the variation triangle:
$$i^!\mathbb{Q}\to \varphi_f\mathbb{Q}\xrightarrow{\text{var}} \psi_f\mathbb{Q}\to.$$
The map $H^{\cdot}(\X_0,\varphi_f\mathbb{Q})\to H^{\cdot}(\X_0,\psi_f\mathbb{Q})$ can be identified with the map $H^{\cdot}(\X,\X_1)\to H^{\cdot-1}(\X_1)\cong H^{\cdot}(\X-\X_0,\X_1)$ from the diagram:

\begin{tikzcd}
H^{\cdot-1}(\X_1)\arrow{d}{\text{id.}} \arrow{r}& H^{\cdot}(\X,\X_1) \arrow{d} \arrow{r}& H^{\cdot}(\X)\arrow{d}\\
H^{\cdot-1}(\X_1)\arrow{r}{T-\text{id}}& H^{\cdot}(\X-\X_0,\X_1)\arrow{r} & H^{\cdot}(\X-\X_0).
\end{tikzcd}
\\
The map $b:H^{\cdot}(\X,\X-\X_0)\to H^{\cdot}(\X,\X_1)$ from the long exact sequence induced by the variation triangle is obtained from the natural restriction.

\begin{prop}\label{factors}
The map $H^{\cdot}(\X)\xrightarrow{\cap \,o} H^{\cdot+2}(\X,\X-\X_0)$ factors through $$H^{\cdot}(\X)\xrightarrow{i_1^*}H^{\cdot}(\X_1)\cong H^{\cdot+1}(\X-\X_0,\X_1)\to
H^{\cdot+1}(\X-\X_0)\xrightarrow{\partial} H^{\cdot+2}(\X,\X-\X_0).$$
\end{prop} 

\begin{proof}
It is enough to show that there exists $u\in H^1(\X-\X_0,\X_1)\cong H^0(\X_1)$ such that $\partial(u)=o$. For this, consider the diagram:

\begin{tikzcd}
H^1(\X-\X_0) \arrow[r, hookrightarrow,"\partial"]& H^2(\X,\X-\X_0) \arrow{r}{b} & H^2(\X,\X_1) \\
H^1(\X-\X_0,\X_1)\arrow{u} \arrow{r} & H^2(\X,\X-\X_0) \arrow[u, "\text{id.}"] \arrow{r}{b} & H^2(\X,\X_1). \arrow{u}{\text{id}.}
\end{tikzcd}
\\
The bottom line is part of the long exact sequence in relative cohomology. The map $$b:H^2(\X,\X-\X_0)=H_{2n}^{BM}(\X_0)\to H^2(\X,\X_1)=H^{BM}_{2n}(\X_0,\varphi_f\mathbb{Q})$$ is part of the variation distinguished triangle, and the kernel of $b$ is generated by $o$. By the bottom line, the kernel of $b$ is generated by $\delta(u)$, where $u$ is a generator of $H^0(\X_1)\cong H^1(\X-\X_0,\X_1)$, and thus we can choose such an $u$ such that $o=\delta(u)$.


\end{proof}

\begin{prop}
The composition $$H^{\cdot}(\X)\xrightarrow{\cap \,o} H^{\cdot+2}(\X,\X-\X_0)\xrightarrow{b} H^{\cdot+2}(\X_0,\varphi_f\mathbb{Q})$$ is zero and the image of the map $H^{\cdot+2}(\X,\X-\X_0)\to H^{\cdot+2}(\X_0,\varphi_f\mathbb{Q})$ lies in the monodromy fixed part $H^{\cdot+2}(\X_0,\varphi_f\mathbb{Q})^{T-1}$.
\end{prop}

\begin{proof}
By Proposition \ref{factors}, the map $\cap \,o$ factors as:
$$H^{\cdot}(\X)\xrightarrow{\cap \,u} H^{\cdot+1}(\X-\X_0,\X_1)\to
H^{\cdot+1}(\X-\X_0)\xrightarrow{\partial} H^{\cdot+2}(\X,\X-\X_0).$$
The variation long exact sequence is:
$$\cdots\to H^{\cdot+1}(\X-\X_0,\X_1)\xrightarrow{\partial} H^{\cdot+2}(\X,\X-\X_0)\xrightarrow{b} H^{\cdot+2}(\X,\X_1)\cong H^{\cdot+2}(\X_0,\varphi_f\mathbb{Q})\to \cdots,$$
so the composite $H^{\cdot}(\X)\xrightarrow{\cap \,o} H^{\cdot+2}(\X,\X-\X_0)\xrightarrow{b} H^{\cdot+2}(\X_0,\varphi_f\mathbb{Q})$ is indeed zero. 

Next, the map $T-1=\text{can}\,\text{var},$ so it is enough to show that the following composite is zero:
$$H^{\cdot}(\X)\xrightarrow{\cap\,u} H^{\cdot+1}(\X,\X-\X_0)\to H^{\cdot+2}(\X,\X_1)\xrightarrow{\text{var}} H^{\cdot+2}(\X-\X_0,\X_1)\xrightarrow{\text{can}} H^{\cdot+2}(\X,\X_1).$$ The composite of the middle two maps is zero because both maps are restrictions on relative cohomology; this means that indeed the image lie in $H^{\cdot+2}(\X_0,\varphi_f\mathbb{Q})^{T-1}$.
\end{proof}

\begin{prop}\label{defcherngr}
There exists a Chern character map $\text{ch}:K_0(D_{sg}(\X_0))\to H^{\cdot}(\X_0, \varphi_f\mathbb{Q})$ such that the following diagram commutes: 

\begin{tikzcd}
K_0(\X_0)\arrow{d}{ch} \arrow{r}& G_0(\X_0) \arrow{d}{ch} \arrow[r, twoheadrightarrow] & K_0(D_{sg}(\X_0))
\arrow{d}{ch}\\
H^{\cdot}(\X_0) \arrow{r}{\cap \,o} & H^{\cdot+2}(\X,\X-\X_0)=H^{BM}_{2d-\cdot}(\X_0)\arrow{r}{b} & H^{\cdot}(\X_0,\varphi_f\mathbb{Q}). 
\end{tikzcd}
\\
The image of $\text{ch}:K_0(D_{sg}(\X_0))\to H^{\cdot}(\X_0,\varphi_f\mathbb{Q})$ lies inside $$H^{\cdot}(\X_0,\varphi_f\mathbb{Q})^{T-1}\cap b\left(H^{BM}(\X_0)_{\text{alg}}\right).$$
\end{prop}

\textbf{Remark 1.} The $\text{ch}:K_0(D_{sg}(\X_0))\to H^{\cdot}(\X_0,\varphi_f\mathbb{Q})$ is not always surjective because there might be classes in $H^{\cdot}(\X_0,\varphi_f\mathbb{Q})$ not generated by algebraic cycles.
To see this, consider the regular function: $$f=t(xy+1):\mathbb{A}^3\to\mathbb{A}^1.$$ Let $Z\subset\mathbb{A}^2$ be the zero locus of $xy+1=0$. By dimensional reduction \cite[Appendix A]{d}: $$H^{3}_c(\mathbb{A}^3,\varphi_f\mathbb{Q})=H_c^{1}(Z)=H_c^1(\mathbb{R}_{>0}\times S^1)=\mathbb{Q}.$$ 
After taking duals, this means that $H^{\text{odd}}(\mathbb{A}^3, \varphi_f\mathbb{Q})$
is nonzero, and its elements cannot be represented by algebraic cycles. 
This example can be modeled by the quiver with one vertex, three loops $t,x,$ and $y$, potential $f$, and dimension vector $d=1$.
\\

\textbf{Remark 2.} The $\text{ch}:K_0(D_{sg}(\X_0))\to H^{\cdot}(\X_0, \varphi_f\mathbb{Q})$ is not always injective, for example, because $K_0(D_{sg}(\X_0))$ might not be finitely generated. Consider $$f=t(y^2z-x^3-xz^2):\mathbb{A}^4\to\mathbb{A}^1.$$ Consider $f$ as a weight $1$ potential, where $t$ has weight $1$ and the other variables have weight zero. Let $X\subset \mathbb{A}^4$ be the zero locus of $f$, and $Z\subset \mathbb{A}^3$ be the zero locus of $y^2z-x^3-xz^2$. Using dimensional reduction in K-theory, see Theorem \ref{dimred0}, and in cohomology \cite[Appendix A]{d}, which are compatible by Proposition \ref{dimred}, the Chern character can be identified with:
$$\text{ch}:G_0(Z)\to H_{\cdot}^{BM}(Z).$$
The locus $y^2z-x^3-xz^2=0 \subset \mathbb{A}^3$ is the cone over an elliptic curve $CE$. Denote by $o$ the vertex of $CE$; then $CE-o\cong E\times\C^*$ and let $q:CE-o\to E$. Restriction to an open subset gives a surjection
$$G_0(CE)\twoheadrightarrow K_0(CE-o)$$
and we also have  $q^*: K_0(E)\hookrightarrow K_0(CE-o).$
The group $K_0(E)$ is infinite dimensional and $H_{\cdot}^{BM}(CE)$ is finite dimensional, so $\text{ch}$ is not injective. 
This example can be modeled by the quiver with one vertex, four loops $t,x,y,$ and $z$, potential $f$, and dimension $d=1$.
\\

\textbf{Remark 3.} The group $K_0(D_{sg}(\X_0))\otimes\mathbb{Q}$ can be zero, even if $H^{\cdot}(\X_0,\varphi_f\mathbb{Q})$ is not. Such an example is $t=x^n:\mathbb{C}\to\mathbb{C}$ for $n\geq 2$. The category $D^b(k[x]/x^n)$ is generated by $k$, and its subcategory $\text{Perf}(k[x]/x^n)$ is generated by $k[x]/x^n$, which can be obtained as an $n$ fold extension of $k$. This means that $$n K_0(D_{sg}(k[x]/x^n))=0\text{ and so }K_0(D_{sg}(k[x]/x^n))\otimes\mathbb{Q}=0.$$ On the other hand, $\text{dim}_{\mathbb{Q}}\,H^{\cdot}(\mathbb{A}^1,\varphi_{x^n}\mathbb{Q})=n-1.$ Note that the monodromy fixed part $H^{\cdot}(\mathbb{A}^1,\varphi_{x^n}\mathbb{Q})^{T-1}$ is trivial.
\\

\begin{prop}\label{dimred}
Let $V\times\mathbb{A}^n$ be a representation of a reductive group $G$ and denote by $\X=X\times\mathbb{A}^n/G$. Consider a regular function:
$$g=t_1f_1+\cdots+t_nf_n:\X\to\mathbb{A}^1,$$
where $t_1,\cdots, t_n$ be the coordinates on 
$\mathbb{A}^n$ and 
$f_1,\cdots, f_n: X\to\mathbb{A}^1$ are 
$G$-equivariant regular functions. Let $\X_0\subset\X$ be the zero locus of $g$ and $i:\mathcal{Z}\hookrightarrow \mathcal{X}_0$ the zero locus of $f_1,\cdots, f_n$. 
The following diagram commutes:

\begin{tikzcd}
\text{gr}\,G_0(\mathcal{Z}) \arrow{r}{i_*} \arrow{d}{ch} & \text{gr}\,K_0(D_{sg}(\X_0)) \arrow{d}{ch}  \\
H^{BM}_{\cdot}(\mathcal{Z}) \arrow{r}{i_*} & H^{BM}_{\cdot}(\X_0,\varphi_f\mathbb{Q}).
\end{tikzcd}
\\
By dimensional reduction in K-theory, see Theorem \ref{dimred0}, and in cohomology \cite[Appendix A]{d}, the horizontal maps are isomorphisms.
\end{prop}

\begin{proof}
Consider the diagram:

\begin{tikzcd}
\text{gr}\,G_0(\mathcal{Z}) \arrow{r}{i_*} \arrow{d}{\text{ch}} & \text{gr}\,G_0(\X_0) \arrow{r}{\text{surj.}} \arrow{d}{\text{ch}} & \text{gr}\,K_0(D_{sg}(\X_0)) \arrow{d}{\text{ch}} \\
H^{BM}_{\cdot}(\mathcal{Z}) \arrow{r}{i_*} & H^{BM}_{\cdot}(\X_0) \arrow{r} & H^{BM}_{\cdot}(\X_0,\varphi_f\mathbb{Q}).
\end{tikzcd}
\\
The right corner of the diagram commutes from Proposition \ref{defcherngr}. The left corner commutes from Grothendieck-Riemann-Roch, so the diagram indeed commutes. 
\end{proof}

\begin{prop}\label{grr}
Let $\X$ and $\Y$ be smooth stacks with a map $\X\to \Y$. Let $h:\Y\to\mathbb{A}^1$ be a regular function, and let $f:\X\to\mathbb{A}^1$ be the regular function induced by $h$.

The map $\text{ch}:\text{gr}\,K_0(D_{sg}(\X_0))\to H^{\cdot}(\X_0,\varphi_f\mathbb{Q})$ commutes with proper pushforwards along maps $p:\X\to\Y$;
with smooth pullbacks along maps $\pi:\X\to\Y$; and with pullbacks $\iota:\X\hookrightarrow\Y$ to closed smooth substacks, where $\X$ and $\Y$ are as above. 
\end{prop}

\begin{proof}
Let $p:\X\to\Y$ be a proper map between smooth stacks, and let $h:\Y\to\mathbb{A}^1$.
If $h=0$ on $\Y$, then $f=0$, and the result follows from Grothendieck-Riemann-Roch for $p:\X\to\Y$, see Subsection \ref{singularGRR}. 

If $h$ is not zero and $f=0$, the diagram becomes: 

\begin{tikzcd}
\text{gr}\,K_0(\X)\arrow{r}{p_*} \arrow{d}{\text{ch}}& \text{gr}\,K_0(D_{sg}(\Y_0))\arrow{d}{\text{ch}}\\
H^{\cdot}(\X)\arrow{r}{p_*}& H^{\cdot}(\Y,\varphi_h\mathbb{Q}).
\end{tikzcd}
\\
It is enough to show that the following diagram commutes:

\begin{tikzcd}
\text{gr}\,K_0(\X)\arrow{r}{p_*} \arrow{d}{\text{ch}}& \text{gr}\,K_0(\Y_0)\arrow{d}{\text{ch}}\\
H^{\cdot}(\X)\arrow{r}{p_*}& H^{\cdot}(\Y_0),
\end{tikzcd}
\\
which follows from the discussion in Subsection \ref{singularGRR}. 

Finally, consider the case when both $h$ and $f$ are not zero.
It is enough to show the corresponding statements for $$\text{ch}:\text{gr}\,G_0(\X_0)\to H^{BM}_{\cdot}(\X_0).$$ The Chern character commutes with proper pushforward by the Grothendieck-Riemann-Roch by Subsection \ref{singularGRR}.
\\

Next, we discuss the case of smooth pullbacks along $\pi:\X\to\Y$. 
If $h=0$ on $\Y$, then we need to show that following diagram commutes:

\begin{tikzcd}
\text{gr}\,K_0(\Y)\arrow{r}{\pi^*}\arrow{d}{\text{ch}}& \text{gr}\,K_0(\X)\arrow{d}{\text{ch}}\\
H^{\cdot}(\Y)\arrow{r}{\pi^*} & H^{\cdot}(\X),
\end{tikzcd}
\\
which follows from discussion in Subsection \ref{singularGRR}.

If $h$ is not zero, then $f$ is not zero, and the statement follows from the commutativity between $\text{ch}$ and $\pi_0^*$ for the smooth map $\pi_0:\X_0\to\Y_0$, see Subsection \ref{singularGRR}.
\\

The argument for the pullback to a closed smooth substack is similar. If $f=0$ and $h=0$, then the proof is as above. If $f=0$ and $h$ is not zero, then it follows from the commutative diagram:

\begin{tikzcd}
\text{gr}\,G_0(\Y_0)\arrow{r}{\iota^*} \arrow{d}& \text{gr}\,K_0(\X)\arrow{d}\\
H^{BM}_{\cdot}(\Y_0)\arrow{r}{\iota^*}& H^{BM}_{\cdot}(\X).
\end{tikzcd}
\\
If both $f$ and $h$ are nonzero, the result follows from the commutative diagram:

\begin{tikzcd}
\text{gr}\,G_0(\Y_0)\arrow{r}{\iota^*} \arrow{d}& \text{gr}\,G_0(\X_0)\arrow{d}\\
H^{BM}_{\cdot}(\Y_0)\arrow{r}{\iota^*}& H^{BM}_{\cdot}(\X_0).
\end{tikzcd}
\\
discussed in Subsection \ref{singularGRR}.
\end{proof}

The multiplication $m$ and comultiplication maps $\Delta$ involve proper pushforwards, smooth pullbacks, and pullbacks along a closed subvariety, so we obtain:
\begin{cor}
The product, coproduct, and $R$ define a braided bialgebra structure on $\text{gr}\,\text{KHA}.$
The Chern character $\text{ch}:\text{gr}\,\text{KHA}\to\text{CoHA}$ induces a morphism of braided bialgebras.
\end{cor}

\section{The wall crossing theorem}\label{4}
Let $\theta\in \mathbb{Q}^I$ be a stability condition for $Q$, and define the slope $$\mu(d)=\frac{\sum_{i\in I} \theta^id^i}{\sum_{i\in I} d^i}.$$
For any partition $d_1+\cdots+d_k=d$, consider the diagram of correspondences:

\begin{tikzcd}
\X(d_1,\cdots, d_k)\arrow{r}{p_{\dd}} \arrow{d}{q_{\dd}}& \X(d) \\
\X(d_1)\times\cdots\X(d_k).&
\end{tikzcd}
\\
The stack $\X(d)$ has a Harder-Narasimhan stratification with strata 
$$p_{\dd}\left(q_{\dd}^{-1}\left(\X^{\text{ss}}(d_1)\times\cdots\times\X^{\text{ss}}(d_k)\right)\right)$$ corresponding to partitions $\dd$ such that $\mu(d_i)=\mu_i$ and $\mu_1<\cdots<\mu_k$. The category $D^b(\X(d))$ has a semi-orthogonal decomposition in categories $D^b(\X(d))_w$ for $w\in\mathbb{Z}$, where $w$ is the weight with respect to the diagonal character. 
As a corollary of \cite[Amplification 2.11]{hl}, we have that:
\begin{prop}
The category $D^b(\X(d))$ has a semi-orthogonal decomposition with summands  $$p_{\dd*}q_{\dd}^*\left(D^b(\X^{\text{ss}}(d_1))_{w_1}\boxtimes\cdots\boxtimes D^b(\X^{\text{ss}}(d_k))_{w_k}\right),$$  where $\dd$ is a partitions such that $\mu_1<\cdots<\mu_k$ for $\mu_i=\mu(d_i)$. 
\end{prop}

\begin{proof}
First, by replacing $\theta^i$ with $\theta^i-\mu(d)$, we can assume that $\mu(d)=0$.
For $e\in\mathbb{N}^I$, let $\mu'(e)=\text{min}\,\{\mu(f)\text{ with }f\leq e\}$.
In the algorithm described in Subsection \ref{window}, the strata that appear have the form $\X(a)\times\X(b)$.
Indeed, consider a character $\lambda(z)=(z^{c_i}):\C^*\to G(d)$ with associated partition $\dd=(d_1,\cdots,d_k)$ and exponents $c_i$ corresponding to $d_i$ for $1\leq i\leq k$. 
Assume that $(\lambda, R(d_1)\times\cdots\times R(d_k))$ is a potential Kempf-ness locus, so:
$$0<\text{inv}(\lambda)=-\frac{\sum_{i\in I}\sum_{j=1}^k \theta^ic_jd^i_j}{\sqrt{\sum_{i\in I}\sum_{j=1}^k (c_j)^2d^i_j}}.$$
Let $1\leq m\leq k$ with the index which minimizes the sum $\sum_{i\in I} \theta^ic_jd^i_j$. Not all the indices minimize the sum; if that was the case, we would have $\text{inv}(\lambda)=0$. 
We also have that $0>\sum_{i\in I} \theta^ic_md^i_m$.
For the partition $d_m+(d-d_m)=d$ and exponents $c_m$ for $d_m$ and $0$ for $d-d_m$, we have that:
$$0<-\frac{\sum_{i\in I}\sum_{j=1}^k \theta^ic_jd^i_j}{\sqrt{\sum_{i\in I}\sum_{j=1}^k (c_j)^2d^i_j}}<-\frac{\sum_{i\in I}\theta^ic_md^i_m}{\sqrt{\sum_{i\in I}c_m^2d^i_j}}.$$
If two strata have equal $\lambda$, then we pick one with maximal $a$ with respect to $\geq$ among all these strata.
\\

We next claim that the Kempf-Ness strata have the form $$\X^{\text{ss}}(a)\times\X(b)\leftarrow \X'(a,b)\rightarrow \X(d)$$ for a partition $a+b=d$ such that $\mu(a)<\mu'(b)$. Consider a potential Kempf-Ness stratum of the form $\X(a)\times \X(b)$. If such a stratum is Kempf-Ness, then $\mu(a)<\mu'(b)$. Indeed, if this is not the case, let $f<b$ such that $\mu(f)\leq \mu(a)$; then the stratum $\X(a+f)\times\X(b-f)$ will have smaller or equal $\lambda$-invariant than $\X(a)\times\X(b)$ and $a+f>a$, so it will be picked before $\X(a)\times\X(b)$. The corresponding Kempf-Ness stratum will have the form: $$\X(a)\times\X(b)-\X(a)\times\X(b)\cap \left(\bigcup_{\mu(e)<\mu(a)} \X'(e,f)\right)=\X(a)^{ss}\times\X(b).$$ 

Using the semi-orthogonal decomposition in \cite[Theorem 2.10]{hl}, we find that there is a semi-orthogonal decomposition of $D^b(\X(d))$ where the summands have the form $p_{a,b*}q_{a,b}^*\left(D^b(\X^{ss}(a))\boxtimes D^b(\X(b))\right)_w$, where $\mu(a)<\mu'(b)$ and $w\in\mathbb{Z}$ is the weight with respect to $\lambda_{a,b}$. We decompose the categories above by the weights with respect to the characters $\lambda_e=z\, 1_e:\C^*\to G(e)$. This implies that
the category $D^b(\X(d))$ has a semi-orthogonal decomposition with summands $p_{a,b*}q_{a,b}^*\left(D^b(\X^{ss}(a))_v\boxtimes D^b(\X(b))_w\right)$, where $\mu(a)<\mu'(b)$ and $v,w\in\mathbb{Z}$. 
Using induction on the dimension vector $b<d$ and decomposing $D^b(\X(b))$, we obtain a semi-orthogonal decomposition with summands $$p_{\dd*}q_{\dd}^*\left(D^b(\X^{\text{ss}}(d_1))_{w_1}\boxtimes\cdots\boxtimes D^b(\X^{\text{ss}}(d_k))_{w_k}\right),$$  where $\dd$ is a partitions such that $\mu_1<\cdots<\mu_k$ for $\mu_i=\mu(d_i)$. 

\end{proof}

Using Proposition \ref{uns}, we see that:
\begin{cor}\label{wccor}
The category $D_{sg}(\X(d))$ has a semi-orthogonal decomposition with summands  $$p_{\dd*}q_{\dd}^*\left(D_{sg}(\X^{\text{ss}}(d_1)_0)_{w_1}\boxtimes\cdots\boxtimes D_{sg}(\X^{\text{ss}}(d_k)_0)_{w_k}\right),$$  where $\dd$ is a partitions such that $\mu_1<\cdots<\mu_k$ for $\mu_i=\mu(d_i)$. The analogous statement holds for $D^{id}_{sg}$.
\end{cor}


\begin{thm}\label{wallcrossing}
We have an isomorphism of vector spaces:
$$KHA(Q,W)\to \bigotimes_{-\infty<\mu<\infty} KHA(Q,W,\mu).$$
Here $KHA(Q,W)$ is the algebra for the trivial stability condition and $KHA(Q,W,\mu)$ is the algebra for the stability condition $\theta$ and fixed slope $\mu$.
The analogous statement holds for $KHA^{\text{id}}$.
\end{thm}

\begin{proof}
We need to prove that for $d\in\mathbb{N}^I$, we have an isomorphism of vector spaces:
$$K_0(D_{sg}(\X(d)_0)) =\bigoplus K_0(D_{sg}(\X^{\text{ss}}(d_1)_0)\boxtimes\cdots\boxtimes K_0(D_{sg}(\X^{\text{ss}}(d_k)_0)),$$ where the sum is after all decompositions $d_1+\cdots+d_k=d$ with slopes $\mu(d_i)=\mu_i$ such that $\mu_1<\cdots<\mu_k$, which follows from the Corollary \ref{wccor}. Same proof works for $D_{sg}^{id}$.
\end{proof}

As an application, we can find generators for the $KHA(Q):=KHA(Q,0)$ where $Q$ is a Dynkin quiver with vertices labelled $1$ to $n$. 
Choose first the stability condition $\theta: \theta^1<\cdots<\theta^n$. Denote by $d_i$
the dimension vector with $1$ in vertex $i$ and zero everywhere else. The $\theta$-semistable representations are at dimension vectors $nd_i$, for $n$ a nonnegative integer. 
For $1\leq i\leq n$ a vertex, let $r_i$ be the unique representation of dimension $d_i$. Then $$\X^{\theta-\text{ss}}(ns)=\left(r_i^{\oplus n}\right)/ GL(n),$$ so we have that $$\bigoplus_{n\geq 0} K_0(\X^{\theta-\text{ss}}(ns))=\bigoplus_{n\geq 0}K_0(BGL(n))=\text{Sym}\,(K_0(B\C^*)).$$
As a corollary of Theorem \ref{wallcrossing}, we get that:
\begin{cor}\label{wccor2}
The $KHA(Q)$ is generated by the $d_i$ with $1\leq i\leq n$ dimensional pieces under the multiplication map: $$KHA(Q)\cong \bigotimes_{i=1}^n \text{Sym}\,(K_0(B\C^*)).$$ 
\end{cor}

For the stability condition $\theta': \theta^1>\cdots>\theta^n$, the semistable representations are at the multiples of the roots $r_1,\cdots, r_N$ of the Lie algebra associated to $Q$, which implies that:
\begin{cor}\label{wccor3}
The $KHA(Q)$ is generated by the $r_i$ with $1\leq i\leq N$ dimensional pieces under the multiplication map: $$KHA(Q)\cong \bigotimes_{i=1}^N \text{Sym}\,(K_0(B\C^*)).$$ 
\end{cor}

The analogues of Corollaries \ref{wccor2} and \ref{wccor3} for CoHA were proved by Rimanyi \cite{r}.

\section{The PBW theorem}\label{5}

In this section we prove:
\begin{thm}\label{thm5}
There exist categories $\mathbb{M}(d) \subset D_{sg}(\X(d)_0)$ defined in Subsection \ref{defMM}, and there exists a filtration $F^{\cdot}$ on the $\text{KHA}$ such that $F^{\leq 1}:=\bigoplus_{d\in\mathbb{N}^I}K_0(\mathbb{M}(d))$, see Definition \ref{secondfiltrations}, such that the associated graded with respect to this filtration is:
$$\text{gr}^F\text{KHA}=\text{dSym}\,\left(\bigoplus_{d\in\mathbb{N}^I}K_0(\mathbb{M}(d))\right),$$ where the right hand side is an algebra called the deformed symmetric algebra, see Definition \ref{qsym}.
\end{thm}

We next explain the main steps in the proof of Theorem \ref{thm5}.

\textbf{Step $1$.} We find a semi-orthogonal decomposition of $D^b(\X(d))$ of the form:
$$D^b(\X(d))=\langle p_*q^*\left(\overline{\mathbb{N}}(d_1)_{w_1}\boxtimes\cdots\boxtimes\overline{\mathbb{N}}(d_k)_{w_k}\right)\rangle,$$
for certain sets of pairs $\{(d_1,w_1),\cdots, (d_k,w_k)\}$
called \textit{admissible with $r>\frac{1}{2}$}, defined in Subsection \ref{admissible}, and categories $\overline{\mathbb{N}}(d)$ defined in Subsection \ref{defnmagic}. The functors $p_*q^*$ are fully faithful.
Using Proposition \ref{sod}, 
we obtain a corresponding semi-orthogonal decomposition:
$$D_{sg}(\X(d)_0)=\langle p_*q^*\left(\overline{\mathbb{M}}(d_1)_{w_1}\boxtimes\cdots\boxtimes\overline{\mathbb{M}}(d_k)_{w_k}\right)\rangle,$$
for categories $\oM(d)$ defined in Subsection \ref{defnmagic}.
This implies that the $\text{KHA}$ has a basis $x_{d_1,w_1}\cdots x_{d_k,w_k}$ for $x_{d_i,w_i}\in K_0(\oM(d))$ and for $\{(d_1,w_1),\cdots,(d_k,w_k)\}$ an admissible set with $r>\frac{1}{2}$, see Subsection \ref{admissible} for definition. 
This is covered in Subsection \ref{step1}.
\\

\textbf{Step $2$.} We next find relations in $\text{gr}^F\text{KHA}$ of the form $$x_{d,w}x_{e,v}=\left(x_{e,v}q^{f(e,d)}\right)\left(x_{d,w}q^{-g(e,d)}\right),$$ 
where $x_{d,w}\in K_0(\oM(d)_w)$, $x_{e,v}\in K_0(\oM(e)_v)$, the set $\{(d,w), (e,v)\}$ is admissible with $r>\frac{1}{2}$, and
$q^{f(e,d)}$ and $q^{g(e,d)}$ are factors that depend on $d$ and $e$ only. This is covered in Subsection \ref{relations}.
\\

\textbf{Step $3.$} We show that the relations from Step $2.$ also hold for an admissible set $\{(d,w), (e,v)\}$ with $r=\frac{1}{2}$. Further, we can decompose $K_0(\oM(d))$ as follows:
$$\bigoplus_S \left(\bigoplus_{\sigma\in \mathfrak{S}_k}\boxtimes_{i=1}^{k} K_0\left(\MM(d_i)_{w_{i}}\right)\right)^{\mathfrak{S}_{k}} \cong K_0(\oM(d)),$$ where the first sum on the left hand side is 
taken after all sets $\{d_1,\cdots, d_k\}$ with sum $d$, and the second sum is taken after all orderings of elements of $S$; further, the weights $w_i$ associated to $d_i$ for $1\leq i\leq k$ are such that the set 
$\{(d_1,w_{1}),\cdots, (d_k,w_{k})\}$ is \textit{admissible set with $r=\frac{1}{2}$}, see Subsection \ref{admissible} for definition. The action of $\mathfrak{S}_k$ on the second sum of the left hand side is defined in Subsection \ref{defMM}.

This shows that $\text{KHA}$ is generated by monomials 
$x_{d_1,w_1}\cdots x_{d_k,w_k}$, where the set $\{(d_1,w_1),\cdots,(d_k,w_k)\}$
is admissible and $x_{d,w}\in K_0(\mathbb{M}(d)_w)$, and that the relations between these monomials are generated by the relations in Step $3$. This is covered in Subsection \ref{step3}.
\\

\textbf{Step $4$.} We define the deformed symmetric algebra $\text{dSym}$ as the algebra generated by monomials $x_{d_1,w_1}\cdots x_{d_k,w_k}$ with $x_{d,w}\in K_0(\mathbb{M}(d)_w)$ and with the relations from Steps $2$ and $3$. We show that $\text{dSym}$ has a generating set formed by the same monomials and with the same relations as in Step $3$, which proves Theorem \ref{thm5}.
This is covered in Subsection \ref{step4}.
\\

\subsection{Preliminaries.}
We introduce some definitions and recall some results that will be used in the proof of Theorem \ref{thm5}.
\\

\subsubsection{Notations}\label{notations}
Let $d=(d^1,\cdots,d^{n})\in\mathbb{N}^I$ be a dimension vector, where $n=|I|$.
Fix maximal tori and Borel subgroups
compatible under multiplication
$$T(d)\subset B(d)\subset G(d).$$
Let $M$ be the character lattice of $T$, $N$ the cocharacter lattice of $T$, and let $M_{\mathbb{R}}=M\otimes\mathbb{R}$
and $N_{\mathbb{R}}=N\otimes\mathbb{R}$.
We fix a Weyl invariant product $\langle\,,\, \rangle$ 
on $M$ and $N$. We let the weights in the Lie algebra of $B(d)$ be negative; this choice of $B(d)$ determines a choice of dominant chamber $M_{\mathbb{R}}^{+}\subset M_{\mathbb{R}}$.
\\

 Denote the simple roots of $G(d)$ by $\beta^i_{j}$ for $1\leq i\leq n$ and $1\leq j\leq d^i$.
For each dimension vector $d\in\mathbb{N}^I$, denote by $$\lambda_d=(z\text{Id}):\C^*\to G(d)$$ the diagonal character, and let 
$$\beta_d:=\frac{1}{|I|}
\sum_{i\in I}\frac{\beta^i_{1}+\cdots+\beta^i_{d^i}}{d^i}\in M_{\mathbb{R}}.$$ We have that $\langle \lambda_d,\beta_d\rangle=1$.
\\

\subsubsection{Partitions}\label{succ}
Let $d\in\mathbb{N}^I$ be a dimension vector. A partition  $\underline{d}=(d_1,\cdots, d_k)$ of the vector $d$ is an ordered set of dimension vectors $d_i\in\mathbb{N}^I$ such that $d_1+\cdots+d_k=d$ as vectors in $\mathbb{N}^{I}$. We consider the lexicographic order $\succ$ on $\mathbb{N}^I$. A partition $\dd$ is called decreasing if each $d_i$ appears $m_i$ times, the $m_i$ copies of $d_i$ appear consecutively for $1\leq i\leq k$, and $d_1\succ\cdots\succ d_k$.
\\

\subsubsection{Attracting loci}\label{attractingloci} 
Let $S(d)=\text{ker}\left(\text{det}:G(d)\to\C^*\right)$.
For any partition $\underline{d}$, there exist anti-dominant characters: 
$$\lambda_{\dd}:\C^*\to S(d)\subset G(d)$$ with fixed locus
$$R(d)^{\lambda_{\dd}}=R(d_1)\times\cdots\times R(d_k).$$
Denote by $L(d)$ the Levi group of $\lambda_{\dd}$; it is independent
on the choice of the character $\lambda_{\dd}$ and we have that: $$L(d)\cong G(d_1)\times\cdots\times G(d_k).$$
We consider the attracting locus $R(\dd)$
acted by the parabolic group $G(\dd)$, and let $\X(\dd)=R(\dd)/G(\dd).$
We have an affine bundle map $q_{\dd}$ and a proper map $p_{\dd}$ such that:

\begin{tikzcd}
\X(\dd) \arrow[r, "p_{\dd}"] \arrow[d, "q_{\dd}"]
    & \X(d)  \\
  \X(d_1)\times\cdots\times \X(d_k). 
\end{tikzcd}
\\
Conversely, for any character $\lambda:\C^*\to G(d)$ we have an associated partition $\dd$ by looking at the decomposition of the $\lambda$ fixed locus: $$R(d)^{\lambda}=R(d_1)\times\cdots\times R(d_k).$$

\subsubsection{Refinements of partitions}\label{compa}
We call a partition $\underline{e}=(e_1,\cdots, e_l)$ a refinement of $\underline{d}$ if there exist integers $1\leq a_1<\cdots< a_k\leq l$ such the sum of the first $a_1$ $e-$terms is $d_1$, the sum of the next $a_2-a_1$ terms is $d_2$, and so on. We write $\ee\geq\dd$.
We can define the notion of refinement for anti-dominant characters, and this coincides to the above notion of refinement once we pass to their associated partitions. 
For characters $\lambda,\mu:\C^*\to G(d)$, we call $\lambda$ a refinement of $\mu$ and write $\lambda\geq \mu$ if for every weight $\beta$ with $\langle \mu, \beta \rangle> 0$ we have that $\langle \lambda, \beta\rangle > 0$. 
\\


\subsubsection{Tree of partitions}\label{tree}
Define the tree $\mathcal{T}$ as follows: the vertices of $\mathcal{T}$ correspond to partitions $\dd$, and draw and edge from $\dd=(d_1,\cdots,d_k)$ to $\ee$ if $\ee$ is a partition of $d_i$ for some $1\leq i\leq k$. 

A set $T$ is called a \textit{tree of partitions} if it is the set of vertices of a finite subtree of $\mathcal{T}$ which has a unique vertex $v\in T$ with trivial in-degree.
Let $I\subset T$ be the set of vertices with trivial out-degree. Define the Levi associated to $T$ as:
$$L(T):=\prod_{\ee\in I} L(\ee).$$
We will also consider sets with elements characters such that the corresponding partitions form a tree of partitions.




\subsubsection{The polytope $\overline{\mathbb{W}}$}\label{polytope} 
Consider the following polytope in $M_{\mathbb{R}}$:
$$\overline{\mathbb{W}}=\bigoplus_{\beta\text{ wt of }R(d)}\,[0,\beta]\oplus\mathbb{R}\beta_d\subset M_{\mathbb{R}},$$
where the sum is taken after the weights $\beta$ of $R(d)$. Observe that the weights $\beta$ of the $G(d)$ representation $R(d)$ have the form $\beta=\beta^i_{j}-\beta^k_{l}$, where $i$ and $k$ are vertices with an edge between them. We use the notation $\mathbb{W}\subset \WW$ for its interior.

The boundary of the polytope $\frac{1}{2}\WW$ is contained in the subspaces given by equations of the form:
$$F(\lambda): \langle \lambda, -\rangle+\langle \lambda, \frac{1}{2}N^{\lambda>0}\rangle=0,$$
where $\lambda:\C^*\to S(d)$ is a non-trivial character. 
Two characters $\lambda$ and $\mu$ are $\lambda\leq \mu$ if and only if $F(\mu)\subset F(\lambda)$.
For a character $\lambda$, we denote by $F(\lambda)^{\text{int}}\subset F(\lambda)$ the open subspace:
$$F(\lambda)^{\text{int}}=F(\lambda)-\bigcup_{\mu>\lambda}F(\mu).$$

For $L=G(d_1)\times\cdots\times G(d_k)$ a Levi subgroup of $G$, denote by $R(d)^L\subset R(d)$ the $L$ fixed subspace. 
Define: $$\WW(L)=\bigoplus_{1\leq i\leq k}\,\WW(d_i)\subset M_{\mathbb{R}}
,$$ where we identify the simple roots $\beta^i_j$ of $G(d_1)$ with the first $d_1$ simple roots of $d$, and so on. 

\subsubsection{Definition of the categories $\overline{\mathbb{N}}(d)$ and $\overline{\mathbb{M}}(d)$}\label{defnmagic}
We define $\overline{\mathbb{N}}(d)\subset D^b(\X(d))$ as the subcategory generated by $\OO_{\X(d)}\otimes V(\chi)$, where $\chi$ is a dominant weight of $G(d)$ such that:
$$\chi+\rho\in \frac{1}{2}\WW.$$ 
In the case of a general potential $W$, we define $\overline{\mathbb{M}}(d):=D_{sg}(\overline{\mathbb{N}}(d))$ a subcategory of $D_{sg}(\X(d)_0)$, see Subsection \ref{singul}.
We similarly define $\overline{\mathbb{N}}^{\text{ab}}(d)\subset D^b(\Y(d))$ and $\overline{\mathbb{M}}^{\text{ab}}(d)\subset D_{sg}(\Y(d)_0)$, where in this case $\chi$ is a $T(d)$-weight.
\\


\subsubsection{The $r$-invariant.}
For $\chi$ a real weight of $M_{\mathbb{R}}$, we define its $r$-invariant to be the smallest positive real number $r$ such that:
$$\chi\in r\WW.$$ 
The boundary $r\partial\WW$ is contained in the subspaces given by the equations
$$F_r(\lambda): \langle \lambda, -\rangle+r\langle \lambda, N^{\lambda>0}\rangle=0,$$
where $\lambda$ is a character with image on $S(d)$. For $\chi$ a real weight in $M_\mathbb{R}$ such that $\chi\in r\partial\WW$, we have that:
$$r=\text{max}\,\frac{\langle \lambda, \chi\rangle}{\langle \lambda, N^{\lambda>0}\rangle}
=-\text{min}\,\frac{\langle \lambda, \chi\rangle}{\langle \lambda, N^{\lambda>0}\rangle},$$
where the min and max are taken after all $S(d)$ characters $\lambda$. When $\chi$ is dominant, we have that:
$$r=-\text{min}\,\frac{\langle \lambda, \chi\rangle}{\langle \lambda, N^{\lambda>0}\rangle},$$ where the min is taken after all $S(d)$ anti-dominant character $\lambda$.
\\

If $\chi\in M_\mathbb{R}$ has $r(\chi)=r$, we can write:
$$\chi=\sum c_{ij}(\beta_i-\beta_j)+ c\beta_d,$$ where the coefficients satisfy $-r\leq c_{ij}\leq 0$ and $c\in\mathbb{R}$. 
The $p$-invariant for $\chi$ with $r(\chi)=r$ is defined as the smallest number of coefficients equal to $r$ in a formula as above.

\begin{prop}\label{rem1} (a) Two weights of dimension $d$ that differ by a multiple of $\beta_d$ have the same $r$ invariant.

(b) Let $\chi$ be a dominant weight and $w\in W$. Then: $$(r,p)(\chi+\rho)=(r,p)(w*\chi+\rho).$$
\end{prop}

\begin{proof}
The proof of $(a)$ is clear. For $(b)$, we have that $\WW\subset M_{\mathbb{R}}$ is Weyl invariant, and also that: $$w*\chi+\rho=w(\chi+\rho)-\rho+\rho=w(\chi+\rho).$$
\end{proof}

\begin{prop}\label{prop}
Let $\chi\in M_{\mathbb{R}}$ be a real weight with $r(\chi)=r$, and let $\lambda$ be the character with $\chi\in F_r(\lambda)^{\text{int}}.$
Then we can write:
$$\chi=\sum c_{ij}(\beta_i-\beta_j)+ c\beta_d,$$ with $c\in\mathbb{R}$ and such that:
\[  c_{ij}= 
     \begin{cases}
       0 \text{ if }\langle \lambda, \beta_i-\beta_j\rangle<0,\\
       -r \text{ if } \langle \lambda, \beta_i-\beta_j\rangle>0, \\
       \text{in }(-r,0] \text{ if }
\langle \lambda, \beta_i-\beta_j\rangle=0.  \\ 
     \end{cases}
\]
\end{prop}

\begin{proof}
Let $\alpha\in M_{\mathbb{R}}$ be a small weight with $\langle \lambda, \alpha\rangle>0$. Then $\chi-\alpha$ will not be in $r\mathbb{W}$ and thus it will have $r(\chi-\alpha)>r(\chi)$. Further, $\chi+\alpha$ will still be inside $r\mathbb{W}$, and thus $r(\chi+\alpha)\leq r(\chi)$. In the above expression:
$$\chi=\sum c_{ij}(\beta_i-\beta_j)+ c\beta_d,$$
assume that there exists a weight $\langle \lambda, \beta_i-\beta_j\rangle>0$ such that its coefficients $c_{ij}>-r$. But then for small $\varepsilon>0$, we will have that: $$\chi-\varepsilon(\beta_i-\beta_j)=\sum c'_{ij}(\beta_i-\beta_j)+c\beta_d$$ with all coefficients $-r\leq c'_{ij}\leq 0$, so $\chi+\varepsilon'(\beta_i-\beta_j)\in r\mathbb{W}$, which contradicts the above observation. The argument for why $c_{ij}=0$ when $\langle \lambda, \beta_i-\beta_j\rangle<0$ is similar.

Next, we show that we can choose the coefficients $0\geq c_{ij}>-r$ for weights $\beta_i-\beta_j$ such that $\langle \lambda, \beta_i-\beta_j\rangle=0$.
Let $\psi:=\chi+rN^{\lambda>0}$, then:
$$\psi=\sum c_{ij}(\beta_i-\beta_j)+c\beta_d,$$ where the sum is after all weights such that $\langle \lambda, \beta_i-\beta_j\rangle=0$, and $-r\leq c_{ij}\leq 0$.
If $r(\psi)=r$, then by the above argument there exists an anti-dominant character $\mu$ such that:
$$\psi=\sum c_{ij}(\beta_i-\beta_j)+c\beta_d,$$ where the sum is after $\langle \lambda, \beta_i-\beta_j\rangle=0$ and:
\[  c_{ij}= 
     \begin{cases}
       0 \text{ if }\langle \mu, \beta_i-\beta_j\rangle<0,\\
       -r \text{ if } \langle \mu, \beta_i-\beta_j\rangle>0, \\
       \text{between }0\text{ and }r \text{ if }
\langle \mu, \beta_i-\beta_j\rangle=0.  \\ 
     \end{cases}
\]
The character $\lambda$ has associated Levi $L(\lambda)$, and $\mu$ is a character of $L(\lambda)$ with associated Levi $L(\mu)$. Let $\mu'$ be a character of $\C^*\to S(d)\subset G(d)$ with associated Levi $L(\mu)\subset G(d)$. If $\lambda$ and $\mu$ are anti-dominant, then $\mu'$ is anti-dominant.
We have that:
$$\langle\mu',\chi\rangle+r\langle\mu',N^{\mu'>0}\rangle=0,$$ 
so $\chi\in F_r(\mu')$, which contradicts the choice of $\lambda$. This implies that indeed $r(\psi)<r$ and then we can choose all the coefficients $c_{ij}$ to be $-r<c_{ij}\leq 0$.
\end{proof}

The following is an immediate corollary of Proposition \ref{prop}:

\begin{prop}\label{godown}
Let $\chi$ be a dominant real weight in $M_\mathbb{R}$ with $r(\chi)=r\geq \frac{1}{2}$,
and let $\lambda$ be the anti-dominant character such that $\chi\in F_r(\lambda)^{\text{int}}$. Then: 
$$\chi=-rN^{\lambda>0}+\psi,$$ where $r(\psi)=s<r$ and $\psi\in s\mathbb{W}(L)$. 

Applying this result repeatedly, for any dominant real weight $\chi$ there exists a tree of anti-dominant partitions $T$ with Levi group $L$, see Subsection \ref{tree}, such that:
$$\chi=-\sum_{\lambda\in T} r_{\lambda}N^{\lambda_k>0}+\psi,$$ and:
\begin{itemize}
    \item if $\lambda, \mu\in T$ are vertices such that there exists a path from $\lambda$ to $\mu$, then $r_{\lambda}> r_{\mu}\geq \frac{1}{2}$,
    \item $\psi\in \frac{1}{2}\mathbb{W}(L)$. 
\end{itemize}
\end{prop}


\begin{defn}\label{gen}
Let $\chi$ be a dominant weight of $G$. Using Proposition \ref{godown}, there exists a tree of anti-dominant partitions $T$ with Levi group $L$ such that:
$$\chi+\rho=-\sum_{\lambda\in T} r_{\lambda}N^{\lambda_k>0}+\chi',$$ where the coefficients satisfy:
\begin{itemize}
    \item if $\lambda, \mu\in T$ are vertices such that there exists a path from $\lambda$ to $\mu$, then $r_{\lambda}> r_{\mu}\geq \frac{1}{2}$,
    \item  $\chi'$ is a weight of $L$ such that $\chi'\in \frac{1}{2}\mathbb{W}(L)$. 
\end{itemize}
We call the above decomposition of $\chi+\rho$ the \textbf{standard form} of $\chi$.
\end{defn}

The following proposition is contained in the proof of \cite[Proposition 3.11]{hls}.

\begin{prop}\label{rgoesdown}
Let $\chi$ be a dominant weight such that $r(\chi)=r>\frac{1}{2}$, and let $\lambda$ be the maximal anti-dominant character such that $\chi\in F_r(\lambda)^{\text{int}}$.
Let $\sigma$ be a nonzero partial sum in $N^{\lambda<0}$. Then $V(\chi-\sigma)\otimes\OO$ is in $r\mathbb{W}$ or on the interior faces:
$$\bigcup_{\mu\leq\lambda}F_r(\mu)^{\text{int}}.$$
In particular, using the Borel-Bott-Weyl theorem, see Subsection \ref{sp}, the complex $p_{\lambda*}q_{\lambda}^*\left(V(\chi)\otimes\OO\right)$ is in $r\WW$ and the only interior faces it touches are $F_r(\mu)^{\text{int}}$ for $\mu\leq\lambda$; it touches the face $F_r(\lambda)^{\text{int}}$ only for $V(\chi)\otimes\OO$. In particular, for a nontrivial sum $\sigma$ as above, we have:
$$(r,p)(\chi-\sigma)<(r,p)(\chi).$$
\end{prop}

\begin{proof}
Use Proposition \ref{prop} to write
$$\chi=\sum c_{ij}(\beta_i-\beta_j)+c\beta_d,$$ 
where the sum is after all the roots of $R(d)$ and:
\[  c_{ij}= 
     \begin{cases}
       -r \text{ if }\langle \lambda, \beta_i-\beta_j\rangle>0,\\
       0 \text{ if } \langle \lambda, \beta_i-\beta_j\rangle<0, \\
       \text{in }(-r, 0] \text{ if }
\langle \lambda, \beta_i-\beta_j\rangle=0.  \\ 
     \end{cases}
\]
We have that $-\sigma=\sum (\beta_i-\beta_j)$, where the weights have $\langle \lambda, \beta_i-\beta_j\rangle>0$.
This means that:
$$\chi-\sigma=\sum c'_{ij}(\beta_i-\beta_j),$$ 
where $c'_{ij}=c_{ij}+1=-r+1$ if $\beta_i-\beta_j$ appears in the sum $-\sigma$, or $c'_{ij}=c_{ij}$ otherwise. 
If $-r+1>0$, replace the weights $(-r+1)(\beta_i-\beta_j)$ with the weights $(r-1)(\beta_j-\beta_i)$ so that the coefficients in $\chi-\sigma$ will all be between $-r$ and $0$.
This means that $$\chi-\sigma\in r\WW.$$
Assume next that
$r(\chi-\sigma)=r$. Use Proposition \ref{prop} for $\chi-\sigma$ and let $\mu$ be the character with $\chi-\sigma\in F_r(\mu)^{\text{int}}$.
We then have that $c_{ij}=-r$ for all weights $\beta_i-\beta_j$ such that 
$\langle \mu, \beta_i-\beta_j\rangle>0.$ The assumption on $\lambda$ implies that $\mu\leq \lambda$, and equality can happen only for $\sigma=0$.
\end{proof}

\begin{prop}\label{boundary}
Let $\chi$ be a weight in $M_{\mathbb{R}}$ such that $r(\chi)=\frac{1}{2}$. Let $\lambda$ be an anti-dominant character with associated partition $\dd$, let $\chi\in F(\lambda)^{\text{int}}$, and let $\sigma$ be a nonzero partial sum in $N^{\lambda<0}$. 
Then $V(\chi-\sigma)\otimes\OO$ is in $\frac{1}{2}\mathbb{W}$ or on the faces:
$$\bigcup_{\mu}F(\mu),$$
where the union is after all characters $\mu$ such that $R(d)^{\lambda}\subset R(d)^{\mu}$.

\end{prop}

\begin{proof}
Use Proposition \ref{prop} to write:
$$\chi=\sum c_{ij}(\beta_i-\beta_j),$$ 
where the sum is after all the roots of $R(d)$ and:
\[  c_{ij}= 
     \begin{cases}
       -\frac{1}{2} \text{ if }\langle \lambda, \beta_i-\beta_j\rangle>0,\\
       0 \text{ if } \langle \lambda, \beta_i-\beta_j\rangle<0, \\
       \text{in }(-\frac{1}{2}, 0] \text{ if }
\langle \lambda, \beta_i-\beta_j\rangle=0.  \\ 
     \end{cases}
\]
We have that $-\sigma=\sum (\beta_i-\beta_j)$, where the weights have $\langle \lambda, \beta_i-\beta_j\rangle>0$.
This means that we can write:
$$\chi-\sigma=\sum c'_{ij}(\beta_i-\beta_j),$$ 
where $c'_{ij}=c_{ij}+1=\frac{1}{2}$ if $\beta_i-\beta_j$ appears in the sum $\sigma$, or $c'_{ij}=c_{ij}$ otherwise. 
Replace $\frac{1}{2}(\beta_i-\beta_j)$ with $-\frac{1}{2}(\beta_j-\beta_i)$ so that the coefficients $c'_{ij}$ will all be between $-\frac{1}{2}$ and $0$.
In particular, we see that if
$\chi-\sigma=\sum c'_{ij}(\beta_i-\beta_j)$ has coefficient $c'_{ij}=-\frac{1}{2}$, then either 
$\langle \lambda, \beta_i-\beta_j\rangle< 0$ and $\beta_i-\beta_j$ does not appear in $\sigma$, or 
$\langle \lambda, \beta_i-\beta_j\rangle> 0$ and $\beta_j-\beta_i$ appears in $\sigma$. Let $\mu$ be the maximal character such that $\chi-\sigma\in F_r(\mu)$. We see using the above expression for $\chi-\sigma$ that
$\langle \mu, \beta_i-\beta_j\rangle<0$ implies $\langle \lambda, \beta_i-\beta_j\rangle\neq 0.$ This means that $R(d)^{\lambda}\subset R(d)^{\mu}$, and the conclusion follows.
\end{proof}

\begin{prop}\label{nulemma}
Let $\dd=(d_1,\cdots,d_k)$ be a partition of $d$.

(a) Consider the natural embedding $\mathfrak{S}_{k}\subset W_d$ which permutes the weights corresponding to the terms $d_i$. Let $w\in \mathfrak{S}_{k}\subset W_d$ be an element of the Weyl group. Then:
$$\frac{1}{2}N^{\lambda<0}-\frac{1}{2}N^{w\lambda<0}=\sum_{U} N^{ab},$$ where the set $U$ contains pairs $(a,b)$ such that $1\leq a<b\leq k$ and $\sigma(a)>\sigma(b)$, and $N^{ab}$ denotes the sum of the weights in $N^{\lambda<0}$ corresponding to $d_a$ and $d_b$.

(b) Let $\chi$ be a weight in $M_{\mathbb{R}}$ such that $r(\chi)=\frac{1}{2}$. Let $\lambda$ be an anti-dominant character, let $\chi\in F(\lambda)^{\text{int}}$, and let $\sigma$ be a nonzero partial sum in $N^{\lambda<0}$. The weight $\chi-\sigma\in F(w\lambda)^{\text{int}}$ for some $w\in W_d$ if and only if $$\sigma=\sum_{U} N^{ab}$$ for a set $U$ containing pairs $(a,b)$ such that $1\leq a<b\leq k$ and $\sigma(a)>\sigma(b)$, and $N^{ab}$ defined as above.

\end{prop}

\begin{proof}
(a) We can write $N^{\lambda<0}$ as $$N^{\lambda<0}=\sum_{1\leq a<b\leq k}N^{ab},$$ and we thus compute that:
$$N^{\lambda<0}-N^{w\lambda<0}=2\sum N^{ab},$$ where the sum is taken after all $1\leq a<b\leq k$ such that $\sigma(a)>\sigma(b)$. 

(b) By Proposition \ref{boundary}, we have that $R(d)^{\lambda}\subset R(d)^{w\lambda}$, and thus they have to be equal. This implies that $w$ permutes blocks of sizes $d_i$ for $1\leq i\leq k$ between them; we can assume that $w$ sends dominant weights for a block to dominant weights for another block because the sum $N^{w\lambda<0}$ does not matter on a permutation inside a block. We are thus in the situation of (a), and the conclusion follows.
\end{proof}

\subsubsection{Admissible sets}
\label{admissible}
Consider an ordered set of pairs $\{(d_1,w_1),\cdots, (d_k,w_k)\}$. Define a weight of $G(d)$ by: $$\chi=w_1\beta_{d_1}+\cdots+w_k\beta_{d_k},$$ where we identify the simple roots $\beta^i_j$ of $d_1$ with the first $d_1$ simple roots of $d$, and so on. 
Such a set is called \textbf{admissible} if its standard form, see Definition \ref{gen}, has the form:
$$\chi+\rho=-\sum r_iN^{\lambda_i>0}-\sum_S \frac{1}{2} N^{\mu_i>0}+\chi',$$
with associated Levi group $L= G(d_1)\times\cdots\times G(d_k)$. 

If the set $S$ is empty, we call the set $\{(d_1,w_1),\cdots, (d_k,w_k)\}$ \textbf{admissible with $r>\frac{1}{2}$}.
If the first sum is empty, we call the set $\{(d_1,w_1),\cdots, (d_k,w_k)\}$ \textbf{admissible with $r=\frac{1}{2}$}.
The proof of the following proposition is immediate:
\begin{prop}\label{propadm}
Consider an ordered set of pairs $\{(d_1,w_1),\cdots, (d_k,w_k)\}$, and consider the weight of $G(d)$: $$\chi=w_1\beta_{d_1}+\cdots+w_k\beta_{d_k}.$$
Then the standard form of $\chi$:
$$\chi+\rho=-\sum r_kN^{\lambda_k>0}+\chi'$$
has Levi group $L$ with associated partition $\ee=(e_1,\cdots,e_s)$ such that $\ee\leq\dd$. Further, $\chi'$ has the form $\chi'=\rho^{L}+v_1\beta_{d_1}+\cdots+v_k\beta_{d_k}$ for some integers $v_1,\cdots,v_k$, where $\rho^L$ is half the sum of positive roots of $L$.

In particular, if the set $\{(d_1,w_1),\cdots, (d_k,w_k)\}$ is admissible, then the Levi group associated to the standard form is $L=G(d_1)\times\cdots\times G(d_k)$ and $\chi'=\rho^{L}+v_1\beta_{d_1}+\cdots+v_k\beta_{d_k}$ for some integers $v_1,\cdots, v_k.$
\end{prop}

The next proposition says that we can replace $w_i\beta_{d_i}$ with any $\chi\in \NN(d_i)_{w_i}$ in the definition of admissible sets.
\begin{prop}\label{rem2}
Consider an ordered set of pairs $\{(d_1,w_1),\cdots, (d_k,w_k)\}$, let
$\chi_i \in \overline{\mathbb{N}}(d_i)_{w_{i}}$ be dominant weights, and define $\tau:=\chi_1+\cdots+\chi_k$. 
Consider the standard form:
$$\tau+\rho=-\sum r_iN^{\lambda_i>0}+\tau'.$$ 
Then $\tau$ is dominant, all $r_i>\frac{1}{2}$, and the associated Levi is $L=G(d_1)\times\cdots\times G(d_k)$ if and only if the set $\{(d_1,w_1),\cdots, (d_k,w_k)\}$ is admissible with $r>\frac{1}{2}$.
\end{prop}

\begin{proof}
Let $\chi:=w_1\beta_{d_1}+\cdots+w_k\beta_{d_k}$. Use Proposition \ref{godown} and group the terms with $r=\frac{1}{2}$ to write:
$$\chi+\rho=-\sum r_iN^{\lambda_i>0}+\chi',$$ with all $r_i>\frac{1}{2}$, $\chi'\in\frac{1}{2}\overline{\mathbb{W}}$, and associated Levi $L$ whose associated partition $\ee$ satisfies $\ee\leq \dd$. 
Write $\chi_i=\chi_i'+w_i\beta_{d_i}$, where $\chi_i'\in \mathbb{N}(d_i)_0$. Then we have that $\tau=\chi+(\chi'_1+\cdots+\chi'_k),$ and using the standard form of $\chi$ we can write:
$$\tau+\rho=-\sum r_kN^{\lambda_k>0}+(\chi'+\chi'_1+\cdots+\chi'_k).$$ We claim that $\chi'+\chi'_1+\cdots+\chi'_k\in \frac{1}{2}\overline{\mathbb{W}}(L).$
First, we can write: $$\chi'=t_1\beta_{d_1}+\cdots+t_k\beta_{d_k}+\rho_{d_1}+\cdots+\rho_{d_k},$$ where $t_1,\cdots,t_k$ are integers.
We have that:
$$t_1\beta_{d_1}+\cdots+t_k\beta_{d_k}\in 
\frac{1}{2}\left(\overline{\mathbb{W}}(L)\cap \text{sum}\,[0,\beta^i_{k}-\beta^j_{l}]\right),$$
where the sum is after all weights $\beta^v_{k}-\beta^j_{l}$
such that the indices $k$ and $l$ are not in the same interval $[\sigma_i,\sigma_{i+1}]$ for $1\leq i\leq n-1$, where $\sigma_i=d_1+\cdots+d_i$. Further, for every $1\leq i\leq k$ we have that:
$$\chi'_i+\rho_{d_i}\in \frac{1}{2}\overline{\mathbb{W}}(d_i)_0,$$
where $\overline{\mathbb{W}}(d_i)_0\subset \overline{\mathbb{W}}(d_i)$ denotes the hyperplane where the coefficient of $\beta_{d_i}$ is zero.
This implies that indeed:
$$\chi'+\chi'_1+\cdots+\chi'_k=t_1\beta_{d_1}+\cdots+t_k\beta_{d_k}+\sum_{i=1}^k(\chi'_i+\rho_{d_i})
\in \frac{1}{2}\overline{\mathbb{W}}(L).$$
The standard form of $\tau$ is:
$$\tau+\rho=-\sum r_kN^{\lambda_k>0}+(\chi'+\chi'_1+\cdots+\chi'_k).$$ Thus the Levi groups of $\tau$ and $\chi$ are the same, and this implies the desired conclusion.

\end{proof}

We will frequently use the following proposition in this section:
\begin{prop}\label{subsetmic}
Let $\{(d_1,w_1),\cdots, (d_k,w_k)\}$ be an admissible set
with $r=\frac{1}{2}$. Then any ordered subset is also admissible with $r=\frac{1}{2}$.
\end{prop}

\subsubsection{Definition of the categories $\mathbb{N}(d)$ and $\MM(d)$}\label{defMM}
The categories $\NN(d)$ and $\oM(d)$ contain complexes supported on attracting loci, namely the complexes contained in the categories $p_{\dd*}q_{\dd}^*\left(\boxtimes_{i=1}^k\oM(d_i)_{w_i}\right)$ for an admissible set $\{(d_1,w_1),\cdots, (d_k,w_k)\}$ with $r=\frac{1}{2}$. We will define subcategories of $\NN(d)$ and $\oM(d)$ who do not contain such complexes.
\\

Let $d$ be a dimension vector and $w,v$ weights such that $\{(d,w), (e,v)\}$ is admissible for $r=\frac{1}{2}$. Then there exist weights $w'$ and $v'$ such that $\{(e,v'), (d,w')\}$ are admissible for $r=\frac{1}{2}$.
Let $\tau\in\mathfrak{S}_2$ be the functor $\tau: \NN(d)_w\boxtimes \NN(e)_v\to \NN(e)_{v'}\boxtimes \NN(d)_{w'}$
which sends:
$$\tau: V(\chi_1)\boxtimes V(\chi_2)\otimes\OO \to V(\chi_2)\boxtimes V(\chi_1)\otimes\text{det}(N_{e,d}).$$ Here $\text{det}(N_{e,d})$ is a sheaf in cohomological degree $0$. It induces a corresponding functor $$\tau:\oM(d)_w\boxtimes\oM(e)_v\to \oM(e)_{v'}\boxtimes \oM(d)_{w'}$$ with $\tau^2=\text{id}$. These functors thus determine an action of $\mathfrak{S}_{k}$ on $\boxtimes_{j=1}^k \oM(d_i)_{w_i}$, for an admissible set $\{(d_1,w_1),\cdots, (d_k,w_k)\}$ with $r=\frac{1}{2}$.
\\

Consider a set $S$ of cardinal $k$ of dimension vectors with sum $d$. For any ordering $d_1,\cdots, d_k$ of elements of $S$, consider all ordered sets of weights $(w_1,\cdots, w_k)$ such that
$\{(d_1,w_{1}),\cdots, (d_k,w_{k})\}$ is admissible with $r=\frac{1}{2}$.

Denote by $\bigoplus_{\mathfrak{S}_k} \left(\boxtimes_{i=1}^k \oM(d_i)_{w_{i}}\right)$ the direct sum after all $|\mathfrak{S}_k|$ orderings of $S$. 
Define the full subcategory: 
$$\mathcal{A}_{S}\subset \bigoplus_{\mathfrak{S}_k} \left(\boxtimes_{i=1}^k \oM(d_i)_{w_{i}}\right)$$
generated by the complexes which are extensions of $F$ and $\sigma(F)[\text{odd}]$, where $F$ is an object of $\bigoplus_{\mathfrak{S}_k} \left(\boxtimes_{i=1}^k \oM(d_i)_{w_{i}}\right)$ and $\sigma\in\mathfrak{S}_k$ is as above.

\begin{prop}\label{propprop}
There is a short exact sequence:
$$0\to K_0(\mathcal{A}_{S})\to \bigoplus_{\mathfrak{S}_k} \boxtimes_{i=1}^k K_0\left(\oM(d_i)_{w_{i}}\right)
\xrightarrow{T}
\left(\bigoplus_{\mathfrak{S}_k} \boxtimes_{i=1}^k K_0\left(\oM(d_i)_{w_{i}}\right)\right)^{\mathfrak{S}_k}
\to 0,$$
where the map $T$ is defined by $T(x)= \sum_{\sigma\in\mathfrak{S}_k}\sigma(x)$.

Further, $\mathcal{A}_{S}$ is the subcategory of $\bigoplus_{\mathfrak{S}_k}\boxtimes_{i=1}^k \oM(d_i)_{w_{i}}$
generated by sheaves $E$ such that: $$[E]\in K_0(\mathcal{A}_d)\subset \bigoplus_{\mathfrak{S}_k} \boxtimes_{i=1}^k K_0\left(\oM(d_i)_{w_{i}}\right).$$
\end{prop}

\begin{proof}
The subcategory $\mathcal{A}_{S}$ is dense in $\bigoplus_{\mathfrak{S}_k} \left(\boxtimes_{i=1}^k \oM(d_i)_{w_{i}}\right)$, and both of the above results follow from \cite[Theorem 2.1 and Corollary 2.3]{Th}.
\end{proof}

Let $\dd$ be a partition of $d$, $w_i$ weights such that the set $\{(d_1,w_1),\cdots, (d_k,w_k)\}$ is admissible with $r=\frac{1}{2}$, and let $w=\sum w_{i}$ be a weight for $d$. We have an inclusion $$p_{\dd*}q_{\dd}^*\left(\boxtimes_{i=1}^k \oM(d_i)_{w_i}\right)\subset \MM(d)_w.$$
In general, the functor $p_{\dd*}q_{\dd}^*$ is not fully faithful. 
Consider the truncation functor \cite[Proposition 3.9]{hl}:
$$\beta_{\geq w'}: D_{sg}(\X(\dd)_0)\to D_{sg}(\X(\dd)_0)$$ for the character $\lambda_{\dd}^{-1}$ and weight $w'=\frac{1}{2}\langle \lambda_{\dd}, \mathcal{N}^{\lambda_{\dd}>0}\rangle$.
Define the functor:
$$\Phi_{\dd}=\beta_{\geq w'}p_{\dd}^*:\oM(d)_w\to  \boxtimes_{i=1}^k \oM(d_i)_{w_{i}}.$$
For a set $S$ as above, consider the functor:
$$\Phi_S:=\bigoplus_{\mathfrak{S}_k}\Phi_{\dd}:\oM(d)_w\to 
\bigoplus_{\mathfrak{S}_k} \boxtimes_{i=1}^k \oM(d_i)_{w_{i}}.$$
We use the same notation $\Phi_S$ for the direct sum of the above functors after all weights $w\in\mathbb{Z}$.

\begin{defn}
Let $\mathbb{M}(d)$ be the full subcategory
of 
$\overline{\mathbb{M}}(d)$ with objects complexes $F$ such that for every set $S$ of sum $d$, we have:
$$\Phi_{S}(F)\in\mathcal{A}_{S}.$$
\end{defn}

Observe that it is enough to check the above condition for sets $T$ of cardinal $2$. Indeed, for $S=\{d_1,\cdots,d_k\}$ of cardinal $k$, split the elements in two sets $S_1$ and $S_2$ of sums $e_1$ and $e_2$ respectively, and let $T=\{e_1, e_2\}$. Then $$\Phi_{S}\Phi_T(F):=\left(\Phi_{S_1}\boxtimes \Phi_{S_2}(\Phi_{(e_1,e_2)}(F))\right)\oplus \left(\Phi_{S_2}\boxtimes \Phi_{S_1}(\Phi_{(e_2,e_1)}(F))\right)\in \bigoplus_{\mathfrak{S}_k}\boxtimes_{i=1}^k \oM(d_i)_{w_i}$$ is a direct summand of $\Phi_{S}(F)\in \bigoplus_{\mathfrak{S}_k}\boxtimes_{i=1}^k \oM(d_i)_{w_i}$. We have that $\Phi_T(F)\in\mathcal{A}_T$, and thus $\Phi_{S}\Phi_T(F)\in \mathcal{A}_S$.

\textbf{Remark.} There are other possible definitions of the categories $\oM(d)$ and $\MM(d)$, for example we can define $\overline{\mathbb{N}}(d)$ as the category generated by $V(\chi)\otimes\OO$ for $\chi$ a dominant weight such that:
$$\chi+\rho+\delta\in \frac{1}{2}\WW\oplus \mathbb{R}\beta_d,$$ where $\delta_d\in M^W_{\mathbb{R}}$ such that $\delta_d=\sum_{i=1}^k \delta_{d_i}$ for all partitions $\dd=(d_1,\cdots, d_k)$.
\\

\subsubsection{Comparison of admissible sets.}\label{compadm}
Let $\chi$ and $\psi$ be real weights. We define a partial order of real weights denoted by $\to$.

Assume that $\chi\in F_r(\lambda)^{\text{int}}$ and $\psi\in F_s(\mu)^{\text{int}}$. If $r>s$, then $\chi\to\psi$; if $r=s$ and $\lambda\geq \mu$, then $\chi\to\psi$; if $r=s$ and $\lambda=\mu$, write $\chi'=\chi+rN^{\lambda>0}$ and $\psi'=\psi+sN^{\mu>0}$, and say that $\chi\to\psi$ if $\chi'\to\psi'$.
\\

For two admissible sets $A$ and $B$ with $r>\frac{1}{2}$, consider $\chi_A$ and $\chi_B$ as in Subsection \ref{admissible}. We say that $A\geq B$ if $\chi_A\to\chi_B$.
\\

\subsection{The semi-orthogonal decomposition}\label{step1}
The theorem we want to prove is:

\begin{thm}\label{sod5}
We have a semi-orthogonal decomposition
$$D_{sg}(\X(d)_0)=\langle p_{\dd*}q_{\dd}^*\,(\oM(d_1)_{w_1}\boxtimes\cdots\boxtimes \oM(d_k)_{w_k})\rangle,$$ where the sets $\{(d_1,w_1),\cdots, (d_k,w_k)\}$ that appear above are admissible with $r>\frac{1}{2}$, and the functors $p_{\dd*}q_{\dd}^*$ are fully faithful. The order of categories in the semi-orthogonal decomposition respects the order from Subsection \ref{compadm}.
\end{thm}

\textbf{Example.} Before discussing the proof, we write explicitly what Theorem \ref{sod5} says in the same case of the quiver $Q$ with two vertices and two edges between them:

\begin{tikzcd}
 1 
\arrow[r, shift left=2,"y"] &2.  \arrow[l,"z"] 
 \end{tikzcd}
\\
Denote the linear space along the $y$ arrow by $Y$ and the one along the $z$ arrow by $Z$. We have that $\X(1,1)=Y\oplus Z/\C^*\times\C^*$. The diagonal copy of $\C^*$ in $\C^*\times\C^*$ acts trivially on $\C^2=Y\oplus Z$.
Assume that the weight lattice of the quotient $\C^*=(\C^*)^2/\text{diag}$ is generated by $\beta$, and that quotient $\C^*$ acts with weight $\beta$ on $Y$ and with weight $-\beta$ on $X$. Denote by $\Y=Y/\C^*$ and $\mathcal{Z}=Z/\C^*$.

The category $\overline{\mathbb{N}}$ is generated by those characters $\chi$ such that
$$\chi+\rho\in \frac{1}{2}[-\beta,\beta].$$ 
We have that $\rho=0$, so the only weights as above is $\chi=0$. This means that $\overline{\mathbb{N}}$ is the subcategory of $D^b(\C^2/\C^*)$ generated by $\OO_\X$. 
There are two possible characters $\lambda$ and $\lambda^{-1}$ corresponding to the two possible partitions of the dimension vector $(1,1)$. 
The attracting loci are of the form $$\text{pt}/\C^* \xleftarrow{q} \C/\C^* \xrightarrow{p} \C^2/\C^*,$$ and there are two possible attracting loci, $\Y$ for $\lambda$, and $\mathcal{Z}$ for $\lambda^{-1}$. 
We denote by $\OO(w)$ the representation of $\C^*$ on which $\lambda$ acts with weight $w\beta$.
The semi-orthogonal decomposition stated by the theorem is:
$$D^b(\C^2/\C^*)=\langle p_{\lambda*}q_{\lambda}^*D^b(\text{pt})_{\leq -1}, p_{\lambda^{-1}*}q_{\lambda^{-1}}^*D^b(\text{pt})_{\geq 1}, \overline{\mathbb{N}}\rangle,$$ 
which can be also written as:
$$D^b(\C^2/\C^*)=\langle p_{\lambda*}\OO_{\Y}(w)_{\leq -1}, p_{\lambda^{-1}*}\OO_{\mathcal{Z}}(w)_{\geq 1}, \overline{\mathbb{N}}\rangle.$$ To see generation, use the exact triangles:
$$\OO_\X(w+1)\to\OO_\X(w)\to p_{\lambda*}\OO_{\Y}(w)\xrightarrow{[1]} \text{ and }
\OO_\X(w-1)\to\OO_\X(w)\to p_{\lambda^{-1}*}\OO_{\mathcal{Z}}(w)\xrightarrow{[1]}.$$
This recovers the semi-orthogonal decomposition constructed by Halpern-Leistner \cite[Theorem 2.10]{hl} for the GIT quotient $\C^2/\C^*$, linearization $\mathcal{L}=\OO(1)$, and weight $w=0$.
\\

Next, assume that the potential is $W=xy$. 
Let $\text{Coh}^{'}(\X_0)$ be the full subcategory of $D^b(\X_0)$ with objects $F$ such that
$i_*F\in\overline{\mathbb{N}}\}.$
If $F\in \text{Coh}^{'}(\X_0)$ has support $\mathcal{W}$, then $\OO_{\mathcal{W}}\in \text{Coh}^{'}(\X_0)$.
The category $\overline{\mathbb{N}}$ is generated by $\OO_{\X(1,1)}$, and the only structure sheaves supported on $\X_0$ in $\overline{\mathbb{N}}$ are $\OO_{\X_0}$, which means that:
$$\text{Coh}'(\X_0)=\langle \OO_{\X_0}\rangle =\text{Perf}'(\X_0),$$ where $\text{Perf}'(\X_0)\subset \text{Coh}'(\X_0)$ is the subcategory of perfect complexes. This means that the category $\oM(1,1)=\text{Coh}'(\X_0)/\text{Perf}'(\X_0)$ is trivial.
The category $D_{sg}(\X_0)$ is generated by the sheaves $\OO_{\Y}(w)$ for $w\in\mathbb{Z}$, or the sheaves $\OO_{\mathcal{Z}}(w)$ for $w\in\mathbb{Z}$. 
The relation between $\OO_{\Y}$ and $\OO_{\mathcal{Z}}$ in $D_{sg}(\X_0)$ is given by the distinguished triangle:
$$\OO_{\mathcal{Z}}(-w-1)\to \OO_{\X_0}\to \OO_{\Y}(w)\xrightarrow{[1]}.$$

Let $\mathcal{A}=D^b(k\xrightarrow{0}k)/\text{Perf}\,(k\xrightarrow{0}k)$, then $\mathcal{A}$ is the dg category of $\mathbb{Z}/2\mathbb{Z}$-graded vector spaces over $k$.
The semi-orthogonal decomposition of $D_{sg}(\X_0)$ from Theorem \ref{sod5} is:
$$D_{sg}(\X_0)= \langle p_{\lambda*}q_{\lambda}^*\mathcal{A}(w)_{w\leq -1}, p_{\lambda^{-1}*}q_{\lambda^{-1}}^*\mathcal{A}(w)_{w\geq 1}\rangle\cong \langle \mathcal{A}(w), w\in\mathbb{Z}\rangle.$$ 
\linebreak

We now begin the proof of Theorem \ref{sod5}. We first prove the statement in the zero potential case. The next proposition shows the generation part of Theorem \ref{sod5}: 

\begin{prop}\label{gen0}
The categories $p_{\dd*}q_{\dd}^*\,\left(\overline{\mathbb{N}}(d_1)_{w_1}\boxtimes\cdots\boxtimes \overline{\mathbb{N}}(d_k)_{w_k}\right)$, where the set $\{(d_1,w_1),\cdots, (d_k,w_k)\}$ is admissible with $r>\frac{1}{2}$, generate $D^b(\X(d))$. 
\end{prop}

\begin{proof}
Let $\chi$ be a dominant weight for $G(d)$.
It is enough to show that the sheaf $V(\chi)\otimes \OO_\X$ is generated by these categories. 
If the $r$-invariant of $\chi$ is $r(\chi)\leq\frac{1}{2}$, then $V(\chi)\otimes \OO_\X$ is in $\mathbb{N}(d)$. 

Assume that $r(\chi)>\frac{1}{2}$. Use Proposition \ref{gen} to write $\chi$ in standard form: $$\chi+\rho=-\sum r_iN^{\lambda_i>0}+\chi',$$
where the sum is taken after a tree of anti-dominant partitions, 
$\chi'=\chi'_1+\cdots+\chi'_k$ is a dominant weight of the associated Levi group $L=G(d_1)\times\cdots\times G(d_k)$ and $\chi'\in \frac{1}{2}\mathbb{W}(L)$. Let $\lambda$ be the dominant character associated to the partition $\dd=(d_1,\cdots,d_k)$.
Write $\rho=\rho^{\lambda>0}+\rho_{d_1}+\cdots+\rho_{d_k}$, where $\rho^{\lambda>0}$ is half the sum of positive roots which pair positively with $\lambda$. Define integers $v_1,\cdots, v_k$ such that:
$$-\sum r_iN^{\lambda_i>0}-\rho^{\lambda>0}=v_1\beta_{d_1}+\cdots+v_k\beta_{d_k}.$$ For $1\leq i\leq k$, consider the weight $\chi_i=\chi'_i+v_i\beta_{d_i}-\rho_{d_i}$ of dimension $d_i$ and weight $w_i$. 
The $G(d_i)$-weight $\chi_i$ has integer coefficients, is dominant, and $\chi_i\in\overline{\mathbb{N}}(d_i)_{w_i}.$
Indeed, we have that
$\chi_i+\rho_{d_i}=\chi_i'+v_i\beta_{d_i}$ and $\chi'_i\in \frac{1}{2}\overline{\mathbb{W}}(d_i)$ and thus $\chi_i'+v_i\beta_{d_i}\in \frac{1}{2}\overline{\mathbb{W}}(d_i)$ by Proposition \ref{rem1}.
The weight $\chi$ can be written as:
$$\chi=\chi_1+\cdots+\chi_k.$$
Further, by Proposition \ref{rem2}, the set of pairs $\{(d_1,w_1),\cdots, (d_k,w_k)\}$ is admissible with $r>\frac{1}{2}$.

Let $\mathcal{Y}=\X(d_1)\times\cdots\times\X(d_k)$, and consider the sheaf: $$ \left(V(\chi_1)\boxtimes\cdots\boxtimes V(\chi_k)\right)\otimes\OO_{\Y}\in \overline{\mathbb{N}}(d_1)\boxtimes\cdots\boxtimes\overline{\mathbb{N}}(d_k).$$
Consider the complex in $p_{\dd*}q_{\dd}^*\,\left(\overline{\mathbb{N}}(d_1)_{w_1}\boxtimes\cdots\boxtimes \overline{\mathbb{N}}(d_k)_{w_k}\right)$:
$$p_{\dd*}q_{\dd}^*\left( \left(V(\chi_1)\boxtimes\cdots\boxtimes V(\chi_k)\right)\otimes\OO_\Y\right).$$ 
By the discussion in Subsection \ref{sp}, the Borel-Bott-Weyl theorem says that there is a spectral sequence with terms $$V(\chi_1+\cdots+\chi_k-\sigma)\otimes\OO_\X=V(\chi-\sigma)\otimes\OO_\X,$$ where $\sigma$ is a partial sum of weights $\beta$ with $\langle \lambda_{\dd},\beta\rangle<0$. By Proposition \ref{rgoesdown}, the weights for $\sigma$ nontrivial have a smaller $(r,p)$-invariant than $\chi$, and the conclusion thus follows on induction on the $(r,p)$-invariant. 
\end{proof}

\begin{prop}\label{orth}
(a) Let $\mathcal{A}=\{(d_1,w_1),\cdots,(d_k,w_k)\}$ and $\mathcal{B}=\{(e_1,v_1),\cdots,(e_s,v_s)\}$ be two admissible sets with $r>\frac{1}{2}$
and assume that $\mathcal{A}$ is not $\geq \mathcal{B}$ as admissible sets, see Subsection \ref{compadm}. 
Consider $A\in \boxtimes_{i=1}^{k}\overline{\mathbb{N}}(d_i)_{w_i}$ and $B\in\boxtimes_{i=1}^s \overline{\mathbb{N}}(e_i)_{v_i}$. Then: 
$$\text{Hom}\,(p_{\dd*}q_{\dd}^*A, p_{\ee*}q_{\ee}^*B)=0.$$

(b) Let $\mathcal{A}=\{(d_1,w_1),\cdots,(d_k,w_k)\}$ be an admissible set with $r>\frac{1}{2}$, and let $A,B\in \boxtimes_{i=1}^{k}\overline{\mathbb{N}}(d_i)_{w_i}$. Then:
$$\text{Hom}\,(p_{\dd*}q_{\dd}^*A, p_{\dd*}q_{\dd}^*B)=\text{Hom}(A,B).$$
\end{prop}

\begin{proof}
It is enough to prove the result for generators of the above categories: $$A=V(\chi_A)\otimes\OO_\X\in \boxtimes_{i=1}^k\mathbb{N}(d_i)_{w_i}\text{ and }B=V(\chi_B)\otimes\OO \in \boxtimes_{i=1}^s \mathbb{N}(e_i)_{v_i}.$$
Let $\lambda$ and $\mu$ be characters that realize the $r$-invariants of $\chi_A$ and $\chi_B$:
$$\chi_A\in F_r(\lambda)^{\text{int}}\text{ and }\chi_B\in F_s(\mu)^{\text{int}}.$$
 Assume for simplicity that we are in the case when $s>r$ or they are equal and $\mu$ is not $\geq\lambda$; if $s=r$ and $\lambda=\mu$, we need to consider the first characters $\lambda_i$ and $\mu_i$ in the standard forms of $\chi_A$ and $\chi_B$ that differ, and the argument is the same for them.
Using adjunction, it is enough to prove that:
$$\text{Hom}\,(p_{\ee}^*p_{\dd*}q_{\dd}^*A, q_{\ee}^*B)=0.$$
We have that $q_{\ee}^*B\in F_s(\mu)^{\text{int}}$ and, by Proposition \ref{rgoesdown}, $p_{\dd*}q_{\dd}^*A\in r\overline{\mathbb{W}}$ and the only faces in $r\overline{\mathbb{W}}$ it touches are $F_r(\nu)$ for $\nu\leq\lambda$. This means that:
$$\langle \mu, p_{\dd*}q_{\dd}^*A\rangle >\langle \mu, B\rangle,$$ and the conclusion for part (a) follows.

When $\mathcal{A}=\mathcal{B}$, we need to show that:
$$\text{Hom}\,(p_{\dd}^*p_{\dd*}q_{\dd}^*A, q_{\dd}^*B)=\text{Hom}\,(A,B).$$
We have that $q_{\dd}^*B\in F_r(\lambda)^{\text{int}}$ and, by Proposition \ref{rgoesdown}, $p_{\dd*}q_{\dd}^*A\in r\overline{\mathbb{W}}$ and it touches the face $F_r(\lambda)^{\text{int}}$ only for the partial sum $\sigma=0$, and the conclusion for part (a) follows.
\end{proof}

\begin{proof}[Proof of Theorem \ref{thm5}]
The zero potential case follows from Propositions $5.2.$ and $5.3.$ so we have a semi-orthogonal decomposition: 
$$D^b(\X(d))=\langle p_{\dd*}q_{\dd}^*\left(\overline{\mathbb{N}}(d_1)_{w_1}\boxtimes\cdots\boxtimes \overline{\mathbb{N}}(d_k)_{w_k}\right)\rangle,$$ and the functors $p_{\dd*}q_{\dd}^*$ are fully faithful.
Use Proposition \ref{sod} to obtain a semi-orthogonal decomposition:
\begin{equation}
    D_{sg}(\X(d)_0)=\langle D_{sg}\left(p_{\dd*}q_{\dd}^*\left(\overline{\mathbb{N}}(d_1)_{w_1}\boxtimes\cdots\boxtimes \overline{\mathbb{N}}(d_k)_{w_k}\right)\right)\rangle
\end{equation} 
Next, observe that: $$D_{sg}\left(p_{\dd*}q_{\dd}^*\left(\overline{\mathbb{N}}(d_1)_{w_1}\boxtimes\cdots\boxtimes \overline{\mathbb{N}}(d_k)_{w_k}\right)\right)=p_{\dd*}q_{\dd}^*\,D_{sg}\left(\overline{\mathbb{N}}(d_1)_{w_1}\boxtimes\cdots\boxtimes \overline{\mathbb{N}}(d_k)_{w_k}\right).$$ Further, 
use the version of Thom-Sebastiani from Corollary \ref{corts} to write:
$$D_{sg}\left(\overline{\mathbb{N}}(d_1)_{w_1}\boxtimes\cdots\boxtimes \overline{\mathbb{N}}(d_k)_{w_k}\right)\cong D_{sg}(\overline{\mathbb{N}}(d_1)_{w_1})\boxtimes\cdots\boxtimes D_{sg}(\overline{\mathbb{N}}(d_k)_{w_k}),$$ so we indeed obtain the desired semi-orthogonal decomposition. 
After taking the idempotent completion of these categories, we obtain the corresponding semi-orthogonal decomposition for $D_{sg}^{id}$.

\end{proof}

\subsection{Relations in $\text{gr}^FKHA$}\label{relations} Consider the semi-orthogonal decompositions from Theorem \ref{sod5} restricted to weight $w\in\mathbb{Z}$:
$$D_{sg}(\X(d)_0)_w=\langle \cdots, \overline{\mathbb{M}}(d)_w\rangle.$$

\begin{defn}\label{filtration}
Let $\mathcal{T}_w$ be the tree with vertices
the categories that appear in the above semi-orthogonal decomposition, and draw an edge between two such subcategories $\mathcal{A}$ and $\mathcal{B}$ if there are objects $A\in\mathcal{A}$ and $B\in\mathcal{B}$ such that $\text{Hom}(A,B)\neq 0$. For any category $\mathcal{A}$ from the semi-orthogonal decomposition, define 
$l(\mathcal{A})$ as the number of vertices on the smallest cycle through the tree $T_w$ connecting $\mathcal{A}$ and $\overline{\mathbb{M}}(d)_w$. We consider the union of all the trees $\mathcal{T}_w$ for $w\in\mathbb{Z}$.
Define a filtration $\overline{F}^{\cdot}$ on $K_0(D_{sg}(\X(d)_0))$ by letting  $$\overline{F}^{\leq 1}=K_0(\overline{\mathbb{M}}(d))\subset K_0(D_{sg}(\X(d)_0)),$$ and for $i\geq 2$, $$\overline{F}^{\leq i}=\text{subspace generated by }K_0(\mathcal{A})\text{ with }l(\mathcal{A})\leq i\text{ in } K_0(D_{sg}(\X(d)_0)).$$ 
For a complex $E\in D_{sg}(\X(d)_0)$, let 
$$l(E)=\text{min }l\text{ such that }[E]\in \overline{F}^{\leq l}K_0(D_{sg}(\X(d)_0)).$$
\end{defn}

\begin{prop}\label{filtcomp}
Assume that $\{(d,w), (e,v)\}$ is an admissible set, let $\chi_d\in \oM(d)_w$ and $\chi_e\in\oM(e)_v$. Then
$\chi=\chi_d+\chi_e$ is a dominant weight with $r$-invariant $r(\chi)>\frac{1}{2}$. Let $\sigma$ be a non-trivial sum of weights in $N_{d,e}$ such that $\chi-\sigma$ is dominant. Then:
\begin{align*}
l(V(\chi)\otimes\OO_{\X(d+e)})=l(p_{d,e*}(V(\chi)\otimes\OO_{\X(d,e)}))>\\
l(V(\chi-\sigma)\otimes\OO_{\X(d+e)})=l(p_{d,e*}(V(\chi-\sigma)\otimes\OO_{\X(d,e)})).\end{align*}
In particular, we have that:
$$l(V(\chi)\otimes\OO_{\X(d+e)})=l(p_{d,e*}(V(\chi)\otimes\OO_{\X(d,e)}))>l(C),$$
where $C$ is the cone:
$$V(\chi)\otimes\OO_{\X(d+e)}\to p_{d,e*}(V(\chi)\otimes\OO_{\X(d,e)})\to C \xrightarrow{[1]}.$$
\end{prop}

\begin{proof}
It suffices to prove the inequality:
$$l(p_{d,e*}(V(\chi)\otimes\OO_{\X(d,e)}))>l(p_{d,e*}(V(\chi-\sigma)\otimes\OO_{\X(d,e)}))$$ for a non-trivial sum of weights $\sigma$ in $N_{d,e}$.
For this, it is enough to argue that for $w'$ and $v'$ such that $V(\chi-\sigma)\otimes\OO_{\X(d)\times\X(e)}$ in $\oM(d)_{w'}\boxtimes\oM(e)_{v'},$ we have that:
$$l(Rp_{d,e*}q_{d,e}^*(\oM(d)_w\boxtimes\oM(e)_v))>
l(Rp_{d,e*}q_{d,e}^*(\oM(d)_{w'}\boxtimes\oM(e)_{v'})).$$ 
Consider $A\in \oM(d)_w\boxtimes\oM(e)_v$, $B\in \oM(d)_{w'}\boxtimes\oM(e)_{v'}$, and use the notations $p=p_{d,e}$ and $q=q_{d,e}$. Let $A'=V(w\beta_d+v\beta_e)\otimes\OO_{\X(d)\times\X(e)}$, and $B'=V(w'\beta_d+v'\beta_e)\otimes\OO_{\X(d)\times\X(e)}$.
We have that:
$$\langle \lambda_{d,e}, A\rangle <\langle \lambda_{d,e}, B\rangle.$$ 
It is enough to show that:
$$\text{Hom}\,(p_*q^*B, p_*q^*A)=0\text{ and }\text{Hom}\,(p_*q^*A',p_*q^*B')\neq 0.$$
For the first claim, by adjunction we need to show that
$\text{Hom}(p^*p_*q^*B, A)=0,$ which follows as in the proof of Proposition \ref{orth}:
$$\langle \lambda_{d,e},p^*p_*q^*B\rangle \geq \langle \lambda_{d,e}, B\rangle >\langle \lambda_{d,e},A\rangle.$$
For the second claim, 
there are nonzero maps $A'\to B'$, so we obtain nonzero maps $p_*q^*A'\to p_*q^*B'$.
\end{proof}

For a vertex $v\in I$, and a natural number $d_v\in\mathbb{N}$, let $$K_0(BG(d_v))=\mathbb{Z}[q_1^{\pm 1},\cdots, q_{d_v}^{\pm 1}]^{\mathfrak{S}(d_v)},$$
and denote by $q_v:=q_1\cdots q_{d_v}$.
Let $d,e\in \mathbb{N}^I$ be two dimension vectors.  We want to write the class of the normal bundle $\text{det}\,\mathcal{N}_{d,e}\in K_0(BG(d)\times BG(e))$, where $\mathcal{N}_{d,e}$ is the normal bundle of the map $p_{d,e}:\X(d,e)\to\X(d+e)$. First, write: $$\det(\mathcal{N}_{d,e})=\det(N_{d,e})\det(2\rho_{d,e})^{-1},$$ where $N_{d,e}$ is the normal bundle of $R(d,e)$ in $R(d+e)$, and $2\rho_{d,e}$ is the sum of positive roots in $G(d+e)$ which pair positively with $\lambda_{d,e}$.
In the normal bundle $[N_{d,e}]=\sum \beta^v_{i}-\beta^w_{j},$ the weights $\beta^v_{i}$ with $1\leq i\leq d^v$ appear 
$\sum_{\text{edge }v\to w} e^w$ times, whereas the weights $\beta^w_{j}$ with $d^v<j\leq d^v+e^v$ appear 
appear $-\sum_{\text{edge }w\to v} d^w$ times. 
In the term $2\rho_{d,e}$, the weights $\beta^v_{i}$ with $1\leq i\leq d^v$ appear 
$e^v$ times, whereas the weights $\beta^w_{j}$ with $d^v<j\leq d^v+e^v$  
appear $-d^v$ times. 
We will use the shorthand notations $$q^{f(d,e)}=\prod_{v\in I} q_v^{-e^v+\sum_{v\to w}e^w},\, q^{-g(d,e)}=\prod_{v\in I} q_v^{d^v-\sum_{w\to v}d^w}.$$
Here, both $f(d,e)$ and $g(d,e)$ depends on the dimension vectors $d$ and $e$ only. We shift $\det(\mathcal{N}_{d,e})$ by $\text{codim}(\X(d,e)\text{ in }\X(d+e))$.
The class of the normal bundle is then: $$[\det \mathcal{N}_{d,e}]=(-1)^{\chi(d,e)}q^{f(d,e)}q^{-g(d,e)}\in K_0(BG(d)\times BG(e)).$$

\begin{thm}\label{mut}
Let $A\in\oM(d)_w$, $B\in\oM(e)_v$, and assume that $\{(d,w), (e,v)\}$ is admissible with $r>\frac{1}{2}$. Consider the diagram of attracting loci:

\begin{tikzcd}
\X(d)\times \X(e) & \arrow{l}{q_{d,e}} \X(d,e) \arrow{d}{p_{d,e}}\\
\X(e,d) \arrow{u}{q_{e,d}} \arrow{r}{p_{e,d}} & \X(d+e).
\end{tikzcd}
\\
There exists a natural map:
$$p_{e,d*}q_{e,d}^*(B\boxtimes A\otimes \det (\mathcal{N}_{e,d}))\to p_{d,e*}q_{d,e}^*(A\boxtimes B)$$
whose cone $C$ has: $$l(C)<l(p_{d,e*}q_{d,e}^*(A\boxtimes B))=l(p_{e,d*}q_{e,d}^*(B\boxtimes A\otimes \det (\mathcal{N}_{e,d}))).$$ 
\end{thm}

An immediate corollary of Theorem \ref{mut} is:
\begin{cor}\label{cormut}
Let $x_{d,w}\in K_0(\mathbb{M}(d)_w)$ and $x_{e,v}\in K_0(\mathbb{M}(e)_v)$ and assume that the set $\{(d,w), (e,v)\}$ is admissible with $r>\frac{1}{2}$.
Then
$$x_{d,w}x_{e,v}=(-1)^{\chi(d,e)}(x_{e,v}q^{f(e,d)})(x_{d,w}q^{-g(e,d)})$$ in the algebra $\text{gr}^F\,\text{KHA},$ where $-\chi(d,e)=\text{dim}\,\text{Ext}^1(d,e)-\text{dim}\,\text{Ext}^0(d,e)=\text{codim}\,p_{d,e}$. 
\end{cor}

We next state some computations that we will use in the proof of Theorem \ref{mut}. Observe that the weights which pair negatively with $\lambda_{d,e}$ are the weights in $N_{d,e}$.
Let $w$ be the Weyl group elements $w=\prod_{v\in I} w^v\in\mathfrak{S}=\prod_{v\in I}\mathfrak{S}_{d^v+e^v}$ which for a vertex $v$ is the permutation: $$(1,\cdots,e^v,e^v+1,\cdots,d^v+e^v)\to (e^v+1,\cdots,d^v+e^v,1,\cdots,e^v).$$

\begin{prop}\label{Wcomputation}
(a) We have that $w^{-1}\rho-\rho=-2\rho_{e,d}$.

(b) Multiplication by $-w$ gives a bijection of sets:
$$(-w)\{\beta\text{ wt such that }\langle \lambda_{e,d},\beta\rangle<0\}=
\{\beta\text{ wt such that }\langle \lambda_{d,e},\beta\rangle<0\}.$$
\end{prop}

\begin{proof}
The proof of $(a)$ is immediate. For $(b)$, it is the same as showing that $-w$ give a bijection:
$$(-w)\{\beta\text{ wt of }N_{e,d}\}=
\{\beta\text{ wt of }N_{d,e}\}.$$
The weights of $N_{e,d}$ are of the form $\beta^a_{i}-\beta^b_{j}$ where $a$ and $b$ be vertices with an edge between them, $1\leq i\leq e^a$ and $e^b+1\leq j\leq e^b+d^b$. Then:
$$(-w)(\beta^a_{i}-\beta^b_{j})=\beta^b_{j-e^b}-\beta^a_{i+d^a},$$ and $\beta^b_{j-e^b}-\beta^a_{i+d^a}$ is a weight of $N_{d,e}$ because there is an edge between $b$ and $a$ as $Q$ is symmetric, $1\leq j-e^b\leq d^b$, and $d^a+1\leq i+d^a\leq d^a+e^a$.
\end{proof}

\textbf{Example.}\label{Wexample} Before starting the proof of Theorem \ref{mut}, we discuss an example. Assume that the potential is zero; we explain the proof of Corollary \ref{cormut} for two generators $V(\chi_d)\otimes\OO_{\X(d)}$ and $V(\chi_e)\otimes\OO_{\X(e)}$ such that $\{(d,w), (e,v)\}$ is admissible with $r>\frac{1}{2}$.

By the Borel-Bott-Weyl theorem, see Subsection \ref{sp}, there is a spectral sequence with terms $ V((\chi_d+\chi_e-\sigma)^+)\otimes \OO_{\X(d+e)}$ converging to: $$p_{d,e*}(V(\chi_d)\boxtimes V(\chi_e)\otimes \OO_{\X(d,e)}),$$
where $\sigma$ is a sum of weights in $N_{d,e}$, and similarly there is a spectral sequence with terms $V((\chi_e+\chi_d+N_{e,d}-2\rho_{e,d}-\tau)^+)\otimes \OO_{\X(d+e)}$ converging to: $$p_{e,d*}(V(\chi_e)\boxtimes V(\chi_d)\otimes \text{det}(\mathcal{N}_{e,d})\otimes\OO_{\X(e,d)}),$$ where $\tau$ is a sum of weights of $N_{e,d}$. 
In particular, passing to $K$-theory, we have that:
\begin{equation}\label{first}
    [p_{d,e*}(V(\chi_d)\boxtimes V(\chi_e)\otimes \OO_{\X(d,e)})]=
    \sum (-1)^{|I|+l(I)}[V((\chi_d+\chi_e-\sigma_I)^+)\otimes \OO_{\X(d+e)}],\end{equation}
where the sum is after all partial sums $I$ of weights in $N_{d,e}$, and $l(I)$ is the length of the Weyl group element $v_I\in\mathfrak{S}$ such that $v_I*(\chi_d+\chi_e-\sigma_I)$ is dominant. Further, we also have that:
\begin{multline}\label{second}
[p_{e,d*}(V(\chi_e)\boxtimes V(\chi_d) \otimes \text{det}(\mathcal{N}_{e,d}))]= \\
\sum (-1)^{|J|+l(J)}[V((\chi_e+\chi_d-N_{d,e}-2\rho_{d,e}-\tau_J)^+)\otimes \OO_{\X(d+e)}],\end{multline}
where the sum is after all
partial sums $J$ of weights in $N_{e,d}$ and $l(J)$ is defined as above.

For every $\tau_J$, either $\chi_e+\chi_d-N_{d,e}-2\rho_{d,e}-\tau_J$ or $w*(\chi_e+\chi_d-N_{d,e}-2\rho_{d,e}-\tau_J)$ is dominant. 
For $\tau_J$, let $\sigma_J:=-w\tau_J$. Use Proposition \ref{Wcomputation} to compute:
\begin{align*}
    w*(\chi_e+\chi_d+N_{e,d}-2\rho_{e,d}-\tau_J)=\\
    w(\chi_e+\chi_d+N_{e,d}-\rho+w^{-1}\rho-\tau_J+\rho)-\rho=\\
    \chi_d+\chi_e-N_{d,e}+\sigma_J.
\end{align*}
This means that the terms of Equations (\ref{first}) and (\ref{second}) are the same up to $(-1)^{|N_{d,e}|-l(w)}=(-1)^{\chi(d,e)}$:
\begin{multline*}[p_{d,e*}(V(\chi_d)\boxtimes V(\chi_e)\otimes \OO_{\X(d,e)})]=
(-1)^{\chi(d,e)}
[p_{e,d*}(V(\chi_e)\boxtimes V(\chi_d) \otimes \text{det}(\mathcal{N}_{e,d}))].
\end{multline*}
In particular, when the potential is zero, the equality from Corollary \ref{cormut} holds in $\text{KHA}$. 
In the general potential case,
we need to prove the above equality up to terms in a smaller part of the filtration $F^{\cdot}$.
The strategy is to find a map between the two complexes which induces an isomorphism between the terms of maximal degree, which in the above case are $$V(\chi_d+\chi_e)\otimes\OO_{\X(d+e)}\text{ and }V((\chi_e+\chi_d-2\rho_{d,e})^+)\otimes\OO_{\X(d+e)}.$$ 
By Proposition \ref{filtcomp}, the terms $V(\chi_d+\chi_e-\sigma)\otimes\OO_{\X(d+e)}$ and the analogues ones for $(e,d)$ are in a smaller piece of the filtration, and this will allow us to prove Theorem \ref{mut}. 
\\

\begin{proof}[Proof of Proposition \ref{mut}]
Write $\det(\mathcal{N}_{e,d})=\det(N_{e,d})\det(2\rho_{e,d})^{-1},$ where the left hand side has a shift $\text{codim}(p_{e,d})$, $\det(N_{e,d})$ has a shift $\text{codim}(i_{e,d})$, and $\det(2\rho_{e,d})^{-1}$ has a shift $-\text{dim}(\pi_{e,d})$, where the maps $i_{e,d}$ and $\pi_{e,d}$ are defined subsequently. 
First, we have to construct a map: 
\begin{equation}
    p_{e,d*}q_{e,d}^*(B\boxtimes A\otimes \det (N_{e,d})\otimes\det(2\rho_{e,d})^{-1})\to p_{d,e*}q_{d,e}^*(A\boxtimes B).
\end{equation}
Using adjunction, it is enough to construct a map:
\begin{equation}
p_{d,e}^*p_{e,d*}q_{e,d}^*(B\boxtimes A\otimes \det (N_{e,d})\otimes\det(2\rho_{e,d})^{-1})\to q_{d,e}^*(A\boxtimes B).
\end{equation}
Consider the following natural maps: 
\begin{itemize}
    \item affine bundle $\widetilde{q}_{d,e}:R(d,e)/G(d)\times G(e)\to R(d)/G(d)\times R(e)/G(e)$,
    \item closed immersion $i_{d,e}:R(d,e)/G(d)\times G(e)\to R(d+e)/G(d)\times G(e)$,
    \item $\overline{q}_{d,e}:R(d+e)/G(d,e)\to R(d+e)/G(d)\times G(e)$,
    \item proper map $\pi_{d,e}: R(d+e)/G(d,e)\to R(d+e)/G(d+e)$,
    \item $\tau_{d,e}: R(d)/G(d)\times R(e)/G(e)\to R(d,e)/G(d)\times G(e)$ is section of the affine bundle $\widetilde{q}_{d,e}$,
\end{itemize}
and consider the similar maps for $(e,d)$. We denote by $r_{d,e}$ the restriction of representations from $G(d,e)$ to $G(d)\times G(e)$. 
Expanding the definitions of $p_*$, $p^*$ and $q^*$, it is enough to construct a natural map:
\begin{equation}\label{third}
    i_{d,e}^*\overline{q}_{d,e*}\pi_{d,e}^*\pi_{e,d*}\overline{q}_{e,d}^*i_{e,d*}\widetilde{q}_{e,d}^*(A\boxtimes B\otimes \det(N_{e,d})\otimes \det(2\rho_{e,d})^{-1})\to \widetilde{q}_{d,e}^*(A\boxtimes B)
\end{equation}
Let $u=\langle \lambda_{e,d}, v\beta_e+w\beta_d\rangle$ be an integer. Then there exists a functor:
$$\beta_{\geq u}: D_{sg}(R(d+e)_0/G(d)\times G(e))\to D_{sg}(R(d+e)_0/G(d)\times G(e))_{\geq u}$$
adjoint to the natural inclusion 
$$D_{sg}(R(d+e)_0/G(d)\times G(e))_{\geq u}\subset D_{sg}(R(d+e)_0/G(d)\times G(e)).$$
Such an adjoint can be constructed using the semi-orthogonal decomposition of $R(d+e)/G(d)\times G(e)$ for the Kempf-Ness stratum:
$$R(d)\times R(e)/G(d)\times G(e) \leftarrow R(d,e)/G(d)\times G(e)\rightarrow R(d+e)/G(d)\times G(e).$$
In particular, we get a map $F\to \beta_{\geq u}F$ for every $F\in D_{sg}(R(d+e)_0/G(d)\times G(e))$ which is an isomorphism if $F\in D_{sg}(R(d+e)_0/G(d)\times G(e))_{\geq u}$.
\\

We claim that for 
$F\in D_{sg}(R(d+e)_0/G(d)\times G(e))_{\geq u}$
there is a natural isomorphism:
\begin{equation}
    \pi_{d,e}^*\pi_{e,d*}\overline{q}_{e,d}^*(F\otimes \det(2\rho_{e,d})^{-1})\cong \overline{q}_{d,e}^*(F).
\end{equation}
We explain that the above complexes are isomorphic for a complex $V(\psi_d)\boxtimes V(\psi_e)\otimes\OO_{\mathcal{T}}$, where $\psi_d\in M_{\mathbb{R}}(d)$ and has weight $w'$, $\psi_e\in M_{\mathbb{R}}(e)$ and has weight $v'$, $T\subset R(d+e)_0$ a closed subscheme and $\mathcal{T}\subset \X(d+e)_0$ the corresponding quotient stack. We have that:
$$\langle \lambda_{e,d}, v'\beta_e+w'\beta_d\rangle\geq u.$$
This implies that the set $\{(d,w'), (e,v')\}$ is admissible and then $\psi_d+\psi_e$ is dominant.
The right hand side is thus $V(\psi_d+\psi_e)\otimes\OO_{\mathcal{T}}$, where $\psi_d+\psi_e$ is a dominant weight of $G(d+e)$.
Use Proposition \ref{Wcomputation} to write the left hand side as:
$$\pi_{d,e}^*\pi_{e,d*}\overline{q}_{e,d}^*\left(V(\psi_e+\psi_d-2\rho_{e,d})\otimes\OO_{\mathcal{T}}[l(w)]\right)=
V(\psi_d+\psi_e)\otimes\OO_{\mathcal{T}}.$$ 
Next, let $E\in D^b(\X(d)\times\X(e))$.
We construct a natural map:
\begin{equation}\label{eqeq}
    i_{d,e}^*i_{e,d*}\widetilde{q}_{e,d}^*(E\otimes\det(N_{e,d}))\to \widetilde{q}_{d,e}^*(E).
\end{equation}
It is enough to construct a map after applying $\tau_{d,e}^*$ to both sides. Using that $\tau_{e,d}^*i_{e,d}^*=\tau_{d,e}^*i_{d,e}^*$ it is enough to construct a map:
$$\tau^*_{e,d}i_{e,d}^*i_{e,d*}\widetilde{q}_{e,d}^*(E\otimes\det(N_{e,d}))\to \tau^*_{e,d}\widetilde{q}_{e,d}^*E=E.$$
This follows from the natural map:
\begin{equation}\label{fourth}
    i_{e,d}^*i_{e,d*}(E\otimes\det(N_{e,d}))\to E.\end{equation}
The map in Equation (\ref{eqeq}) induces the same natural map for $E\in D_{sg}\left((\X(d)\times\X(e))_0\right)$.
\\

We obtain the map from Equation (\ref{third}) as follows:
\begin{align*}
i_{d,e}^*r_{d,e}\pi_{d,e}^*\pi_{e,d*}\overline{q}_{e,d}^*i_{e,d*}\widetilde{q}_{e,d}^*(A\boxtimes B\otimes \det(N_{e,d})\otimes \det(2\rho_{d,e})^{-1})\to\\
i_{d,e}^*r_{d,e}\pi_{d,e}^*\pi_{e,d*}\overline{q}_{e,d}^*\beta_{\geq u}i_{e,d*}\widetilde{q}_{e,d}^*(A\boxtimes B\otimes \det(N_{e,d})\otimes \det(2\rho_{d,e})^{-1})\to\\
i_{d,e}^*r_{d,e}\overline{q}_{d,e}^*\beta_{\geq u}i_{e,d*}\widetilde{q}_{e,d}^*(A\boxtimes B\otimes \det(N_{e,d}))\cong \\
\beta_{\geq u}i_{d,e}^*i_{e,d*}\widetilde{q}_{e,d}^*(A\boxtimes B\otimes \det(N_{e,d}))\to
\beta_{\geq u}\widetilde{q}_{e,d}^*(A\boxtimes B)\cong
\widetilde{q}_{d,e}^*(A\boxtimes B).
\end{align*}


Next, we want to argue that the cone $C$ of the map 
$$p_{e,d*}q_{e,d}^*(B\boxtimes A\otimes \det (N_{e,d})\otimes\det(2\rho_{d,e})^{-1})\to p_{d,e*}q_{d,e}^*(A\boxtimes B)\to C\to $$ has a smaller length \begin{equation}\label{fifth}
    l(C)<l(p_{e,d*}q_{e,d}^*(B\boxtimes A\otimes \det (N_{e,d})\otimes\det(2\rho_{d,e})^{-1}))=l(p_{d,e*}q_{d,e}^*(A\boxtimes B)).\end{equation}
We first prove the statement in the potential zero case. 
The lengths $$l(p_{e,d*}q_{e,d}^*(B\boxtimes A\otimes \det (N_{e,d})\otimes\det(2\rho_{e,d})^{-1}))\text{ and }l(p_{d,e*}q_{d,e}^*(A\boxtimes B))$$
depend on the set $\{(d,w), (e,v)\}$ only, so it is enough to 
prove the statement for generators 
$A=V(\chi_d)\otimes \OO_\X$ and $B=V(\chi_e)\otimes\OO_\X$.
In this case, the map factors as:

\begin{tikzcd}
p_{e,d*}(V(\chi_e+\chi_d+N_{e,d}-2\rho_{e,d})\otimes\OO_{\X(e,d)}[l(w)]) \arrow{r} \arrow{dr}{s} & p_{d,e*}(V(\chi_d+\chi_e)\otimes\OO_{\X(d,e)})\\
 & V(\chi_d+\chi_e)\otimes\OO_{\X(d+e)} \arrow{u}{t}.
\end{tikzcd}
\\
The map $t$ is constructed from the natural transformation $\text{id}\to p_{d,e*}p_{d,e}^*$.
We need to explain the construction of the map $s$:
\begin{equation}
    p_{e,d*}(V(\chi_e+\chi_d+N_{e,d}-2\rho_{e,d})\otimes\OO_{\X(e,d)}[l(w)])\to V(\chi_d+\chi_e)\otimes\OO_{\X(d+e)}.
\end{equation}
We first use that $$V(\chi_d+\chi_e)\otimes\OO_{\X(d+e)}\cong R\pi_{e,d*}\left(V(\chi_e+\chi_d-2\rho_{e,d})\otimes\OO_{\X(d+e)}[l(w)]\right).$$
It is thus enough to construct a map:
$$p_{e,d*}(V(\chi_e+\chi_d+N_{e,d})\otimes\OO_{\X(e,d)})
\to V(\chi_e+\chi_d)\otimes\OO_{\X(d+e)},$$
which follows from the natural map in Equation \ref{fourth}.
Next, we need to argue that the cones of the maps $s$ and $t$ satisfy:
\begin{multline*}
    l(C_s), l(C_t)<l\left(p_{e,d*}q_{e,d}^*\left(V(\chi_e+\chi_d+N_{e,d}-2\rho_{e,d})\otimes\OO_{\X(e,d)}\right)\right)=\\
    l\left(p_{d,e*}q_{d,e}^*(V(\chi_d+\chi_e)\otimes\OO_{\X(d,e)})\right).\end{multline*}
Using Proposition \ref{filtcomp}, we have that indeed:
$$l(C_t)<l\left(p_{d,e*}q_{d,e}^*(V(\chi_d+\chi_e)\otimes\OO_{\X(d,e)})\right)=l\left(V(\chi_d+\chi_e)\otimes\OO_{\X(d+e)}\right).$$
By the Borel-Bott-Weyl theorem, the cone of $s$ is generated by $$V(\chi_e+\chi_d+N_{e,d}-2\rho_{e,d}-\tau_J)\otimes\OO_{\X(d)}$$ where $\tau_J$ is a partial sum of weights in $N_{e,d}$ which is not the full sum. We have explained in the Example \ref{Wexample} that these are the same as the vector bundles $$V(\chi_d+\chi_e-\sigma_J)\otimes\OO_{\X(d)}$$ where $\sigma_J$ is a nonzero partial sum of weights of $N_{d,e}$. All these sheaves have: $$l<l\left(p_{d,e*}(V(\chi_d+\chi_e)\otimes\OO_{\X(d,e)})\right)=l\left(p_{e,d*}(V(\chi_e+\chi_d-2\rho_{e,d})\otimes\OO_{\X(e,d)})\right),$$ and this ends the proof of claim from Equation (\ref{fifth}).
\\

For an arbitrary potential, in order to show that the cone $C$ of the map
$$p_{e,d*}q_{e,d}^*(A\boxtimes B \otimes \det \mathcal{N}_{e,d})\to p_{d,e*}q_{d,e}^*(A\boxtimes B)\to C\to $$
lies in a smaller piece of the filtration, consider the diagram:

\begin{tikzcd}
\X(d)\times\X(e) & \X(d,e) \arrow{r}{p_{d,e}} \arrow{l}{q_{d,e}} & \X(d+e) \\
(\X(d)\times\X(e))_0 \arrow{u}{i} & \X(d,e)_0 \arrow{u}{i} \arrow{r}{p_{d,e}} \arrow{l}{q_{d,e}} & \X(d+e)_0. \arrow{u}{i}
\end{tikzcd}
\\
We also use the notations $A$ and $B$ for representatives in $\X(d)_0$ and $\X(e)_0$ of the above objects such that $i_{d*}A\in\mathbb{N}(d)$ and $i_{e*}B\in\mathbb{N}(e)$. We see that:
$$i_{d+e*}p_{d,e*}q_{d,e}^*(A\boxtimes B)=p_{d,e*}q_{d,e}^*\left(i_{d*}A\boxtimes i_{e*}B\right).$$
Because
$i_{d*}A\in\mathbb{N}(d)$ and $i_{e*}B\in\mathbb{N}(e)$, the zero potential case implies that $i_{d+e*}C$ lies in a smaller piece of the filtration, and thus the same is true for $C$, which is what we wanted to prove.
\end{proof}


\subsection{Decomposition of $K_0(\oM(d))$}\label{step3}
Recall the definition of $\MM(d)$ from Subsection \ref{defMM}. 
Recall also Proposition \ref{nulemma}, and for $\sigma\in \mathfrak{S}_{k}\subset W_d$ write:
$$\frac{1}{2}N^{\lambda<0}-\frac{1}{2}N^{\sigma\lambda<0}=\sum_{U}N^{ab},$$ where the set $U$ contains pairs $(a,b)$ such that $1\leq a<b\leq k$ and $\sigma(a)>\sigma(b)$, and $N^{ab}$ denotes the sum of the weights corresponding to $d_a$ and $d_b$. 
Define $\text{codim}(\sigma):= \sum_{(a,b)\in U} \text{codim}(p_{a,b})$. 
\begin{prop}\label{even}
Let $\{(d_1,w_{1}),\cdots, (d_k, w_{k})\}$ be an admissible set with $r=\frac{1}{2}$ and assume that the category $\boxtimes_{i=1}^k\boxtimes_{j=1}^{m_i}\oM(d_i)_{w_{i}}$ is non-empty. 
All $\text{codim}(\sigma)$ as above are even.
\end{prop}

\begin{proof}
By Proposition \ref{subsetmic}, we can reduce the proof to the showing that $\text{codim}(p_{d,e})$ are even for an admissible set $\{(d,w), (e,v)\}$ with $r=\frac{1}{2}$ such that there exist integer weights $\chi_1\in \NN(d)_w$ and $\chi_2\in\NN(e)_v$. 
Let $\lambda_{d,e}=\prod_{i\in I}(1,\cdots, 1, z,\cdots,z):\C^*\to G(d+e)$ be a character where in the above formula $1$ appears $d^i$ times and $z$ appears $e^i$ times. 
The assumption that $r=\frac{1}{2}$ implies that: $$\langle \lambda_{d,e}, \chi_1+\chi_2+\rho_{d,e}\rangle=\frac{1}{2}\langle \lambda_{d,e}, N_{d,e}\rangle,$$ which further implies that:
$$2\langle \lambda_{d,e}, \chi_1+\chi_2\rangle=\langle \lambda_{d,e}, N_{d,e}-2\rho_{d,e}\rangle=\text{codim}(p_{d,e}),$$ so $\text{codim}\,(p_{d,e})$ is indeed even.
\end{proof}

The proof of Theorem \ref{mut} can be adapted to show the following:

\begin{prop}\label{switch}
Let $S$ be a set of cardinal $k$ with sum $d$. Consider $d_1,\cdots,d_k$ an ordering of $S$, and let $w_i$ be weights such that $\{d_1,w_1),\cdots, (d_k,w_k)\}$ is admissible with $r=\frac{1}{2}$. Consider $F\in \boxtimes_{i=1}^k \oM(d_i)_{w_{i}}$. 
We have a natural isomorphism:
$$\bigoplus_{\sigma\in\mathfrak{S}_{k}} \sigma F[\text{codim}(\sigma)]\xrightarrow{\text{iso}}\Phi_{S}(p_{\dd*}q_{\dd}^*F).$$ 
\end{prop}

\begin{proof}
One can define a natural map:
$$\bigoplus_{\sigma\in\mathfrak{S}_{\dd}} \sigma F[\text{codim}(\sigma)]\rightarrow \Phi_{\dd}(p_{\dd*}q_{\dd}^*F)$$ as in Theorem \ref{mut}; the map defined there works for the case $(d,e)$. Following the argument in Theorem \ref{mut}, it is enough to show it is an isomorphism in the potential zero case, and in that case we can assume that $F=\boxtimes_{i=1}^k V(\chi_{i})\otimes\OO$. The isomorphism follows from Proposition \ref{boundary}.
\end{proof}

The following result is an immediate corollary of Proposition \ref{switch}.

\begin{cor}\label{iso}
Let $S$ be a set with $k$ elements with sum $d$.
The composition
$$\bigoplus_{\sigma\in\mathfrak{S}_k}\boxtimes_{i=1}^k  K_0\left(\oM(d_i)_{w_{i}}\right)
\xrightarrow{\Phi_{\dd}p_{\dd*}q_{\dd}^*}
\bigoplus_{\sigma\in\mathfrak{S}_k}\boxtimes_{i=1}^k K_0\left(\oM(d_i)_{w_{i}}\right)$$ is the symmetrization map $T(x)=\sum_{\sigma\in\mathfrak{S}}\sigma(x)$.
\end{cor}

The next result is the analogue of Theorem \ref{mut} for $r=\frac{1}{2}$.
\begin{cor}\label{change}
Consider an admissible set $\{(d,w), (e,v)\}$ with $r=\frac{1}{2}$. Let $x_{d,w}\in K_0(\MM(d)_w)$ and $y_{e,v}\in K_0(\MM(d)_v)$. Then:
$$x_{d,w}y_{e,v}- \left(y_{e,v}q^{f(e,d)}\right)\left(x_{d,w}q^{-g(e,d)}\right)\in K_0(\MM(d+e)_{w+v}).$$ 
\end{cor}


\begin{proof}
Consider the triangle:
$$p_{e,d*}q_{e,d}^*\left(B\boxtimes A\otimes\text{det}(\mathcal{N}_{e,d})\right)\to p_{d,e*}q_{d,e}^*\left(A\boxtimes B\right)\to C\xrightarrow{[1]}$$
defined as in Theorem \ref{mut}. 
It is enough to show that the cone $C$ is in $\mathbb{M}(d+e)$, so we need to check that $\Phi_{S}(C)\in\mathcal{A}_{S}$ for all sets $S$ with sum $d+e$. It is enough to check this statement for sets $S$ with cardinal $2$. If $S$ is not $\{d,e\}$, then
the results follows from Proposition \ref{upper}, so we can assume that $S=\{d,e\}$. 
Apply $\Phi_{S}$ to the above exact triangle:
$$\Phi_{S}\left(p_{e,d*}q_{e,d}^*\left(B\boxtimes A\otimes\text{det}(\mathcal{N}_{e,d})\right)\right)\to \Phi_{S}\left(p_{d,e*}q_{d,e}^*\left(A\boxtimes B\right)\right)\to \Phi_{S}(C)\xrightarrow{[1]}.$$ Using Proposition \ref{switch} and \ref{even} we obtain that:
$$B\boxtimes A[\text{even}]\oplus A\boxtimes B[\text{even}]\to B\boxtimes A[\text{even}]\oplus A\boxtimes B \to \Phi_{S}(C)\xrightarrow{[1]},$$ which implies that $\Phi_{S}(C)\in\mathcal{A}_{S}$.
\end{proof}

The main result of this section is the following:
\begin{thm}\label{decomp}
 The following map induced by multiplication is an isomorphism of vector spaces:
$$m:\bigoplus_S\left(\bigoplus_\mathfrak{S} \boxtimes_{i=1}^k\, K_0\left(\MM(d_i)_{w_{i}}\right)\right)^{\mathfrak{S}_{k}} \to K_0(\oM(d)),$$ where the sum on the left hand side is taken after all sets $S$ of cardinal $k\geq 1$ with sum of elements $d$, orderings $d_1,\cdots,d_k$ of the set $S$, and weights $w_i$ for $1\leq i\leq k$ such that $\{(d_1,w_1),\cdots, (d_k,w_k)\}$ is admissible with $r=\frac{1}{2}$.
\end{thm}

Recall the functors $\Phi_S$ and the symmetrization map $T$ from Subsection \ref{defMM}.
Before we start the proof of Theorem \ref{decomp}, we collect some preliminary results.

\begin{prop}\label{surj}
The subspace $K_0(\MM(d))\hookrightarrow K_0(\oM(d))$ can be described as:
$$K_0(\MM(d))=\bigcap_{|S|\geq 2} \text{ker}\left(K_0(\oM(d))\xrightarrow{T\,\Phi_{S}} \left(\bigoplus_{\mathfrak{S}_{k}}\,\boxtimes_{i=1}^k K_0\left(\oM(d_i)_{w_{i}}\right)\right)^{\mathfrak{S}_{k}}\right).$$

\end{prop}

\begin{proof}
The category $\MM(d)$ is dense in $\oM(d)$ because for every $F\in \oM(d)$, the complex $F\oplus F[1]$ is in $\MM(d)$. By \cite[Corollary 2.3]{Th}, we then have that $$K_0(\MM(d))\hookrightarrow K_0(\oM(d)).$$

To show that $K_0(\MM(d))$ is included in the right hand side, we need to check that the composition:
$$K_0(\MM(d))\xrightarrow{\Phi_{S}} 
\bigoplus_{\mathfrak{S}_k}\boxtimes_{i=1}^k  K_0\left(\oM(d_i)_{w_{i}}\right)
\xrightarrow{T} 
\bigoplus_{\mathfrak{S}_k}\boxtimes_{i=1}^k\, K_0\left(\oM(d_i)_{w_{i}}\right)^{\mathfrak{S}_{k}}$$ is zero; this is the case because for $F\in \MM(d)$ we have that $\Phi_{S}(F)\in\mathcal{A}_{S}$. 

To show that the right hand side is included in $K_0(\MM(d))$, observe that the right hand side is the subspace of $K_0(\oM(d))$ such that $\Phi_{S}\in K_0(\mathcal{A}_{S})$.
Let $F$ be a complex of $\oM(d)$ such that for every set $S$ with $|S|\geq 2$ and with sum $d$ we have $[\Phi_{S}(F)]\in K_0(\mathcal{A}_{S})$. But then $\Phi_{S}(F)\in \mathcal{A}_{S}$ by Proposition \ref{propprop}, and so $F\in\MM(d)$.

\end{proof}



\begin{prop}\label{upper}
Let $S=\{(d_1,w_1),\cdots, (d_k,w_k)\}$ and $U=\{(e_1,v_1),\cdots, (e_l,v_l)\}$ be admissible sets with $r=\frac{1}{2}$, and assume that $k<l$ or $k=l$ and the sets $\{d_1,\cdots,d_k\}$ and $\{e_1,\cdots, e_l\}$ are different.
Then the following composition is zero:
$$\boxtimes_{i=1}^k K_0\left(\MM(d_i)_{w_{i}}\right)
\xrightarrow{p_{\dd*}q_{\dd}^*} K_0(\oM(d))\xrightarrow{T\Phi_{U}}
\left(\bigoplus_{\mathfrak{S}_l}
\boxtimes_{i=1}^l K_0\left(\oM(e_i)_{v_{i}}\right)\right)^{\mathfrak{S}_l},$$ where $T$ denotes the symmetrization map.
\end{prop}

\begin{proof}
We use localization with respect to the character $\lambda_{\dd}:\C^*\to G(d)$, see Subsection \ref{local}. The $\lambda_{\dd}$ fixed locus of $\X(e_1)\times\cdots\times\X(e_l)$ can be showed to be, as in the proof of Theorem \ref{bialgebra}:
$$\bigcup_P \prod_{i=1}^k \prod_{j=1}^l \X(f_{ij}),$$ where the set $P$ has elements matrices of partitions $(f_{ij})$ with $1\leq i\leq k$, $1\leq j\leq l$ such that:
$$\sum_{1\leq j\leq l} f_{ij}=d_i\text{ and }
\sum_{1\leq i\leq k} f_{ij}=e_j.$$
All the vector spaces we consider are localized at the ideal $I_{\dd}$, see Theorem \ref{loc} for its definition. By Theorem \ref{loc}, it is enough to show the result after localization with respect to $I_{\dd}$. Consider the diagram: 

\begin{tikzcd}
K_0(D_{sg}(\X(d)_0)\arrow{r}{\Phi_{\ee}}& \boxtimes_{j=1}^l K_0\left(\oM(e_j)_{v_{j}}\right)\\
\boxtimes_{i=1}^k  K_0(\MM(d_i)_{w_{i}}) \arrow{u}\arrow{r}{\prod \Phi_{\underline{f}}}& \boxtimes_{i=1}^{k}\boxtimes_{j=1}^{l} K_0\left(D_{sg}(\X(f_{ij})_0)_{w}\right).\arrow{u}{\prod m_{\underline{f}}}
\end{tikzcd}
\\
The square commutes up to an Euler factor $$\frac{\prod_{j=1}^k \text{eu}(f_{1j},\cdots, f_{kj})}{\text{eu}(d_1,\cdots,d_k)},$$ which is not a zero divisor by Proposition \ref{zerodiv}. It suffices to show that the composite:
\begin{multline}\label{mama}
\boxtimes_{i=1}^k  K_0(\MM(d_i)_{w_{i}}) \xrightarrow{\bigoplus \prod \Phi_{\underline{f}} } 
\bigoplus \boxtimes_{i=1}^{k}\boxtimes_{j=1}^{l} K_0\left(D_{sg}(\X(f_{ij})_0)_{w}\right)
\xrightarrow{\bigoplus m_{\underline{f}}}\\
\bigoplus_{\mathfrak{S}_l}\boxtimes_{i=1}^l K_0\left(\oM(e_i)_{v_{i}}\right)\xrightarrow{T} 
\left(\bigoplus_{\mathfrak{S}_l}
\boxtimes_{i=1}^l K_0\left(\oM(e_i)_{v_{i}}\right)\right)^{\mathfrak{S}_l}
\end{multline}
is zero. By the assumption on the sets $S$ and $U$, there exists an index $1\leq i\leq k$ such that the partition $(f_{i1},\cdots, f_{il})$ of $d_i$ has at least two nonzero terms. Fix such an index $i$. All permutations of the partition $(f_{i1},\cdots, f_{il})$ appear in the second summand of Equation (\ref{mama}). The symmetrization maps $T$ for $(f_{i1},\cdots, f_{il})$ and $(e_1,\cdots, e_l)$ commute with multiplication, so
it suffices to show that the map:
$$K_0(\MM(d_i)_{w_i})\to \left(\bigoplus_{\mathfrak{S}_l}
\boxtimes_{j=1}^{l} K_0\left(D_{sg}(\X(f_{ij})_0)_{w}\right)
\right)^{\mathfrak{S}_l}$$
is zero. 
This follows from Proposition \ref{surj}.

\end{proof}

\begin{proof}[Proof of Theorem \ref{decomp}]
Injectivity of the map $m$ follows from Proposition \ref{upper}. For sujectivity, Proposition \ref{surj} implies that the map:
$$K_0(\MM(d))\oplus \bigoplus_{|S|\geq 2}
\left(\bigoplus_{\mathfrak{S}_k}\boxtimes_{i=1}^k  K_0(\oM(d_i)_{w_{i}})\right)^{\mathfrak{S}_{k}}\to K_0(\oM(d))$$ is surjective. Using induction on $d$ we see that this implies that $m$ is also surjective.
\end{proof}

\begin{defn}\label{secondfiltrations}
We refine the filtration $\overline{F}$, see Definition \ref{filtration}, as follows.
Let $w\in\mathbb{Z}$.
Consider a graph $\mathcal{G}_w$ with vertices admissible sets $\{(d_1,w_1),\cdots,(d_k,w_k)\}$ of total weight $w$. 
Let $\mathcal{A}$ and $\mathcal{B}$ be two admissible sets. If $r(\mathcal{A})>\frac{1}{2}$, draw an edge between $\mathcal{A}$ and $\mathcal{B}$ if there exist objects $A\in\mathcal{A}$ and $B\in\mathcal{B}$ such that $\text{Hom}\,(A,B)\neq 0$.
If $r(\mathcal{A})=r(\mathcal{B})=\frac{1}{2}$,
draw an edge from $\mathcal{A}$ to $\mathcal{B}$ if $|\mathcal{A}|\geq |\mathcal{B}|$. 
\\

The only admissible set with out-degree zero is $\{(d,w)\}$.
For any set $\mathcal{A}$ as above, define 
$l(\mathcal{A})$ as the number of vertices of the smallest path connecting $\mathcal{A}$ and $\MM(d)_w$. 
Define a filtration $F^{\cdot}$ on $K_0(D_{sg}(\X(d)_0))$ by
$$F^{\leq i}=\text{subspace generated by }K_0(\mathcal{A})\text{ with }l(\mathcal{A})\leq i\text{ in } K_0(D_{sg}(\X(d)_0)).$$ 
Observe that $F^{\leq 1}=K_0(\MM(d))\subset K_0(D_{sg}(\X(d)_0))$.
For a complex $E\in D_{sg}(\X(d)_0)$, define the length of $E$ by: 
$$l(E)=\text{min }l\text{ such that }[E]\in F^{\leq l}K_0(D_{sg}(\X(d)_0)).$$
\end{defn}

\begin{cor}\label{basisimp}
Consider the associated graded $\text{gr}^F\text{KHA}$ with respect to the filtration $F^{\cdot}$ from Definition \ref{secondfiltrations}. The monomials 
$x_{d_1,w_1}\cdots x_{d_k,w_k}$ for an admissible set $\{(d_1,w_1),\cdots, (d_k,w_k)\}$ and $x_{d_i,w_i}\in K_0(\MM(d_i)_{w_i})$ generate $\text{gr}^F\text{KHA}$, and the only relations between these monomials are those generated by the above relations for admssible sets with $r=\frac{1}{2}$, see Corollary \ref{change}. 
\end{cor}

\subsection{The deformed symmetric algebra}\label{step4}

\begin{defn}\label{qsym}
Let $\text{dSym}\left(\bigoplus K_0(\mathbb{M}(d)_w)\right)$
be the free algebra generated by monomials
$x_{d_1,w_1}\cdots x_{d_k,w_k}$, where $x_{d,w}\in K_0(\mathbb{M}(d)_w)$, with relations: $$x_{d,w}x_{e,v}=(x_{e,v}q^{f(e,d)})(x_{d,w}q^{-g(e,d)}).$$ 
\end{defn}

In this section, we prove:
\begin{prop}\label{PBWdef}
The algebra $\text{dSym}\left(\bigoplus K_0(\mathbb{M}(d)_w)\right)$ is generated by the monomials $x_{d_1,w_1}\cdots x_{d_k,w_k}$ for an admissible set $\{(d_1,w_1),\cdots, (d_k,w_k)\}$ with $x_{d_i,w_i}\in K_0(\MM(d_i)_{w_i})$, and the only relations between these monomials are those generated by the above relations for admssible sets with $r=\frac{1}{2}$. 
\end{prop}

\begin{proof}[Proof of Theorem \ref{thm5}]
Consider the natural map $$\text{dSym}\left(\bigoplus K_0(\mathbb{M}(d))\right)\to \text{gr}^F\,KHA.$$
It is well defined by Theorem \ref{mut} and Corollary \ref{change}. 
Corollary \ref{basisimp} and Proposition \ref{PBWdef} imply that it is an isomorphism.
\end{proof}

Before we start the proof of Proposition \ref{PBWdef}, we define what it means for two sets of pairs to be conjugate. First, we define this notion for weights of the form $\chi_A$ for $A$ a set of pairs.
We call the weights $w\beta_d+v\beta_e\text{ and } v\beta_e+w\beta_d+N_{e,d}$ conjugate; two weights $\chi=\sum w_i\beta_{d_i}$ and $\psi=\sum v_i\beta_{e_i}$ are conjugate if $\psi$ can be obtained from $\chi$ by repeatedly applying the above transformation for pairs $(d_i,d_{i+1})$.
We say that two sets $A$ and $B$ are conjugate if their corresponding weights $\chi_A$ and $\chi_B$ are conjugate.
In particular, the above transformation is an involution on sets of cardinal $2$. Proposition \ref{PBWdef} follows from the following:

\begin{prop}
The set $A=\{(d_1, w_1),\cdots, (d_k, w_k)\}$ is conjugate to an admissible set. 
\end{prop}

\begin{proof}
For a permutation $\sigma\in \mathfrak{S}_k$, we denote by $\chi_{\sigma(A)}$ the weight obtained by applying the above conjugations for the transpositions in $\sigma$. 
Consider a permutation $\sigma$ of $\mathfrak{S}_k$ which minimizes the invariant $$r=\frac{\langle \lambda, \chi_{\sigma(A)}+\rho\rangle}{\langle \lambda, N^{\lambda>0}\rangle}$$ for $\lambda$ an anti-dominant character. We have that $r\leq -\frac{1}{2}$. 
Choose the maximal such character $\lambda$ which realizes the minimum of the above fraction. We can assume without loss of generality that $\sigma=\text{id},$ and we can write $$\chi+\rho=rN^{\lambda>0}+\chi'+\rho^L.$$ Here $\rho^L$ is half the sum of positive roots of the Levi group $L=G(e_1)\times\cdots\times G(e_s)$ associated to $\lambda$, the corresponding partition satisfies $\ee\leq\dd$, and we have that $$\chi'=\chi_1+\cdots+\chi_s\text{ where }\chi_i=\sum_{j\in P_i} v_{i,j}\beta_{d_j}$$ for integers $v_{i,j}$ and a partition $\bigcup_{1\leq i\leq s}P_i=\{1,\cdots, k\}$. 

The above conjugation for two consecutive indices $d_j, d_{j+1}$ both in $P_i$ for some index $1\leq i\leq s$ is the same if we apply it for $\chi$ or for $\chi_i$.
By induction on the dimension vector $d$, we can find permutations $\sigma_i\in\mathfrak{S}_{|P_i|}$ for $1\leq i\leq s$ such that after applying $\sigma_i$ to $\chi_i$ we obtain admissible weights. We can assume without loss of generality that all these permutations are the identity. In particular, all the weights $\chi_i$ are dominant and because $rN^{\lambda>0}-\rho^{\lambda>0}$ is dominant, the weight $\chi$ is dominant as well. 

It remains to show that $r<r_i=r(\chi_i)$, where $r_i$ is the minimum of the above fraction and the $r$-invariant of $\chi_i$ because $\chi_i$ is dominant. We have that:
$$\chi+\rho=rN^{\lambda>0}+r_iN^{\lambda_i>0}+\psi,$$ where $\psi$ is generated by weights $\beta$ such that $\langle \lambda,\beta\rangle=\langle \lambda_i,\beta\rangle=0$. Let $\mu$ be the anti-dominant character corresponding to the partition $(e_1, \cdots, e_{i-1}, P_i, e_{i+1}, \cdots, e_s)$.
The choice of $\lambda$ implies that:
$$\frac{\langle \lambda, \chi+\rho\rangle}{\langle \lambda, N^{\lambda>0}\rangle}<
\frac{\langle \mu, \chi+\rho\rangle}{\langle \mu, N^{\mu>0}\rangle}.$$ This implies that:
$$r<\frac{\langle \mu, r_1N^{\lambda>0}+r_1N^{\lambda_1>0}\rangle}{\langle \lambda, N^{\lambda>0}\rangle+\langle\lambda_1, N^{\lambda_1>0}\rangle}=\frac{r\langle \lambda, N^{\lambda>0}\rangle+r_i\langle \lambda_i, N^{\lambda_i>0}\rangle}{\langle \lambda, N^{\lambda>0}\rangle+\langle\lambda_i, N^{\lambda_i>0}\rangle},$$ so $r<r_i$ as desired.

\end{proof}

\subsection{Quiver with cuts.}\label{qcuts} Let $Q$ be a symmetric quiver with potential $W$ which admits a cut, that is, for which there exists a set of edges $E'\subset E$ such that each term in $W$ contains exactly one edge in $E'$. Let $$E(d)=\prod_{a\in E'} \text{Hom}\,(V^{s(a)}, V^{t(a)})\text{ and }F(d)=\prod_{a\in E-E'} \text{Hom}\,(V^{s(a)}, V^{t(a)}).$$ There exists a splitting $R(d)=F(d)\oplus E(d)$.
The dimension reduction Theorem \ref{dimred0}, \cite{i} implies that:
$$K_0(D_{sg}(\X(d)_0))\cong G_0(\mathcal{K}(d))\cong G_0(\mathcal{K}^{\text{cl}}(d)),$$ 
where $\mathcal{K}(d)$ is the Koszul stack $\mathcal{K}(d)=\OO_{\X(d)}[E(d)^{\vee}[1];d]$ and $\mathcal{K}^{\text{cl}}(d)$ is its underlying scheme.
One source of examples of $(Q,W)$ such that $W$ admits a cut are tripled quivers $(\widetilde{Q},\widetilde{W})$, see Subsection \ref{tripledef} for definition. For a tripled quiver $(\widetilde{Q},\widetilde{W})$ we have that $\mathcal{K}(d)=\mathcal{P}(d)$ is the stack of representations of the preprojective algebra of $Q$, so:
$$\text{KHA}(Q,W)\cong \bigoplus_{d\in\mathbb{N}^I} G_0(\mathcal{P}(d))\cong \bigoplus_{d\in\mathbb{N}^I} G_0(\mathcal{P}^{\text{cl}}(d)).$$
Theorem \ref{thm5} thus implies that
a K-theoretic Hall algebra defined in terms of the preprojective stack $\mathcal{P}(d)$
satisfies a PBW type theorem.

\subsection{Framed quivers}
The category $\MM(d)$ is a semi-orthogonal summand for all categories of singularities of framed representations. We expect this result to be useful in comparing the $\text{KHA}$ and $\text{CoHA}$.

\begin{prop}\label{framedsod}
Let $\X^{ss}(f,d)$ be the framed moduli space for a quiver $Q$, and consider the map $\pi:\X^{ss}(f,d)\to\X(d)$.
We have a semi-orthogonal decomposition:
$$D_{sg}(\X^{\text{ss}}(f,d)_0)=\langle\cdots, \pi^*\oM(d)\rangle.$$ In particular, we have an injection:
$$K_0(\MM(d))\hookrightarrow K_0(\oM(d))\hookrightarrow K_0(D_{sg}(\X^{\text{ss}}(f,d)_0)).$$
\end{prop}

\begin{proof}
We prove the zero potential case, the general case follows from Proposition \ref{sod}. We will use \cite[Theorem 2.10]{hl}, see also Subsection \ref{window}, to find a category $\mathbb{G}_w\subset D^b(\X(f,d))$ such that $\text{res}:\mathbb{G}_w\to D^b(\X^{ss}(f,d))$ is an equivalence. We now explain what are the Kempf-Ness strata and we choose weights $w$ to define $\mathbb{G}_w$. We use GIT for the equivariant line bundle $\mathcal{L}=\text{det}\,(\C^{fd})$. 

The Kempf-Ness strata appear for characters $\tau=\lambda\mu^{-a}$, where $\lambda:\C^*\to G(d)$, $\mu:\C^*\xrightarrow{(\text{Id},z\text{Id})} G(d)\times\C^*$, and $a>0$. 
The exponent $a$ is positive. Indeed, to find a Kempf-Ness stratum, we need to maximize the invariants 
$$\text{inv}(\tau, Z):=-\frac{\text{wt}_{\lambda}\mathcal{L}|_Z}{|\lambda|}$$
for pairs $(\tau, Z)$ where $\tau$ is a character of $G(d)\times\C^*$ $Z$ is fixed by $\tau$, and $\text{inv}(\lambda, Z)>0$.
The algorithm continues as long as there are pairs such that $\text{inv}(\lambda, Z)<0$,
so $a$ is indeed positive.
The normal bundle $N^{\tau>0}$ in $R(d)\times\C^{fd}$ at a fixed point $p$ fixed by $\tau$ has the decomposition: 
$$N^{\tau>0}=N^{\lambda>0}\oplus (\C^{fd})^{\tau>0}.$$ We choose $w$ for the Kempf-Ness stratum such that:
$$-\frac{1}{2}\langle \lambda, N^{\lambda>0}\rangle \leq \langle \tau, j_{\tau}^*F\rangle <
\frac{1}{2}\langle \lambda, N^{\lambda>0}\rangle +\langle \tau, (\C^{fd})^{\tau>0}\rangle.$$
If the character $\tau$ fixes $\C^{fd}$, the invariant $\text{inv}(\lambda,Z)$ is zero.
\\

We next show that $\pi^*\overline{\mathbb{N}(d)}\subset \mathbb{G}_w$. For this, we need to show that for $F\in\overline{\mathbb{N}(d)}$, the $\tau$-weights of $\pi^*F$ satisfy the above condition. This follows from $\langle \tau, (\C^{fd})^{\tau>0}\rangle>0$ and the definition of $\overline{\mathbb{N}(d)}$ which implies the following $\lambda$-conditions:
$$-\frac{1}{2}\langle \lambda, N^{\lambda>0}\rangle \leq \langle \lambda, F\rangle=\langle \tau, \pi^*F\rangle \leq 
\frac{1}{2}\langle \lambda, N^{\lambda>0}\rangle.$$

Next, the functor $\pi^*: \overline{\mathbb{N}(d)}\to D^b(\X^{ss}(f,d))$ is fully faithful because $\pi_*\pi^*\OO=\OO$ and thus:
$$\text{Hom}\,(\pi^*A,\pi^*B)=\text{Hom}\,(A,\pi_*\pi^*B)=\text{Hom}\,(A,B).$$

In order to obtain the desired semi-orthogonal decomposition, we need to show that the above fully faithful functor $\pi^*$ has an adjoint, see Subsection \ref{sodadjoint}. Consider the functor: $$\mathbb{G}_w\xrightarrow{\pi_*} D^b(\X(d))\xrightarrow{\Phi}\overline{\mathbb{N}(d)},$$ where $\Phi$ is the adjoint to the inclusion 
$\overline{\mathbb{N}(d)}\subset D^b(\X(d))$. This functor is indeed an adjoint to $\pi^*$:
$$\text{Hom}(\pi^*A, B)=\text{Hom}(A,\pi_*B)=\text{Hom}(A,\Phi\pi_*B).$$

\end{proof}

\subsection{Abelian decompositions}

In this section, we provide a replacement of $\MM(d)$ in terms of the abelian stacks $\Y(d)$. We first need to define the abelian version of partition.

\subsubsection{Abelian partitions}
Let $k\geq 1$ be an integer. 
We call an ordered $k$-tuple of sets $$P=(P_1,\cdots, P_k)$$ 
a \textit{abelian partition} of $d\in\mathbb{N}^{I}$ 
if every $P_i$ consists of sets $P_{i}^v$ for $v\in I$, and the sets $P_{i}^v$ for $1\leq 1\leq k$ form a partition of the set $$\{1,\cdots,d^v\}=P_{1}^v\cup\cdots\cup P_{k}^v.$$ 
Some of the sets $P_{i}^v$ might be empty. 
For a given abelian partition $P=(P_1,\cdots, P_k)$, we get a partition $d=d_1+\cdots+d_k$ of $d\in\mathbb{N}^{I}$ by letting $d_i=(\text{card}\,P_{i}^v)\in\mathbb{N}^{I}$. 

For a set $P^v\subset\{1,\cdots,d^v\}$,
let $\mathfrak{gl}(P^v)\subset \mathfrak{gl}(d^v)$ be the affine subspace generate by $\beta^v_{i}$ with $i\in P^v$. For a set $P=(P^v)_{v\in I}$, let $R(P) \subset R(d)$ be the subspace generated by the weights $\beta^v_{i}-\beta^w_{j}$ with $i\in P^v$, $j\in P^w$, and $v$ and $w$ vertices with an edge between them.
For an abelian partition $P$, let $R(P)=R(P_1)\times\cdots\times R(P_k)\subset R(d)$.

For any abelian partition $P$, we can choose a character $\lambda:\C^*\to T'(d)\subset T(d)$, where $T'(d)=\text{ker}\,(\text{det}:T(d)\to\mathbb{C}^*)$, such that the $\lambda$-fixed locus is $R(d)^{\lambda}=R(P)$. Consider the diagram of attracting loci for $\lambda$:

\begin{tikzcd}
  R(P)^{\lambda>0}/T(d) \arrow[r, "p"] \arrow[d, "q"]
    & R(d)/T(d)=\Y(d)  \\
  R(P)/T(d). 
 \end{tikzcd}
\linebreak

\subsubsection{Definition of $\MM^{ab}(d)$}\label{yetanother}
Let $\NN^{ab}(d)\subset D^b(\Y(d))$ be the subcategory generated by sheaves $V(\chi)\otimes\OO$, where $V(\chi)$ is the one-dimensional representation of $T(d)$ with weight $
\chi$ such that:
$$\chi+\rho\in\frac{1}{2}\WW.$$
We define $\oM^{ab}(d):=D_{sg}(\NN^{ab}(d))\subset D_{sg}(\Y(d)_0)$. It is immediate to see that $$(\Pi^*)^{-1}\pi^*\oM^{ab}(d)\subset \oM^{ab}(d)\text{ and }\pi_*\Pi^*\oM^{ab}(d)=\oM(d).$$

Let $P=(P_1,P_2)$ be an abelian partition with associated set of pairs $\{(d,w), (e,v)\}$ admissible with $r=\frac{1}{2}$. Denote by $\tau P$ the set partition $(P_2, P_1)$. Let $w'$ and $v'$ be weights such that $\{(e,v'), (d,w')\}$ is admissible with $r=\frac{1}{2}$.
Let $\lambda$ be a character associated to $P$ such that:
$$\langle \lambda, \chi_P+\rho\rangle=-\frac{1}{2}\langle \lambda, N^{\lambda>0}\rangle,$$ where $\chi_P$ is defined as in the non-abelian case.
Let $V(\chi_1)\otimes\OO_{\Y(d)}\in \NN^{\text{ab}}(d)_w$ and $V(\chi_2)\otimes\OO_{\Y(e)}\in \NN^{\text{ab}}(d)_v$, and consider the functor $$\tau:\NN^{\text{ab}}(P)\cong \NN^{\text{ab}}(d)_w\boxtimes\NN^{\text{ab}}(e)_v\to \NN^{\text{ab}}(\tau P)\cong \NN^{\text{ab}}(e)_{v'}\boxtimes\NN^{\text{ab}}(d)_{w'}$$ defined on generators by:
$$\tau: V(\chi_1+\chi_2)\otimes\OO_{\Y(d)\times\Y(e)}\to V(\chi_2+\chi_1+N^{\lambda<0})\otimes\OO_{\Y(e)\times\Y(d)}[l(\sigma)],$$ where $\sigma$ is the element of the Weyl group
$\sigma=\prod_{v\in I} \sigma^v\in\mathfrak{S}=\prod_{v\in I}\mathfrak{S}_{d^v+e^v}$ which for a vertex $v$ is the permutation: $$(1,\cdots,e^v,e^v+1,\cdots,d^v+e^v)\to (e^v+1,\cdots,d^v+e^v,1,\cdots,e^v).$$
Observe that applying $\pi_*\Pi^*$ to the functor $\tau$ induces the functor $\tau$ defined in Subsection \ref{defMM}:
$$\tau: V(\chi_1+\chi_2)\otimes\OO_{\X(d+e)}
\to V(\chi_2+\chi+1+N_{e,d})\otimes\OO_{\X(d+e)}.$$
The functor $\tau$ in the abelian case induces an analogous functor $\tau: \oM^{\text{ab}}(P)\to \oM^{\text{ab}}(\tau P)$.
Observe that $\tau^2=[\text{even}]$.
Consider a set $S$ of cardinal $k$ corresponding to a set partition of $d$. Let $P$ be an ordering of the elements of $S$. From the above construction, we obtain an action of $\mathfrak{S}_k$ on: $$\bigoplus_{\mathfrak{S}_k}\boxtimes_{i=1}^k
K_0(\oM^{\text{ab}}(P))\cong \bigoplus_{\mathfrak{S}_k}\boxtimes_{i=1}^k K_0(\oM^{\text{ab}}(d_i)_{w_i}).$$ 
We define $\mathcal{A}(S)\subset \bigoplus_{\mathfrak{S}_k}\boxtimes_{i=1}^k
\oM^{\text{ab}}(P)$
as the subcategory generated by extensions of $E$ and $\sigma(E)[\text{odd}]$, where $E$ is an object of $\bigoplus_{\mathfrak{S}_k}\boxtimes_{i=1}^k
\oM^{\text{ab}}(P)$ and $\sigma\in\mathfrak{S}_{k}$. 
We define $\mathbb{M}^{\text{ab}}(d)\subset \oM^{\text{ab}}(d)$ as the subcategory of objects $F$ such that $\Phi_{S}(F)\in \mathcal{A}_{S}$. 

Consider the shifted action $w*\chi=w(\chi+\rho)-\rho$ of the Weyl group $W$ on weights in $M_{\mathbb{R}}$.
Then one can check that for $S$ and $\lambda$ as above, $\Phi_{w\lambda}(w*F)$ is in $\mathcal{A}_{wS}$ is an only if $\Phi_{\lambda}(F)$ is in $\mathcal{A}_S$.
This implies that $$(\Pi^*)^{-1}\pi^*\MM(d)\subset \MM^{\text{ab}}(d)\text{ and }\pi_*\Pi^*\MM^{\text{ab}}(d)=\MM(d).$$

\begin{prop}\label{abelianmagic}
The map defined in Subsection \ref{abeliansubsection} induces an isomorphism:
$$(\Pi^*)^{-1}\pi^*: K_0(\MM(d))\cong K_0(\MM^{ab}(d))^{W}.$$
\end{prop}

\begin{proof}
Recall the natural injection $K_0(\MM(d))\hookrightarrow K_0(\oM(d))\hookrightarrow K_0(D_{sg}(\X(d)_0))$. 
There exists a semi-orthogonal decomposition $$D_{sg}(\Y(d)_0)=\langle \mathcal{B}, \oM^{ab}(d)\rangle,$$ where the complement $\mathcal{B}$ is supported on attracting loci $p_{\lambda*}q_{\lambda}^*$; the proof of the semi-orthogonal decomposition is the same as in the non-abelian case, and it also follows from the theorem of \cite{sp}. We thus have an injection: $$K_0(\MM^{ab}(d))\hookrightarrow K_0(\oM^{ab}(d))\hookrightarrow K_0(D_{sg}(\Y(d)_0)).$$
Consider the diagram:

\begin{tikzcd}
K_0(\MM(d))\arrow{r}{(\Pi^*)^{-1}\pi^*} \arrow[d, hookrightarrow] & 
K_0(\MM^{ab}(d)) \arrow[d, hookrightarrow] \arrow{r}{\pi_*\Pi^*} &
K_0(\MM(d)) \arrow[d, hookrightarrow]\\
K_0(D_{sg}(\X(d)_0))\arrow[r, hookrightarrow, "(\Pi^*)^{-1}\pi^*"]& K_0(D_{sg}(\Y(d)_0)) \arrow{r}{\pi_*\Pi^*}& K_0(D_{sg}(\X(d)_0)).
\end{tikzcd}
\\
The image of $(\Pi^*)^{-1}\pi^*: K_0(D_{sg}(\X(d)_0)\to K_0(D_{sg}(\Y(d)_0))$ is $K_0(D_{sg}(\Y(d)_0))^W$ by Proposition \ref{abel2}, so the image of the top left arrow lies in $K_0(\MM^{ab}(d))^W$. We can thus rewrite the above diagram as:

\begin{tikzcd}
K_0(\MM(d))\arrow{r}{(\Pi^*)^{-1}\pi^*} \arrow[d, hookrightarrow] & 
K_0(\MM^{ab}(d))^W \arrow[d, hookrightarrow] \arrow{r}{\pi_*\Pi^*} &
K_0(\MM(d)) \arrow[d, hookrightarrow]\\
K_0(D_{sg}(\X(d)_0))\arrow{r}{\text{iso.}}& K_0(D_{sg}(\Y(d)_0)) \arrow{r}{\text{iso.}}& K_0(D_{sg}(\X(d)_0)).
\end{tikzcd}
\\
This implies that the top maps are injections. The composition of the top arrow maps is $\pi_*\Pi^*(\Pi^*)^{-1}\pi^*=\text{id},$ so the top maps are isomorphisms.
\end{proof}

\section{GrKHA and the KBPS Lie algebra}\label{8}

\subsection{An alternative definition of the coproduct on grKHA}\label{sodcop}

Let $d,e\in\mathbb{N}^I$, 
let $\delta\in M_{\mathbb{R}}$ be a weight, and let $\lambda_{d,e}$ be a character as in Subsection \ref{attractingloci}.
Consider the Kempf-Ness stratum: \[\mathcal{Y}(d)\times\mathcal{Y}(e)\xleftarrow{q_{d,e}} \mathcal{Y}(d,e)\xrightarrow{p_{d,e}} \mathcal{Y}(d+e).\] 
We obtain a semi-orthogonal decomposition from Theorem \ref{hl}:
\begin{equation}
    D^b(\mathcal{Y}(d+e))=\langle p_*q^*D^b\left(\mathcal{Y}(d)\times \mathcal{Y}(e)\right)_{<w}, \mathbb{G}_{\delta}, p_*q^*D^b\left(\mathcal{Y}(d)\times \mathcal{Y}(e)\right)_{\geq w}\rangle,\end{equation} where $w=\langle \lambda_{d,e}, \delta\rangle$.
Consider the adjoints to the inclusion of the first and third summands in the above semi-orthogonal decomposition:
\begin{multline*}
    \Psi=(\Psi_{<w},\Psi_{\geq w}): D^b(\mathcal{Y}(d+e))\to 
\langle p_*q^*\mathcal{Y}(d)\times \mathcal{Y}(e)_{<w}, p_*q^*\mathcal{Y}(d)\times \mathcal{Y}(e)_{\geq w}\rangle\cong \\
D^b(\mathcal{Y}(d)\times\mathcal{Y}(e)).\end{multline*}
For the definition of adjoint, see Subsection \ref{window} or \cite[Section 3.6]{hl}.
Using Proposition \ref{uns}, we have a semi-orthogonal decomposition:
\begin{equation}\label{sod8}
    D_{sg}(\mathcal{Y}(d+e)_0)=\langle p_*q^*D_{sg}\left(\left(\mathcal{Y}(d)\times \mathcal{Y}(e)\right)_0\right)_{<w}, \mathbb{D}_w, p_*q^*D_{sg}\left(\left(\mathcal{Y}(d)\times \mathcal{Y}(e)\right)_0\right)_{\geq w}\rangle,\end{equation} where $\mathbb{D}_w:=D_{sg}(\mathbb{G}_w)$.
The functors above thus induce:
$$TS^{-1}\Psi: D_{sg}(\mathcal{Y}(d+e)_0)\to D_{sg}\left(\left(\Y(d)\times\Y(e)\right)_0\right)\to D_{sg}\left(\mathcal{Y}(d)_0\right)\times D_{sg}\left(\mathcal{Y}(e)_0\right).$$
We define the abelian coproduct by:
$$\overline{\Delta}^{ab}_{e,d,\delta}=\text{sw}_{d,e}\,TS^{-1}\Psi: K_0\left(D_{sg}(\mathcal{Y}(d+e)_0)\right)\to K_0\left(D_{sg}\left(\mathcal{Y}(e)_0\right)\right)\boxtimes K_0\left(D_{sg}\left(\mathcal{Y}(d)_0\right)\right).$$
The coproduct on Weyl invariants is defined by multiplying the above with an Euler class $\text{eu}_2(e,d)$:
\begin{multline*}
\overline{\Delta}_{e,d,\delta}=\text{sw}_{d,e}\,TS^{-1}\Psi(\text{eu}_2(-)): K_0\left(D_{sg}(\mathcal{Y}(d+e)_0)\right)^{W_{d+e}}\to\\ K_0\left(D_{sg}\left(\mathcal{Y}(e)_0\right)\right)^{W_d}\boxtimes K_0\left(D_{sg}\left(\mathcal{Y}(d)_0\right)\right)^{W_d}.
\end{multline*}

The theorem we prove in this section is:
\begin{thm}\label{coincide}
The abelian coproduct $\overline{\Delta}^{ab}_{d,e,\delta}$ defined above using semi-orthogonal decompositions induces a coproduct for $\text{gr}\,\text{KHA}$ which coincides with the coproduct $\Delta_{d,e}^{ab}$ from Definition \ref{coproduct}. Thus $\overline{\Delta}_{d,e,\delta}=\widetilde{\Delta}_{d,e}.$ In particular, the coproduct map $\overline{\Delta}_{d,e,\delta}$ defined above is independent of the choice of the weight $\delta\in M_{\mathbb{R}}$.
\end{thm}

The main ingredient in the proof of the above theorem is the following:
\begin{prop}\label{gkernel}
Recall the setting of Equation (\ref{sod8}).
The subspace $\text{gr}\,K_0\left(\mathbb{D}_{\delta}\right)$ from Equation (\ref{sod8}) is in the kernel of $\Delta^{ab}_{d,e}$.
\end{prop}

Before we start the proof of Proposition \ref{gkernel}, we collect some preliminary results. 

\begin{prop}\label{independence}
The image of $$\text{gr}\,K_0\left(\mathbb{D}_{\delta}(d)\right)\subset \text{gr}\,K_0\left(D_{sg}(\Y(d)_0)\right)$$ is independent of the choice of $\delta\in M_{\mathbb{R}}$. 
\end{prop}

\begin{proof}
The hyperplanes $\langle \lambda_{d,e}, \beta\rangle=w$ divide $M_{\mathbb{R}}$ in regions, and the definition of $\mathbb{D}_{\delta}(d)$ depends only on the region in which $\delta$ lies. Let $\beta$ be a simple root such that $\langle \lambda_{d,e}, \beta\rangle=1$. It is enough to compare the definitions for $\delta$ and $\delta+\beta$. Consider a nonzero map $\OO\to\OO(\beta)$, and let $Z$ be the support of its cokernel. We have a short exact sequence:
$$0\to\OO\to\OO(\beta)\to\OO_Z\to 0,$$
which means that for every $F\in \mathbb{D}_{\delta}(d)$ we have that:
$$F\to F\otimes\OO(\beta)\to i_*i^*\left(F\otimes\OO(\beta)\right)\xrightarrow{[1]}.$$
If $F\in \text{gr}^i K_0(D_{sg}(\Y(d)_0))$, then $i_*i^*\left(F\otimes\OO(\beta)\right)\in \text{gr}^{i+2} K_0(D_{sg}(\Y(d)_0))$, which means that $F\otimes\OO(\beta)\in \text{gr}^iK_0(D_{sg}(\Y(d)_0))$ and $[F]=[F\otimes\OO(\beta)]\in \text{gr}^iK_0(D_{sg}(\Y(d)_0))$. Thus the images of $\text{gr} K_0(D_{sg}(\Y(d)_0))$ and $\text{gr} K_0(\mathbb{D}_{\delta+\beta}(d))$ inside $\text{gr} K_0(D_{sg}(\mathcal{Y}(d)_0))$ are the same. 
\end{proof}

\begin{prop}\label{projec}
Let $T$ be a torus, $\lambda:\C^*\to T$ a character, and let $T'=T/\C^*$.
Consider a morphism $\pi:S/T\to Z/T'$, where $S\to Z$ is an affine bundle over a variety, both $S$ and $Z$ have an action of $T$, $\lambda$ fixes $Z$ and acts with weight $1$ on the normal bundle $N_{Z/S}$.
Let $\langle \lambda, S\rangle=n+1$. Consider $F$ a sheaf on $S/T$ such that 
$-n\leq \langle \lambda, F\rangle \leq -1.$
Then $R\pi_*F=0.$
\end{prop}

\begin{proof}
Let $E$ be a sheaf on $Z$. We will show that:
$$R\pi_*R\mathcal{H}om(\pi^*E,F)=R\mathcal{H}om(E,R\pi_*\,F)=0.$$ 
It is enough to prove the statement for the $\C^*$-equivariant $\text{Hom}$.
By restricting to a point in $Z$ with fiber $U$, we need to show that:
$$R\text{Hom}^{\C^*}(\OO_U,F|_U)=0.$$
The fiber $U$ is an affine space of dimension $n+1$ where $\C^*$ acts with weight $1$, 
The sheaf $F|_U$ has $\C^*$ weights between $-n$ and $-1$. By \cite[Theorem 2.10]{hl}, we have that the restriction to $\mathbb{P}(U)$ gives an isomorphism:
$$R\text{Hom}^{\C^*}(\OO_U,F|_U)\cong R\text{Hom}(\OO_{\mathbb{P}(U)}, F|_{\mathbb{P}(U)}).$$
The sheaf $F|_{\mathbb{P}(U)}$ is generated by $\OO_{\mathbb{P}(U)}(w)$, where $-n\leq w\leq -1$, so 
$$R\text{Hom}(\OO_{\mathbb{P}(U)}, F|_{\mathbb{P}(U)})=0,$$ and this implies the desired conclusion for the first part.
\end{proof}

\begin{proof}[Proof of Proposition \ref{gkernel}]
Consider the diagram:

\begin{tikzcd}
\mathbb{D}_{\delta}(d+e) \arrow{r}{\tau^*p_{d,e}^*} &
D_{sg}\left(R(d+e)_0\oplus\C/T(d+e)\times\C^*\right) \arrow{d}{\widetilde{q}_{d,e*}} \\
& D_{sg}\left(\left(R(d)\times R(e)\right)_0/T(d)\times T(e)\right).
\end{tikzcd}
\\
The image of $\text{gr}\,K_0\left(\mathbb{D}_{\delta}\right)\subset \text{gr}\,K_0(D_{sg}(\Y(d+e)_0))$ is independent of $\delta$ by Proposition \ref{independence}, so we can choose $\delta$ such that the $\lambda_{d,e}$-weights for $F\in \mathbb{G}_{\delta}(d+e)$ satisfy:
\[-n\leq \langle \lambda_{d,e}, F\rangle \leq -1,\]
where $n=\langle \lambda_{d,e}, \text{det}\,N^{\lambda>0}\rangle$. For $F\in \mathbb{D}_{\delta}(d+e)$ choose a lift also denoted by $F\in D^b(\Y(d+e)_0)$ such that $i_*F\in \mathbb{G}_{\delta}(d+e)$.
We have that $-n\leq \langle \lambda_{d,e}, \tau^*p^*_{d,e} F\rangle \leq -1$. The conclusion follows from Proposition \ref{projec}.
\end{proof}

\begin{prop}\label{abelinverse}
For any dimension vectors $d,e\in\mathbb{N}^I$, we have that $$\Delta^{ab}_{e,d}m^{ab}_{d,e}=\text{sw}_{d,e},$$ in both cohomology and graded K-theory.
\end{prop}

\begin{proof}
We prove the statement in cohomology, the statement in graded K-theory follows from it.
Let $x\in H^{\cdot}(\Y^{ss}(d),\varphi\,\mathbb{Q})\boxtimes H^{\cdot}(\Y^{ss}(e),\varphi\,\mathbb{Q})$.
The composition $\Delta^{ab}_{d,e}m^{ab}_{d,e}$ has the form: $$\Delta^{ab}_{e,d}m^{ab}_{d,e}(x)=\text{sw}_{d,e}\,\widetilde{q}_*(\tau^*p^*p_*q^*(x)).$$ We have that $p^*p_*(y)=y\,\text{eu}_1(d,e)$ for $y\in H^{\cdot}(\Y^{ss}(d,e),\varphi\,\mathbb{Q})$.
Using the projection formula, 
we have that
$$\widetilde{q}_*(\widetilde{q}^*(x)\tau^*\text{eu}_1(a,b))=x\widetilde{q}_*(\tau^*\text{eu}_1(d,e)).$$ Let $\mathcal{Z}=\Y(d)\times\Y(e)$ and $N=R(d,e)^{\lambda_{d,e}>0}$ the normal bundle of $R(d)\times R(e)$ in $R(d,e)$. Recall that $\lambda_{d,e}$ acts with weight $1$ on $N$.
It suffices to show that for
$\widetilde{q}: \mathbb{P}_{\mathcal{Z}}(N\oplus 1)\to\mathcal{Z}$ we have that:
$$\widetilde{q}_*(\tau^*\text{eu}_1(d,e))=1.$$
The vector space $H^{\cdot}(\mathbb{P}_{\mathcal{Z}}(N\oplus 1),\varphi\,\mathbb{Q})$ has a basis $\widetilde{q}^*\left(H^{\cdot}(\mathcal{Z},\varphi\,\mathbb{Q})\right)h^i$ for $0\leq i\leq n=\text{dim}\,N$, where $h=c_1(\OO(1))\in H^2(\mathbb{P}_{\mathcal{Z}}(N\oplus 1),\varphi\,\mathbb{Q})$ and $\OO(1)$ is the canonical line bundle on 
$\mathbb{P}_{\mathcal{Z}}(N\oplus 1)$. We have that $\widetilde{q}_*(h^i)=0$ for $0\leq i\leq n-1$ and $\widetilde{q}_*(h^n)=1$.

The Euler class $\tau^*(\text{eu}_1(d,e))$ is a monic polynomial of degree $n$ in $h$ because $\text{dim}\,R(d,e)=\text{dim}\,R(e,d)$. This implies the desired result:
$$\widetilde{q}_*(\tau^*\text{eu}_1(d,e))=1.$$
\end{proof}

\begin{proof}[Proof of Theorem \ref{coincide}]
Using the semi-orthogonal description, it is immediate that:
\[\overline{\Delta}^{ab}_{e,d, \delta}\left(m^{ab}_{d,e}\right)=\text{sw}_{d,e}\text{ and }\overline{\Delta}^{ab}_{e,d, \delta}(\mathbb{D}_{\delta})=0.\]
Further, the semi-orthogonal decomposition also implies the decomposition in K-theory:
\[K_0(D_{sg}(\Y(d+e)_0))=m^{ab}_{d,e}\left(K_0(D_{sg}(\Y(d)_0))\boxtimes K_0(D_{sg}(\Y(e)_0))\right)\oplus K_0(\mathbb{D}_{\delta}).\]
By the above formulas, we see that $\overline{\Delta}^{ab}$ descends to $\text{gr}\,\text{KHA}$. Further, 
using Propositions \ref{gkernel} and \ref{abelinverse}, we have that:
\[\overline{\Delta}_{e,d,\delta}^{ab}(m^{ab}_{d,e})=\Delta^{ab}_{e,d}(m^{ab}_{d,e})=\text{sw}_{d,e}\text{ and }\overline{\Delta}_{e,d,\delta}^{ab}(\mathbb{D}_{\delta})=\Delta_{e,d}^{ab}(\mathbb{D}_{\delta})=0,\]
which implies the desired equality.
\end{proof}



\subsection{The filtration E on grKHA}

Consider the inclusion $\C^*\xrightarrow{z\text{Id}}G(d)$, and let $q_d\in K_0(BG(d))$ be the generator of $K_0(BG(d))\twoheadrightarrow K_0(B\C^*)=\C[q_d^{\pm 1}]$. Define $h_d:=q_d-1\in \text{gr}^2 K_0(B\C^*)$. Observe that the map $\det: G(d)\to\C^*$ induces an injection $\text{gr}\,K_0(B\C^*)\to \text{gr}\,K_0(BG(d))$ with image $\mathbb{C}[\![ h_d]\!]$. 

\begin{prop}\label{lemmaimp}
There is a natural isomorphism: $$\text{gr}\,K_0(\mathbb{M}(d))=\text{gr}\,K_0(\mathbb{M}(d)_0)[\![ h_d]\!].$$
\end{prop}
 
\begin{proof}
Let $\X^p(d)=R(d)/PG(d)$ and $\X^s(d)=R(d)/S(d)$,
where: 
$$PG(d)=\text{coker}\,\left(\C^*\xrightarrow{z\text{Id}}G(d)\right)\text{ and }
S(d)=\text{ker}\left(\text{det}:G(d)\to\C^*\right).$$
Consider the commutative diagram induced by the map $G(d)\to PG(d)\times\C^*$, where the first map is the natural quotient and the second map is $\text{det}:G(d)\to\C^*$:

\begin{tikzcd}
\text{gr}K_0(D_{sg}(\X^p(d)_0))[\![ h_d]\!]\arrow{r} \arrow[d, hookrightarrow] & \text{gr}K_0(D_{sg}(\X(d)_0))\arrow[d, hookrightarrow]\\
H^{\cdot}(\X^p(d)_0,\varphi\,\mathbb{Q})^{\wedge}[\![ h_d]\!]\arrow{r}{\text{iso.}}& H^{\cdot}(\X(d)_0, \varphi\,\mathbb{Q})^{\wedge},
\end{tikzcd}
\\
which implies that the top map is an inclusion as well. Next, consider the commutative diagram induced by the map $S(d)\times\C^*\to G(d)$ via the natural inclusion $S(d)\subset G(d)$ and $\C^*\xrightarrow{z\text{Id}} G(d)$:

\begin{tikzcd}
\text{gr}\,K_0(D_{sg}(\X(d)_0))\arrow[d, hookrightarrow]\arrow{r} & \text{gr}\,K_0(D_{sg}(\X^s(d)_0))[\![ h_d]\!]\arrow[d, hookrightarrow]\\
H^{\cdot}(\X(d)_0, \varphi\,\mathbb{Q})^{\wedge}\arrow{r} & H^{\cdot}(\X^s(d)_0,\varphi\,\mathbb{Q})^{\wedge}[\![ h_d]\!],
\end{tikzcd}
\\
which implies that the top map is an inclusion. Next, we explain that the map $\X^s(d)\to \X^p(d)$ coming from $S(d)\hookrightarrow G(d)\twoheadrightarrow PG(d)$ induces a surjection:
\[ \text{gr}\,K_0(D_{sg}(\X^p(d)_0))\twoheadrightarrow \text{gr}\,K_0(D_{sg}(\X^s(d)_0)).\]
Consider the diagram:

\begin{tikzcd}
\text{gr}\,G_0(\X^p(d)_0)\arrow{d} \arrow[r, twoheadrightarrow]&
\text{gr}\,K_0(D_{sg}(\X^p(d)_0)) \arrow{d}\\
\text{gr}\,G_0(\X^s(d)_0) \arrow[r, twoheadrightarrow]& \text{gr}\,K_0(D_{sg}(\X^s(d)_0)).
\end{tikzcd}
\\
It is enough to show that the map:
\[ \text{gr}\,G_0(\X^p(d)_0)\to \text{gr}\,G_0(\X^s(d)_0)\]
is surjective. Stratify $R(d)_0=\bigcup W_L$ by locally closed subset $W\subset R(d)_0$ fixed by Levi subgroups $L\subset G(d)$. Denote by:
$$PL=\text{coker}\,\left(\C^*\xrightarrow{z\text{Id}}L\right)\text{ and }
SL=\text{ker}\left(\text{det}:L\to\C^*\right).$$
Using the excision sequence, it is enough to show that for $W$ a scheme and $L$ a Levi group of $G(d)$ acting trivially on it, 
the map 
\[\text{gr}\,G_0(W/PL)=\text{gr}\,G_0(W)\otimes \text{gr}\,K_0(BPL)\to \text{gr}\,G_0(W/SL)=\text{gr}\,G_0(X)\otimes \text{gr}\,K_0(BSL)\] is a surjection, which is true because in this case \[\text{gr}\,K_0(BPL)\cong H^{\cdot}(BPL)\cong H^{\cdot}(BSL)=\text{gr}\,K_0(BSL).\]
Putting together these results, we get that:
$$\text{gr}\,K_0(D_{sg}(\X(d)_0))\cong \text{gr}\,K_0(D_{sg}(\X^p(d)_0)).$$
Next, consider the diagram:

\begin{tikzcd}
\text{gr}\,K_0(\oM(d)) \arrow[r, hookrightarrow]& \text{gr}\,K_0(D_{sg}(\X(d)_0)) \arrow[r, twoheadrightarrow]& 
\text{gr}\,K_0(\oM(d)) \\
\text{gr}\,K_0(\oM(d)_0)[\![ h_d]\!] \arrow{u} \arrow[r, hookrightarrow] & \text{gr}\,K_0(D_{sg}(\X^p(d)_0))[\![ h_d]\!] \arrow[r, twoheadrightarrow] \arrow{u}{\text{iso.}} & 
\text{gr}\,K_0(\oM(d)_0)[\![ h_d]\!], \arrow{u}
\end{tikzcd}
\\
where the left horizontal maps are induced by the inclusions of $\oM(d)$ and $\oM(d)_0$ and the right horizontal maps are induced by the adjoint of these inclusions coming from the semi-orthogonal decomposition from Theorem \ref{thm5}. The left and right horizontal maps are induced from the inclusion $\oM(d)_0\subset \oM(d)$ and tensoring with $\OO(q_d)$ given by the character $\C^*\xrightarrow{z\text{Id}} G(d)$. It follows that:
$$\text{gr}\,K_0(\oM(d)_0)[\![ h_d]\!]\cong \text{gr}\,K_0(\oM(d)).$$
The analogous statement holds if we replace $\MM(d)_0$ with $\MM(d)_w$. Consider the diagram:

\begin{tikzcd}
K_0(\MM(d))  & \boxtimes_{i=1}^k K_0(\MM(d_i)_{w_i}) \arrow{l}{m_{\dd}}\\
K_0(\MM(d)_0)[q^{\pm 1}]\arrow[u, hookrightarrow]  & \boxtimes_{i=1}^k K_0(\MM(d_i)_{w_i})[q^{\pm 1}]. \arrow[u, hookrightarrow] \arrow{l}{m_{\dd}}
\end{tikzcd}
\\
After passing to the $\text{gr}$, both vertical maps become isomorphisms. Using Theorem \ref{decomp} we obtain that indeed:
$$\text{gr}\,K_0(\MM(d)_0)[\![ h_d]\!]\cong \text{gr}\,K_0(\MM(d)).$$
\end{proof}

We are now ready to define the filtration $E^{\cdot}$ on $\text{grKHA}$: 
 
\begin{defn}\label{E}

For a generator $x_dh_d^i\in \text{gr}\,K_0(\mathbb{M}(d)_0)[h_d]$, let $\text{deg}\,(x_dh_d^i)=1+2i$. For a monomial $\left(x_{d_1}h_{d_1}^{i_1}\right)\cdots \left(x_{d_k}h_{d_k}^{i_k}\right)$, define its degree by:
\[\text{deg}\left(x_{d_1}h_{d_1}^{i_1}\cdots x_{d_k}h_{d_k}^{i_k}\right)=\left(1+2i_1\right)+\cdots+\left(1+2i_k\right).\]
Define $E^{\leq i}\subset \text{gr}\,\text{KHA}$ as the subspace generated by monomials of degree $\text{deg}\leq i$.
\end{defn}

\subsection{The PBW theorem for grKHA}
Recall the modification of the multiplication of the $\text{CoHA}$ given in Definition \ref{supercom}. The bilinear maps $\psi$ can be used to modify the multiplication of $\text{KHA}$ as well. We will use the notation $\text{KHA}^{\psi}$ for the algebra with the twisted multiplication by $\psi$.
In this section, we prove:
\begin{thm}\label{PBWgr}
There is a natural isomorphism of algebras: $$\text{gr}^E\,\text{gr}\,KHA^{\psi}\cong \text{Sym}\,\left(\bigoplus \text{gr}\,K_0(\mathbb{M}(d)_0)[\![ h_d]\!]\right).$$
\end{thm}

We first show that the algebra $\text{gr}^E\text{gr}\,\text{KHA}$ is commutative:

\begin{prop}\label{commutative}
Let $d,e$ be dimension vectors and $a,b\geq 0$ be two integers. Consider $x_d\in K_0(\MM(d)_0)$ and $x_e\in K_0(\MM(e)_0)$. We have the following equality in $\text{gr}^{2a+2b+2}\,K_0(D_{sg}(\X(d+e)_0))$:
$$(x_dh_d^a)(x_eh_e^b)=(-1)^{\chi(d,e)}(x_eh_e^b)(x_dh_d^a).$$
\end{prop}

Before we start the proof of Proposition \ref{commutative}, we collect some results which compare the filtrations $E^{\cdot}$ and $F^{\cdot}$. 
First, define a category $\oM^s(d)\subset D_{sg}(\X^s(d)_0)$ as in Subsection \ref{defnmagic}; by the argument in Theorem \ref{sod5} or the \v{S}penko-van den Bergh \cite{sp}, we have a semi-orthogonal decomposition:
$$D_{sg}(\X^s(d)_0)=\langle \cdots, \oM^s(d)\rangle.$$
By the argument in Proposition \ref{lemmaimp}, the inclusion of categories $\oM(d)_0\subset \oM^s(d)$ induces an isomorphism:
$$\text{gr}\,K_0(\oM(d)_0)\cong \text{gr}\,K_0(\oM^s(d)).$$
Let $\widetilde{q}_d\in K_0(BS(d))$ be the class of $\beta_d$ defined in Subsection \ref{notations}. 

For an admissible set $\{(d_1,w_1),\cdots, (d_k,w_k)\}$ of $r=\frac{1}{2}$ and total weight $w$, 
we can define functors $$\Phi_{\dd}:\oM^{s}(d)_w\to \boxtimes_{i=1}^k \oM^{s}(d_i)_{w_i}$$ as in Subsection \ref{defMM}. The functors $\Phi_{\dd}$ send $\oM(d)$ to $\boxtimes_{i=1}^k \oM(d_i)_{w_i}$ because they preserve integer weights, and we thus obtain a commutative diagram:

\begin{tikzcd}
\text{gr}\,K_0(\oM(d)_0)\arrow{r}{\Phi_{\dd}} \arrow{d}{\text{iso.}}& \boxtimes_{i=1}^k \text{gr}\,K_0\left(\oM(d_i)_{w_{i}}\right)\ \arrow{d}{\text{iso.}}\\
\text{gr}\,K_0(\oM^{s}(d))\arrow{r}{\Phi_{\dd}}& \boxtimes_{i=1}^k \text{gr}\,K_0\left(\oM^s(d_i)_{w_{i,j}}\right).
\end{tikzcd}

Multiplication by $\widetilde{q}_d^w$ induces an isomorphism: $$\widetilde{q}_d^w: K_0(\oM^s(d)_0)\cong K_0(\oM^s(d)_w).$$

For a set $S$ of dimension vectors with sum $d$, we can define functors $\Phi_S:\oM^s(d)\to \boxtimes_{i=1}^k \oM^s(d_i)_{w_i}$ and
categories $\mathcal{A}_S^s\subset \bigoplus_{\mathfrak{S}_k}\boxtimes_{i=1}^k \oM^s(d_i)_{w_i}$.
Define $\oM^s(d)\subset \MM^s(d)$ as the full subcategory with sheaves $F$ such that $\Phi_S(F)\in\mathcal{A}_S$ for every $S$ as above with $|S|\geq 2$. 
By the analogue of Proposition \ref{surj}, we have that:
$$K_0(\MM^s(d))=\bigcap_{|S|\geq 2} \text{ker}\left(K_0(\oM^s(d))\xrightarrow{\Phi_{S}} \left(\bigoplus_{\mathfrak{S}_{k}}\,\boxtimes_{i=1}^k K_0\left(\oM^s(d_i)_{w_{i}}\right)\right)^{\mathfrak{S}_{k}}\right),$$ where $k=|S|$.
Using Proposition \ref{surj}, this implies that:
\begin{prop}\label{commsp}
There is an isomorphism $\text{gr}\,K_0(\MM(d)_0)\cong \text{gr}\,K_0(\MM(d)_w)$ making the following diagram commute:

\begin{tikzcd}
\text{gr}\,K_0(\MM(d)_0)\arrow{r}\arrow{d}& \text{gr}\,K_0(\MM(d)_w)\arrow{d}\\
\text{gr}\,K_0(\MM^s(d)_0)\arrow{r}{\widetilde{q}_d^w}& \text{gr}\,K_0(\MM^s(d)_w).
\end{tikzcd}
\end{prop}

Let $A\subset \text{KHA}$ be the subspace generated by the pieces in $F^{\leq b+2}$ corresponding to the partitions $(d,e)$ and $d+e$ inside $\bigoplus_{w=0}^a K_0(D_{sg}(\X(d+e)_0)_w)$. Consider the subspace $B\subset A$ generated by $F^{\leq b+2}$ inside $\bigoplus_{w=0}^{a-1} K_0(D_{sg}(\X(d+e)_0)_w)$ and by $F^{\leq b+1}$ inside $K_0(D_{sg}(\X(d+e)_0)_a)$.

\begin{prop}\label{anotherprop}
(a) The product $(x_dh_d^a)(x_eh_e^b)$ is in $\text{gr}\,A\subset \text{gr}\,\text{KHA}$.

(b) We have that $\text{gr}\,B\subset E^{\leq 2a+2b+1}$.
\end{prop}

\begin{proof}
By Theorem \ref{sod5} and Theorem \ref{mut}, the pieces $$F^{\leq b+2}\subset K_0(D_{sg}(\X(d+e)_0)_w)$$ are generated by $K_0(\MM(d+e)_w)$ and by:
$$p_{d,e*}\left(K_0(\MM(d)_{u-i})\boxtimes K_0(\MM(e)_{v+i})\right)$$ for $0\leq i\leq b$.
Use Proposition \ref{commsp} to write:
$$\text{gr}\,A=\bigoplus_{0\leq i\leq a, 0\leq j\leq b}
\text{gr}\,K_0(\MM(d)_0)\,h_d^i\boxtimes \text{gr}\,K_0(\MM(e)_0)\,h_e^j.$$ 

The subspace $\text{gr}\,B\subset \text{gr}\,A$ corresponds to the above sum with the summand $\text{gr}\,K_0(\MM(d)_0)\,h_d^a\boxtimes \text{gr}\,K_0(\MM(e)_0)\,h_e^b$ removed. Both part $(a)$ and $(b)$ follow immediately from these decompositions.
\end{proof}


\begin{proof}[Proof of Proposition \ref{commutative}]
Consider lifts $x_dh_d^a\in A$ and $x_eh_e^b\in B$. We have that $(x_dh_d^a)(x_eh_e^b)\in F^{\leq b+2}$. Theorem \ref{mut} implies that:
$$(x_dh_d^a)(x_eh_e^b)-(-1)^{\chi(d,e)}(x_eh_e^bq^{g(e,d)})(x_dh_d^bq^{-f(e,d)})\in F^{\leq b+1}.$$ Further, the construction in Proposition \ref{anotherprop} implies that it actually lies in:
$$F^{\leq b+1}\cap \bigoplus_{w=0}^{a}K_0(D_{sg}(\X(d+e)_0)_w)\subset B.$$  
Passing to $\text{gr}\,\text{KHA}$, we get that:
$$(x_dh_d^a)(x_eh_e^b)-(-1)^{\chi(d,e)}(x_eh_e^b)(x_dh_d^b)\in \text{gr}\,B\subset E^{\leq 2a+2b+1},$$ and the conclusion follows.
\end{proof}

Let $d_1,\cdots,d_k$ be different dimension vectors and $m_1,\cdots,m_k$ positive integers such that $m_1d_1+\cdots+m_kd_k=d$. By Proposition \ref{commutative}, the multiplication map $\boxtimes_{i=1}^k \text{gr}\,K_0(\MM(d_i))^{\boxtimes m_i}\to \text{gr}\,K_0(D_{sg}(\X(d)_0))$ factors through:
$$\boxtimes_{i=1}^k \text{gr}\,K_0(\MM(d_i))^{\boxtimes m_i}\xrightarrow{T_{m_1}\times\cdots\times T_{m_k}}
\boxtimes_{i=1}^k\,S^{m_i}\text{gr}\,K_0(\MM(d_i))\to
\text{gr}\,K_0(D_{sg}(\X(d)_0)),$$ where $T_{m_i}$ for $1\leq i\leq k$ denote symmetrization maps. The similar result holds in cohomology by \cite[Theorem C]{dm}.

\begin{prop}\label{surjec}
The natural map:
$$\bigoplus_{\dd\text{ descending}} \boxtimes_{i=1}^k S^{m_i}\text{gr}\,K_0(\MM(d_i))\to \text{gr}^E\text{gr}\,K_0(D_{sg}(\X(d)_0))$$ is a surjection, where the sum on the left hand side is taken after all descending partitions $m_1d_1+\cdots+m_kd_k=d$ with $d_1\succ\cdots\succ d_k$.
\end{prop}

\begin{proof}
Theorem \ref{thm5} implies that the map:
$$\bigoplus \boxtimes_{i=1}^k \text{gr}\,K_0(\MM(d_i))^{\boxtimes m_i}\to \text{gr}^E\text{gr}\,K_0(D_{sg}(\X(d)_0))$$ is a surjection, where the sum of the left hand side is taken all partitions $\dd$. The map factors through $\boxtimes_{i=1}^k S^{m_i}\text{gr}\,K_0(\MM(d_i))$ as explained above. Further, using Proposition \ref{commutative}, the image of 
$\boxtimes_{i=1}^k S^{m_i}\text{gr}\,K_0(\MM(d_i))\to \text{gr}^E\text{gr}\,K_0(D_{sg}(\X(d)_0))$ depends only on the set $\{d_1,\cdots, d_k\}$ of terms of the partitions $\dd$ and not on their order, so we obtain the desired conclusion.
\end{proof}

\begin{prop}\label{filtchern}
Consider dimension vectors $d_1\succ\cdots\succ d_k$ and integers $m_1,\cdots, m_k$ such that $m_1d_1+\cdots+m_kd_k=d$.
The following diagram commutes:

\begin{tikzcd}
\boxtimes_{i=1}^k S^{m_i}\text{gr}\,K_0(\mathbb{M}(d_i)) \arrow[d, hookrightarrow] \arrow{r} & \text{gr}\,K_0(D_{sg}(\X(d)_0)) \arrow[d, hookrightarrow]\\
\boxtimes_{i=1}^k S^{m_i}H^{\cdot}(\X(d_i),\varphi\,\mathbb{Q})\arrow[r, hookrightarrow] & H^{\cdot}(\X(d),\varphi\,\mathbb{Q}).
\end{tikzcd}
\\
where the bottom map is an inclusion.
In particular, the cohomological filtrations on $\boxtimes_{i=1}^k S^{m_i}\text{gr}\,K_0(\mathbb{M}(d_i))$ induced by:
\[\boxtimes_{i=1}^k S^{m_i}\text{gr}\,K_0(\mathbb{M}(d_i))\to \boxtimes_{i=1}^k S^{m_i}H^{\cdot}(\X(d_i),\varphi\,\mathbb{Q})\] and by 
$\boxtimes_{i=1}^k S^{m_i}\text{gr}\,K_0(\mathbb{M}(d_i))\subset K_0(D_{sg}(\X(d)_0))\to H^{\cdot}(\X(d),\varphi\,\mathbb{Q})$ coincide. 

\end{prop}

\begin{proof}
The diagram commutes from Proposition \ref{grr}, so
it is enough to show that the map \[\boxtimes_{i=1}^kS^{m_i}H^{\cdot}(\X(d_i),\varphi\,\mathbb{Q})\to H^{\cdot}(\X(d),\varphi\,\mathbb{Q})\] is an inclusion. We show the statement for $\mathcal{Y}$:
\[p_*:H^{\cdot}(\mathcal{Y}(d_1,\cdots,d_k),\varphi\,\mathbb{Q})\to H^{\cdot}(\mathcal{Y}(d),\varphi\,\mathbb{Q})
,\]
and then the statement above follows after taking Weyl invariants. Using \cite[Section 7]{dp}, we have natural isomorphism:

\begin{tikzcd}
H^{\cdot}(\mathcal{Y}(d_1,\cdots,d_k),\varphi\,\mathbb{Q})\arrow{r}{p_*} & H^{\cdot}(\mathcal{Y}(d),\varphi\,\mathbb{Q})\\
H^{\cdot}(\mathcal{Y}(\dd),\mathcal{Y}(\dd)_1)\arrow{r}{p_*} \arrow{u}{\text{iso.}} & H(\mathcal{Y}(d),\mathcal{Y}(d)_1). \arrow{u}{\text{iso.}}
\end{tikzcd}
\\
Both maps $p_*:H^{\cdot}(\mathcal{Y}(\dd))\to H^{\cdot}(\mathcal{Y}(d))$ and $p_*:H^{\cdot}(\mathcal{Y}(\dd)_1)\to H^{\cdot}(\mathcal{Y}(d)_1)$ are injective because $\mathcal{Y}(\dd)\subset \Y(d)$ and $\mathcal{Y}(\dd)_1\subset \Y(d)_1$ are Kempf-Ness strata. Using the long exact sequence for relative cohomology, we get that the map:
$$p_*:H^{\cdot}(\mathcal{Y}(\dd),\Y(\dd)_1)\to H^{\cdot}(\mathcal{Y}(d),\Y(d)_1)$$ is injective.
\end{proof}

\begin{prop}\label{injec}
The natural map:
$$\Phi: \bigoplus_{\dd\text{ descending}} \boxtimes_{i=1}^k S^{m_i}\text{gr}\,K_0(\MM(d_i))\to \text{gr}^E\text{gr}\,K_0(D_{sg}(\X(d)_0))$$ is an injection, where the sum on the left hand side is taken after all descending partitions $m_1d_1+\cdots+m_kd_k=d$ with $d_1\succ\cdots\succ d_k$. 
\end{prop}

\begin{proof}
Theorem \ref{thm5} says that multiplication induces an isomorphism of vector spaces:
$$\Psi: \bigoplus_{\text{admissible}}\boxtimes_{i=1}^k\boxtimes_{j=1}K_0(\MM(d_i)_{w_{i}})\rightarrow K_0(D_{sg}(\X(d)_0)),$$ where the terms corresponding to admissible sets with $r=\frac{1}{2}$ also have an action of the corresponding $\mathfrak{S}$. Using Proposition \ref{filtchern}, the following map is injective:
$$\Psi: \bigoplus_{\text{admissible}}\boxtimes_{i=1}^k\boxtimes_{j=1}\text{gr}\,K_0(\MM(d_i)_{w_{i}})\rightarrow \text{gr}\,K_0(D_{sg}(\X(d)_0)).$$
Let $x$ be in the domain of $\Phi$, and assume that $x\in E^{\leq i}$. By Proposition \ref{commutative}, there exists $y$ in the domain of $\Psi$ such that $x-y\in E^{\leq i-1}$, so $\Phi(x)=\Psi(y)=y$ in $\text{gr}^E\text{gr}\,K_0(D_{sg}(\X(d)_0))$, which implies that $\Phi$ is indeed injective. 
\end{proof}

\begin{proof}[Proof of Theorem \ref{PBWgr}]
Propositions \ref{surjec} and \ref{injec} imply that the natural map:
$$\Phi: \bigoplus_{\dd\text{ descending}} \boxtimes_{i=1}^k S^{m_i}\text{gr}\,K_0(\MM(d_i))\to \text{gr}^E\text{gr}\,K_0(D_{sg}(\X(d)_0))$$ is an isomorphism of vector spaces, where the sum on the left hand side is taken after all descending partitions $m_1d_1+\cdots+m_kd_k=d$ with $d_1\succ\cdots\succ d_k$. 
To show the isomorphism for algebras,
observe that the left hand side is supercommutative, which follows from Proposition \ref{commutative} and the definition of the twisted multiplication \ref{supercom}.
\end{proof}

As a corollary of Theorem \ref{PBWgr}, we define the KBPS Lie algebra on $E^{\leq 1}$.
\begin{cor}\label{KBPS}
The subspace $E^{\leq 1}\subset \text{gr}\,\text{KHA}$ is preserved by the Lie bracket $[x,y]=xy-yx$, so $E^{\leq 1}$ is a Lie algebra under this bracket. 
\end{cor}

\begin{proof}
By Theorem \ref{PBWgr}, the algebra $\text{gr}^E\text{gr}\, KHA$ is supercommutative, so for $x,y\in E^{\leq 1}$ we have that $xy-yx=0$ in $E^{\leq 2}/E^{\leq 1}$, which means that indeed $[x,y]\in E^{\leq 1}$.
\end{proof}

\subsection{BPS Lie algebras}\label{bpslie}
In this section, we show that the filtrations $E^{\leq i}$ and $P^{\leq i}$ are compatible via $\text{ch}$:
\begin{thm}\label{comp}
Let $(Q,W)$ be a symmetric quiver with potential, and consider the Chern character $\text{ch}:\text{gr}\,\text{KHA}\to\text{CoHA}$. Then for any $i\geq 1$ we have that:
$$\text{ch}\,\left(E^{\leq i}\right)\subset P^{\leq i}.$$
\end{thm}

We will prove this by showing that $E^{\leq 1}$ and $P^{\leq 1}$ are the spaces of primitive elements in $\text{gr}\,\text{KHA}$ and $\text{CoHA}$, respectively, and using that $\text{ch}$ is a morphism of bialgebras. We begin with preliminary results. 

\begin{prop}\label{Kprimitive}
The elements in $E^{\leq 1}$ are primitive for the coproduct $\Delta$ of $\text{gr}\,\text{KHA}$.
\end{prop}

\begin{proof}
Let $d,e\in\mathbb{N}^I$ be non-zero dimension vectors. We need to show that: $$\Delta_{d,e}(\mathbb{M}(d+e))=0.$$
Consider the category $\overline{\mathbb{G}}\subset D^b(\Y(d+e))$ defined by the conditions:
$$-\frac{1}{2}\langle \lambda_{d,e},N^{\lambda>0}\rangle+\langle \lambda_{d,e}, \rho\rangle
\leq \langle \lambda_{d,e},F\rangle\leq  
\frac{1}{2}\langle \lambda_{d,e}, N^{\lambda>0}\rangle-\langle \lambda_{d,e}, \rho\rangle.$$ Let $\overline{\mathbb{D}}:=D_{sg}(\overline{\mathbb{G}})$.
If the left hand side and the right hand side are not integers, then $\overline{\mathbb{G}}=\mathbb{G}_{\delta}$ for a weight $\delta\in M_{\mathbb{R}}$, for example for $\delta=u\beta$ where $\beta$ is a simple root with $\langle \lambda_{d,e},\beta\rangle=1$ and $u=\lceil -\frac{1}{2}\langle \lambda_{\dd},N^{\lambda>0}\rangle\rceil $. Fix such a weight $\beta$.
Further, $\MM^{ab}(d+e)\subset \mathbb{D}$ and Theorem \ref{coincide} implies that: $$\Delta^{ab}_{d,e}(\MM^{ab}(d+e))=0.$$ Proposition \ref{abelianmagic} implies then the desired conclusion. 
\\


Assume that the left hand side and right hand side are integers. Let their values be $u$ and $u'$, respectively, and consider an admissible set $\{(d,w), (e,v)\}$ of $r=\frac{1}{2}$. Let $w'$ and $v'$ be such that $\{(e,v'), (d,w')\}$ is admissible with $r=\frac{1}{2}$. 
Recall the functors 
from Subsections \ref{defMM}, \ref{yetanother}:
\begin{align*}
\Phi_u=\beta_{\leq -u}p_{\lambda^{-1}}^*: \mathbb{M}^{ab}(d+e)_{w+v}\rightarrow \MM^{ab}(d)_w\boxtimes \MM^{ab}(e)_v\\
\Phi_{u'}=\beta_{\leq u}p_{\lambda}^*:
\mathbb{M}^{ab}(d+e)_{w+v}\rightarrow \MM^{ab}(e)_{v'}\boxtimes \MM^{ab}(d)_{w'}.
\end{align*}
Let $\sigma\in W_{d+e}$ be the transposition defined for a vertex $i\in I$ as:
$$\sigma^i(1,\cdots, d^i+e^i)=(d^i+1,\cdots, d^i+e^i,1,\cdots,d^i).$$ Let $\tau\in\mathfrak{S}_2$ be the involution corresponding to the admissible set $\{(d,w), (e,v)\}$. 
For a complex $E\in K_0(\MM^{ab}(d)_w)\boxtimes K_0(\MM^{ab}(e)_v)$ we have that:
$$\tau(E)=\sigma(E)\otimes\det(N_{e,d})[l(\sigma)].$$ The argument in Proposition \ref{independence} implies that: $$[\tau E]=(-1)^{l(\sigma)}[\sigma E]\text{ in  }\text{gr}\,K_0(\MM^{ab}(d)_w)\boxtimes \text{gr}\,K_0(\MM^{ab}(e)_v).$$ 
For a complex $F\in\MM^{ab}(d+e)$ invariant under $\sigma$, the transposition $\sigma$ interchanges on $p_{\lambda^{-1}}^*$ and $p_{\lambda}^*$ for the weights $u$ and $u'$, respectively, which implies that:
$$\Phi_u(\sigma F)=\sigma \Phi_{u'}(F).$$
Next, $F\in \MM^{ab}(d+e)$ implies that $T_2(\Phi_u(F))=0$ in $K_0(\MM^{ab}(d)_w)\boxtimes K_0(\MM^{ab}(e)_v)\oplus K_0(\MM^{ab}(e)_{v'})\boxtimes K_0(\MM^{ab}(d)_{w'})$, which means that:
\begin{equation}\label{zerodum}
    0=[\Phi_u(F)]+[\sigma \Phi_u(F)\otimes\, \det\,N_{e,d}[l(\sigma)]]=[\Phi_u(F)]+(-1)^{l(\sigma)}[\Phi_{u'}(F)]\end{equation} in
$\text{gr}\,K_0(\MM^{ab}(d)_w)\boxtimes \text{gr}\,K_0(\MM^{ab}(d)_v)\oplus \text{gr}\,K_0(\MM^{ab}(e)_{v'})\boxtimes \text{gr}\,K_0(\MM^{ab}(d)_{w'})$.
Further, for $F$ in $\MM^{ab}(d+e)$ invariant under $\sigma$, we have decompositions: $$[F]=m^{ab}_{d,e}(\Phi_{u}(F))+M_{u+1}=M_u+m^{ab}_{e,d}(\Phi_{u'}(F))$$
for some $M_u\in K_0(\mathbb{G}_u)$ and $M_{u+1}\in K_0(\mathbb{G}_{u+1})$. By the argument used to prove Theorem \ref{mut}, for $A\boxtimes B\in \oM^{ab}(d)_w\boxtimes \oM^{ab}(e)_v$
there exists a triangle $$p_{e,d*}q_{e,d}^*(B\boxtimes A\otimes \text{det}(N^{\lambda<0}))\to
p_{d,e*}q^*_{d,e}(A\boxtimes B)\to C\xrightarrow{[1]}$$ with $C\in \mathbb{G}_w$. This implies, using Proposition \ref{abelinverse}, that $$\Delta^{ab}_{d,e}m^{ab}_{d,e}(\Phi_u(F))=
(-1)^{\dim\,\text{Ext}^1(d,e)}\Delta^{ab}_{d,e}m^{ab}_{e,d}(\Phi_u(F))=(-1)^{\dim\,\text{Ext}^1(d,e)}\,\Phi_u(F).$$
By Theorem \ref{coincide} and Proposition \ref{abelinverse}, this means that:
$$\Delta_{d,e}^{ab}(F)=(-1)^{\dim\,\text{Ext}^1(d,e)}[\Phi_{u}(F)]=[\Phi_{u'}(F)]$$ in $\text{gr}K_0(D_{sg}(\mathcal{Y}(d)_0))\boxtimes \text{gr}K_0(D_{sg}(\mathcal{Y}(e)_0)).$
By Proposition \ref{even}, we have that $$\dim\,\text{Ext}^1(d,e)\equiv l(\sigma) \,(\text{mod} 2).$$
Using Equation (\ref{zerodum}) we obtain the desired conclusion.
\end{proof}

We next prove the analogous result in cohomology:

\begin{prop}\label{Cprimitive}
All elements in $P^{\leq 1}$ are primitive for the coproduct $\Delta$ of $\text{CoHA}$.
\end{prop}

\begin{proof}
Let $\pi_d:\X(d)\to X(d)$ be the Hilbert-Chow morphism for the dimension vector $d\in\mathbb{N}^I$.
Recall that:
\begin{multline*}
P^{\leq 1}:=H^{\cdot}(X(d),
\varphi_{\text{Tr}\,W}{}^p{R}^{\leq 1}\pi_{d*}\mathbb{Q}_{\X(d)}[-(d,d)])\hookrightarrow H^{\cdot}(X(d),\varphi_{\text{Tr}\,W}R\pi_{d*}\mathbb{Q}_{\X(d)}[-(d,d)/2])\cong\\ H^{\cdot}(\X(d),\varphi_{\text{Tr}\,W}\mathbb{Q}_{\X(d)}[-(d,d)]).\end{multline*}
Consider two dimension vectors $d,e\in\mathbb{N}^I$, and recall the diagram:

\begin{tikzcd}
\widetilde{\mathcal{X}}(d,e) \arrow{d}{\widetilde{q}_{d,e}} \arrow{r}{\tau} & \mathcal{X}(d,e) \arrow{r}{p_{d,e}} & \mathcal{X}(d+e)\\
\mathcal{X}(d)\times\mathcal{X}(e)
\end{tikzcd}
\\
The map $\tau$ is cohomologically proper, see Definition \ref{proper}. We have a map of sheaves on $\X$:
$$\mathbb{Q}[-(d,d)]\to (\tau p_{d,e})_*(\tau p_{d,e})^*\mathbb{Q}[-(d,d)].$$
Apply $\pi_{d*}$ to these equality and use that $$(\pi_a\boxtimes\pi_b)_*\widetilde{q}_{*}=\pi_{d*}(\tau p)_*$$ and that 
the dimension of the map $\widetilde{q}_{d,e}$ is $\text{dim}\,\text{Ext}^1(a,b)$
to obtain a map:
$$R\pi_{d*}\mathbb{Q}[-(d,d)/2]\to (\pi_a\boxtimes\pi_b)_*\Pi_{*}\mathbb{Q}[-(d,d)]\to R\pi_{a*}\mathbb{Q}[-(a,a)]\boxtimes R\pi_{b*}\mathbb{Q}[-(b,b)].$$ 
To obtain the maps in the nonzero potential case, we apply the vanishing cycle functor and use that it commutes with $R\pi_{d*}$ \cite{dm}:
$$\varphi_{\text{Tr}\,W}R\pi_{d*}\mathbb{Q}[-(d,d)]=R\pi_{d*}\varphi_{\text{Tr}\,W}\mathbb{Q}[-(d,d)]\to  \varphi_{\text{Tr}\,W}R\pi_{a*}\mathbb{Q}[-(a,a)]\boxtimes\varphi_{\text{Tr}\,W} R\pi_{b*}\mathbb{Q}[-(b,b)].$$
It is enough to show that ${}^pR^{\leq 1}\pi_{d*}\mathbb{Q}[-(d,d)]$ is in the kernel of the maps:
$$\Delta:R\pi_{d*}\mathbb{Q}[-(d,d)]\to R\pi_{a*}\mathbb{Q}[-(a,a)]\boxtimes R\pi_{b*}\mathbb{Q}[-(b,b)].$$

The complexes $R\pi_{a*}\mathbb{Q}[-(a,a)]$ and $R\pi_{b*}\mathbb{Q}[-(b,b)]$ lie in perverse degrees $\geq 1$, so $R\pi_{a*}\mathbb{Q}[-(a,a)]\boxtimes R\pi_{b*}\mathbb{Q}[-(b,b)]$ lies in perverse degrees $\geq 2$, which implies the desired conclusion.
\end{proof}

\begin{prop}\label{braiding2}
Let $x\in \text{gr}^{\leq a}K_0(D_{sg}(\X(d)_0))\subset E^{\leq a},\,y\in\text{gr}^{\leq b}K_0(D_{sg}(\X(e)_0))\subset E^{\leq b}$. 
Then: $$R_{d,e}(x\boxtimes y)\in \bigoplus_{i+j\leq a+b} E^{\leq i}\boxtimes E^{\leq j}.$$
\end{prop}

\begin{proof}
Assume first that $x,y\in E^{\leq 1}$. Corollary \ref{KBPS} implies that $xy-yx=z$, where $z\in E^{\leq 1}$. By Proposition \ref{Kprimitive}, we have that $\Delta(z)=z\boxtimes 1+1\boxtimes z$.
Using Theorem \ref{bialgebra}, we get that $R_{d,e}(x\boxtimes y)=\pm y\boxtimes x$, which proves the statement in this case.

Assume that $x$ is not in $E^{\leq 1}$; it suffices to show the results for $x=x_1x_2$, where $x_1\in E^{\leq a_1}$ is of dimension $d_1$ and $x_2\in E^{\leq a_2}$ is of dimension $d_2$. We will use induction on $a$ to prove the result. Using Proposition \ref{braiding}, it suffices to show that:
$$\left(R_{d_1,d_2}\boxtimes 1\right)\left(x_1\boxtimes R_{d_2,e}(x_2\boxtimes y)\right)\in \bigoplus_{i+j+k\leq a+b} E^{\leq i}\boxtimes E^{\leq j}\boxtimes E^{\leq k}.$$ 
We have that $x_1\boxtimes R_{d_2,e}(x_2\boxtimes y)\in \bigoplus E^{\leq a_1}\boxtimes E^{\leq i}\boxtimes E^{\leq a_2+b-i}$, and using induction again for the product in $E^{\leq a_1}\boxtimes E^{\leq i}$ we get the desired conclusion.
\end{proof}

\begin{prop}\label{gradedbialgebra}
The coproduct and the filtration on $\text{grKHA}$ are compatible, that is, for every $c\geq 1$ we have:
$$\Delta(E^{\leq c})\subset \bigoplus_{i+j\leq c}E^{\leq i}\boxtimes E^{\leq j}.$$
Thus the associated graded $\text{gr}^E\text{gr}\,\text{KHA}$ is a bialgebra with the product and coproduct induced from $\text{gr}\,\text{KHA}$.
\end{prop}

\begin{proof}
Proposition \ref{Kprimitive} implies the statement for elements in $c=1$.
For general $c\geq 1$, it suffices to prove the claim for $xy\in E^{\leq c}$ with $x\in E^{\leq a}$ and $y\in E^{\leq b}$ such that $a,b\geq 1$ and $a+b\leq c$; assume that $x\in \text{gr}^{\leq a}K_0(D_{sg}(\X(d)_0))$ and $y\in \text{gr}^{\leq b}K_0(D_{sg}(\X(e)_0))$. 
Theorem \ref{bialgebra} says that:
$$\Delta_{d^0,e^0}(xy)=\sum_P \left(m\boxtimes m\right)\left(1\boxtimes R_{d^2,e^1}\boxtimes 1\right)\left(\Delta_{d^1,d^2}(x)\boxtimes \Delta_{e^1,e^2}(y)\right).$$
Using the induction hypothesis, we have that: $$\Delta_{d^1,d^2}(x)\in \bigoplus E^{\leq i}\boxtimes E^{\leq a-i}\text{ and }\Delta_{e^1,e^2}(y)\in\bigoplus E^{\leq j}\boxtimes E^{\leq b-j}.$$ For $x_2\in E^{\leq a-i}$ and $y_1\in E^{\leq j}$, Proposition \ref{braiding2} implies that:
$$R_{d^2,e^1}(x_2\boxtimes y_1)\in \bigoplus E^{\leq k}\boxtimes E^{\leq a-i+j-k}.$$
Putting these facts together, we see that: 
$$\Delta_{d^0,e^0}(xy)\in \bigoplus m^2 \left(E^{\leq i}\boxtimes E^{\leq k}\boxtimes E^{\leq a-i+j-k}\boxtimes E^{\leq b-j}\right)\subset \bigoplus E^{i+k}\boxtimes E^{a+b-i-k},$$ which implies the desired conclusion.
\end{proof}

\begin{prop}\label{primitiveKHA}
The primitive elements of the $\text{grKHA}$ bialgebra are $E^{\leq 1}\subset \text{grKHA}$.
\end{prop}

\begin{proof}
Let $\text{Prim}$ be the primitive elements of $\text{gr}_I\text{KHA}$, and $\text{Prim}'$ the primitive elements of $\text{gr}^E\text{gr}_I\text{KHA}$. We have that $\text{Prim}\subset \text{Prim}'$.
Proposition \ref{Kprimitive} shows that $E^{\leq 1}\subset \text{Prim}.$ Theorem \ref{PBWgr} and Proposition \ref{gradedbialgebra} say that $\text{gr}^E\text{gr}_I\text{KHA}=\text{Sym}\left(E^{\leq 1}\right)$ is a supersymmetric algebra and that the coproduct is compatible with the multiplication, so 
$E^{\leq 1}=\text{Prim}'$, which implies that
$E^{\leq 1}=\text{Prim}.$
\end{proof}

\begin{prop}\label{primitivecoha}
The primitive elements of the $\text{CoHA}$ bialgebra are $P^{\leq 1}\subset \text{CoHA}$.
\end{prop}

\begin{proof}
The conclusion of Proposition \ref{gradedbialgebra} holds in cohomology because $\text{CoHA}$ is a bialgebra with braiding $R$. The proof of Proposition \ref{primitiveKHA} follows formally once we know the statements of Propositions \ref{Kprimitive}, \ref{braiding2}, and \ref{gradedbialgebra}, and thus the above Proposition follows from Proposition \ref{Cprimitive} and the analogues of Propositions \ref{braiding2} and \ref{gradedbialgebra} in cohomology.
\end{proof}

\begin{proof}[Proof of Theorem \ref{comp}]
Recall the definition of the filtration $E^{\leq i}$: a monomial $\left(x_{d_1}h_{d_1}^{i_1}\right)\cdots \left(x_{d_k}h_{d_k}^{i_k}\right)$ has degree $\text{deg}=\left(1+2i_1\right)+\cdots+\left(1+2i_k\right)$, where $x_{d_i}\in \text{gr}\,K_0(\mathbb{M}(d_i)_0)$ is in $E^{\leq 1}$. The subspace $E^{\leq i}\subset \text{gr}_I\,\text{KHA}$ is defined as the subspace generated by monomials of degree $\text{deg}\leq i$. 

Similarly, for a monomial 
$\left(x_{d_1}h_{d_1}^{i_1}\right)\cdots \left(x_{d_k}h_{d_k}^{i_k}\right)$,
where $x_{d_i}\in H^{\cdot}(X(d),\varphi_{\text{Tr}\,W} IC_{d})\subset P^{\leq 1}$,
define its degree as $\text{deg}=\left(1+2i_1\right)+\cdots+\left(1+2i_k\right)$. Then 
$P^{\leq i}\subset \text{CoHA}$ is the subspace generated by monomials of degree $\text{deg}\leq i$. 

It thus suffices to show that generators $x_dh_d^{\leq i}$ are sent to elements in $E^{\leq 2i+1}$. We have that $\text{ch}(x_dh_d^i)=\text{ch}(x_d)h_d^i$, so it is enough to prove the statement for $x_d\in E^{\leq 1}.$ These elements are primitive by Propositions \ref{primitiveKHA} and \ref{primitivecoha}, and $\text{ch}$ is a bialgebra morphism, so we indeed have that:
$$\text{ch}\,E^{\leq 1}\subset \text{Prim(CoHA)}=P^{\leq 1}.$$
\end{proof}

An immediate corollary of Theorem \ref{comp} is:

\begin{cor}
The map $\text{ch}: KBPS\to BPS$ is a morphism of Lie algebras.
\end{cor}

The following corollary follows easily from Theorem \ref{comp} in the zero potential case.

\begin{cor}
Let $Q$ be a symmetric quiver and let $d\in\mathbb{N}^I$ be a dimension vector. Then the Chern character restricts to an isomorphism:
$$\text{ch}: \text{gr}\,K_0(\mathbb{N}(d))\to IH^{\cdot}(X(d)).$$ 
The Chern character $\text{ch}: K_0(\X(d))\to H^{\cdot}(\X(d))$ is injective, so $K_0(\mathbb{N}(d))\cong \text{gr}\,K_0(\mathbb{N}(d))$ and thus $$\text{dim}\,K_0(\mathbb{N}(d))=\text{dim}\,IH^{\cdot}(X(d)).$$
Further, the map of Lie algebras 
$$\text{ch}: KBPS(Q,0)\to BPS(Q,0)$$ is an isomorphism of Lie algebras with trivial Lie bracket.
\end{cor}

\begin{proof}
The Lie algebra $BPS$ has trivial Lie bracket because $\text{CoHA}(Q,0)$ is supersymmetric, for example by the computation in \cite[Section 2.4]{ks}. All the other statements are immediate.
\end{proof}

We next discuss two examples where we check that $\text{dim}\,K_0(\mathbb{N}(d))=\text{dim}\,IH^{\cdot}(X(d))$ directly.
\\

\textbf{Example.} Let $Q$ be the quiver with two vertices $1$ and $2$ and $n$ arrows in both directions between the vertices. The diagonal copy of $\C^*$ acts trivially on $R(1,1)=\C_{1}^n\oplus\C_{-1}^n$. The coarse space $X:=\C_{1}^n\oplus\C_{-1}^n//\C^*$ is singular and has two small resolutions of singularities by $X^+$ and $X^-$, the GIT quotients of $\C_{1}^n\oplus\C_{-1}^n$ by $\C^*$ for the two possible linearizations. 
We have that 
$$IH^{\cdot}(X)=H^{\cdot}(X^+)=H^{\cdot}(X^-)$$ because the resolutions $\pi^+:X^+\to X$ and $\pi^-:X^-\to X$ are small.
Next, by \cite[Theorem 3.2]{hls}, we have that $\mathbb{N}\cong\overline{\mathbb{N}}\cong D^b(X^+)\cong D^b(X^-)$. The varieties $X^+$ and $X^-$ are $\mathbb{P}^{n-1}$ bundles over an affine space, so $$\text{ch}: K_0(X^{\pm})\to H^{\cdot}(X^{\pm})$$ is an isomorphism, and this explain why $\text{dim}\,K_0(\mathbb{N})=\text{dim}\,IH^{\cdot}(X).$ 
\\

\textbf{Example.} Let $Q$ be the quiver with $m$ loops, and choose the dimension vector $d=2$. There are no nontrivial GIT stability conditions, so we cannot obtain a resolution of singularities as in the example above. Denote by $\beta_0$ and $\beta_1$ the roots of $PGL(2)$. 
The subcategory $\overline{\mathbb{N}}(2)$ of $D^b(\mathfrak{gl}(2)^m/PGL(2))$ is generated by the dominant weights $\chi$ such that:
$$\chi+\rho=c(\beta_0-\beta_1)+\frac{\beta_0-\beta_1}{2}\in \frac{m}{2}[\beta_1-\beta_0, \beta_0-\beta_1].$$ The weight $\chi$ is dominant,
so $c\geq 0$. The condition above implies that $0\leq c\leq \frac{m-1}{2}.$ Further, if $m$ is even, then: $\mathbb{N}(2)=\overline{\mathbb{N}}(2)$; if $m=2k+1$ is odd, then $$K_0(\overline{\mathbb{N}}(2))=K_0(\mathbb{N}(2))\oplus \left(K_0(\mathbb{N}(1)_{-k})\boxtimes K_0(\mathbb{N}(1)_{k})\right)^{\mathfrak{S}_2}.$$
This implies that:
\[ \text{dim}\,K_0(\mathbb{N}(2))= 
\begin{cases} \frac{m-1}{2}, & \mbox{if }m\mbox{ is odd,}  \\ \frac{m}{2}, & \mbox{if }m\mbox{ is even.} \end{cases} \]

Meinhardt and Reineke \cite[Theorem 4.7]{mr} have a combinatorial formula of the intesection cohomology groups of the coarse space $X(d)$ for all dimensions $d\geq 1$. When $d=2$, one can check that their formula says that:
\[ \text{dim}\,IH^{\cdot}(X(2))= 
\begin{cases} \frac{m-1}{2}, & \mbox{if }m\mbox{ is odd,}  \\ \frac{m}{2}, & \mbox{if }m\mbox{ is even.} \end{cases} \]

$$G_0(\mathcal{P}(d))\to \text{gr}\,G_0(\mathcal{P}(d))\to K_0(\mathcal{P}^{\text{ss}}(f,d)/G(d))$$ which fits in the diagram:


\section{Examples of KHA}\label{6}

\subsection{Description of $KHA$ for $W=0$}
Let $Q$ be an arbitrary quiver, and consider the algebra for the zero potential $KHA(Q)=KHA(Q,0)$. For $i,i'\in I$, denote by $c(i,i')$ the number of edges between $i$ and $i'$ in $Q$. For $d\in\mathbb{N}^I$, let $W_d:=\times_{i\in I}\mathfrak{S}_{d_i}$ be the Weyl group of $G(d)$. We have that:
$$K_0(\X(d))=\mathbb{Z}[z_{i,j}^{\pm 1}]^{W_d},$$ 
where $i\in I$ and $1\leq j\leq d_i$. 
The multiplication is defined by the composition:
$$
\X(d)\times\X(e)\xrightarrow{q_{d,e}^*} \X(d,e)\xrightarrow{i_{d,e*}} R(d+e)/G(d,e)\xrightarrow{\pi_{d,e*}} \X(d+e),$$
where $i_{d,e}$ is the embedding 
$i_{d,e}: R(d,e)\to R(d+e)$
and $\pi_{d,e}$ is the projection 
$\pi_{d,e}: R(d+e)/G(d,e)\to \X(d+e)$.
We have that: $$q_{d,e}^*: K_0(\X(d)\times\X(e))\cong K_0(\X(d,e)).$$
The weights of the normal bundle of $R(d,e)$ in $R(d+e)$ are 
$\beta^i_{j}-\beta^{i'}_{d_{i'}+j'}$, where $i,i'$ are vertices with an edge from $i$ to $i'$, 
and $1\leq j\leq d_i$ and $1\leq j'\leq e_{i'}$. The first map: 
$$i_{d,e*}: K_0(R(d,e)/G(d,e))\to K_0(R(d+e)/G(d,e))$$
is equal to: 
$$i_{d,e*}\left(f(z_{i,j})g(z_{i',j'})\right)=f(z_{i,j})g(z_{i',j'})\prod_{i\in I, 1\leq j\leq d_i}\prod_{i'\in I, d_{i'}+1\leq j'\leq d_{i'}+e_{i'}} \left(1-z_{i,j}z^{-1}_{i',j'}\right)^{c(i,i')}.$$
Using the formula for multiplication in terms of the abelian stacks in Proposition \ref{abel}, we have that:

\begin{prop}\label{shuffle}
The $\text{KHA}(Q,0)$ has graded vector spaces 
$$K_0(\X(d))=\mathbb{Z}[z_{i,j}^{\pm 1}]^{W_d},$$ 
where $i\in I$ and $1\leq j\leq d_i$, and the multiplication for $f\in K_0(\X(d))$ and $g\in K_0(\X(e))$ is given by the formula:
$$(fg)(z_{i,j})=\text{Sym}\,\left(f(z_{i,j})g(z_{i',j'})\frac{\prod_{i\in I, 1\leq j\leq d_i}\prod_{i'\in I, d_{i'}+1\leq j'\leq d_{i'}+e_{i'}} \left(1-z_{i,j}z^{-1}_{i',j'}\right)^{c(i,i')}}{\prod_{i\in I}\prod_{1\leq j\leq d, d+1\leq j'\leq d+e}(1-z_{i,j}z^{-1}_{i,j'})}\right).$$
\end{prop}

\begin{cor}
Let $Q$ be a symmetric quiver. Then 
$$KHA(Q,0)=\text{qSym}\,\left(\bigoplus_{d\in\mathbb{N}^I} K_0(\mathbb{M}(d))\right).$$
\end{cor}

\begin{proof}
One can either check the relations directly using the shuffle product formula, or use the equality in $K_0(\X(d))$ proved in Example \ref{Wexample}:
$$[p_{d,e*}\left(V(\chi_d+\chi_e)\otimes\OO_{d,e}\right)]=[p_{e,d*}\left(V(\chi_e+\chi_d-N_{d,e})\otimes \OO_{e,d}\right)].$$ 
\end{proof}

A shuffle formula for the product of $\text{CoHA}$ appears in \cite[Section 2.4]{ks}, and for a general oriented cohomological theory in \cite{yz}.
\subsection{Equivariant Jordan quiver}\label{Jordan}

Let $Q$ be the quiver with one vertex and one loop. The representations of $Q$ of dimension $d\in\mathbb{N}$ are given by linear maps $V\to V$, where $\text{dim}\,(V)=d$. Consider the action of $\C^*$ by scaling the map $V\to V$. 
\begin{prop}
The equivariant KHA algebra 
$$KHA_{\C^*}(Q,0)=\bigoplus_{d\in\mathbb{N}}K^{\C^*}_0(\X(d))=\bigoplus_{d\in\mathbb{N}}\mathbb{Z}[z_1^{\pm 1},\cdots,z_d^{\pm 1},q^{\pm 1}]^{W_d}$$
has a shuffle product description:
$$(fg)(z_1,\cdots,z_{d+e},q)=\text{Sym}\left(f(z_1,\cdots, z_d, q)g(z_{d+1},\cdots,z_{d+e}, q)\prod_{1\leq i\leq d, d+1\leq j\leq d+e} \zeta(z_iz_j^{-1})\right),$$ for $f\in K^{\C^*}_0(\X(d))$, $g\in K^{\C^*}_0(\X(e))$, and for $\zeta(z)=\frac{1-qz}{1-z}.$ We thus have an isomorphism: $$KHA_{\C^*}(Q,0)\cong U_q^{>}(L\mathfrak{sl}_2).$$
\end{prop}

\begin{proof}

The dimension $d$ components of $KHA_{\C^*}(Q,0)$ is $$K_0(\X(d)/\C^*)= \mathbb{Z}[z_1^{\pm 1},\cdots, z_d^{\pm 1}, q^{\pm 1}]^{W_d}.$$
The multiplication is defined by the composition:
$$K_0(\X(d,e)/\C^*\times\C^*)\to K_0(\X(d,e)/\C^*)\to K_0(\X(d+e)/\C^*).$$
The first map is induced by the diagonal embedding $\C^*\to \C^*\times \C^*$, which induces the map:
$$\mathbb{Z}[q_1^{\pm 1}, q_2^{\pm 1}]\to \mathbb{Z}[q^{\pm 1}],$$ which sends $q_1$ and $q_2$ to $q$. The second map  is:
$$K_0(R(d,e)/G(d,e)\times\C^*)\xrightarrow{i_{d,e*}} K_0(R(d+e)/G(d,e)\times\C^*)\xrightarrow{\pi_{d,e*}} K_0(R(d+e)/G(d+e)\times\C^*).$$ The immersion $i_{d,e}:R(d,e)\to R(d+e)$ has normal bundle with weights $qz_iz_j^{-1}$, where $1\leq i\leq d$, $d+1\leq j\leq d+e$, so we have $$i_{d,e*}(f(z_i)g(z_j))=f(z_1,\cdots, x_d)g(z_{d+1},\cdots,z_{d+e})\prod_{1\leq j\leq d,d+1\leq i\leq d+e}(1-qz_iz_j^{-1}).$$
Using the formula for multiplication in Proposition \ref{abel}, we have that:
$$(fg)(z_1,\cdots, z_{d+e})=\text{Sym}\left(fg\prod_{1\leq j\leq d,d+1\leq i\leq d+e} \zeta\left(\frac{z_i}{z_j}\right)\right),$$ where $\zeta(z)=\frac{1-qz}{1-z}.$ By \cite[Theorem 3.5]{ts}, the quantum group $U_q^{>}(L\mathfrak{sl}_2)$ has the same shuffle product description.

\end{proof}

\subsection{}
Let $(Q,W)$ be an arbitrary quiver with potential and let $T$ be a torus as in Section \ref{2}; in particular, $T$ preserves the potential $W$. 

\begin{prop}\label{def}
The inclusion $i:\X(d)_0\to \X(d)$ induces an algebra morphism 
$$KHA_T(Q,W)\to KHA_T(Q,0).$$
\end{prop}

\begin{proof}
First, we explain that the map from $i_*:D^b(\X(d)_0)\to D^b(\X(d))$ commutes with the maps used in the definition of multiplication. For this, we need to check that the following diagram is commutative:

\begin{tikzcd}
D^b(\X(d)_0)\boxtimes D^b(\X(e)_0) \arrow{d}{i_{d*}\times i_{e*}}\arrow{r}{i_*p_*q^*} & D^b(\X(d+e)_0) \arrow{d}{i_{d+e*}}\\
D^b(\X(d))\boxtimes D^b(\X(e)) \arrow{r}{p_*q^*} & D^b(\X(d+e)),
\end{tikzcd}
\\
where $i:\X(d)_0\times \X(e)_0\to \X(d+e)_0$. It is enough to show that the following diagram commutes:

\begin{tikzcd}
D^b\left((\X(d)\times\X(e))_0\right) \arrow{d}{i_{d+e*}} \arrow{r}{q^*} & D^b(\X(d,e)_0) \arrow{r}{p_*} \arrow{d}{i_{d+e*}} & D^b(\X(d+e)_0) \arrow{d}{i_{d+e*}}\\
D^b(\X(d)\times\X(e)) \arrow{r}{q^*} & D^b(\X(d,e)) \arrow{r}{p_*} & D^b(\X(d+e)).
\end{tikzcd}
\\
The left corner commutes from proper base change, and the right corner is clear.
Next, the map $G^T_0(\X(d)_0)\to K^T_0(\X(d))$ factors through: $$K^T_0(D_{sg}(\X(d)_0))\to K_0^T(\X(d)).$$ For this, we need to show that $i_*K^T_0(\X(d)_0)=0$, which follows from a $T$-equivariant version of Proposition \ref{pushzero}. 

The induced maps $K^T_0(D_{sg}(\X(d)_0))\to K^T_0(\X(d))$ respect multiplication, so: $$i_*: KHA_T(Q,W)\to KHA_T(Q,0)$$ is an algebra morphism.

\end{proof}

\subsection{Type $A$ Dynkin quivers.}\label{Dynkin}
Let $n\geq 2$, and consider $Q$ a Dynkin quiver of type $A_{n}$. Let $(c_{ij})$ be its Cartan matrix. 
Let $SH_n$ be the $\mathbb{N}^I$-graded vector space 
$$SH_n=\bigoplus_{d\in \mathbb{N}^I} \text{qf}\,K_0(B(G(d)\times\C^*)),$$
where $\text{qf}\,R$ denotes the quotient field of the domain $R$.
One can define a shuffle product on $SH_n$.
Let $f_1(z_{i,r})\in \text{qf}\,K_0(B(G(d)\times\C^*))$ where $i\in I$ and $1\leq r\leq d_i$, and $f_2(z_{i',r'})\in \text{qf}\,K_0(B(G(e)\times\C^*))$ where $i'\in I$ and $1\leq r'\leq e_i$. The shuffle product formula on $SH_n$ is:
$$(f_1\, f_2)(z_{1,1},\cdots,z_{n,d_n+e_n})=
\text{Sym}\left(f_1(z_{i,r})f_2(z_{i',r'})\prod_{i,i'\in I, r\leq d_i, r'>d_{i'}} \zeta_{i,i'}\left(\frac{z_{i,r}}{z_{i',r'}}\right)\right),$$ where $\zeta_{i,j}(z)=\frac{z-q^{-c_{ij}}}{z-1}.$
Recall the following result from \cite[Theorem 3.5]{ts}:

\begin{prop}\label{tsresult}
Let $SH'$ be the subalgebra of $SH_n$ generated by $K_0(B(G(e_i)\times\C^*))$ for $1\leq i\leq n$, where $e_i$ is a dimension vector with $1$ in position $i$ and $0$ everywhere else. 
Then there exists an isomorphism: $$SH'\cong U_q^{>}(L\mathfrak{sl}_{n+1})$$ which sends the generators of $K_0(B(G(e_i)\times\C^*))$ to the generators in dimensions $e_i$ of the quantum group for $1\leq i\leq n$. 
\end{prop}

Recall the definition of the tripled quiver $(\widetilde{Q}, \widetilde{W})$ associated to a quiver $Q$ from Subsection \ref{tripledef}. 

Let $T'=\C_{q_1}\times\C_{q_2}$ be the two dimensional torus that acts by multiplication with $q_1$ on the edges $e\in E$, by multiplication with $q_2$ on the edges $\overline{e}$, and by multiplication with $q_1^{-1}q_2^{-1}$ on $\omega_i$ for $i\in I$. Let $T\subset T'$ be the one dimensional torus which acts with $q_1=q_2$. 

\begin{thm}
There exists an isomorphism $$KHA_T(\widetilde{Q},\widetilde{W})\cong U_q^{>}(L\mathfrak{sl}_{n+1}).$$
\end{thm}

\begin{proof}

By Proposition \ref{liecomp}, the KBPS Lie algebra for $(\widetilde{Q},\widetilde{W})$ is $$\bigoplus_{d\in\mathbb{N}_I} K_0(\mathbb{M}(d))=\bigoplus_{d\in\mathbb{N}_I} \text{gr}\,K_0(\mathbb{M}(d))=\mathfrak{sl}^{>}_{n+1},$$ so $KHA_T(\widetilde{Q},\widetilde{W})$ is generated by the components of dimensions $e_i$, for $1\leq i\leq n$.
The map from Proposition \ref{def}: $$KHA_T(\widetilde{Q},\widetilde{W})\to KHA_T(\widetilde{Q},0)$$ is injective because it is injective in dimensions $e_i$, for $1\leq i\leq n$.
The formula for the multiplication for $KHA_T(\widetilde{Q},\widetilde{W})$ is given by the formula for the multiplication on $KHA_T(\widetilde{Q},0)$ by \ref{def}.

Next, we follows the same reasoning as for Proposition \ref{shuffle} to write the shuffle product for $\text{KHA}_T(\widetilde{Q},0)$. The weights appearing in the normal bundle $N_{d,e}$ are $1-z_{i,j}z_{i,j'}q^{-1}_1q^{-1}_2$ for the loops $\omega_i$, $1-z_{i,j}z_{i',j'}^{-1}q_1$ for the edges $e\in E$ with multiplicity $e(i,i')$, $1-z_{i,j}z_{i',j'}^{-1}q_2$ for the edges $\bar{e}$, for $e\in E$ with multiplicity $e(i',i)$. 
For $f$ of dimension $d$ and $g$ of dimension $e$, the shuffle product has the form:
$$(fg)(z_{i,j})=\text{Sym}\,\left(f(z_{i,j})g(z'_{i',j'})\text{ prod}\right),$$ where the factor $\text{prod}$ is: 
$$\text{prod}=\prod_{i\in I, j\leq d_i, j'>d_i} \frac{1-z_{i,j}z_{i,j'}^{-1}q^{-1}_1q^{-1}_2}{1-z_{i,j}z_{i,j'}^{-1}} \prod_{i,i'\in I}(1-z_{i,j}z_{i',j'}^{-1}q_1)^{e(i,i')}(1-z_{i,j}z_{i',j'}^{-1}q_2)^{e(i',i)}  .$$
The formula can be rewritten as:
$$(f\, g)(z_{i,j})=
\text{Sym}\left(f(z_{i,j})\,g(z_{i',j'})\prod_{i,i'\in I, j\leq k_i, j'>k_{i'}} \tau_{i,i'}\left(\frac{z_{i,r}}{z_{i',r'}}\right)\right),$$ where the formula for $\tau_{i,i'}$ is:
\[   
\tau_{i,i'}(z) = 
     \begin{cases}
       \frac{1-z q_1^{-1}q_2^{-1}}{1-z} \text{ for }i=i',\\
       1-zq_1 \text{ for }i-i'=1, \\
       1-zq_2 \text{ for }i-i'=-1,\\
       1 \text{ otherwise.} \\ 
     \end{cases}
\]

Define the map $\Phi: SH'_n\to K$ by sending $\Phi(q_1)=\Phi(q_2)=q^{-1}$, and by: $$\Phi(f)=f\prod_{i,i'=1}^{n+1} \prod_{j,j'\leq d_i} \frac{\tau_{i,i'}(z_{i,j}/z_{i',j'})}{\zeta_{i,i'}(z_{i,j}/z_{i',j'})}.$$ 

We first need to argue that the map is well defined. For this, we need to show that $\frac{\tau_{i,i'}}{\zeta_{i,i'}}$ has poles at most along $z$. We compute the ratio:
\[   
\frac{\tau_{i,j}(z)}{\zeta_{i,j}(z)} = 
     \begin{cases}
       q^2 \text{ for }i=j,\\
       q(1-z) \text{ for }i-j=1, \\
       q(1-z) \text{ for }i-j=-1,\\
       1 \text{ otherwise.} \\ 
     \end{cases}
\]

The map $\Phi$ is thus indeed well defined, and it is an algebra morphism.
Both $SH_n'$ and $KHA_T(\widetilde{Q},\widetilde{W})$ are generated in dimension $e_i$, with $1\leq i\leq n+1$, and in these dimensions they are isomorphic: 
$$SH'_{e_i}=\mathbb{Q}[q^{\pm 1}, t_{i,1}^{\pm 1}]\to K^T_0(D_{sg}(\X(e_i)_0))=K_0(B(\C^*\times\C^*)).$$

This means that $\Phi: SH_n'\to KHA_T(\widetilde{Q},\widetilde{W})$ is an algebra isomorphism, and the result follows from Proposition \ref{tsresult}.

\end{proof}

We have left with the proof of the following:

\begin{prop}\label{liecomp}
The KBPS Lie algebra for $(\widetilde{Q},\widetilde{W})$ is $\mathfrak{sl}^{>}_{n+1}$.
\end{prop}

\begin{proof}
First, we argue that $\mathbb{M}(d)=D^b(\text{pt})$ for $d$ a root of $A_{n}$, and that $\mathbb{M}(d)$ is trivial otherwise. Indeed, choose the stability condition $ \theta: \theta_1>\cdots>\theta_{n}$. The stack $T^*\X^{\text{ss}}(d)$ is either empty or a point, so by \cite[Theorem 3.3]{hl2} we have that:
$$\mathbb{M}(d)\cong D^b(T^*(\X^{\text{ss}}(d))).$$ 
The space $\X^{\text{ss}}(d)$ is empty for $d$ not a root of $A_{n}$, and it is a point otherwise.

Next, assume $n=2$. We use dimensional reduction to get that $\X(1,0)=\text{pt}$, $\X(0,1)=\text{pt}$, and $\X(0,1)=\text{Spec}\, k[x,y]/xy,$ where $x$ corresponds to the edge from $0$ to $1$, $y$ to the edge from $1$ to $0$, and the group $\C^*\times\C^*$ acts by multiplication $q_1\cdot x=q_1x$, $q_1\cdot y=q_1^{-1}y$, and $q_2\cdot x=q_2^{-1}x$, $q_2\cdot y=q_2y$. The diagonal copy of $\C^*$ acts trivially, and we let the quotient $\C^*$ act by $q$ on $x$ and $q^{-1}$ on $y$.
Let $X=\text{Spec}\,k[x]$ and $Y=\text{Spec}\,k[y]$.
Further, we have that: $$\left(q^i1_{(1,0)}\right)\left( q^{-i}1_{(0,1)}\right)=[\OO_X(i)]$$ and similarly:
$$\left( q^i1_{(0,1)}\right)\left(q^{-i}1_{(1,0)}\right)=[\OO_Y(-i)]$$ in $K_0(\X(1,1))$. 
By checking the weight conditions, we have that $\mathbb{M}(1,0)= \OO_{\text{pt}} $, $\mathbb{M}(0,1)= \OO_{\text{pt}},$, and $\mathbb{M}(1,1)=\langle \OO_{\X}\rangle \subset D^b(\X(1,1))$.
Using the exact sequence $$0\to \OO_Y(-1)\to \OO_{\X}\to \OO_X\to 0,$$ 
we get that: $$e_{(1,1)}= e_{(1,0)}e_{(0,1)}+(qe_{(0,1)})(q^{-1}e_{(1,0)}).$$ Taking into account the definition of the twisted multiplication and of the cohomological filtration on $KHA$, we get the relation 
$$e_{(1,1)}= e_{(1,0)}e_{(0,1)}-e_{(0,1)}e_{(1,0)}$$ in the KBPS Lie algebra. All other Lie brackets between the generators are zero by dimension reasons, so the Lie algebra is $\mathfrak{sl}_3^{>}$ in this case. 

For $n\geq 3$, we use induction to find the relations between the generators $e_i$ and $e_j$, for all pairs $1\leq i,j\leq n$ except $(1,n)$. For the pair $(1,n)$, there are no nontrivial extensions between $d_1$ and $d_n$, where $d_i$ is the dimension vector with $1$ at vertex $i$ and $0$ everywhere else. This means that:
$$[e_{1}, e_{n}]=0$$ in the KBPS Lie algebra. This means that we have a surjection $\mathfrak{sl}^{>}_{n+1}\to \text{KBPS}$. The two Lie algebras have the same dimension, so they are isomorphic. 
\end{proof}

\subsection{The three loop quiver.}\label{jordan}
For $Q$ the Jordan quiver, the tripled quiver construction is the pair 
$(\widetilde{Q}, \widetilde{W})$, where $\widetilde{Q}$ is the quiver with one vertex and three edges $x,y,$ and $z$, and the potential is
$\widetilde{W}=xyz-xzy$. Let $T=\C^*_{q_1}\times\C^*_{q_2}$ act by multiplication with $q_1$ on $x$, by $q_2$ on $y$, and by $q_1^{-1}q_2^{-1}$ on $z$.


The Feigin-Odeskii shuffle algebra is defined as:
$$SH=\bigoplus_{d\geq 0} K^{\C^*\times\C^*}_0(BGL(d)),$$ with multiplication for $f$ of dimension $d$ and $g$ of dimension $e$ given by:
$$(f g)(z_1,\cdots, z_{d+e})=\text{Sym}\left(fg\prod_{1\leq j\leq d, d+1\leq i\leq d+e} \zeta\left(\frac{z_i}{z_j}\right)\right),$$ where $\zeta$ is defined by: $$\zeta(z)=\frac{(1-q_1z)(1-q_2z)}{(1-z)(1-q_1q_2z)}.$$
This algebra is also known under other names in the literature, for example as $U^{>}_q(\widehat{\widehat{\mathfrak{gl}_1}})$ or as the (positive part of the) elliptic Hall algebra.

\begin{prop}
Let $\text{KHA}^{sph}_T(\widetilde{Q},\widetilde{W})\subset \text{KHA}_T(\widetilde{Q},\widetilde{W})$ be the subalgebra generated by the dimension $1$ part. Then there exists a natural isomorphism
$$\Phi: SH\to \text{KHA}^{sph}_T(\widetilde{Q},\widetilde{W}).$$ 
\end{prop}

\begin{proof}
First, by Proposition \ref{def}, we have a natural map: $$i_*: KHA_T(\widetilde{Q},\widetilde{W})\to KHA_T(\widetilde{Q},0).$$ For $d=1$ the map is an isomorphism, so we have a natural inclusion:
$$i_*: KHA^{\text{sph}}_T(\widetilde{Q},\widetilde{W})\hookrightarrow KHA_T(\widetilde{Q},0).$$

The multiplication in $KHA_T(\widetilde{Q},0)$ has the shuffle formula, for $f$ of degree $d$ and $g$ of degree $e$:
$$(f\,g)(z_1,\cdots, z_{d+e})=\text{Sym}
\left(fg\prod_{1\leq j\leq d, d+1\leq i\leq d+e}\tau\left(\frac{z_1}{z_j}\right)
\right),$$ where the factor $\tau$ is defined by:
$$\tau(z)=\frac{(1-zq^{-1}_1q^{-1}_2)(1-zq_1)
(1-zq_2)}{1-z}.$$
The ratio between the factors involved in the shuffle multiplications is:
$$\frac{\tau(z)}{\zeta(z)}=\left(1-zq_1^{-1}q_2^{-1}\right)\left(1-zq_1q_2\right).$$
Consider the algebra morphism
$\Phi: SH\to KHA_T(\widetilde{Q},0)$ defined by:
$$\Phi(f)=f\prod_{i\leq j\leq i\leq n}\left(1-q_1^{-1}q_2^{-1}\frac{z_i}{z_j}\right)\left(1-q_1q_2\frac{z_i}{z_j}\right).$$
In dimension $d=1$ it is an isomorphism, and $SH$ is generated by the dimension $1$ elements \cite[Theorem 2.5]{ne3}, so it identified $SH$ with the subalgebra of $KHA_T(\widetilde{Q},0)$ generated by the dimension $1$ part. By the above discussion, this is the same as $\text{KHA}^{\text{sph}}_T(\widetilde{Q},\widetilde{W})$. 

\end{proof}

\textbf{Remark.} The Lie algebra $\text{KBPS}(\widetilde{Q},\widetilde{W})=\bigoplus_{d>0}k$ has graded dimension $1$ for every $d\geq 1$ and trivial Lie bracket. 

Indeed, we have an isomorphism between the KBPS and BPS Lie algebras for $(\widetilde{Q},\widetilde{W})$. To show this, it is enough to show that the Chern character $\text{ch}: G_0(\mathcal{P}(d))\to H^{BM}(\mathcal{P}(d))$ is surjective, or that $\text{ch}: \text{gr}\,K_0(\MM(d))\to H^{\cdot}(X(d),\varphi_{\text{Tr}\,W}\mathbb{Q})$ is surjective. 
This follows from the following diagram constructed from Proposition \ref{framedsod} and dimensional reduction:

\begin{tikzcd}
K_0(\MM(d))\arrow{d}{\text{ch}}\arrow[r, hookrightarrow]& K_0(\text{Hilb}(\mathbb{A}^2,d))\arrow{d}{\text{iso}}\\
H^{\cdot}(X(d),\varphi_{\text{Tr}\,W}\mathbb{Q})\arrow[r, hookrightarrow]& H^{BM}(\text{Hilb}(\mathbb{A}^2,d)).
\end{tikzcd}
\\

The Lie algebra $\text{BPS}(\widetilde{Q},\widetilde{W})=\bigoplus_{d>0}k$ has graded dimension $1$ for every $d\geq 1$ and trivial Lie bracket. This follows from \cite[Section 5.1]{d2} for the vector space isomorphism, and the Lie bracket is trivial because the Lie bracket respect the cohomological grading and all the generators are in cohomological degree $-2$. 
\\

\textbf{Remark.} We expect the $KHA_T(\widetilde{Q},\widetilde{W})$ for $Q$ an affine type $\widehat{A}_n$ quiver to be isomorphic to the quantum toroidal algebras discussed in \cite{ne1}, possibly after tensoring with $\text{Frac}\,K_0(BT)$ where $T$ is a two dimensional torus defined similarly with the one from the current section.

\section{Representations of KHA}\label{7}

Let $(Q^f,W^f)$ be a quiver with potential, and let $\theta^f$ be a stability condition for $Q^f$. 
Let $Q\subset Q^f$ be a subquiver, let $W=W^f|_{Q}$ be the restriction of the potential to $Q$, and $\theta=\theta^f|_Q$ the restriction of the stability condition. Let $T$ be a torus as in Section \ref{2}.
Fix a vector $f\in\mathbb{N}^{I^f-I}$; we will call $f$ a framing vector. We use the notation $\X^{ss}(f,d)$ for the moduli stack of $\theta$-semistable representations of $Q^f$.

We denote by $\X^{ss}(d)$ the stack of $\theta$-semistable sheaves on $Q$, which is the same as the stack of $\theta$-semistable sheaves on $\X^{ss}(0,d)$. Consider the maps:

\begin{tikzcd}
\X^{ss}(f,d,e) \arrow{r}{p_{d,e}} \arrow{d}{q_{d,e}}& \X^{ss}(f,d+e) \\
\X^{ss}(d)\times\X^{ss}(f,e). & 
\end{tikzcd}

\begin{prop}\label{rep}
Let $f$ be a framing vector. Then $KHA_T(Q,W, \theta)$ acts on $$\bigoplus_{d\in\mathbb{N}^I} K^T_0(D_{sg}(\X^{ss}(f,d)_0))$$
via the maps: 
$$p_{d,e*}q_{d,e}^*: D^T_{sg}(\X^{ss}(d)_0)\boxtimes D^T_{sg}(\X^{ss}(f,e)_0)\to 
D^T_{sg}(\X^{ss}(f,d+e)_0).$$
\end{prop}

\begin{proof}
Let $c,d,e\in\mathbb{N}^I$ be dimension vectors. We need to check that the following diagram commutes:

\adjustbox{scale=0.95, right}{
\begin{tikzcd}
D^T_{sg}(\X^{ss}(d)_0)\boxtimes D^T_{sg}(\X^{ss}(e)_0)\boxtimes D^T_{sg}(\X^{ss}(f,c)_0) \arrow{r}{} \arrow{d} & D^T_{sg}(\X^{ss}(d)_0)\boxtimes D^T_{sg}(\X^{ss}(f,c)_0) \arrow{d}\\%
D^T_{sg}(\X^{ss}(d+e)_0)\boxtimes D^T_{sg}(\X^{ss}(f,c)_0) \arrow{r}{} & D^T_{sg}(\X^{ss}(f,c+d+e)_0).
\end{tikzcd}
}
\\
The diagram commutes from the associativity for $(Q^f,W^f)$ and the dimension vectors $(0,d), (0,e),$ and $(f,c)$.
\end{proof}

\textbf{Remark.} Using the algebra morphism $\text{KHA}(Q,W)\to \text{KHA}(Q,W,\theta)$, Proposition \ref{rep} constructs representations of $\text{KHA}(Q,W)$. 
\\

\textbf{Remark.} Consider the pair $(Q_3,W)$, with $Q_3$ the quiver with one vertex and three loops $x,y,z$ and potential $W=xyz-xzy$. Let $Q_3^f$ be the framed quiver where we add a new vertex and one edge to the vertex of $Q_3$ and potential $W$. Consider the stability condition $\theta$ from Subsection \ref{framed}. The critical locus of $\text{Tr}(W)$ on $\X^{ss}(1,d)$ is $\text{Hilb}(\C^3,d)$, and we use the notation $K_{\text{crit}}(\text{Hilb}(\C^3, n)):=K_0(D_{sg}(\X^{ss}(1,d)_0))$.
The above proposition constructs an action of $\text{KHA}_T(Q_3,W)$ on: $$\bigoplus_{d\geq 0} K_0(D_{sg}(\X^{ss}(1,d)_0))=\bigoplus_{d\geq 0}K_{\text{crit}}(\text{Hilb}(\C^3,d)).$$

We next apply this result to the geometry of Nakajima quiver varieties, following the proof as given in the cohomological case \cite[Section 6.3.]{d2}. First, we need to introduce new quivers. Given a quiver $Q=(I,E)$, and $f\in\mathbb{N}^I$, we denote by $Q^f$ the framed quiver with vertices $I\cup\{\infty\}$ and $f_i$ edges from $\infty$ to $i\in I$. Recall from Subsection \ref{tripledef} the construction of the double quiver $\overline{Q}$ and tripled quiver $(\widetilde{Q}, \widetilde{W})$. 
Let $\theta$ be a stability condition for $Q$, extended as in Subsection \ref{framed} to $Q^f$. Consider the map that forgets the action of $\omega_i$:
$$\pi: \X(\widetilde{Q^f},1,d)\to \X(\overline{Q^f},1,d)$$ the map that forgets the action of $\omega_i$.
The following result in proved in \cite[Lemma 6.5.]{d2}:
\begin{prop}\label{davi}[Davison]
The inclusion $\pi^{-1}\X^{\text{ss}}(\overline{Q^f},1,d)\hookrightarrow \X^{\text{ss}}(\widetilde{Q^f},1,d)$ induces an equality:
$$\pi^{-1}\X^{\text{ss}}(\overline{Q^f},1,d)\cap \text{crit}\,(\text{Tr}\,\widetilde{W^f})=\X^{\text{ss}}(\widetilde{Q^f},1,d)\cap \text{crit}\,(\text{Tr}\,\widetilde{W^f}).$$
\end{prop}

In particular, this means that
$$D_{sg}\left(\X^{\text{ss}}(\widetilde{Q^f},1,d)_0\right)= D_{sg}\left(\left(\pi^{-1}\X^{\text{ss}}(\overline{Q^f},1,d)\right)_0\right).$$
We use dimensional reduction, see Subsection \ref{dimred0}, \cite{i}, for the space
$$\pi^{-1}\X^{\text{ss}}(\overline{Q^f},1,d)=\X^{\text{ss}}(\overline{Q^f},1,d)\times\prod_{i\in I^f} \mathfrak{gl}(d_i)$$ and potential $\widetilde{W^f}$ with respect to the coordinates on $\prod_{i\in I^f} \mathfrak{gl}(d_i)$ to get that
$$K_0\left( D_{sg}\left(\left(\pi^{-1}\X^{\text{ss}}(\overline{Q^f},1,d)\right)_0\right)\right)=G_0(\mathcal{K}^{\text{ss}}(\overline{Q^f},1,d)).$$
Here $\mathcal{K}^{\text{ss}}(\overline{Q^f},1,d)$ is the (Koszul) stack obtained from dimensional reduction, where the relations are the moment map equations at the vertices $i\in I^f$. The moment map equations for $i\in I$ are the relations for the Nakajima quiver varieties $N^{\text{ss}}(f,d)$, and the moment map equation for $\infty$ is superfluous \cite[page 33]{d2}. This means that, for $\theta$ generic, 
$G_0(\mathcal{K}^{\text{ss}}(\overline{Q^f},1,d))=K_0(N^{\text{ss}}(f,d))$.

Let $\mathcal{P}(d)$ be the stack of representatiuons of the preprojective algebra  $\mathcal{P}=\mu^{-1}(0)/G(d)\subset \X(\overline{Q})$. 
Using dimensional reduction for $(\widetilde{Q},\widetilde{W})$, we have that:
$$\text{KHA}(\widetilde{Q},\widetilde{W})=\bigoplus_{d\in\mathbb{N}^I} G_0(\mathcal{P}(d)).$$ 
Proposition \ref{rep} implies the following:

\begin{thm}
There exists a natural action of $$KHA(\widetilde{Q},\widetilde{W})=\bigoplus_{d\in\mathbb{N}^I} G_0(\mathcal{P}(d)),$$ see Subsection \ref{qcuts}, on the $K_0$ of Nakajima quiver varieties: $$\bigoplus_{d\in\mathbb{N}^I}K_0(N^{ss}(f,d)),$$ for any framed vector $f\in\mathbb{N}^I$. 
\end{thm}


\textbf{Example.} Let $Q$ be the Jordan quiver. Then $(\widetilde{Q},\widetilde{W})=(Q_3,W)$, and we obtain an action of $\text{KHA}_T(Q_3,W)$ on
$\bigoplus_{d\geq 0} K^T_0(\text{Hilb}\,(\C^2,d)).$

\end{document}